\documentclass{article}
\usepackage[hmarginratio=1:1]{geometry}
\usepackage{titlesec}

\titleformat{\section}  
  {\fontsize{14}{16}\bfseries} 
  {\thesection} 
  {1em} 
  {} 
  [] 

\titleformat{\subsection}  
  {\fontsize{12}{14}\bfseries} 
  {\thesubsection} 
  {1em} 
  {} 
  [] 

\usepackage[utf8]{inputenc}
\usepackage[utf8]{inputenc}
\usepackage{tikz, tikz-cd}
\usepackage{url, hyperref}
\usepackage{amsmath, amsthm, amssymb, stmaryrd, xcolor, enumitem}
\usepackage{graphicx}
\usepackage{setspace}
\usepackage{fancyref}
\usepackage{hyperref}
\hypersetup{colorlinks,linkcolor={blue},citecolor={cyan},urlcolor={red}}
\usepackage{biblatex}
\usepackage{comment}
\usepackage{breqn}
\newtheorem{theorem}{Theorem}[section]
\usepackage{mathtools}
\usepackage{floatrow}
\usepackage{newunicodechar}

\theoremstyle{definition}
\newtheorem{definition}[theorem]{Definition}
\newtheorem{example}[theorem]{Example}
\newtheorem{remark}[theorem]{Remark}
\newtheorem{question}[theorem]{Question}

\newtheorem{assumption}[theorem]{Assumption}
\newtheorem{setup}[theorem]{Setup}
\newtheorem{keyremark}[theorem]{Key Remark}
\newtheorem{keyexample}[theorem]{Key Example}
\theoremstyle{plain}
\newtheorem{proposition}[theorem]{Proposition}
\newtheorem{corollary}[theorem]{Corollary}
\newtheorem{lemma}[theorem]{Lemma}
\newtheorem{conjecture}[theorem]{Conjecture}
\newtheorem{conjecture/theorem}[theorem]{Conjecture/Theorem}

\newtheorem{proposition/definition}[theorem]{Proposition/Definition}
\newtheorem{tentative definition}[theorem]{Tentative Definition}

\newenvironment{proofsketch}{%
  \proof}{\endproof}

\newenvironment{explanation of the claim}{%
  \proof}{\endproof}

\title{De Rham-Betti Groups of Type IV Abelian Varieties}
\author{Zekun Ji}
\date{}
\begin{document}
\newcommand\scalemath[2]{\scalebox{#1}{\mbox{\ensuremath{\displaystyle #2}}}}

\newcommand{\gdrb}{$\mathrm{G}_{\mathrm{dRB}}$ }
\newcommand{\gdrbmath}{\mathrm{G}_{\mathrm{dRB}}}
\newcommand{\gm}{$\mathbb{G}_{\mathrm{m}}$ }
\newcommand{\resgm}{$\mathrm{Res}_{E/Q}\mathbb{G}_{\mathrm{m}}$  }
\newcommand{\resgmmath}{\mathrm{Res}_{E/\mathbb{Q}}\mathbb{G}_{\mathrm{m}}}
\newcommand{\gmmath}{\mathbb{G}_{\mathrm{m}} }
\newcommand{\gdrbh}{$\mathrm{G}_{\mathrm{dRB}}^{\mathrm{h}}$ }
\newcommand{\gdrbhmath}{\mathrm{G}_{\mathrm{dRB}}^{\mathrm{h}} }

\newcommand{\absgalois}{\mathrm{Gal}(\overline{\mathbb{Q}}/\mathbb{Q})}
\newcommand{\galoisL}{\mathrm{Gal}(L/\mathbb{Q})}
\newcommand{\charUE}{X^{*}(\mathrm{U}_{E})}
\newcommand{\qbar}{\overline{\mathbb{Q}}}

\newcommand{\liegdrbmath}{\mathrm{g}_{\mathrm{dRB}}}
\newcommand{\liegdrbssmath}{\mathrm{g}_{\mathrm{dRB}}^{\mathrm{ss}}}
\newcommand{\liegdrbhmath}{\mathrm{g}_{\mathrm{dRB}}^{\mathrm{h}}}
\newcommand{\algebraicderham}{\mathcal{V}_{\mathrm{dR}}}
\newcommand{\localsystem}{\mathbb{V}_{\mathrm{B}}}

\newcommand{\flatholomorphic}{\mathcal{V}^{\mathrm{an}}_{\mathrm{B}}}
\newcommand{\holomorphicderham}{\mathcal{V}^{\mathrm{an}}_{\mathrm{dR}}}
\newcommand{\analyticu}{U^{\mathrm{an}}}
\newcommand{\familygdrb}{\mathrm{G}_{\mathrm{dRB}}^{Y \rightarrow U}}
\newcommand{\betti}{\mathrm{H}^1(A,\mathbb{Q})}
\newcommand{\derham}{\mathrm{H}^1_{\mathrm{dR}}(A/\overline{\mathbb{Q}})}
\newcommand{\sigmabar}{\overline{\sigma}}

\maketitle
\begin{abstract}
We study the de Rham-Betti structure of a simple abelian variety of type IV. We will take a Tannakian point of view inspired by Andr{\'e} (\cite{andre2004introduction}). The main results are that the de Rham-Betti groups of simple CM abelian fourfolds and simple abelian fourfolds defined over $\qbar$ whose endomorphism algebra is a degree 4 CM-field coincide with their Mumford-Tate groups. The method of proof involves a thorough investigation of the reductive subgroups of the Mumford-Tate groups of these abelian varieties, inspired by Kreutz-Shen-Vial (\cite{kreutz2023rhambetti}). The condition that the underlying abelian variety is simple and the condition that the de Rham-Betti group is a $\mathbb{Q}$-algebraic group are also used in a crucial way. The proof is different from the method of computing Mumford-Tate groups of these abelian varieties by Moonen-Zarhin in \cite{mz-4-folds}. We will also study a family of de Rham-Betti structures, in the formalism proposed by Saito-Terasoma in \cite{saito1997determinant}. For such families with geometric origin, we will characterize properties of fixed tensors of the de Rham-Betti group associated with such a family.
\end{abstract}
\tableofcontents

\section{Introduction}
Given $\{p_1,\dots,p_{n}\}$ an arbitrary set of $n$ numbers in $\mathbb{C}$, it is a fascinating problem to find $\qbar$-polynomial relations among $p_{i}$s i.e. determining polynomials $f$ in $n$ variables with coefficients in $\qbar$ such that $$f(p_1,\dots,p_{n})=0$$ or proving such nonzero $f$ does not exist. This question becomes more tractable if we require that $p_{i}$s are ``periods" on a fixed smooth projective variety $X$ defined over $\qbar$. Very roughly speaking, we require that $$p_{i}=\int_{\gamma}\omega$$ where $\gamma \in \mathrm{H}_{i}(X,\mathbb{Q})$ is the representative of a closed chain on $X(\mathbb{C})$ with $\mathbb{Q}$-coefficients and $\omega$ is a rational differential form on $X$ with $\qbar$-coefficients. 

To put things in a more precise setting, let $X$ be a smooth projective variety defined over $\qbar$. In the famous paper \cite{grothendieck1966rham}, Grothendieck established the following isomorphism $$\rho_{m}: \mathrm{H}^{n}_{\mathrm{B}}(X,\mathbb{Q})\otimes_{\mathbb{Q}} \mathbb{C} \cong \mathrm{H}^{n}_{\mathrm{dR}}(X/\qbar) \otimes_{\qbar} \mathbb{C}$$ which we call the Grothendieck comparison isomorphism. Then once we fix a $\mathbb{Q}$-basis for $\mathrm{H}^{n}_{\mathrm{B}}(X,\mathbb{Q})$ and a $\qbar$-basis for $\mathrm{H}^{n}_{\mathrm{dR}}(X/\qbar)$, we obtain a matrix whose entries we call \textit{periods} on $X$. Inspired by this setup, in \cite[Section 7.1.6]{andre2004introduction}, Andr{\'e} makes the following definition. 
\begin{definition}\label{introdef1}
A triple $$(V_{\mathrm{B}},V_{\mathrm{dR}},\rho_{m})$$ where $V_{\mathrm{B}}$ is a finite dimensional $\mathbb{Q}$-vector space,  $V_{\mathrm{dR}}$ is a finite dimensional $\qbar$-vector space and $\rho_{m}:V_{\mathrm{B}}\otimes_{\mathbb{Q}}\mathbb{C}\cong V_{\mathrm{dR}}\otimes_{\qbar}\mathbb{C}$ an isomorphism between $\mathbb{C}$-vector spaces is called a \textit{de Rham-Betti structure}. 
\end{definition}
Moreover, in \cite[Section 7.5]{andre2004introduction}, the notion of a de Rham-Betti class is introduced.
\begin{definition}
    Given a de Rham-Betti structure $(V_{\mathrm{B}},V_{\mathrm{dR}},\rho_{m})$, an element $\alpha\in V_{\mathrm{B}}$ is called a \textit{de Rham-Betti class} if $\rho_{m}(\alpha\otimes 1)\in V_{\mathrm{dR}}\otimes 1$.
\end{definition}
In \cite[Section 7.1.6]{andre2004introduction}, Andr{\'e} made the observation that the category of all such triples, denoted by $\mathcal{C}_{\mathrm{dRB}}$, has a canonical tensor category structure. Furthermore $\mathcal{C}_{\mathrm{dRB}}$ admits a natural forgetful functor $\omega$ to $\mathrm{Vec}_{\mathbb{Q}}$, the category of finite dimensional $\mathbb{Q}$-vector spaces. Then $\mathcal{C}_{\mathrm{dRB}}$ together with $\omega$ forms a neutral Tannakian category.

This formalism is closely related to the question posed in the beginning. For a single de Rham-Betti structure $V_{0}=(V_{\mathrm{B}},V_{\mathrm{dR}},\rho_{m})\in \mathrm{ob}\mathcal{C}_{\mathrm{dRB}}$, denote the Tannakian category generated by $V_{0}$ in $C_{\mathrm{dRB}}$ by $\langle V_{0}\rangle^{\otimes}$. Recall that objects in $\langle V_{0}\rangle^{\otimes}$ are subquotient de Rham-Betti structures of $\oplus_{n_{i},m_{i}\in \mathbb{Z}_{\geq 0}}V_{0}^{\otimes n_{i}} \otimes V_{0}^{*\otimes m_{i}}$. Given $W\in\mathrm{ob}\langle V_{0}\rangle^{\otimes}$, one can see that de Rham-Betti classes in $W$ indeed give $\qbar$-algebraic relations among entries of comparison isomorphism i.e. \textit{periods} associated with the de Rham-Betti structure $V_{0}$.

The abstract formalism by Andr{\'e} has the following powerful application. By Tannakian duality (see \cite[Theorem 2.11]{deligne2012tannakian} for example), de Rham-Betti classes in $\langle V_{0}\rangle^{\otimes}$ are 
precisely fixed tensors of the deRham-Betti group (see Definition \ref{gdrbdefn}) associated with $V_{0}$ which we denote by $\gdrbmath(V_{0})$. By construction it is a linear algebraic group defined over $\mathbb{Q}$. Moreover, it admits a canonical inclusion $\gdrbmath(V_{0})\xhookrightarrow{}\mathrm{GL}(V_{\mathrm{B}})$. Therefore the question of determining certain algebraic relations among transcendental numbers can now be described in the language of algebraic group theory. 
\begin{question}\label{coarsequestion}
We may pose the following two questions which turn out to be main theme of this article.

\begin{enumerate}
    \item\label{question1} Given a de Rham-Betti structure $V_{0}$ and limited amount of information about the $\qbar$-algebraic relations of the periods of $V_{0}$, can we already determine the algebraic group $\gdrbmath(V_{0})$?
    \item\label{question2}On the other hand, without knowing $\gdrbmath(V_{0})$ explicitly, suppose we can extract certain properties of $\gdrbmath(V_{0})$ as an algebraic group, what kind of information can we gather about $\qbar$-algebraicity and transcendence among periods of $V_0$?
\end{enumerate}   
\end{question}

We now present the following two results addressing Question \ref{question1} from \ref{coarsequestion}.

\begin{theorem}[=Theorem \ref{mainthmcm}]\label{introthm1}
  Suppose $A$ is a simple CM abelian fourfold. Then the de Rham-Betti group associated with the following dRB structure $$\mathrm{H}^1_{\mathrm{dRB}}(A,\mathbb{Q}):=(\betti,\mathrm{H}^1_{\mathrm{dR}}(A/\qbar),\rho_{m})$$ is equal to the Mumford-Tate group of $A$.
\end{theorem}
Recall that $\mathrm{MT}(A)$ is already computed in \cite[Theorem 7.6]{mz-4-folds}. Hence one can explicitly write down the de Rham-Betti groups in this setting.
\begin{theorem}[=Proposition \ref{weiltype} and Theorem \ref{deg4theorem}]\label{introthm2}
  Suppose $A$ is a simple abelian fourfold of type IV defined over $\qbar$(see Definition \ref{type4defn}), which is not of anti-Weil type (see Definition \ref{defnweiltype}). Then the de Rham-Betti group associated with $\mathrm{H}^1_{\mathrm{dRB}}(A,\mathbb{Q})$ is equal to the Mumford-Tate group of $A$.
\end{theorem}
Also in this case $\mathrm{MT}(A)$ is already computed in \cite[Section 7.4 and 7.5]{mz-4-folds}. The current state of the art concerning the de Rham-Betti group of an abelian variety is as follows. It is shown in \cite[Theorem 7.12]{kreutz2023rhambetti} that for an arbitrary elliptic curve $E$ defined over $\qbar$, we have that $\gdrbmath(E)=\mathrm{MT}(E)$. For higher dimensional abelian varieties, it is shown in \cite[Theorem 7.15]{kreutz2023rhambetti} that for an arbitrary simple abelian surface $A$ defined over $\qbar$ which satisfies that $\mathrm{End}^{\circ}(A)\neq\mathbb{Q}$, we have that $\gdrbmath(A)=\mathrm{MT}(A)$. Moreover, one can easily adapt the proof of \cite[Case (d) of Theorem 7.15]{kreutz2023rhambetti} to show that for a simple CM abelian threefold we have that $\gdrbmath(A)=\mathrm{MT}(A)$. 

The starting point of proving Theorem \ref{introthm1} and Theorem \ref{introthm2} actually already employs the philosophy of Question \ref{question2} of \ref{coarsequestion}. We show the following general property concerning the de Rham-Betti group of a simple type IV abelian variety.

\begin{proposition}[=Proposition \ref{gm in CM gdrb} and Proposition \ref{gmintype4}]\label{introprop1}
Given a simple type IV abelian variety $A$ defined over $\qbar$, we have a homomorphism of $\mathbb{Q}$-algebraic groups $\gmmath\xhookrightarrow{}\gdrbmath(A)$. Moreover, the image of $\gmmath$ in $\gdrbmath(A)$ is precisely the set of homotheties in $\mathrm{GL}(\betti)$.
\end{proposition}

We remark that simply from the definition of the de Rham-Betti group in the Tannakian formalism, there is no telling whether it contains $\gmmath$ as homotheties or not. The proof of the above proposition relies on the observation from \cite[Theorem 7.11]{kreutz2023rhambetti} that for an arbitrary abelian variety $A$ defined over $\qbar$, we have that \begin{equation}\label{introinclusion}
  \gdrbmath(A)\xhookrightarrow{}\mathrm{MT}(A)  
\end{equation} which is deduced from the equality $\mathrm{MT}(A)=G_{\mathrm{And}}(A)$ proven by Andr{\'e} in \cite{andre2004introduction}. Furthermore, properties of the Hodge group of $A$ and CM number fields are also used to prove Proposition \ref{introprop1}. Immediately from Proposition \ref{introprop1}, we also obtain the following two corollaries. Both of them address Question \ref{question2} of \ref{coarsequestion}.
\begin{corollary}
    Given a simple type IV abelian variety $A$ defined over $\qbar$, there is no de Rham-Betti class in $\mathrm{H}^{j}(A,\mathbb{Q})$ where $j$ is an odd number.
\end{corollary}

With some more work we can deduce
\begin{corollary}[=Corollary \ref{onecaseofconj}]
 Suppose $A$ is a simple abelian variety of type IV defined over $\qbar$. Moreover, suppose that the centre of $\mathrm{End}^{\circ}(A)$ is a CM-field of degree 2 or 4 and that the centre of the de Rham-Betti group of $A$ has dimension bigger than or equal to 2. Then the one-dimensional deRham-Betti structures in $\langle\mathrm{H}^1_{\mathrm{dRB}}(A,\mathbb{Q})\rangle^{\otimes}$ are of the form $\mathbb{Q}_{\mathrm{dRB}}(k)(k\in\mathbb{Z})$. 
\end{corollary}

See Definition \ref{drbweights} for the definition of the one-dimensional de Rham-Betti structure $\mathbb{Q}_{\mathrm{dRB}}(k)$. From Theorem \ref{introthm1} and Theorem \ref{introthm2}, one catches a glimpse of the relation between deRham-Betti theory and Hodge theory. However, simply from the definition, it is not so clear how these two are related. Nevertheless, we do have the following conjecture, whose origin can be traced back to Grothendieck.

\begin{conjecture}[\cite{andre2004introduction}, Conjecture 7.5.1.1]\label{conj1intro}
  Let $X$ be a smooth projective variety defined over $\qbar$. Then consider the following dRB structure $$(V_{\mathrm{B}},V_{\mathrm{dR}},\rho_{m}')=(\mathrm{H}^{2p}_{\mathrm{B}}(X,\mathbb{Q}),\mathrm{H}^{2p}_{\mathrm{dR}}(X/\qbar),\rho_{m}) \otimes \mathbb{Q}_{\mathrm{dRB}}(p)$$ Then every dRB class in $V_{\mathrm{B}}$ comes from a $\mathbb{Q}$-coefficient algebraic cycle on $X$ for $0\leq p \leq \mathrm{dim}(X)$.
\end{conjecture}

On the other hand, we can compare with the renowned Hodge conjecture.
\begin{conjecture}[Hodge Conjecture]\label{conjintro2}
 Let $X$ be a smooth projective variety defined over $\mathbb{C}$. Then every Hodge class in $ \mathrm{H}^{2p}(X,\mathbb{Q})\otimes \mathbb{Q}(p)$ for $0\leq p \leq \mathrm{dim}(X)$ comes from a $\mathbb{Q}$-coefficient algebraic cycle on $X$.
\end{conjecture}

Very roughly speaking, algebraic cycles and Andr{\'e}'s motivated cycles are the bridge between deRham-Betti theory and Hodge theory. Indeed this is already used to deduce formula (\ref{introinclusion}) by the authors of \cite{kreutz2023rhambetti}. We refer the reader to \cite[Section 6.2-6.3]{kreutz2023rhambetti} for a detailed exposition. However, compared with the Hodge conjecture, very little is known about Conjecture \ref{conj1intro}. Nevertheless, we do have the following theorem.

\begin{theorem}[\cite{bost2016some}, Theorem 3.8]\label{introbost1}
 For an abelian variety $A$ defined over $\qbar$, denote the dRB structure $(\betti,\derham,\rho_{m})$ by $\mathrm{H}^1_{\mathrm{dRB}}(A,\mathbb{Q})$. Then every dRB class on $\mathrm{H}^2_{\mathrm{dRB}}(A,\mathbb{Q})\otimes \mathbb{Q}_{\mathrm{dRB}}(1)$ comes from a $\mathbb{Q}$-coefficient algebraic cycle on $A$. Moreover, we have that $$\mathrm{End}_{\mathrm{dRB}}(\mathrm{H}^1_{\mathrm{dRB}}(A,\mathbb{Q}))\cong\mathrm{End}^{\circ}(A)$$
\end{theorem}

The above theorem about invariants of de Rham-Betti group at degree 2 cohomology groups is crucial to prove Theorem \ref{introthm1} and Theorem \ref{introthm2}, analogous to the method of computing Mumford-Tate groups of abelian varieties, where the Lefschetz (1,1) Theorem is used. Also crucial is the following observation first made by Andr{\'e} in \cite[Section 7.5.3]{andre2004introduction} and whose proof can be found in \cite[Theorem 7.4]{kreutz2023rhambetti}.
\begin{theorem}[Andr{\'e},Bost,W{\"u}stholz]\label{introreductiveness}
    Given an abelian variety $A$ defined over $\qbar$. Then the de Rham-Betti group associated with $\mathrm{H}^1_{\mathrm{dRB}}(A,\mathbb{Q})$ is a reductive algebraic group. 
\end{theorem}

To determine the de Rham-Betti group of simple abelian fourfolds, however, Theorem \ref{introbost1} and Theorem \ref{introreductiveness} still do not suffice. The major obstacle is that de Rham-Betti groups are only defined in the Tannakian formalism. Mumford-Tate groups, however, can be defined to be the smallest algebraic group over $\mathbb{Q}$ whose $\mathbb{R}$-points contain the image of the Deligne torus (see Definition \ref{delignetorusmt}). This formalism is used in a crucial way by authors of \cite{mz-4-folds} and 
\cite{ribet1983hodge} to compute Mumford-Tate groups of abelian fourfolds. We lack such formalism for de Rham-Betti groups.

Therefore for the proof of Theorem \ref{introthm1} and Theorem \ref{introthm2}, we perform a detailed analysis of reductive algebraic subgroups of $\mathrm{MT}(A)$. The properties that $A$ is a simple abelian variety and that $\gdrbmath(A)$ is defined over $\mathbb{Q}$ are exploited in a thorough manner. This is done in Section \ref{longcomp} and Section \ref{cmnonsimplicity} as well as Section \ref{deg4section}. This forms the most technical part of the article.

The method mentioned in the previous paragraph eventually fails for anti-Weil type abelian fourfolds. See Proposition \ref{existenceofsl2timessl2} and Remark \ref{existenceexplained} for the precise obstruction. Roughly speaking we construct a family of anti-Weil type abelian fourfolds whose Mumford-Tate groups properly contain a particular algebraic subgroup, which we do not know how to rule out as their de Rham-Betti groups. Therefore we are left with the following open question.
\begin{question}
  For a simple anti-Weil type abelian variety that falls in the family constructed in Proposition \ref{existenceofsl2timessl2}, how do we compute its de Rham-Betti group if it is defined over $\qbar$?
\end{question}

The last section of this article is devoted to 
the study of de Rham-Betti systems (Definition \ref{defndrbsystem}), which formalizes the notion of a family of de Rham-Betti structures parameterized
by a fixed quasi-projective smooth algebraic variety defined over $\qbar$, which we denote by $U$. This formalism is already proposed in \cite{saito1997determinant}. Given a de Rham-Betti system $\mathcal{V}$ on $U$, which satisfies certain geometric constraints (Assumption \ref{keyassumption}), we show 
\begin{proposition}[=Proposition \ref{mainprincipleB}]\label{thmintro3}
Deligne's ``Principle B" holds for de Rham-Betti systems.
\end{proposition}
Roughly speaking, let $\mathcal{V}=(\mathbb{V}_{\mathrm{B}},(\mathcal{V}_{\mathrm{dR}},\nabla),\rho_{m})$ be a deRham-Betti system satisfying Assumption \ref{keyassumption}. Then for a $\qbar$-point $u_0$ of $U$, we may define the restriction of $\mathcal{V}$ at $u_0$ (see Definition \ref{drbsystemrestriction}). It is a de Rham-Betti structure in the sense of Definition \ref{introdef1} and denote it by $\mathcal{V}|_{u_0}$. If an element $\alpha\in\mathbb{V}_{\mathrm{B},u_0}$ is monodromy invariant and it is a de Rham-Betti class with respect to $\mathcal{V}|_{u_0}$, then the parallel transform of $\alpha$ is a de Rham-Betti class above every $\qbar$-point in $U$. The proof of Proposition \ref{thmintro3} closely follows the proof for the ``Deligne's Principle B" in the setting of Hodge theory.

Inspired by the notion of generic Mumford-Tate group, we also give the notion of the de Rham-Betti group associated with a deRham-Betti system $\mathcal{V}$ (Definition \ref{drbsystemtannakian}). We denote this group by $\mathcal{G}_{\mathrm{dRB}}(\mathcal{V})$. It is defined in the abstract Tannakian formalism. Using Proposition \ref{thmintro3}, we have the following equivalent characterization
\begin{proposition}[=Proposition \ref{principleB}]
 Let $\mathcal{V}$ be a deRham-Betti system satisfying Assumption \ref{keyassumption}. Moreover, suppose that $\mathcal{G}_{\mathrm{dRB}}(\mathcal{V})$ is a reductive algebraic group. Then the fixed tensors of $\mathcal{G}_{\mathrm{dRB}}(\mathcal{V})$ are precisely monodromy invariant classes which are fiberwise de Rham-Betti classes. 
\end{proposition}

However, the following question remains open.
\begin{question}
    Given a de Rham-Betti system $\mathcal{V}$ satisfying the geometric constraint from Assumption \ref{keyassumption}, does there exist a $\qbar$-point such that $\mathcal{G}_{\mathrm{dRB}}(\mathcal{V})=\gdrbmath(\mathcal{V}|_{u_0})$?
\end{question}\noindent\textbf{Notations and Conventions}
\begin{enumerate}
    \item We fix an algebraic closure of $\mathbb{Q}$ and denote it by $\qbar$. We also fix an embedding of fields $\qbar\xhookrightarrow{}\mathbb{C}$.
    \item In this article, every algebraic variety is assumed to be defined over a field of characteristics zero.
    \item Given an abelian variety $A$, we denote by $\mathrm{End}^{\circ}(A)$ its endomorphism algebra $\mathrm{End}(A)\otimes_{\mathbb{Z}}\mathbb{Q}$. 
\end{enumerate}

\section*{Acknowledgment}
I would like to thank my advisor Mingmin Shen for introducing me to this topic and his guidance and encouragement along the way. I also would like to thank Javier Fresán, Tobias Kreutz, Charles Vial, Yichen Qin, Ziyang Gao and Mingyu Ni for useful discussions. The research was carried out as part of my PhD project and was supported by an NWO cluster grant with project number 613.009.153.

\section{Preliminaries}
\subsection{De Rham-Betti Structures and Groups}\label{drbsect}
We start by introducing the de Rham-Betti structures and de Rham-Betti groups. The main references are \cite{bost2016some} and \cite{kreutz2023rhambetti}.
\begin{definition}[Section 2.1.1, \cite{bost2016some}]\label{drb}
A \textit{de Rham-Betti structure} is a triple $$(V_{\mathrm{B}},V_{\mathrm{dR}},\rho_{m})$$ such that
\begin{enumerate}
    \item $V_{\mathrm{B}}$ is a finite dimensional $\mathbb{Q}$-vector space,
    \item $V_{\mathrm{dR}}$ is a finite dimensional $\qbar$-vector space, and
    \item $\rho_{m}$ is an isomorphism of $\mathbb{C}$-vector spaces $\rho_{m}: V_{\mathrm{B}} \otimes_{\mathbb{Q}}\mathbb{C} \cong V_{\mathrm{dR}} \otimes_{\qbar} \mathbb{C}$.
\end{enumerate}
In the sequel we sometimes abbreviate de Rham-Betti structure as dRB structure. 
\end{definition} 
One can define a morphism between two dRB structures.
\begin{definition}
    Given two dRB structures $$V=(V_{\mathrm{B}},V_{\mathrm{dR}},\rho_{m})$$ and $$V'=(V'_{\mathrm{B}},V'_{\mathrm{dR}},\rho'_{m})$$ a \textit{morphism $f$ of dRB structures} between $V$ and $V'$ consists of the following data
\begin{enumerate}
    \item $f_{\mathrm{B}}:V_{\mathrm{B}}\rightarrow V'_{\mathrm{B}}$ is a homomorphism between $\mathbb{Q}$-vector spaces.
    \item $f_{\mathrm{dR}}:V_{\mathrm{dR}}\rightarrow V'_{\mathrm{dR}}$ is a homomorphism between $\qbar$-vector spaces.
\end{enumerate}
Moreover, we require that $f_{\mathrm{B}}$ and $f_{\mathrm{dR}}$ make the following diagram commute. $$\begin{tikzcd}
V_{\mathrm{B}}\otimes\mathbb{C} \arrow[r, "\rho_{m}"] \arrow[d, "f_{\mathrm{B}}\otimes1"'] & V_{\mathrm{dR}}\otimes\mathbb{C} \arrow[d, "f_{\mathrm{dR}}\otimes1"] \\
V'_{\mathrm{B}}\otimes\mathbb{C} \arrow[r, "\rho'_{m}"]                                    & V'_{\mathrm{dR}}\otimes\mathbb{C}                                    
\end{tikzcd}$$
\end{definition}
\begin{definition}\label{drbclassdefinition}
Given a dRB structure $V$, a \textit{de Rham-Betti class} in $V$ is an element $\alpha \in V_{\mathrm{B}}$ such that $\rho_{m}(\alpha\otimes 1)\in V_{\mathrm{dR}}\otimes 1$. 
\end{definition}
\begin{example}
There are already ample examples of one-dimensional dRB structures. Fixing a $\mathbb{Q}$-basis for $V_{\mathrm{B}}$ and a $\qbar$-basis for $V_{\mathrm{dR}}$, the triple $(V_{\mathrm{B}},V_{\mathrm{dR}},\rho_{m})=(\mathbb{Q},\qbar,\rho_{m})$ where $\rho_{m}:\mathbb{C}\rightarrow \mathbb{C},z\rightarrow cz$ with $c \in \mathbb{C}^{\times}$ gives an example of a dRB structure.
\end{example}
In the above example, when $c=(2\pi i)^{k},k\in \mathbb{Z}$, such dRB structures are worthy of a proper and explicit definition below.
\begin{definition}\label{drbweights}
Given a one-dimensional dRB structure $(V_{\mathrm{B}},V_{\mathrm{dR}},\rho_{m})$, if we can find a basis for $V_{\mathrm{B}}$ and a basis for $V_{\mathrm{dR}}$ such that $\rho_{m}$ is the scalar multiplication by $(2\pi i)^{k}$ then we denote this dRB structure by $\mathbb{Q}_{\mathrm{dRB}}(k)$. It is clear that for a fixed $k\in\mathbb{Z}$ all such dRB structures are isomorphic. \end{definition}
Continuing with examples of dRB structures we have the following
\begin{example}
Let $X$ be a smooth projective variety defined over $\qbar$. Then the famous Grothendieck comparison isomorphism states that (See \cite{grothendieck1966rham})$$\rho_{m}: \mathrm{H}^{j}_{\mathrm{B}}(X,\mathbb{Q})\otimes_{\mathbb{Q}} \mathbb{C} \cong \mathrm{H}^{j}_{\mathrm{dR}}(X/\qbar) \otimes_{\qbar} \mathbb{C}$$ Therefore the triple $$(\mathrm{H}^{j}_{\mathrm{B}}(X,\mathbb{Q}),\mathrm{H}^{j}_{\mathrm{dR}}(X/\qbar),\rho_{m})$$ is a dRB structure which we abbreviate as $\mathrm{H}^{j}_{\mathrm{dRB}}(X,\mathbb{Q})$.
\end{example}
We can also define the tensor product of two dRB structures.
\begin{definition}
Given two dRB structures $$V=(V_{\mathrm{B}},V_{\mathrm{dR}},\rho_{m})$$ and $$V'=(V'_{\mathrm{B}},V'_{\mathrm{dR}},\rho'_{m})$$ their tensor product $V\otimes V'$ is given by $$(V_{\mathrm{B}}\otimes V'_{\mathrm{B}}, V_{\mathrm{dR}}\otimes V'_{\mathrm{dR}}, \rho_{m}\otimes \rho'_{m})$$
\end{definition}
\begin{example}
    We have that $\mathbb{Q}_{\mathrm{dRB}}(k)=\mathbb{Q}_{\mathrm{dRB}}(1)^{\otimes k}$. For an abelian variety $A$ defined over $\qbar$, we have that $$\mathrm{H}^{2}_{\mathrm{dRB}}(A,\mathbb{Q})\cong \wedge^2\mathrm{H}^{1}_{\mathrm{dRB}}(A,\mathbb{Q})$$
\end{example}
\begin{example}
For $X$ a smooth projective variety defined over $\qbar$, consider the de Rham-Betti structure $$V:=\mathrm{H}^{2j}_{\mathrm{dRB}}(X,\mathbb{Q})\otimes \mathbb{Q}_{\mathrm{dRB}}(j)$$ If $\alpha \in \mathrm{cl}_{\mathrm{B}}(Z^{k}_{\mathbb{Q}}(X)) \subset \mathrm{H}^{2j}_{\mathrm{B}}(X,\mathbb{Q})$, then $\alpha$ is a de Rham-Betti class with respect to the dRB structure $V$. See \cite[Section 1.1.2]{bost2016some} for a detailed explanation of this and also \cite[Page 6]{charles2011notes} for an explanation about the compatibility of the Betti cycle class map and the de Rham cycle class map under the Grothendieck comparison isomorphism, up to a power of $2\pi i$.
\end{example}

\begin{remark}\label{uncountablymany}
There are uncountably many isomorphism classes of one-dimensional dRB structures. Suppose $V_1=(\mathbb{Q},\qbar,\rho_{m})$ where $\rho_{m}:\mathbb{C}\rightarrow \mathbb{C},z\rightarrow c_1z$ with $c_1 \in \mathbb{C}^{\times}$ and $V_2=(\mathbb{Q},\qbar,\rho'_{m})$ where $\rho'_{m}:\mathbb{C}\rightarrow \mathbb{C},z\rightarrow c_2z$ with $c_2 \in \mathbb{C}^{\times}$ Then $V_1$ and $V_2$ are isomorphic as de Rham-Betti structures if and only if $c_1=qc_2$ with $q\in \qbar^{\times}$. Since $\mathbb{C}$ is an uncountable set and $\qbar$ is a countable set, we deduce that there are uncountably many isomorphism classes of one-dimensional dRB structures. Compared with the setting of Hodge structures, it is clear we have only countably many isomorphism classes of one-dimensional Hodge structures.
\end{remark}

We also have a natural notion of the dual of a dRB structure.
\begin{definition}
    Given a dRB structure $V=(V_{\mathrm{B}},V_{\mathrm{dR}},\rho_{m})$, its dual dRB structure $V^{*}$ is defined by $$(\mathrm{Hom}_{\mathbb{Q}}(V_{\mathrm{B}},\mathbb{Q}),\mathrm{Hom}_{\qbar}(V_{\mathrm{dR}},\qbar),(\rho_{m}^{*})^{-1})$$
\end{definition}
\begin{example}
    Given an abelian variety $A$ defined over $\qbar$, the de Rham-Betti structure of $A^{\vee}$ satisfies $$\mathrm{H}^1_{\mathrm{dRB}}(A^{\vee},\mathbb{Q})\otimes\mathbb{Q}_{\mathrm{dRB}}(1)\cong\mathrm{H}^1_{\mathrm{dRB}}(A,\mathbb{Q})^{*}$$ See Section 3.3 of \cite{bost2016some}.
\end{example}

In \cite[Section 7.1.6]{andre2004introduction}, Andr\'e proposes to study the category of all dRB structures $C_{\mathrm{dRB}}$ in the Tannakian formalism. Recall from above that $C_{\mathrm{dRB}}$ has a natural tensor structure. Then denote by $\omega_{\mathrm{B}}$ the forgetful functor from $C_{\mathrm{dRB}}$ to the category of $\mathbb{Q}$-vector spaces which remembers only the Betti part of the dRB structure. The main property of $C_{\mathrm{dRB}}$ is summarized in the following proposition from \cite{kreutz2023rhambetti}.
\begin{proposition} [Section 4, \cite{kreutz2023rhambetti}]
The category of all dRB structures $C_{\mathrm{dRB}}$ together with the forgetful functor $\omega_{\mathrm{B}}$ to the category of $\mathbb{Q}$-vector spaces forms a neutral Tannakian category.   
\end{proposition}
\begin{remark}
One can easily write down the identity objects in $C_{\mathrm{dRB}}$. They are one-dimensional de Rham-Betti structures whose comparison isomorphism is given by the scalar multiplication by an element in $\qbar^{\times}$.
\end{remark}
Given a dRB structure $V$, we denote the Tannakian category generated by $V$ in $C_{\mathrm{dRB}}$ by $\langle V\rangle^{\otimes}$. Objects in $\langle V\rangle^{\otimes}$ are subquotient de Rham-Betti structures of $\oplus_{n_{i},m_{i}\in \mathbb{Z}_{\geq0}}V^{\otimes n_{i}} \otimes V^{*\otimes m_{i}}$. We now define the de Rham-Betti group of $V$ via the Tannakian formalism. 
\begin{definition}\label{drbgdefn}
The \textit{de Rham-Betti group} of $V$ is defined as $\underline{\mathrm{Aut}}(\omega_{\mathrm{B}}|_{\langle V\rangle^{\otimes}})$ (see \cite[Page 20]{deligne2012tannakian} for a definition). We denote it by $\gdrbmath(V)$. It is an algebraic subgroup of $\mathrm{GL}(V_{\mathrm{B}})$.
\end{definition}

By the Tannakian duality (see \cite[Theorem 2.11]{deligne2012tannakian}), we have an equivalence of tensor categories between $\langle V\rangle^{\otimes}$ and $\mathrm{Rep}_{\mathbb{Q}}(\gdrbmath(V))$. Therefore we have the following equivalent definition of the de Rham-Betti group of $V$.
\begin{definition}\label{gdrbdefn}
The \textit{de Rham-Betti group} of $V$ is the subgroup of $\mathrm{GL}(V_{\mathrm{B}})$ which fixes all the de Rham-Betti classes in $\langle V\rangle^{\otimes}$. A de Rham-Betti class in the category $\langle V\rangle^{\otimes}$ is an element of one of the identity objects in $\langle V\rangle^{\otimes}$.
\end{definition}
The de Rham-Betti group $\gdrbmath(V)$ of a de Rham-Betti structure $V$ is closely related to the torsor of periods associated with $V$. We do not recall the complete setup but refer the reader to \cite[Chapter 4]{kreutz2023rhambetti} or \cite[Section 2]{bost2016some} for more details. Nevertheless, we need the following two results that come out of this formalism.

\begin{lemma}[Theorem 4.7, \cite{kreutz2023rhambetti}]\label{gdrbconnectedness}
For an arbitrary dRB structure $V$, we have that $\gdrbmath(V)$ is a connected algebraic group. 
\end{lemma}
\begin{lemma}[Chapter 5, \cite{kreutz2023rhambetti}]\label{drbgroupdim}
The dimension of $\gdrbmath(V)$ as an algebraic group is greater than or equal to $\mathrm{trdeg}_{\qbar}\qbar(\alpha_{i,j})$, where $\alpha_{i,j}$s are the coefficients of the matrix of comparison isomorphism $\rho_{m}$ with respect to a $\mathbb{Q}$-basis for $V_{\mathrm{B}}$ and a $\qbar$-basis for $V_{\mathrm{dR}}$.
\end{lemma}

\begin{example}
In dimension one, if $V=(\mathbb{Q},\qbar,\rho_{m}=c\in\qbar^{\times})$, then $\gdrbmath(V)=\{\mathrm{id}\}$. If $V=(\mathbb{Q},\qbar,\rho_{m}=c)$ with $c$ being a transcendental number, then $\gdrbmath(V)=\gmmath$. 
\end{example}

Finally we recall one version of the famous Grothendieck period conjecture, which is one of the motivations of the theory of de Rham-Betti structures.
\begin{conjecture}[1.1.3, \cite{bost2016some}]\label{wgpc}
  Let $X$ be a smooth projective variety defined over $\qbar$. Then consider the dRB structure $$(V_{\mathrm{B}},V_{\mathrm{dR}},\rho_{m}')=\mathrm{H}^{2j}_{\mathrm{dRB}}(X,\mathbb{Q}) \otimes \mathbb{Q}_{\mathrm{dRB}}(j)$$ for some integer $0\leq j\leq \mathrm{dim}(X)$. If an element $v \in V_{\mathrm{B}}$ satisfies that $\rho'_{m}(v\otimes 1)\in V_{\mathrm{dR}} \otimes 1$ then we call such $v$ a dRB class on $X$. The conjecture states that every dRB class on $X$ comes from a $\mathbb{Q}$-coefficient algebraic cycle on $X$. 
\end{conjecture}
\begin{remark}
The above is called the Grothendieck Period conjecture in \cite{bost2016some}. However, there is an even stronger conjecture about periods of smooth algebraic varieties defined over $\qbar$, which would imply the above conjecture. See \cite[Conjecture 6.6]{kreutz2023rhambetti} for the precise statement and a detailed exposition. This stronger conjecture is called the Grothendieck period conjecture by Andr\'e (\cite[Conjecture 7.5.2.1]{andre2004introduction}) and in \cite{kreutz2023rhambetti}. In the sequel, we will refer to Conjecture \ref{wgpc} as the weak Grothendieck period conjecture. Already very little is known about Conjecture \ref{wgpc}. For example, it is not yet known whether it holds for $\mathrm{H}^{2}_{\mathrm{dRB}}(X,\mathbb{Q}) \otimes \mathbb{Q}_{\mathrm{dRB}}(1)$ when $X$ is an arbitrary smooth projective variety defined over $\qbar$. 
\end{remark}

However, the weak Grothendieck period conjecture is known to hold in the following case, which will be used throughout the entire article. The proof of the following theorem relies on a deep result from transcendental number theory called the ``Analytic Subgroup Theorem" proven by W{\"u}stholz in \cite{wustholz1989algebraische} (see also \cite{bost2016some} for a detailed exposition).

\begin{theorem}[Theorem 3.8, \cite{bost2016some}]\label{bostoriginal}
 For an abelian variety $A$ defined over $\qbar$,  every dRB class in $\mathrm{H}^2_{\mathrm{dRB}}(A,\mathbb{Q})\otimes \mathbb{Q}_{\mathrm{dRB}}(1)$ comes from a $\mathbb{Q}$-coefficient algebraic cycle on $A$. Moreover, we have that $\mathrm{End}(\mathrm{H}^1_{\mathrm{dRB}}(A,\mathbb{Q}))=\mathrm{End}^{\circ}(A)$.
\end{theorem}

It is well known that if $A$ is a simple abelian variety, then the Hodge structure $\betti$ is also simple. Using the analytic subgroup theorem, the authors of \cite{kreutz2023rhambetti} show that the following is true.

\begin{theorem}[Theorem 7.11, \cite{kreutz2023rhambetti}]\label{reductive}
  If $A$ is a simple abelian variety defined over $\qbar$, then the de Rham-Betti structure $\mathrm{H}^1_{\mathrm{dRB}}(A,\mathbb{Q})$ is simple. Moreover, for any abelian variety $A$ defined over $\qbar$, we have that $\gdrbmath(A)$ is a reductive algebraic group defined over $\mathbb{Q}$.
\end{theorem}

\subsection{Mumford-Tate Groups}
The Mumford-Tate group is the fundamental symmetry group associated with a Hodge structure. In this section, we collect some well known facts about the Mumford-Tate group. We add the proof for certain facts for usage in later sections. The main reference is \cite{moonen2004introduction}.

Given a $\mathbb{Q}$-Hodge structure $V$ of weight $m$, i.e. $$V \otimes_{\mathbb{Q}} \mathbb{C}=\oplus_{p+q=m}V^{p,q}$$ such that $$\overline{V^{p,q}}=V^{q,p}$$ we will give three equivalent definitions of its Mumford-Tate group (Definition \ref{delignetorusmt}, Definition \ref{tannakianformalism1} and Definition \ref{tannakianformalism2}).

We first associate $V$ with an $\mathbb{R}$-representation of the Deligne torus \begin{equation}\label{delignetorus}\mu:\mathbb{S}:=\mathrm{Res}_{\mathbb{C}/\mathbb{R}}\gmmath \rightarrow \mathrm{GL}(V_{\mathbb{R}})\end{equation} The representation $\mu$ is defined as follows. Given $z \in \mathbb{S}(\mathbb{R})=\mathbb{C}^{\times}$, its image under $\mu$ acts via scalar multiplication by $z^{-p}\overline{z}^{-q}$ on the summand $V^{p,q}$. One can check that the image $\mu(z)$ sends $V_{\mathbb{R}}$ to $V_{\mathbb{R}}$.

\begin{definition}\label{delignetorusmt}
The \textit{Mumford-Tate group} of the $\mathbb{Q}$-Hodge structure $V$ is the smallest $\mathbb{Q}$-algebraic subgroup of $\mathrm{GL}(V)$ whose set of $\mathbb{R}$-points contains the image of $\mu$. And we denote it by $\mathrm{MT}(V)$.
\end{definition}
\begin{remark}
    One can alternatively define $\mathrm{MT}(V)$ as smallest $\mathbb{Q}$-algebraic subgroup of $\mathrm{GL}(V)$ whose set of $\mathbb{C}$-points contains the image of $\mu_{\mathbb{C}}$. More explicitly the image of $\mu_{\mathbb{C}}$ is the set of diagonal matrices which acts as scalar multiplication by $z_1^{-p}z_2^{-q}$ on the Hodge summand $V^{p,q}$ where $z_1,z_2\in\mathbb{C}^{\times}$.
\end{remark}
Immediately from this definition, we have the following observation about Mumford-Tate groups.
\begin{lemma}\label{existenceofhomothety}
    We have the following commutative diagram of $\mathbb{Q}$-algebraic groups $$\begin{tikzcd}
\gmmath \arrow[r, "i_{w}"] \arrow[rd] & \mathrm{MT}(V) \arrow[d, hook] \\
                                      & \mathrm{GL}(V)                
\end{tikzcd}$$ For  $t \in \gmmath(\mathbb{Q})=\mathbb{Q}^{\times}$, $i_{w}(t)$ is equal to the scalar multiplication by $t^{-m}$ on $V$.
\end{lemma}
\begin{corollary}\label{existenceofhomothetycor}
Given any Hodge structure $V$ which is not of weight 0, its Mumford-Tate group contains the homotheties in $\mathrm{GL}(V)$ i.e. diagonal matrices whose entries along the diagonal are the same. In other words, we have a homomorphism of $\mathbb{Q}$-algebraic groups $i:\gmmath\xhookrightarrow{}\mathrm{MT}(V)$ which sends $t\in\gmmath$ to the scalar multiplication by $t^{-1}$ in $\mathrm{GL}(V)$.
\end{corollary}
\begin{proof}
The scheme theoretic image of $i_{w}$ is precisely the homotheties in $\mathrm{GL}(V)$. Since $\mathrm{MT}(V)$ is an algebraic group, we deduce that $\mathrm{MT}(V)$ contains the homotheties. The map $i$ defined in the lemma then indeed factors through $\mathrm{MT}(V)$.
\end{proof}

\begin{remark}
We stress that the theory of de Rham-Betti structures does not have a formalism similar to Definition \ref{delignetorusmt}. One reason is the lack of a notion of weights in the theory of de Rham-Betti structures. 
\end{remark}

One can also define the Mumford-Tate group of $V$ in the Tannakian formalism. By \cite{deligne2012tannakian}, the category of $\mathbb{Q}$-Hodge structures together with the forgetful functor $\omega$, which remembers the underlying $\mathbb{Q}$-vector space, forms a neutral Tannakian category. Within it, we can consider the Tannakian subcategory generated by $V$, which we denote by $\langle V\rangle^{\otimes}$. Recall that the objects in $\langle V\rangle^{\otimes}$ are subquotient Hodge structures of $\oplus_{n_{i},m_{i}\in \mathbb{Z}_{\geq0}}V^{\otimes n_{i}} \otimes V^{*\otimes m_{i}}$.
\begin{definition}\label{tannakianformalism1}
    The Mumford-Tate group of a weight $m$ $\mathbb{Q}$-Hodge structure $V$ is precisely $\underline{\mathrm{Aut}}(\omega|_{\langle V\rangle^{\otimes}})$ (See \cite[Page 20]{deligne2012tannakian} for a definition), which is a $\mathbb{Q}$-algebraic subgroup of $\mathrm{GL}(V)$.
\end{definition} 
Immediately from the Tannakian formalism, we can give the third equivalent definition of Mumford-Tate group.
\begin{definition}\label{tannakianformalism2}
 The Mumford-Tate group of a weight $m$ $\mathbb{Q}$-Hodge structure $V$ is precisely the group in $\mathrm{GL}(V)$ which fixes the Hodge classes in  $\langle V\rangle^{\otimes}$.
\end{definition}

Given an abelian variety $A$, its first singular cohomology group $\mathrm{H}^1(A,\mathbb{Q})$ can be endowed with a weight 1 Hodge structure. Then we define
\begin{definition}
    $\mathrm{MT}(A):=\mathrm{MT}(\mathrm{H}^1(A,\mathbb{Q}))$.
\end{definition}

Using the formalism from Definition \ref{delignetorusmt}, we can also define the Hodge group $\mathrm{Hdg}(V)$ associated with the Hodge structure $V$. Denote by $\mathbb{U}$ the $\mathbb{R}$-subtorus of $\mathbb{S}$ whose set of $\mathbb{R}$-points is $\{z \in \mathbb{C}^{\times}|z\overline{z}=1\}$.
\begin{definition}\label{hodgegroupdefn}
    The \textit{Hodge group} of $V$, denoted by $\mathrm{Hdg}(V)$, is the smallest $\mathbb{Q}$-algebraic subgroup of $\mathrm{GL}(V)$ whose set of $\mathbb{R}$-points contains the image of $\mathbb{U}$ under $\mu$ (see formula (\ref{delignetorus})). 
\end{definition}
\begin{remark}\label{hdgroupcdefn}
    One can also define $\mathrm{Hdg}(V)$ as the smallest $\mathbb{Q}$-algebraic subgroup of $\mathrm{GL}(V)$ whose set of $\mathbb{C}$-points contains the image of $\mathbb{U}(\mathbb{C})$ under $\mu_{\mathbb{C}}$.
\end{remark}
Immediately from the definition, we have the following convenient lemma.
\begin{lemma}\label{hodgetimesgm}
Let $V$ be a Hodge structure of weight $k$ where $k\neq0$. The Hodge group $\mathrm{Hdg}(V)$ is a connected algebraic subgroup of $\mathrm{MT}(V)$. Moreover, there is an isogeny $\Psi: \gmmath \times \mathrm{Hdg}(V) \rightarrow \mathrm{MT}(V)$ which takes $(t,g) \in \gmmath \times \mathrm{Hdg}(V)$ to $t^{-1}g \in \mathrm{MT}(V)$.
\end{lemma}
\begin{proof}
It follows from the definition that $\mathrm{Hdg}(V)$ is a connected algebraic subgroup of $\mathrm{MT}(V)$. By Corollary \ref{existenceofhomothetycor}, we also have the inclusion of $\gmmath$ inside $\mathrm{MT}(V)$ as homotheties. Therefore we have a well-defined homomorphism of algebraic groups $\Psi: \gmmath \times \mathrm{Hdg}(V) \rightarrow \mathrm{MT}(V)$ which sends $(t,g) \in \gmmath \times \mathrm{Hdg}(V)$ to $t^{-1}g \in \mathrm{MT}(V)$. Now we claim that $\mathrm{Im}(\Psi_{\mathbb{C}})$ contains the image of the $\mathbb{C}$-points of the Deligne torus under $\mu_{\mathbb{C}}$. Denote the diagonal matrix in  $\mathrm{Im}(\mu_{\mathbb{C}})$ which acts as scalar multiplication by $z_1^{-p}z_2^{-q}$ on $V^{p,q}$s by $(z_1^{-p}z_2^{-q}|_{V^{p,q}})$. Note that $p+q=k$ and $z_1,z_2\in\mathbb{C}^{\times}$. Fix two square roots for $z_1$ and $z_2$ and denote them as $z_1^{1/2}$ and $z_2^{1/2}$ respectively. Let $m=z_1^{1/2}z_2^{-1/2}$. Then one can check that $(z_1z_2)^{-\frac{k}{2}}(m^{-p}m^{q}|_{V^{p,q}})=(z_1^{-p}z_2^{-q}|_{V^{p,q}})$. Note that $(z_1z_2)^{\frac{k}{2}}\in\gmmath(\mathbb{C})$ and $(m^{-p}m^{q}|_{V^{p,q}})\in \mu_{\mathbb{C}}(\mathbb{U}(\mathbb{C}))\subset\mathrm{Hdg}(V)(\mathbb{C})$. By the definition of the Mumford-Tate group (see Definition \ref{delignetorusmt}), we can conclude that $\mathrm{Im}(\Psi)=\mathrm{MT}(V)$. The kernel of $\Psi$ is a proper $\mathbb{Q}$-algebraic subgroup of $\gmmath$, hence is a finite group.
\end{proof}
Within the tensor category $\langle V\rangle^{\otimes}$, we can describe the fixed tensors of $\mathrm{Hdg}(V)$ as follows.
\begin{lemma}\label{hodgefixedtensors}
   The fixed tensors of $\mathrm{Hdg}(V)$ in $\langle V\rangle^{\otimes}$ are precisely the elements in one-dimensional objects in $\langle V\rangle^{\otimes}$. These are precisely the so called $(p,p)$-Hodge classes for $p \in \mathbb{Z}$.
\end{lemma}

We end this section with the following Theorem from \cite{bost2016some} and \cite{kreutz2023rhambetti} relating the de Rham-Betti group and Mumford-Tate group of an abelian variety defined over $\qbar$. For an abelian variety $A$, it is proven by Andr{\'e} in \cite{andrearticle} that every Hodge tensor in $\langle \betti\rangle^{\otimes}$ is motivated and in fact $G_{\mathrm{And}}(A)=\mathrm{MT}(A)$. By definition, every motivated tensor is a de Rham-Betti tensor. Therefore every Hodge tensor is a de Rham-Betti tensor. Since $\gdrbmath(A)$ and $\mathrm{MT}(A)$ are reductive algebraic subgroups of $\mathrm{GL}(\betti)$, using Definition \ref{tannakianformalism2}, we can deduce the following.
\begin{theorem}[Theorem 7.11, \cite{kreutz2023rhambetti}]\label{keyinclusion}
    For an abelian variety $A$ defined over $\qbar$, we have the following commutative diagram $$\begin{tikzcd}
\gdrbmath(A) \arrow[r, hook] \arrow[rd, hook] & \mathrm{MT}(A) \arrow[d, hook]            \\
                                              & {\mathrm{GL}(\mathrm{H}^1(A,\mathbb{Q}))}
\end{tikzcd}$$
\end{theorem}
For abelian varieties defined over $\qbar$, the authors of \cite{bost2016some} proved that the Weak Grothendieck Period Conjecture holds for degree 2 cohomology groups.

\begin{theorem}[Theorem 3.8, \cite{bost2016some}]\label{mainfact}
Given an abelian variety $A$ defined over $\qbar$
\begin{enumerate}
\item The following chain of equalities holds $$\mathrm{End}_{\mathrm{dRB}}(\betti)=\mathrm{End}_{\mathrm{Hdg}}(\betti)=\mathrm{End}^{\circ}(A)$$
\item The Lefschetz (1,1) theorem holds true in this setting: every dRB class in $\mathrm{H}^2_{\mathrm{dRB}}(A,\mathbb{Q})\otimes\mathbb{Q}_{\mathrm{dRB}}(1)$ is the image of some $\mathbb{Q}$-coefficient algebraic cycle under the cycle class map.
\end{enumerate}
\end{theorem}

\subsection{Algebraic Tori}\label{algtori}
In this section we summarize some basic properties of algebraic tori. The main reference is \cite[Chapter 14]{milne2014algebraic}. 
\begin{definition}
An algebraic group $T$ defined over $\mathbb{Q}$ is a (connected) \textit{$\mathbb{Q}$-algebraic torus} if $T_{\qbar}$ is isomorphic to a finite product of $\mathbb{G}_{m,\qbar}$.
\end{definition}
It is essential to study the character group of a $\mathbb{Q}$-algebraic torus, which we now define.
\begin{definition}\label{chardefn}
  Given an algebraic torus $T$, we define its \textit{character group} to be $\mathrm{Hom}(T_{\qbar},\mathbb{G}_{m,\qbar})$. We denote it by $X^{*}(T)$. It is a finite rank free $\mathbb{Z}$-module.  
\end{definition}
\begin{remark}
By definition, we have $T_{\qbar}\cong {\displaystyle \prod_{i=1}^{r} \mathbb{G}_{m,\qbar}}$. Fixing a coordinate $t_{i}$ for each $\mathbb{G}_{m,\qbar}$ in the product, then an element $f\in \mathrm{Hom}(T_{\qbar},\mathbb{G}_{m,\qbar})$ is of the form $f(t)={\displaystyle \prod_{i=1}^{r}t_{i}^{n_{i}}}$ with $n_{i} \in \mathbb{Z}$. The rank of $X^{*}(T)$ is $r$. 

\end{remark}

The finite rank $\mathbb{Z}$-module $X^{*}(T)$ contains more information. Namely, it admits an action of the absolute Galois group $\mathrm{Gal}(\qbar/\mathbb{Q})$. The action can be described as follows. Because both $\gmmath$ and $T$ are defined over $\mathbb{Q}$, an element $g \in \mathrm{Gal}(\qbar/\mathbb{Q})$ induces morphisms of algebraic groups $g_{T}: T_{\qbar} \rightarrow T_{\qbar}$ and $g_{0}: \mathbb{G}_{m,\qbar} \rightarrow \mathbb{G}_{m,\qbar}$.  Then for a character $f \in X^{*}(T)$, $g\circ f$ is equal to the composition $g_{0} \circ f \circ g_{T}^{-1}\in X^{*}(T)$. By \cite[Chapter 14.f]{milne2014algebraic} this is a continuous group action on the left where we equip $\mathrm{Gal}(\qbar/\mathbb{Q})$ with the Krull topology and $X^{*}(T)$ with the discrete topology. 

\begin{theorem}[Theorem 14.17, \cite{milne2014algebraic}] \label{chartorus}
 Taking the character group $X^{*}$ defines a contravariant functor from the category of $\mathbb{Q}$-algebraic tori to the category of finite rank free $\mathbb{Z}$-modules equipped with a continuous $\mathrm{Gal}(\overline{\mathbb{Q}}/\mathbb{Q})$-action. This is an equivalence of categories. Moreover, under this equivalence, exact sequences of algebraic tori correspond to exact sequences of $\mathrm{Gal}(\overline{\mathbb{Q}}/\mathbb{Q})$-modules.
     
\end{theorem}
An essential example of a non-split $\mathbb{Q}$-algebraic torus is $T=\mathrm{Res}_{E/\mathbb{Q}}\gmmath$, where $E$ is a number field. We give a description of the $\mathrm{Gal}(\qbar/\mathbb{Q})$ structure of $X^{*}(T)$. More details can be found in \cite[Chapter 14]{milne2014algebraic}.

\begin{example} \label{firstexample}
Recall that $T=\mathrm{Res}_{E/\mathbb{Q}}\gmmath$ is a $\mathbb{Q}$-algebraic torus such that for a $\mathbb{Q}$-algebra $R$, we have that $\mathrm{Res}_{E/\mathbb{Q}}\gmmath(R)=(E\otimes_{\mathbb{Q}}R)^{*}$. In particular, we have that $$\mathrm{Res}_{E/\mathbb{Q}}\gmmath(\qbar)=\prod_{\mathrm{Hom}(E,\overline{\mathbb{Q}})}\qbar^{\times}$$
Then one can identify $X^{*}(T)$ with $\mathbb{Z}^{\mathrm{Hom}(E,\overline{\mathbb{Q}})}$, which has a $\mathbb{Z}$-basis labeled by the set of embeddings of $E$ into $\overline{\mathbb{Q}}$. Under this identification, an element $$(n_{\sigma})_{\sigma\in \mathrm{Hom}(E,\qbar)}\in \mathbb{Z}^{\mathrm{Hom}(E,\overline{\mathbb{Q}})}$$ is sent to $$f \in \mathrm{Hom}(\resgmmath(\qbar),\mathbb{G}_{m}(\qbar)):(t_{\sigma})_{\sigma\in \mathrm{Hom}(E,\qbar)} \rightarrow \prod_{\sigma\in \mathrm{Hom}(E,\qbar)}t_{\sigma}^{n_{\sigma}}$$ Note that $\absgalois$
acts on the left on $\mathrm{Hom}(E,\overline{\mathbb{Q}})$, and hence this induces a left group action on $X^{*}(\resgmmath)$ by permuting the labeled basis for $$X^{*}(\resgmmath)\cong\mathbb{Z}^{\mathrm{Hom}(E,\overline{\mathbb{Q}})}$$ Let $L$ be the Galois closure of $E/\mathbb{Q}$ in $\overline{\mathbb{Q}}$.
We then note that $T_{L}$ already splits as a product of $\mathbb{G}_{\mathrm{m},L}$. Hence the action of $\absgalois$ on  $X^{*}(\resgmmath)$ in fact descends to the action of $\galoisL$ on $\mathbb{Z}^{\mathrm{Hom}(E,L)}$. 
\end{example}
Now suppose $E$ is a CM-field (see Definition \ref{CMdefinitions}). Another important example of a $\mathbb{Q}$-algebraic torus is $\mathrm{U}_E$ which can be viewed naturally a algebraic subtorus of $\resgmmath$.
\begin{example}\label{uexample}
Suppose $E$ is a CM-field, then $\mathrm{U}_E$ is an algebraic torus defined over $\mathbb{Q}$ such that for a $\mathbb{Q}$-algebra $R$, we have $$\mathrm{U}_E(R)=\{x \in (E\otimes_{\mathbb{Q}}R)^{*}|x\overline{x}=1\}$$ where the complex conjugation on $E$ extends $\mathbb{Q}$-linearly to $E\otimes_{\mathbb{Q}}R$. Then $$\mathrm{U}_E(\qbar)=\{(z_{\sigma})_{\sigma \in \mathrm{Hom}(E,\qbar)}\in\prod_{\mathrm{Hom}(E,\overline{\mathbb{Q}})}\qbar^{\times}|z_{\sigma}z_{\overline{\sigma}}=1\} \subset \resgmmath(\qbar)$$ 
Using the explicit description of the Galois action on $X^{*}(\resgmmath)$ from Example \ref{firstexample}, one can see that for any element $$f \in X^{*}(\resgmmath)=\mathrm{Hom}(\resgmmath(\qbar),\mathbb{G}_{m}(\qbar))$$ and for any element $x\in \mathrm{U}_{E}(\qbar)$ we have that $$f(x)=\overline{f}(x)^{-1}$$ where the complex conjugation on $f$ is induced by the action of $\absgalois$ on $X^{*}(\resgmmath)$. Now we denote by $M$ the $\mathbb{Z}$-submodule of $X^{*}(\resgmmath)$ generated by elements of the form $m+\overline{m}, m\in X^{*}(\resgmmath)$. Because $E$ is a CM-field, by Lemma \ref{comjugationcommuting}, the action of the complex conjugation in $\absgalois$ on $\mathrm{Hom}(E,\qbar)$ commutes with the action of every element of $\absgalois$ on $\mathrm{Hom}(E,\qbar)$. Therefore we have that $M$ is in fact a $\absgalois$ submodule. One can then show that
$X^{*}(\mathrm{U}_E)$ is precisely the quotient of $X^{*}(\resgmmath)$ by $M$. In particular the complex conjugation in $\absgalois$ acts as scalar multiplication by -1 on $X^{*}(\mathrm{U}_E)$.   \end{example}

\subsection{Simple Abelian Varieties with Complex Multiplication}\label{cmintrosection}
In this section we give a brief introduction to simple abelian varieties with complex multiplication. The main reference is \cite[Chapter 1.1-1.2]{milne2006complex}. Along the way, we collect some well-known properties of the Hodge structure associated with an abelian variety whose endomorphism algebra contains a number field.

We start by recalling some basic and well known properties of a CM-field.
\begin{definition} [Proposition 1.1.4, \cite{milne2006complex}]\label{CMdefinitions}
A number field $E$ is a \textit{CM-field} if one of the following three equivalent definitions holds true
\begin{enumerate}
    \item $E$ is a totally imaginary degree two extension of a totally real subfield $F$ of $E$ i.e. there is no embedding of $E$ into $\mathbb{R}$ but every embedding from $F$ into $\mathbb{C}$ lies inside $\mathbb{R}$
    \item There exists a unique nontrivial field automorphism $\tau \in \mathrm{Aut}(E/\mathbb{Q})$ such that for any embedding $i: E \rightarrow \mathbb{C}$ and any $e\in E$, we have $\overline{i(e)}=i(\tau e)$
    \item $E=F[x]/(x^2-\alpha)$ where $F$ is a totally real number field and $\alpha$ is an element of $F$ such that for every embedding $i: F \xhookrightarrow{} \mathbb{R}$, $i(\alpha)<0$.
\end{enumerate} We will call $\tau$ the complex conjugation on $E$.
\end{definition}

\begin{proof}
See the proof of \cite[Proposition 1.1.4]{milne2006complex} for a proof that the above three definitions are equivalent to each other.  
\end{proof}

\begin{lemma} \label{comjugationcommuting}
    For a CM-field $E$, its complex conjugation commutes with any other field automorphism of $E$.
\end{lemma}

\begin{proof}
According to the third definition above $E=F[x]/(x^2-\alpha)$ with $F$ a totally real number field. We first prove that any field automorphism $\sigma$ of $E$ preserves the totally real subfield $F:=E^{\tau}$ and sends $x$ to $fx$ for some $f \in F-\{0\}$. Suppose $\sigma \in \mathrm{Aut}(E/\mathbb{Q})$ sends an element $f\in F-\{0\}$ to $f'+f''x \in E-F$ where $f'' \neq 0$. Then for any embedding $$i: E \xhookrightarrow{} \mathbb{C}$$ precomposing it with $\sigma$ we obtain another embedding $$i'=i \circ \sigma: E \xhookrightarrow{} \mathbb{C}$$ Then we have that $i'(f)=i(f')+i(f'')i(x) \in \mathbb{C}$. Now $i(f'), i(f'') \in \mathbb{R}$ because $F$ is a totally real field and $i(x)$ is a purely imaginary number in $\mathbb{C}$ because $i(\alpha)<0$. Therefore $i'(f) \in \mathbb{C}-\mathbb{R}$, which gives a contradiction.

Suppose now $x$ is sent to $g'+g''x \in E$ by $\sigma$ where $g',g''\in F$. Then note that $\sigma(x)^2=g^{'2}+g^{''2}\alpha+2g'g''x=\sigma(\alpha) \in F$, hence either $g'=0$ or $g''=0$. Suppose $g''=0$, i.e. $\sigma(x)=g'$. Then for any embedding $i: E \xhookrightarrow{} \mathbb{C}$, precomposing it with $\sigma$ we obtain another embedding $i'=i \circ \sigma: E \xhookrightarrow{} \mathbb{C}$. But then for any element $e=f_1+f_2x \in E$, we have $i'(e)=i(\sigma(f_1))+i(\sigma(f_2))g' \in \mathbb{R}$, which contradicts the condition that $E$ does not admit an embedding into $\mathbb{R}$. Hence $\sigma(x)=g''x$ for some $g'' \in F$.

Denote by $\tau$ the complex conjugation on $E$. For any field automorphism $\sigma \in \mathrm{Aut}(E/\mathbb{Q})$ and any element $e=f_1+f_2 x$ in $E$, we have $\tau \circ \sigma(e)=\tau(\sigma(f_1)+\sigma(f_2)g''x)=\sigma(f_1)-\sigma(f_2)g''x$ while $\sigma \circ \tau(e)=\sigma(f_1-f_2x)=\sigma(f_1)-\sigma(f_2)g''x$. Hence $\sigma \circ \tau=\tau \circ \sigma$, and therefore the complex conjugation on $E$ commute with every field automorphism of $E$.
\end{proof}

\begin{corollary}\label{galoisclosureofCMisCM}
    The splitting closure $L$ of a CM-field $E/\mathbb{Q}$ in $\qbar$ is also a CM-field and moreover the totally real subfield $L^{\tau}$ is also a Galois extension over $\mathbb{Q}$. 
\end{corollary}

\begin{proof}
    See \cite[Corollary 1.1.5]{milne2006complex} for a proof of the first statement. For the second one, note that $F=L^{\tau}$ but $\{\mathrm{id},\tau\}$ is a normal subgroup of $G=\mathrm{Gal}(L/\mathbb{Q})$ because $\tau$ commutes with every element of $G$.
\end{proof}

\begin{corollary} \label{subofcm}
    A subfield of a CM-field is either a CM-field or a totally real field.
\end{corollary}
\begin{proof}
    By Corollary \ref{galoisclosureofCMisCM}, it suffices to show any subfield $K$ of a Galois CM-field $L$ is either totally real or CM. Since the complex conjugation on $L$ commute with every element of $\mathrm{Gal}(L/\mathbb{Q})$ by Lemma \ref{comjugationcommuting}, and $K$ is of the form $L^{H}$ for some subgroup $H$ of $\mathrm{Gal}(L/\mathbb{Q})$, the complex conjugation $\tau$ on $L$ preserves the field $K$. Then we claim that $\tau$ restricted to $K$ is either a complex conjugation on $K$ or is the identity on $K$. Note that given any embedding $i':K \rightarrow \mathbb{C}$, it factors through an embedding $i: L \rightarrow \mathbb{C}$ because $L/K$ is a Galois extension. Hence for any $k \in K \subset L$, we have $\overline{i'(k)}=\overline{i(k)}=i(\tau(k))=i'(\tau|_{K}(k))$. Therefore if $\tau|_{K}$ is not the identity, then $K$ is a CM-field according to Definition \ref{CMdefinitions}. If $\tau|_{K}$ is equal to the identity, then $K$ lies inside the totally real Galois subfield $L^{\tau}$. This implies that $K$ is totally real since any embedding of $K$ into $\mathbb{C}$ factors through the embedding of $L^{\tau}$ into $\mathbb{C}$.
\end{proof}
We now recall the definition of a simple CM abelian variety. 

\begin{definition}\label{cmavdefinition}
Given a simple abelian variety $A$ and a CM-field $E$ with $\mathrm{deg}(E/\mathbb{Q})=2\mathrm{dim}(A)$, we say $A$ is a \textit{simple abelian variety with complex multiplication by $E$} or a \textit{simple CM abelian variety} if $E \cong\mathrm{End}^{\circ}(A)$ as $\mathbb{Q}$-algebras.
\end{definition}

The datum of a simple CM abelian variety $A$ with complex multiplication $E$ induces a morphism of $\mathbb{Q}$-algebras
\begin{equation*}
    i: E \xhookrightarrow{} \mathrm{End}(\mathrm{H}^1(A,\mathbb{Q}))
\end{equation*} and the image of $i$ preserves the weight 1 Hodge structure on $\mathrm{H}^1(A,\mathbb{Q})$. This gives rise to the notion of CM-type (see Definition \ref{cmtypedefn}) on $\mathrm{Hom}(E,\mathbb{C})$. To explain how this comes about, we will put things in a more general setting for usage in later sections. In Setup \ref{firstkeysetup} we will explore properties of the Hodge structure associated with an abelian variety whose endomorphism algebra contains a number field.
\begin{setup}\label{firstkeysetup}
Let $A$ be an abelian variety and $E$ be a number field. Denote the Galois closure $E/\mathbb{Q}$ in $\qbar$ by $L$. We denote $\mathrm{H}^1(A,\mathbb{Q})$ by $V$.    

Suppose we have a morphism of $\mathbb{Q}$-algebras
\begin{equation*}
    i: E \xhookrightarrow{} \mathrm{End}_{\mathrm{Hdg}}(V)
\end{equation*} 
Then we have the following eigenspace decomposition induced by $i$.
\begin{lemma}\label{neweigenspacetranslate}
We keep the same notations as above. Then we have \begin{equation}\label{eigendecomp}V\otimes_{\mathbb{Q}} L=\bigoplus_{\sigma \in \mathrm{Hom}(E,L)}V_{\sigma}\end{equation} where $V_{\sigma}$'s are $L$-vector spaces such that for any element $v \in V_{\sigma}$, we have $e \circ v=\sigma(e)v$. Moreover, $V_{\sigma}$s have dimensions equal to $\frac{\mathrm{dim}(V)}{\mathrm{deg}(E/\mathbb{Q})}$.
\end{lemma}
\begin{proof}
We first explain how to obtain the decomposition (\ref{eigendecomp}). The field structure on $E$ gives the following isomorphism of \textit{$\mathbb{Q}$-algebras}
\begin{equation}
    E\otimes_{\mathbb{Q}}L\cong \prod_{\sigma\in\mathrm{Hom}(E,L)}L
\end{equation} where a basic tensor $e\otimes l\in E\otimes_{\mathbb{Q}}L$ is sent to $\prod_{\sigma\in\mathrm{Hom}(E,L)}l(\sigma(e))$. Then the $L$-linear extension of $i$ induces a morphism of $L$-\textit{algebras}
$$i_{L}: \prod_{\sigma\in\mathrm{Hom}(E,L)}L\rightarrow \mathrm{End}(V\otimes_{\mathbb{Q}} L)$$ In particular, this induces the desired decomposition (\ref{eigendecomp}).

As for the second statement, note that $\mathrm{Gal}(L/\mathbb{Q})$ acts $\mathbb{Q}$-linearly on $V\otimes_{\mathbb{Q}} L$ by acting on the extended scalars. Given $g\in\mathrm{Gal}(L/\mathbb{Q})$, for any $v \in V_{\sigma}$ and $e \in E$ we have \begin{equation}\label{galoistranslateeq}
i(e) \circ (g \circ v)=g \circ (i(e) \circ v)=g \circ (\sigma(e)v)=(g \circ \sigma(e))(g\circ v)    
\end{equation} where the first equality holds because $i(e)$ is a linear transformation with $\mathbb{Q}$-coefficients. Therefore formula (\ref{galoistranslateeq}) implies that the $\mathbb{Q}$-linear transformation $g \in \mathrm{Gal}(L/\mathbb{Q})$ on $V \otimes_{\mathbb{Q}} L$ maps $V_{\sigma}$ to $V_{g \circ \sigma}$ and henceforth $g^{-1} \in \mathrm{Gal}(L/\mathbb{Q})$ maps $V_{g \circ \sigma}$ to $V_{\sigma}$. Therefore $V_{\sigma}$ and $V_{g \circ \sigma}$ are isomorphic as $\mathbb{Q}$-vector spaces. Hence they also have the same dimension as $L$-vector spaces. Since the action of $\mathrm{Gal}(L/\mathbb{Q})$ on $\mathrm{Hom}(E,L)$ is transitive and $V_{L}$ contains at least one nonzero eigenspace $V_{\sigma_0}$, the eigenspaces $V_{\sigma}$s have the same dimension. 
\end{proof}
Now we explore how the above eigenspace decomposition interacts with the Hodge decomposition on $V$. 
\begin{lemma}\label{qbarhodge}
 We keep the same notations as above. Each $V_{\sigma}$ is stable under the action of the $\mathrm{MT}(A)(\qbar)$.
\end{lemma}
\begin{proof}
 Since $E$ preserves the Hodge structure on $V$, the image of each element $e\in E$ under $i$ in $\mathrm{End}(V)$ commutes with $\mathrm{MT}(A)(\qbar)$. For any $v\in V_{\sigma}$ and any $g\in\mathrm{MT}(A)(\qbar)$ we therefore have $i(e)\circ(g\circ v)=g\circ (i(e)\circ v)=g\circ (\sigma(e)v)=\sigma(e)(g\circ v)$. Hence $g\circ v\in V_{\sigma}$.
\end{proof}
The following lemma states that ``weight one $\qbar$-Hodge structure'' admits a decomposition as expected.
\begin{lemma}\label{qbarhodgedecomp}
Suppose $V$ is a weight one Hodge structure. Let $W$ be a $\qbar$-vector subspace of $V\otimes{\qbar}$ and furthermore suppose $W$ is stable under the action of $\mathrm{MT}(V)(\qbar)$. Then $W_{\mathbb{C}}:=W\otimes_{\qbar}\mathbb{C}$ has the following direct sum decomposition
    \begin{equation*}
    W_{\mathbb{C}}= W_{\mathbb{C}} \cap  V^{0,1} \oplus W_{\mathbb{C}} \cap  V^{1,0}
\end{equation*} 
\end{lemma}
\begin{proof}
 Since $W$ is stable under the action of $\mathrm{MT}(V)(\qbar)$, its $\mathbb{C}$-linear extension $W_{\mathbb{C}}$ is stable under $\mathrm{MT}(V)(\mathbb{C})$. Therefore  $W_{\mathbb{C}}$ is stable under the image of the $\mathbb{C}$-points of the Deligne torus under the homomorphism (\ref{delignetorus}) $$\mu_{\mathbb{C}}:\mathrm{Res}_{\mathbb{C}/\mathbb{R}}\gmmath(\mathbb{C})\rightarrow\mathrm{MT}(V)(\mathbb{C})$$ Given any $w\in W_{\mathbb{C}}$ we have the Hodge decomposition $w=w^{1,0}+w^{0,1}$ where $w^{1,0}\in V^{1,0}$ and $w^{0,1}\in V^{0,1}$. Then for any $(z_1,z_2)\in \mathrm{Res}_{\mathbb{C}/\mathbb{R}}\gmmath(\mathbb{C})=\mathbb{C}^{*}\times\mathbb{C}^{*}$ we have $$\mu_{\mathbb{C}}(z_1,z_2)\circ w=z_1^{-1}w^{1,0}+z_2^{-1}w^{0,1}\in W_{\mathbb{C}}$$ Hence $z_1^{-1}w^{1,0}+z_2^{-1}w^{0,1}-w^{1,0}-w^{0,1}\in W_{\mathbb{C}}$. Swapping $(z_1,z_2)=(\frac{1}{2},1)$ and $(z_1,z_2)=(1,\frac{1}{2})$ in this formula, we obtain that $w^{1,0}\in W_{\mathbb{C}}\cap V^{1,0}$ and $w^{0,1}\in W_{\mathbb{C}}\cap V^{0,1}$. 
\end{proof}
\begin{remark}
    We warn the reader that there is no reason for $W_{\mathbb{C}}\cap V^{1,0}$ and $W_{\mathbb{C}}\cap V^{0,1}$ to have the same dimension.
\end{remark}
\begin{definition}
\label{multiplicitydefn}
Back to the setup of Lemma \ref{neweigenspacetranslate} and Lemma \ref{qbarhodge} i.e. $K$ a number field together with a morphism of $\mathbb{Q}$-algebras $$i: K \xhookrightarrow{} \mathrm{End}_{\mathrm{Hdg}}(V)$$ 
   we define the \textit{multiplicity} $m_{\sigma}$ of an embedding $$\sigma: K \rightarrow \qbar$$ to be the dimension of the $\mathbb{C}$-vector space $V_{\sigma}^{1,0}=V_{\sigma}\otimes_{\qbar}\mathbb{C}\cap V^{1,0}$.
\end{definition}

\begin{remark} \label{conjugatedecomp}
    If $K$ is a CM-field, we have $\overline{V_{\sigma}\otimes_{\qbar}\mathbb{C}}=V_{\sigmabar}\otimes_{\qbar}\mathbb{C}$, where the complex conjugation is the canonical one on $V\otimes_{\mathbb{Q}}\mathbb{C}$. Hence we have $\overline{V_{\sigma}^{0,1}}=V_{\overline{\sigma}}^{1,0}$, which implies that $m_{\sigma}+m_{\overline{\sigma}}=\frac{\mathrm{dim}(V)}{\mathrm{deg}(K/\mathbb{Q})}$.
\end{remark}

\end{setup}
We now apply the discoveries made in Setup \ref{firstkeysetup} to a simple CM abelian variety $A$ i.e. $\mathrm{End}^{\circ}(A)=E$ with $\mathrm{deg}(E/\mathbb{Q})=\mathrm{dim}(V)$ and $E$ is a CM-field.
\begin{lemma}\label{famouscmtype}
We assume the setup of Definition \ref{cmavdefinition}. For any $\sigma \in \mathrm{Hom}(E,\qbar)$, $V_{\sigma} \otimes_{\qbar} \mathbb{C}$ lies either in $\mathrm{H}^{1,0}(A)$ or in $\mathrm{H}^{0,1}(A)$. Moreover, if $V_{\sigma}\otimes_{\qbar} \mathbb{C} \subset \mathrm{H}^{1,0}(A)$, then $V_{\overline{\sigma}}\otimes_{\qbar} \mathbb{C} \in \mathrm{H}^{0,1}(A)$ and vice versa.
\end{lemma}
\begin{proof}
For the first statement, note that By Lemma \ref{neweigenspacetranslate}, we have $\mathrm{dim}V_{\sigma}=1$ for any $\sigma\in\mathrm{Hom}(E,\qbar)$. We can then directly apply Lemma \ref{qbarhodgedecomp} to conclude. The second statement follows from Remark \ref{conjugatedecomp}.
\end{proof}
Inspired by the above lemma, given a simple abelian variety $A$ with complex multiplication by $E$, we can define the following function on $\mathrm{Hom}(E,\qbar)$
\begin{equation*}
  \Phi_{A}: \mathrm{Hom}(E,\qbar) \rightarrow \{0,1\}
\end{equation*}
\begin{equation*}
\Phi_{A}(\sigma)=
\begin{cases}
    1,& \text{if } V_{\sigma} \subset \mathrm{H}^{1,0}\\
    0,              & \text{if } V_{\sigma} \subset \mathrm{H}^{0,1}
\end{cases}
\end{equation*}
Inspired by the geometric scenario, we have the following definition.
\begin{definition}\label{cmtypedefn}
  For any CM-field $E$, a \textit{CM-type} on $\mathrm{Hom}(E,\qbar)$ is a function \begin{equation*}
  \Phi: \mathrm{Hom}(E,\qbar) \rightarrow \{0,1\}
\end{equation*} such that $\Phi(\sigma)+\Phi(\overline{\sigma})=1$.
\end{definition}
\begin{remark}
The famous notion of the CM-type associated with a CM abelian variety is a special case of the notion of the \textit{multiplicity} from Definition \ref{multiplicitydefn}. 
\end{remark}
We end this section with a general lemma about the Hodge group of a Hodge structure admitting a CM-field of largest degree in its ring of endomorphism. The $\mathrm{H}^1(A,\mathbb{Q})$ of a CM abelian variety $A$ is one example of such Hodge structures.
\begin{lemma} \label{generalizedcm}
    Given a $\mathbb{Q}$-Hodge structure $W$ of weight $k$ whose Hodge endomorphism algebra contains a CM-field $E$ with $\mathrm{dim}_{\mathbb{Q}}(W)=\mathrm{deg}(E/\mathbb{Q})$, then the Hodge group of $W$ is contained in the $\mathbb{Q}$-algebraic torus $\mathrm{U}_{E}$.
\end{lemma}
\begin{proof}
We first show that there is generalized notion of CM-type on $\mathrm{Hom}(E,\mathbb{C})$. Tensoring the eigenspace decomposition from Lemma \ref{neweigenspacetranslate} by $\mathbb{C}$ we obtain $$W \otimes_{\mathbb{Q}} \mathbb{C}=\oplus_{\sigma \in \mathrm{Hom}(E,\mathbb{C})} W_{\sigma}$$ where each $W_{\sigma}$ is one-dimensional. Bringing the Hodge decomposition into the picture, for any $w_{\sigma}\in W_{\sigma}$ we have $w_{\sigma}=\Sigma_{p+q=k}w^{p,q}_{\sigma}$ where $w^{p,q}_{\sigma}\in W^{p,q}$. Then for any $e\in E$ we have that $i(e) \circ w_{\sigma}=\Sigma_{p+q=k}i(e) \circ w^{p,q}_{\sigma}=\Sigma_{p+q=k}\sigma(e)w^{p,q}_{\sigma}$ for $w_{\sigma}\in W_{\sigma}$. Note that the image of $E$ preserves the Hodge structure, we have $i(e)\circ w^{p,q}_{\sigma}\in W^{p,q}$. Since the Hodge decomposition $W\otimes{\mathbb{C}}=\oplus_{p+q=k} W^{p,q}$ is a direct sum, we have $i(e) \circ w^{p,q}_{\sigma}=\sigma(e) \circ w^{p,q}_{\sigma}$. Hence $w^{p,q}_{\sigma} \in W_{\sigma}$. But $W_{\sigma}$ is one-dimensional, hence there is a unique tuple $(p,q)$ such that $W_{\sigma} \subset W^{p,q}$. Denote the $(p,q)$ associated with $\sigma$ by $(p_{\sigma},q_{\sigma})$. Moreover, since $E$ is a CM-field, we also have $W_{\overline{\sigma}} \subset W^{q,p}$.

Hence with respect to the same basis given by the eigenspace decomposition, the image of $\mathbb{U}(\mathbb{C})$ under the group homomorphism (see Definition \ref{hodgegroupdefn}) $$\mu_{\mathbb{C}}: \mathbb{U}(\mathbb{C}) \xhookrightarrow{} \mathrm{Res}_{\mathbb{C}/\mathbb{R}}\mathbb{G}_{m}(\mathbb{C}) \rightarrow \mathrm{GL}(W \otimes_{\mathbb{Q}}\mathbb{C})$$ is equal to the set of following diagonal matrices
$$\left\{\begin{pmatrix}
    z_1^{-p_{\sigma_1}}z_2^{-q_{\sigma_1}} & & \\
    & \ddots & \\
    & & z_1^{-p_{\sigma_n}}z_2^{-q_{\sigma_{n}}}
  \end{pmatrix}|z_1,z_2 \in \mathbb{C}, z_1z_2=1,(p_{\sigma_{i}},q_{\sigma_{i}})=(q_{\overline{\sigma_{i}}},p_{\overline{\sigma_{i}}})\right\}
$$ where $n$ is the dimension of $W$. Moreover, the image of $\mathrm{U}_{E}(\mathbb{C})$ inside $\mathrm{GL}(W \otimes_{\mathbb{Q}} \mathbb{C})$ is equal to the following set of diagonal matrices 
$$\left\{\begin{pmatrix}
    g_{\sigma_1} & & \\
    & \ddots & \\
    & & g_{\sigma_{n}}
  \end{pmatrix}|g_{\sigma_{i}} \in \mathbb{C}, g_{\sigma_{i}}g_{\overline{\sigma_{i}}}=1\right\}
$$ Hence the image of $\mathbb{U}(\mathbb{C})$ inside $\mathrm{GL}(W \otimes_{\mathbb{Q}}\mathbb{C})$ is contained in the image of $\mathrm{U}_{E}(\mathbb{C})$. Recall by Remark \ref{hdgroupcdefn}, the Hodge group of $W$ is the smallest $\mathbb{Q}$-algebraic subgroup of $\mathrm{GL}(W)$ whose set of $\mathbb{C}$-points contains the image of $\mathbb{U}(\mathbb{C})$ under $\mu_{\mathbb{C}}$, and $\mathrm{U}_{E}$ is an algebraic group defined over $\mathbb{Q}$, we have that the Hodge group of $W$ is contained in $\mathrm{U}_{E}$.
\end{proof}

\subsection{Weil Structures}\label{weilstruc}
In this section we are going to recall a well known construction in linear algebra made by Weil (for example see \cite{Weil1979AbelianVA}). See Section 5.5 of \cite{charles2011notes} or Lemma 12 of \cite{moonen1998weil} for a detailed discussion. In order to suit our purpose better, we present the construction of Weil structure in a slightly different but equivalent way to the standard sources. Lemmas in this section unless specified are well known but their proofs are recorded for the sake of completeness or the usefulness in later sections.

This section should be seen as a natural continuation of Setup \ref{firstkeysetup}. Recall we are given an abelian variety $A$ and a number field $K$ such that $K\xhookrightarrow{}\mathrm{End}^{\circ}(A)$. This induces a map of $\mathbb{Q}$-algebras $$i: K \rightarrow \mathrm{End}(\mathrm{H}^1(A,\mathbb{Q}))$$ which preserves the weight $1$ Hodge structure on $\mathrm{H}^1(A,\mathbb{Q})$. We denote $\mathrm{H}^1(A,\mathbb{Q})$ by $V$.    

Denote the splitting closure of $K/\mathbb{Q}$ in $\qbar$ by $L$. Recall by Lemma \ref{neweigenspacetranslate}, we have the eigenspace decomposition
\begin{equation} \label{eigenspacedecompatC}
V_{L}:=V\otimes_{\mathbb{Q}} L=\bigoplus_{\sigma \in \mathrm{Hom}(K,L)} V_{\sigma}
\end{equation}
Letting $n=\frac{\mathrm{dim}(V)}{\mathrm{deg}(K/\mathbb{Q})}$, we have a natural inclusion of $L$-vector spaces
\begin{equation*}
    W:=\bigoplus_{\sigma \in \mathrm{Hom}(E,L)}{\bigwedge}^n V_{\sigma} \xhookrightarrow{} {\bigwedge}^n V_{L}
\end{equation*} According to the proof of Lemma \ref{neweigenspacetranslate}, the action of $\mathrm{Gal}(L/\mathbb{Q})$ on $V_{L}$ permutes the eigenspace $V_{\sigma}$s. Therefore, the $L$-subspace $W$ is preserved by the action of $\mathrm{Gal}(L/\mathbb{Q})$, which implies that $W$ descends to a $\mathbb{Q}$-subspace $W'$ of $\bigwedge^n V$. Namely, $W'$ is the set of elements in $W$ fixed by the action of $\mathrm{Gal}(L/\mathbb{Q})$. In the sequel, we will denote $W'$ by $\bigwedge_{K}^n V$.

\begin{definition} \label{weilstructure}
 Following the construction from above, we call $$W'=\wedge_{K}^n V$$ the \textit{Weil structure} associated with $i: K \xhookrightarrow{} \mathrm{End}(V)$.
\end{definition}
\begin{remark}
The construction and definition of Weil structures from above differs from standard ones for example from Lemma 12 of \cite{moonen1998weil}.
\end{remark}
We have the following lemma stating that if $K$ preserves the Hodge or dRB structure on $V$, then the Weil structure associated with $K$ is a sub-Hodge or sub-dRB structure.
\begin{lemma} \label{weilhodge}
Suppose the image of $i: K \xhookrightarrow{} \mathrm{End}(\mathrm{H}^1(A,\mathbb{Q}))$ preserves the Hodge structure or the dRB structure. Then the inclusion of $\mathbb{Q}$-vector space
$\bigwedge_{K}^n \mathrm{H}^1(A,\mathbb{Q}) \xhookrightarrow{} \bigwedge_{\mathbb{Q}}^n \mathrm{H}^1(A,\mathbb{Q})$ induces a natural weight $n$ Hodge structure or a de Rham-Betti structure on $\bigwedge_{K}^n \mathrm{H}^1(A,\mathbb{Q})$.
\end{lemma}
\begin{proof}
If the image of $i: K \xhookrightarrow{} \mathrm{End}(\betti)$ preserves the Hodge structure or the dRB structure, then it commutes with the Mumford-Tate group or the dRB group inside $\mathrm{GL}(\betti)$, which we denote by $G$ for both cases. Then it suffices to show that the 
$\mathbb{Q}$-vector subspace $\bigwedge_{K}^n \betti$ inside $\bigwedge_{\mathbb{Q}}^n \betti$ is a $G$-subrepresentation. But then it suffices to show that $\bigwedge_{K}^n \betti \otimes_{\mathbb{Q}} \qbar=\bigoplus_{\sigma \in \mathrm{Hom}(K,\qbar)}{\bigwedge}^n V_{\sigma}$ inside $(\bigwedge_{\mathbb{Q}}^n \mathrm{H}^1(A,\mathbb{Q}))\otimes_{\mathbb{Q}} \qbar$ is preserved by the $\qbar$-points of $G$. We abuse the notation by denoting $V_{\sigma}\otimes_{L}\qbar$ by $V_{\sigma}$ as well. But since $K$ commutes with $G$, each eigenspace $V_{\sigma}$ is preserved by $G(\qbar)$ by Lemma \ref{qbarhodge}. Hence indeed the $\qbar$-linear extension of the Weil structure is preserved by the $\qbar$-points of $G$.
\end{proof}

\begin{remark}\label{weilhodgenumber}
Denote the multiplicity of $\sigma \in \mathrm{Hom}(K,\mathbb{C})$ (see Definition \ref{multiplicitydefn}) by $m_{\sigma}$. Then the Hodge decomposition of the weight $n$ Hodge structure $\bigwedge_{K}^n \mathrm{H}^1(A,\mathbb{Q})$ can be written in the following manner \begin{equation*}
\begin{split}
\wedge_{K}^n \mathrm{H}^1(A,\mathbb{Q}) \otimes_{\mathbb{Q}} {\mathbb{C}}&=\bigoplus_{\sigma \in \mathrm{Hom}(K,\qbar)}\wedge^n V_{\sigma}\otimes_{\qbar}\mathbb{C}=\bigoplus_{\sigma \in \mathrm{Hom}(K,\qbar)}\wedge ^n(V_{\sigma}^{1,0} \oplus V_{\sigma}^{0,1})\\
&=\bigoplus_{\sigma \in \mathrm{Hom}(K,\qbar)}\wedge ^{m_{\sigma}}V_{\sigma}^{1,0}\otimes \wedge ^{n-m_{\sigma}}V_{\sigma}^{0,1}
\end{split}
\end{equation*} where each summand $\wedge ^{m_{\sigma}}V_{\sigma}^{1,0}\otimes \wedge ^{n-m_{\sigma}}V_{\sigma}^{0,1}$ is a subspace of $\mathrm{H}^{m_{\sigma},n-m_{\sigma}}(A)$.
\end{remark}

\begin{remark}\label{explicityweildrb}
We can also write out explicitly the dRB structure on $\wedge_{K}^{n}\mathrm{H}^1(A,\mathbb{Q})$ as follows. Note that if $K \xhookrightarrow{} \mathrm{End}(\mathrm{H}^1(A,\mathbb{Q}))$ preserves the dRB structure, then for any element $k \in K$, the morphism $$\rho_{m} \circ k \circ \rho^{-1}_{m}: \mathrm{H}^1_{\mathrm{dR}}(A/\mathbb{C}) \rightarrow \mathrm{H}^1_{\mathrm{dR}}(A/\mathbb{C})$$ takes $\mathrm{H}^1_{\mathrm{dR}}(A/\qbar)$ to $\mathrm{H}^1_{\mathrm{dR}}(A/\qbar)$ and is $\qbar$-linear. And in this way we also obtain a natural embedding of $K$ into $\mathrm{End}_{\qbar}(\mathrm{H}^1_{\mathrm{dR}}(A/\qbar))$. With respect to this embedding, $\mathrm{H}^1_{\mathrm{dR}}(A/\qbar)$ also admits an eigenspace decomposition $$\mathrm{H}^1_{\mathrm{dR}}(A/\qbar)=\bigoplus_{\sigma \in \mathrm{Hom}(K,\qbar)}W_{\sigma}$$ Now by construction we have $\rho_{m}^{-1}(W_{\sigma} \otimes_{\qbar}\mathbb{C}) \subset V_{\sigma} \otimes_{\qbar} \mathbb{C}$. By Lemma \ref{neweigenspacetranslate}, $V_{\sigma}$s are equidimensional, hence all $W_{\sigma}$s have the same dimension. Again letting $n=\frac{\mathrm{dim}(\mathrm{H}^1(A,\mathbb{Q}))}{\mathrm{deg}(K/\mathbb{Q})}$, then we may define $$\wedge_{K \otimes \qbar}^{n}\mathrm{H}^1_{\mathrm{dR}}(A/\qbar):=\bigoplus_{\sigma \in \mathrm{Hom}(K,\qbar)}\wedge^{n}_{\qbar}W_{\sigma} \xhookrightarrow{} \wedge^{n}_{\qbar}\mathrm{H}^1_{\mathrm{dR}}(A,\qbar)$$ Then by Lemma \ref{weilhodge}, the triple $(\wedge_{K}^{n}\mathrm{H}^1(A,\mathbb{Q}), \wedge_{K \otimes \qbar}^{n}\mathrm{H}^1_{\mathrm{dR}}(A/\qbar), \rho_{m})$ is a sub-dRB structure of $(\wedge^{n}_{\mathbb{Q}}\mathrm{H}^1(A,\mathbb{Q}), \wedge^{n}_{\qbar}\mathrm{H}^1_{\mathrm{dR}}(A/\qbar), \rho_{m})$.
\end{remark}
\begin{lemma}\label{morphismofweilstruc}
   If we are in the same setup as Lemma \ref{weilhodge}, then we have an embedding $K \xhookrightarrow{} \mathrm{End}_{\mathbb{Q}}(\wedge_{K}^{n}\mathrm{H}^1(A,\mathbb{Q}))$ as a morphism of $\mathbb{Q}$-algebras. Moreover, the image of $K$ preserves the Hodge structure or the dRB structure on $\wedge_{K}^{n}\mathrm{H}^1(A,\mathbb{Q})$. 
\end{lemma}

\begin{proof}
By construction we have $\wedge_{K}^{n}\mathrm{H}^1(A,\mathbb{Q}) \otimes_{\mathbb{Q}} L=\oplus_{\sigma \in \mathrm{Hom}(K,L)} \wedge^nV_{\sigma}$ viewed as an $L$-subspace of $\wedge^n\mathrm{H}^1(A,\mathbb{Q}) \otimes_{\mathbb{Q}} L$. Then 
we can define a morphism of $\mathbb{Q}$-algebras 
\begin{equation*}
    K \rightarrow \mathrm{End}_{L}(\wedge^n_{K}\mathrm{H}^1(A,\mathbb{Q}) \otimes_{\mathbb{Q}} L)
\end{equation*} by letting an element $k \in K$ act as scalar multiplication by $\sigma(k)\in L$ on each 1-dimensional $L$-vector space $\wedge^n V_{\sigma}$. Using that the action of the Galois group $\mathrm{Gal}(L/\mathbb{Q})$ permutes the set of summands $\wedge^n V_{\sigma}$, one can verify that the image of $K$ inside $\mathrm{End}_{L}(\wedge^n_{K}\mathrm{H}^1(A,\mathbb{Q}) \otimes_{\mathbb{Q}} L)$ is in fact $\mathrm{Gal}(L/\mathbb{Q})$-equivariant. Hence, we in fact have a morphism of $\mathbb{Q}$-algebras
\begin{equation*}
    K \rightarrow \mathrm{End}(\wedge^n_{K}\mathrm{H}^1(A,\mathbb{Q}))
\end{equation*} As for the second statement in the lemma, similar to the proof of Lemma \ref{weilhodge} we denote both the Mumford-Tate group and the de Rham-Betti group by $G$. Then each $V_{\sigma}$ is preserved by $\qbar$-points of $G$. Therefore each one-dimensional $\qbar$-vector space $\wedge^n V_{\sigma}$ is also preserved by the induced action of $G(\qbar)$. Hence by construction the image of $K$ in $\mathrm{End}(\wedge^n_{K}\mathrm{H}^1(A,\mathbb{Q})\otimes_{\mathbb{Q}}\qbar)$ commutes with the induced action of $G(\qbar)$. We conclude that the action of $K$ on $\wedge^n_{K}\mathrm{H}^1(A,\mathbb{Q})$ preserves the underlying Hodge or dRB structure.
\end{proof}
\begin{lemma} \label{Weilstructurelayer}
Under the setup of Lemma \ref{morphismofweilstruc}, for any subfield $k \xhookrightarrow{} K$, we denote $m=\frac{\mathrm{dim}(\mathrm{H}^1(A,\mathbb{Q}))}{\mathrm{deg}(k/\mathbb{Q})}$, $n=\frac{\mathrm{dim}(\mathrm{H}^1(A,\mathbb{Q}))}{\mathrm{deg}(K/\mathbb{Q})}$ and $l=\frac{\mathrm{deg}(K/\mathbb{Q})}{\mathrm{deg}(k/\mathbb{Q})}$. Then within $\wedge_{\mathbb{Q}}^m\mathrm{H}^1(A,\mathbb{Q})$ we have a natural isomorphism of Hodge structures or dRB structures
\begin{equation*}
 \wedge_{k}^{l}((\wedge_{K}^n \mathrm{H}^1(A,\mathbb{Q})))\cong\wedge_{k}^m\mathrm{H}^1(A,\mathbb{Q})
\end{equation*}
\end{lemma}
\begin{proof}
We will show that as $\mathbb{Q}$-vector subspaces of  $\wedge_{\mathbb{Q}}^m V$, there is an equality
\begin{equation}\label{equalityneeded}\wedge_{k}^{l}((\wedge_{K}^n V))=\wedge_{k}^m V\end{equation}
For this it suffices to show the equality of $\qbar$-vector subspaces
$$\wedge_{k}^{l}((\wedge_{K}^n V)) \otimes_{\mathbb{Q}} \qbar=\wedge_{k}^m\mathrm{H}^1(A,\mathbb{Q}) \otimes_{\mathbb{Q}} \qbar$$ 
Note that we have a canonical map $$\mathrm{Hom}(K,\qbar) \rightarrow \mathrm{Hom}(k,\qbar)$$ by restriction. It is a surjective map with the size of each fiber precisely $l$. Hence if $V \otimes \qbar=\bigoplus_{\sigma \in \mathrm{Hom}(K,\qbar)}V_{\sigma}$ is the eigenspace decomposition with respect to the embedding $K\xhookrightarrow{}\mathrm{End}(V)$, then with respect to the action of $k$ on $V$, we have the eigenspace decomposition \begin{equation}\label{keigenspace}
   V\otimes_{\mathbb{Q}}\qbar=\bigoplus_{\tau \in \mathrm{Hom}(k,\qbar)}(\oplus_{\sigma \in \mathrm{Hom}(K,\qbar),\sigma|_{k}=\tau}V_{\sigma})  
\end{equation} Therefore applying the Weil structure construction to decomposition (\ref{keigenspace}) we have that
\begin{equation*}
\begin{split}\wedge_{k}^m V\otimes_{\mathbb{Q}} \qbar&=\bigoplus_{\tau \in \mathrm{Hom}(k,\qbar)}(\wedge^{m=nl}_{\qbar}(\oplus_{\sigma \in \mathrm{Hom}(K,\qbar),\sigma|_{k}=\tau}V_{\sigma}))\\&=\bigoplus_{\tau \in \mathrm{Hom}(k,\qbar)}(\otimes_{\sigma \in \mathrm{Hom}(K,\qbar),\sigma|_{k}=\tau}\wedge^n_{\qbar}V_{\sigma})\end{split}\end{equation*} On the other hand, applying the Weil structure construction to the embedding $K \xhookrightarrow{} \mathrm{End}(V)$ we have that $$\wedge_{K}^n V\otimes_{\mathbb{Q}} \qbar=\bigoplus_{\sigma \in \mathrm{Hom}(K,\qbar)}\wedge^{n}V_{\sigma}=\bigoplus_{\tau \in \mathrm{Hom}(k,\qbar)}(\oplus_{\sigma \in \mathrm{Hom}(K,\qbar),\sigma|_{k}=\tau}\wedge^n_{\qbar}V_{\sigma})$$Recall that in Lemma \ref{morphismofweilstruc}, we have constructed an embedding $K \xhookrightarrow{} \mathrm{End}((\wedge_{K}^n V))$.  This induces a morphism of $\mathbb{Q}$-algebras \begin{equation}\label{kembedding}k \xhookrightarrow{} K \xhookrightarrow{} \mathrm{End}(\wedge_{K}^n V)\end{equation} Checking the proof of Lemma \ref{morphismofweilstruc}, with respect to formula (\ref{kembedding}) the eigenspace associated to $\tau\in \mathrm{Hom}(k,\qbar)$ is precisely $\oplus_{\sigma \in \mathrm{Hom}(K,\qbar),\sigma|_{k}=\tau}\wedge^n_{\qbar}V_{\sigma}$. Hence applying the Weil structure construction to (\ref{kembedding}) we obtain that \begin{equation*}
\begin{split}
\wedge_{k}^{l}((\wedge_{K}^n V)) \otimes_{\mathbb{Q}} \qbar&=\bigoplus_{\tau \in \mathrm{Hom}(k,\qbar)}\wedge^l_{\qbar}(\oplus_{\sigma \in \mathrm{Hom}(K,\qbar),\sigma|_{k}=\tau}\wedge^n_{\qbar}V_{\sigma})\\
&=\bigoplus_{\tau \in \mathrm{Hom}(k,\qbar)}(\otimes_{\sigma \in \mathrm{Hom}(K,\qbar),\sigma|_{k}=\tau}\wedge^n_{\qbar}V_{\sigma})
\end{split}   
\end{equation*} Therefore we obtain the desired equality (\ref{equalityneeded}).
\end{proof}

We conclude this section with a lemma relating elements in a Weil structure to invariants of a certain algebraic group, which will be used in future.
\begin{lemma}\label{weilstructurefixer}
Let $K$ be a number field and suppose we are given $K \xhookrightarrow{}\mathrm{End}(V)$ an inclusion of $\mathbb{Q}$-algebras. Denote by $\mathrm{SL}_{K}(V)$ the $\mathbb{Q}$-algebraic subgroup of $\mathrm{GL}(V)$ whose $\mathbb{Q}$-points are $\{g \in \mathrm{GL}_{K}(V)|\mathrm{det}_{K}(V \xrightarrow{\cdot g} V)=1\}$ where we view $V$ as a $K$-vector space. Then elements in the Weil structure $W:=\wedge_{K}^{n:=\frac{\mathrm{dim}(V)}{\mathrm{deg}(K/\mathbb{Q})}}\xhookrightarrow{}\wedge_{\mathbb{Q}}^{n}V$ are fixed by the induced action of $\mathrm{SL}_{K}(V)$ on $\wedge_{\mathbb{Q}}^{n}V$.
\end{lemma}

\begin{proof}
We will show that elements in $W \otimes \qbar$ are fixed by elements in $\mathrm{SL}_{K}(V)(\qbar)$. As usual we decompose $V_{\qbar}$ into direct sum of eigenspaces $\bigoplus_{\sigma \in \mathrm{Hom}(K,\qbar)}V_{\sigma}$ with respect to the embedding $K \rightarrow \mathrm{End}(V)$. Recall from Lemma \ref{neweigenspacetranslate} that $\mathrm{dim}(V_{\sigma})=n$. Now by definition $\mathrm{GL}_{K}(V)(\qbar)$ consists of invertible $\qbar$-linear transformations of $V_{\qbar}$ which commute with the action of $K$ on $V_{\qbar}$. Therefore each element $g\in \mathrm{GL}_{K}(V)(\qbar)$ preserves each direct summand $V_{\sigma}$. In fact $\mathrm{GL}_{K}(V)(\qbar)$ consists of block diagonal matrices $$\left\{\begin{pmatrix}
    g_{\sigma_1} & & \\
    & \ddots & \\
    & & g_{\sigma_{n}}
  \end{pmatrix}|\forall \sigma_{i}\in\mathrm{Hom}(K,\qbar), g_{\sigma_{i}} \in \mathrm{GL}_{\qbar}(V_{\sigma_{i}})\right\}
$$ Then by definition $\mathrm{SL}_{K}(V)(\qbar)$ consists of block diagonal matrices $$\left\{\begin{pmatrix}
    g_{\sigma_1} & & \\
    & \ddots & \\
    & & g_{\sigma_{n}}
  \end{pmatrix}|\forall \sigma_{i}\in\mathrm{Hom}(K,\qbar), g_{\sigma_{i}} \in \mathrm{SL}_{\qbar}(V_{\sigma_{i}})\right\}$$ Moreover we have that $$W\otimes_{\mathbb{Q}}\qbar=\bigoplus_{\sigma \in \mathrm{Hom}(K,\qbar)} \wedge^n_{\qbar}V_{\sigma}$$ therefore every element in $W\otimes_{\mathbb{Q}}\qbar$ is fixed by elements in $\mathrm{SL}_{K}(V)(\qbar)$.
\end{proof}

\subsection{Rosati Involution}
In this section we collect results about the Rosati involution associated with a polarization form on the endomorphism ring of an abelian variety $A$. The main reference is \cite{mumford1970abelian}.

Given a non-degenerate bilinear form $\phi$ on a finite dimensional vector space $V$, both defined over $\mathbb{Q}$, we can define an involution on $\mathrm{End}(V)$ as follows. Denote the isomorphism of vector spaces $$V \rightarrow V^{\vee}, w \rightarrow (v \rightarrow \phi(v,w))$$ by $\lambda_{\phi}$. Given $f\in \mathrm{End}(V)$, then we define $f^{\dag}$ to be $$\lambda^{-1}_{\phi} \circ f^{\vee} \circ \lambda_{\phi}$$ The $\dag$ construction satisfies that for any $v,w \in V$ we have that 
$$\phi(f(v),w)=\phi(v,f^{\dag}(w))$$ 
Now suppose $V=\betti$ and $\phi$ is a polarization form on $V$. Denote the image of $\mathrm{End}^{\circ}(A) \xhookrightarrow{} \mathrm{End}(V)$ by $E$. Then if $f\in E$, we have that $f^{\dag}=\lambda^{-1}_{\phi} \circ f^{\vee} \circ \lambda_{\phi}$ is also a morphism of Hodge structures. Hence $f^{\dag}\in E$.
\begin{definition}\label{rosatidefn}
In the setup of above, we call the involution $\dag: E\rightarrow E$ the \textit{Rosati involution} on $E$ associated with the polarization form $\phi$.     
\end{definition}
Now we restrict further to the case where $E$ is a field.
\begin{lemma}
   The Rosati involution satisfies the following positivity condition $$\mathrm{Tr}_{E/\mathbb{Q}}(ee^{\dag})>0$$ when $e\neq 0$.
\end{lemma}
If $E$ is a CM-field, there is another involution which satisfies a similar positivity condition, namely, the complex conjugation. We end this section with the following well known lemma which will be used later. 

\begin{lemma}\label{rosaticmcoincides}
    For an abelian variety $A$ with $\mathrm{End}^{\circ}(A)=E$ where $E$ is a CM-field, the Rosati involution associated with any polarization form $\phi$ coincides with the complex conjugation on $E$. Moreover, we have $\mathrm{dim}_{\mathbb{Q}}(\mathrm{H}^{1,1}(A)\cap \mathrm{H}^{2}(A,\mathbb{Q}))=\frac{\mathrm{deg}(E/\mathbb{Q})}{2}$.
\end{lemma}
\begin{proof}
The first statement of the lemma follows by the analysis in \cite[Page 201-202]{mumford1970abelian}. We now prove the second statement. By \cite[Proposition 3.7]{bost2016some}, we have the following identification $$\mathrm{NS}(A)=\mathrm{Hom}^{\mathrm{sym}}(A,A^{\vee}):=\{\phi\in\mathrm{Hom}(A,A^{\vee})|\phi^{\vee}=\phi\}$$ where $A^{\vee}$ denotes the dual abelian variety of $A$. Let $\lambda_{\phi}: A\rightarrow A^{\vee}$ be the isogeny associated with the polarization form $\phi$. According to \cite[Section 2.9]{mz-4-folds}, there is an isomorphism between $F:=\{e\in\mathrm{End}^{\circ}(A)|e^{\dag}=e\}$ and $\mathrm{Hom}^{\mathrm{sym}}(A,A^{\vee})$ where $e\in F$ is sent to $\lambda_{\phi}\circ e$. Therefore we have an isomorphism $\mathrm{NS}(A)\cong F$. The first part of the lemma tells us that the Rosati involution $\dag$ associated with $\phi$ is the complex conjugation on $E=\mathrm{End}^{\circ}(A)$. Hence the set $F$ can in fact be identified with the maximal totally real subfield of $E$. Hence $\mathrm{dim}(F)=\mathrm{dim}(\mathrm{H}^{1,1}(A)\cap\betti)=\frac{\mathrm{deg}(E/\mathbb{Q})}{2}$.
\end{proof}

\section{De Rham-Betti Groups of Simple CM Abelian Fourfolds}\label{cmgdrbsection}
In this section, we are going to show that for a simple CM abelian fourfold $A$, its de Rham-Betti group coincides with its Mumford-Tate group (Theorem \ref{mainthmcm}). 
\subsection{Existence of Homotheties inside De Rham-Betti Groups of CM-Abelian Varieties} \label{gmincmsection}
As the first step for the proof of Theorem \ref{mainthmcm}, in this section we will show that for a simple CM abelian variety $A$ (see Section \ref{cmintrosection} for a brief introduction) the canonical inclusion of the homotheties $\gmmath \xhookrightarrow{} \mathrm{MT}(A)$ (see Lemma \ref{existenceofhomothety}) factors through the inclusion of $\gdrbmath(A) \xhookrightarrow{} \mathrm{MT}(A)$ (see Proposition \ref{gm in CM gdrb}). The proof method of this section is inspired by \cite[Section 7.5]{kreutz2023rhambetti}.

To begin with, we will construct an algebraic subgroup of $\gdrbmath(A)$ called $\gdrbhmath(A)$ which serves the role of Hodge group in the setting of de Rham-Betti structures. Fix an abelian variety $A$ over $\overline{\mathbb{Q}}$. Denote by $$\phi: \gmmath\times \mathrm{Hdg}(A) \twoheadrightarrow \mathrm{MT}(A)$$ the natural isogeny relating its Hodge group and Mumford-Tate group (see Lemma \ref{hodgetimesgm}). By Theorem \ref{keyinclusion}, there is a natural inclusion of the $\gdrbmath(A)$ into $\mathrm{MT}(A)$. Take the identity component of the preimage of $\gdrbmath(A)$ by $\phi$ within the product $\gmmath\times \mathrm{Hdg}(A)$ and then project it down to $\mathrm{Hdg}(A)$. We then obtain a connected algebraic subgroup inside the Hodge group, which we call $\gdrbhmath(A)$. Observe that $\gdrbhmath(A)$ is also a reductive algebraic group.  To summarize, we have constructed $\gdrbhmath(A)$ such that the following diagram commutes
\begin{equation}\label{gdrbhdiagram1}
\begin{tikzcd}
                                                                                         &                                                                                                & \gdrbmath(A) \arrow[d, hook] \\
(\phi^{-1}\gdrbmath(A))^{\circ} \arrow[r, hook] \arrow[d, two heads] \arrow[rru, two heads] & \gmmath \times \mathrm{Hdg}(A) \arrow[r, "\phi", two heads] \arrow[d, "\mathrm{pr}_2", two heads] & \mathrm{MT}(A)               \\
\gdrbhmath(A) \arrow[r, hook]                                                               & \mathrm{Hdg}(A)                                                                                   &                          
\end{tikzcd}
\end{equation}
\begin{definition}\label{defgdrbh}
Following the setup and notations from the above, we define $$\gdrbhmath(A):=\mathrm{pr}_{2}((\phi^{-1}\gdrbmath(A))^{\circ})$$    
\end{definition}
For later usage, we single out an easy observation from the diagram above.
\begin{lemma}\label{keysurjection}
The algebraic group homomorphism from
 $$(\phi^{-1}\gdrbmath(A))^{\circ}\rightarrow\gdrbmath(A)$$ in the above diagram (\ref{gdrbhdiagram1}) is indeed surjective. 
\end{lemma}
\begin{proof}This is because $\phi$ is an isogeny and $\gdrbmath(A)$ is a connected algebraic group by Lemma \ref{gdrbconnectedness}.\end{proof}

The main goal of this section is to prove the following proposition.
\begin{proposition} \label{gm in CM gdrb}
For a simple CM abelian variety $A$, we have the following commutative diagram of algebraic groups defined over $\mathbb{Q}$ 
$$\begin{tikzcd}
                                                 & \gdrbmath(A) \arrow[d]         \\
\gmmath \arrow[r, "i_{w}"] \arrow[rd] \arrow[ru] & \mathrm{MT}(A) \arrow[d, hook] \\
                                                 & \mathrm{GL}(V)                
\end{tikzcd}$$
where $i_{w}$ sends $t\in\gmmath$ to scalar multiplication by $t^{-1}$ on $V$ (see Lemma \ref{existenceofhomothety}).
Moreover we can upgrade (\ref{gdrbhdiagram1}) to the following commutative diagram 
\begin{equation}\label{gdrbhdiagram}
\begin{tikzcd}
\gmmath \times \gdrbhmath(A) \arrow[r, two heads] \arrow[d, hook] & \gdrbmath(A) \arrow[d, hook] \\
\gmmath \times \mathrm{Hdg}(A) \arrow[r, two heads]               & \mathrm{MT}(A)              
\end{tikzcd}
\end{equation}
\end{proposition}

A immediate corollary of the above proposition is the following.

\begin{corollary}
For a simple CM abelian variety $A$, we have a natural inclusion $\gdrbhmath(A)\xhookrightarrow{}\gdrbmath(A)$.    
\end{corollary}
\begin{proof}
    We note that from diagram (\ref{gdrbhdiagram}), the subgroup $\{\mathrm{id}\}\times\gdrbhmath(A)$ of $\gmmath\times\gdrbhmath(A)$ naturally maps to $\gdrbmath(A)$ under $\phi$, and its image is precisely $\gdrbhmath(A)$.
\end{proof}

The proof of Proposition \ref{gm in CM gdrb} and the preparation for it will occupy the rest of this subsection. It is based on a detailed analysis of (\ref{gdrbhdiagram1}) and the theory of algebraic tori.

We start with the following lemma from \cite{kreutz2023rhambetti}.
\begin{lemma}[Theorem 7.11 (iii), \cite{kreutz2023rhambetti}]\label{surjectiontogmcompatible}
For any abelian variety defined over $\qbar$, there exist surjective algebraic group homomorphisms from $\mathrm{MT}(A)$ and $\gdrbmath(A)$ to $\gmmath$. Moreover, they are compatible with each other. In other words, the following diagram commutes 
$$\begin{tikzcd}
\gdrbmath(A) \arrow[r, two heads] \arrow[d, hook] & \gmmath \\
\mathrm{MT}(A) \arrow[ru, two heads]              &        
\end{tikzcd}$$
\end{lemma}
\begin{proof}
We include the proof for later usage. Suppose the dimension of the abelian variety $A$ is $g$. Then $\wedge^{2g}\mathrm{H}^1(A,\mathbb{Q})$ is a weight $2g$ Hodge structure of dimension 1. Hence $$\mathrm{MT}(\wedge^{2g}\mathrm{H}^1(A,\mathbb{Q}))=\gmmath$$ On the other hand, notice that elements in the one-dimensional de Rham-Betti structure $$\mathrm{H}_{\mathrm{dRB}}^{2g}(A,\mathbb{Q}) \otimes \mathbb{Q}_{\mathrm{dRB}}(g)$$ are algebraic classes and hence are de Rham-Betti classes. Therefore we obtain $$\wedge^{2g}\mathrm{H}_{\mathrm{dRB}}^1(A,\mathbb{Q})=\mathrm{H}_{\mathrm{dRB}}^{2g}(A,\mathbb{Q}) \cong \mathbb{Q}_{\mathrm{dRB}}(-g)$$ Thus $\gdrbmath(\wedge^{2g}\mathrm{H}_{\mathrm{dRB}}^1(A,\mathbb{Q}))=\gmmath$. By \cite[Proposition 2.21]{deligne2012tannakian} and using the Tannakian formalism for Mumford-Tate groups (Definition \ref{tannakianformalism1}) the inclusion of neutral Tannakian categories $$\langle\wedge^{2g}\mathrm{H}^1(A,\mathbb{Q})\rangle^{\otimes} \xhookrightarrow{} \langle \mathrm{H}^1(A,\mathbb{Q})\rangle^{\otimes}$$ and $$\langle\wedge^{2g}\mathrm{H}^1_{\mathrm{dRB}}(A,\mathbb{Q})\rangle^{\otimes} \xhookrightarrow{} \langle \mathrm{H}_{\mathrm{dRB}}^1(A,\mathbb{Q})\rangle^{\otimes}$$ induces surjective algebraic group homomorphisms $\mathrm{MT}(A) \twoheadrightarrow \gmmath$ and $\gdrbmath(A) \twoheadrightarrow \gmmath$. Upon close inspections, these two homomorphisms are the restrictions of the usual determinant maps on $\mathrm{GL}(\betti)$, hence they are compatible with each other.
\end{proof}

In preparation for the proof of Proposition \ref{gm in CM gdrb}, we will use the following lemma concerning properties of certain algebraic tori defined over $\mathbb{Q}$. We refer the reader to Section \ref{algtori} for a brief introduction of the relevant theory.
 \begin{lemma}\label{nosurjectiontogm}
Let $E$ be a CM number field. Then no $\mathbb{Q}$-algebraic subgroup of $\mathrm{U}_E$ admits a surjective homomorphism to $\gmmath$. Recall from Section \ref{algtori} that $\mathrm{U}_E$ is an algebraic subtorus of \resgm satisfying $\mathrm{U}_E(\mathbb{Q})=\{x \in E^{*}|x\overline{x}=1\}$.
 \end{lemma}

 \begin{proof}[Proof of lemma \ref{nosurjectiontogm}]
Denote $L$ the Galois closure of $E/\mathbb{Q}$ in $\qbar$ which is a CM-field by Lemma \ref{galoisclosureofCMisCM}. Given $T$ an arbitrary $\mathbb{Q}$-algebraic subgroup of $\mathrm{U}_E$, by Theorem \ref{chartorus}, there is a $\mathrm{Gal}(L/\mathbb{Q})$-equivariant surjection of $\mathbb{Z}$-modules: $$X^{*}(\mathrm{U}_E) \xrightarrow{f} X^{*}(T)$$ In particular, denote by $\sigma$ be the complex conjugation of the CM-field $L$, then for any $m \in X^{*}(\mathrm{U}_E)$ we have $$\sigma \circ f(m)=f(\sigma \circ m)=f(-m)=-f(m)$$ Because $f$ is surjective, $\sigma$ acts as -1 on $X^{*}(T)$. Suppose $T$ admits a surjective homomorphism to $\gmmath$. Then by Theorem \ref{chartorus} we have a $\mathrm{Gal}(L/\mathbb{Q})$-equivariant inclusion $$X^{*}(\gmmath) \xhookrightarrow{i} X^{*}(T)$$ Note that \gm is already split over $\mathbb{Q}$, hence $\sigma \in \mathrm{Gal}(L/\mathbb{Q})$ acts as identity on $X^{*}(\gmmath)$. Therefore, for $a \in X^{*}(\gmmath)$, we have that $$i(a)=i(\sigma \circ a)=\sigma \circ i(a)=-i(a)$$ which is a contradiction. 
 \end{proof}

According to the construction in diagram (\ref{gdrbhdiagram1}), we have that $\mathrm{pr}_2:(\phi^{-1}\gdrbmath(A))^{\circ}\rightarrow \gdrbhmath(A)$ is a surjective algebraic group homomorphism. We also have the following.
\begin{lemma}\label{pr1surj}
Under the first projection map from $\gmmath \times \mathrm{Hdg}(A)$ to $\gmmath$, we have that $(\phi^{-1}\gdrbmath(A))^{\circ}$ surjects to $\gmmath$.
\end{lemma}
\begin{proof}
    Because $(\phi^{-1}\gdrbmath(A))^{\circ}$ is a connected algebraic group, the image of $\mathrm{pr}_1:(\phi^{-1}\gdrbmath(A))^{\circ}\rightarrow\gmmath$ is either $\gmmath$ or $\{\mathrm{id}\}$. Suppose it is $\{\mathrm{id}\}$. Then $$(\phi^{-1}\gdrbmath(A))^{\circ}\xhookrightarrow{}\{\mathrm{id}\}\times\mathrm{Hdg}(A)\xhookrightarrow{}\{\mathrm{id}\}\times\mathrm{U}_{E}$$ by Lemma \ref{generalizedcm}. However, we have the following composition of surjective homomorphism of algebraic groups $$(\phi^{-1}\gdrbmath(A))^{\circ} \twoheadrightarrow \gdrbmath(A) \twoheadrightarrow \gmmath$$ This is then a contradiction to Lemma \ref{nosurjectiontogm}. Hence the image of $$\mathrm{pr}_1:(\phi^{-1}\gdrbmath(A))^{\circ}\rightarrow\gmmath$$ is the entire $\gmmath$.
\end{proof}

We warn the reader that the above lemma does not ensure that $\gmmath \times \{\mathrm{id}\}\xhookrightarrow{} (\phi^{-1}\gdrbmath(A))^{\circ}$. However, combining Lemma \ref{pr1surj} and Lemma \ref{nosurjectiontogm} with a well known fact from group theory called Goursat's lemma, we will obtain Corollary \ref{gminside} which is used throughout this section. We first state Goursat's lemma following \cite[Proposition 2.16]{lewis1999survey}.

\begin{lemma}[Goursat's lemma] \label{goursat}
Let $G_1$ and $G_2$ be two algebraic groups, and $H$ be an algebraic subgroup of $G_1 \times G_2$ whose projections to both $G_1$ and $G_2$ are surjective. Denote $H \cap G_1 \times \{\mathrm{id}_{G_2}\}$ by $N_1$ and $H \cap \{\mathrm{id}_{G_1}\}\times G_2$ by $N_2$. Then $N_1$ is a normal algebraic subgroup of $G_1$ and $N_2$ is a normal algebraic subgroup of $G_2$. Moreover, the image of $H$ in $G_1/N_1 \times G_2/N_2$ is the graph of an algebraic group isomorphism between $G_1/N_1$ and $G_2/N_2$.

\end{lemma}

\begin{corollary} \label{gminside}
Suppose $E$ is a CM-field. Given a connected $\mathbb{Q}$-algebraic subgroup $$T \subset \gmmath \times \mathrm{U}_{E}$$ such that $\mathrm{pr}_1: T \rightarrow \gmmath$ is a surjective homomorphism, we have that $$T \cap \gmmath \times \{\mathrm{id}\}=\gmmath \times \{\mathrm{id}\}\xhookrightarrow{} \gmmath \times \mathrm{U}_{E}$$
\end{corollary}

\begin{proof}
Denote the image of $T$ under $\mathrm{pr}_2: T \rightarrow \mathrm{U}_{E}$ by $T'$, which is a connected algebraic subgroup of $\mathrm{U}_{E}$. Then we have $$T \xhookrightarrow{} \gmmath \times T'$$ where both $\mathrm{pr}_1$ and $\mathrm{pr}_2$ when restricted to $T$ are surjective algebraic group homomorphisms. Denote $T \cap \gmmath \times \{\mathrm{id}\}$ by $N$ which is a normal algebraic subgroup of $\gmmath$ by Lemma \ref{goursat}. By the same lemma we then have a surjection of $\mathbb{Q}$-algebraic groups $$T' \twoheadrightarrow \gmmath/N$$ Note that $\gmmath/N$ is either isomorphic to $\gmmath$ or is equal to a trivial algebraic group. However according to Lemma \ref{nosurjectiontogm}, $T'$ being a connected subtorus of $\mathrm{U}_{E}$, cannot admit a surjective algebraic group homomorphism to $\gmmath$. Hence $\gmmath/N$ is equal to the trivial algebraic group. Hence we have $N=T \cap \gmmath \times \{\mathrm{id}\}=\gmmath \times \{\mathrm{id}\}$.
\end{proof}
\begin{proof}[Proof of Proposition \ref{gm in CM gdrb}]\label{gdrbhmakesense}
Applying Lemma \ref{pr1surj} and Corollary \ref{gminside} to the following diagram 
$$\begin{tikzcd}
\phi^{-1}(\gdrbmath(A))^{\circ}  \arrow[rr, hook] &         & \gmmath \times \gdrbhmath(A) \arrow[ld, two heads] \arrow[rd, two heads] \arrow[r, hook] & \gmmath \times \mathrm{U}_{E} \\
                                               & \gmmath &                                                                                          & \gdrbhmath(A)                
\end{tikzcd}$$
we obtain that $\gmmath \times \{\mathrm{id}\} \xhookrightarrow{} \phi^{-1}(\gdrbmath(A))^{\circ}$. By the construction of $\gdrbhmath(A)$, we then obtain the following commutative diagram of algebraic groups $$\begin{tikzcd}
\gmmath \arrow[rd, hook] \arrow[r, hook] & \gdrbmath(A) \arrow[d, hook] \\
                                                       & \mathrm{MT}(A)              
\end{tikzcd}$$
Now because $\gmmath \times \{\mathrm{id}\}\xhookrightarrow{}\phi^{-1}(\gdrbmath(A))^{\circ}$ and by definition $\gdrbhmath(A)=\mathrm{pr}_2(\phi^{-1}(\gdrbmath(A))^{\circ})\subset \mathrm{Hdg}(A)$, we obtain the following commutative diagram
$$\begin{tikzcd}
\gmmath \times \gdrbhmath(A) \arrow[r, two heads] \arrow[d, hook] & \gdrbmath(A) \arrow[d, hook] \\
\gmmath \times \mathrm{Hdg}(A) \arrow[r, two heads]               & \mathrm{MT}(A)              
\end{tikzcd}$$ Note that the homomorphism from $\gmmath \times \gdrbhmath(A)$ to $\gdrbmath(A)$ is indeed surjective because of Lemma \ref{keysurjection}. 
\end{proof}

In fact Proposition \ref{gm in CM gdrb} tells us that for simple abelian varieties with complex multiplication, invariant tensors of $\gdrbmath(A)$ are related to invariant tensors under the smaller torus $\gdrbhmath(A)$, up to twisting by $\mathbb{Q}_{\mathrm{dRB}}(k)$ for some $k \in \mathbb{Z}$. This is analogous to the relations between Hodge groups and Mumford-Tate groups. This observation is made more precise in Corollary \ref{gdrbhinv}.
\begin{lemma}[Lemma 2.8 \cite{bost2016some}]\label{weightsinav}
For any abelian variety $A$ defined over $\qbar$, the dRB structure $\mathbb{Q}_{\mathrm{dRB}}(1)$ is an object in the Tannakian category $\langle\mathrm{H}_{\mathrm{dRB}}^1(A,\mathbb{Q})\rangle^{\otimes}$. 
\end{lemma}
\begin{proof}
Any polarization form on $A$ induces a nonzero morphism of dRB structures $$\mathrm{H}_{\mathrm{dRB}}^1(A,\mathbb{Q})\otimes \mathrm{H}_{\mathrm{dRB}}^1(A,\mathbb{Q}) \rightarrow \mathbb{Q}_{\mathrm{dRB}}(-1)$$ which realizes $\mathbb{Q}_{\mathrm{dRB}}(-1)$ as a quotient object. But by definition $\mathbb{Q}_{\mathrm{dRB}}(1)$ is the dual object of $\mathbb{Q}_{\mathrm{dRB}}(-1)$ and a Tannakian category is closed under taking the dual of an object.
\end{proof}
Recall by Tannakian duality $\langle\mathrm{H}_{\mathrm{dRB}}^1(A,\mathbb{Q})\rangle^{\otimes}$ is tensor equivalent to $\mathrm{Rep}_{\mathbb{Q}}\gdrbmath(A)$. The action of $\gmmath\subset\gdrbmath(A)$ on the de Rham-Betti structure $\mathbb{Q}_{\mathrm{dRB}}(-1)$ is stated as follows.
\begin{corollary}\label{gmweights}
Under the setup of Proposition \ref{gm in CM gdrb}, $\gmmath\subset\gdrbmath(A)$ acts as $\gmmath \rightarrow \gmmath:t \rightarrow t^{-2}$ on $\mathbb{Q}_{\mathrm{dRB}}(-1) \in \mathrm{ob}\langle\mathrm{H}_{\mathrm{dRB}}^1(A,\mathbb{Q})\rangle^{\otimes}$.
\end{corollary}
\begin{proof}
  The polarization form induces a $\gdrbmath(A)$-equivariant morphism $\mathrm{H}_{\mathrm{dRB}}^1(A,\mathbb{Q}) \otimes \mathrm{H}_{\mathrm{dRB}}^1(A,\mathbb{Q}) \rightarrow \mathbb{Q}_{\mathrm{dRB}}(-1)$. In particular, the action of the $\gmmath$ on both sides have to be compatible, which implies the second statement.
\end{proof}
We now record a well known convenient lemma from invariant theory, which follows directly from the definition of invariants of a linear algebraic group.
\begin{lemma}\label{invarintslemma}
Suppose $W$ is a finite dimensional $\mathbb{Q}$-vector space. Let $G \rightarrow \mathrm{GL}(W)$ be a morphism of algebraic groups defined over $\mathbb{Q}$. Then we have $W^{G}\otimes_{\mathbb{Q}}\qbar=(W \otimes_{\mathbb{Q}}\qbar)^{G(\qbar)}$.
\end{lemma}

\begin{corollary} \label{gdrbhinv}
Let $M:=H^1_{\mathrm{dRB}}(A,\mathbb{Q})$ be the de Rham-Betti structure of a simple CM abelian variety $A$  and let $m$ and $n$ be two non negative integers such that $m-n$ is an even number. Then
    $$(M^{\otimes m}\otimes M^{*\otimes n})^{\gdrbhmath(A)} \otimes \mathbb{Q}_{\mathrm{dRB}}({\frac{m-n}{2}}) = (M^{\otimes m}\otimes M^{*\otimes n}\otimes \mathbb{Q}_{\mathrm{dRB}}({\frac{m-n}{2}}))^{\gdrbmath(A)}$$ In particular, we have $\mathrm{End}_{\mathrm{dRB}}(M)=\mathrm{End}_{\gdrbhmath(A)}(M)$.
    
\end{corollary}
\begin{proof}
Unwinding Proposition \ref{gm in CM gdrb}, we have that $\gmmath\xhookrightarrow{}\gdrbmath(A)$ and $\gdrbhmath(A)\xhookrightarrow{}\gdrbmath(A)$. Moreover, there is a surjective algebraic group homomorphism $\phi:\gmmath \times \gdrbhmath(A)\rightarrow\gdrbmath(A)$ sending $(t,g)$ to $t^{-1}g$. In particular it is surjective on $\qbar$-points. Therefore by Lemma \ref{invarintslemma} we have $$(M^{\otimes m}\otimes M^{*\otimes n}\otimes \mathbb{Q}_{\mathrm{dRB}}({\frac{m-n}{2}}))^{\gdrbmath(A)}=(M^{\otimes m}\otimes M^{*\otimes n}\otimes \mathbb{Q}_{\mathrm{dRB}}({\frac{m-n}{2}}))^{\gmmath \times \gdrbhmath(A)}$$ The RHS of the above is equal to $
(M^{\otimes m}\otimes M^{*\otimes n}\otimes \mathbb{Q}_{\mathrm{dRB}}({\frac{m-n}{2}}))^{\gdrbhmath(A)}$ because of Lemma \ref{gmweights}. But note that $\gdrbhmath(A)$ is a connected subtorus of $\mathrm{U}_{E}$ by construction. Hence by Lemma \ref{nosurjectiontogm}, its action on any one-dimensional de Rham-Betti object is trivial. We thus obtain the first equality in the corollary. The second statement follows by taking $m=n=1$.
\end{proof}

\begin{corollary}
For any simple CM abelian variety $A$, there is no dRB class in $$\mathrm{H}^{m}_{\mathrm{dRB}}(A,\mathbb{Q}) \otimes \mathbb{Q}_{\mathrm{dRB}}(n)$$ if $m \neq 2n$.
\end{corollary}

\begin{proof}
    Note that $\mathrm{H}^{m}_{\mathrm{dRB}}(A,\mathbb{Q})\cong \wedge^{m}\mathrm{H}^{1}_{\mathrm{dRB}}(A,\mathbb{Q})$. Therefore if $m \neq 2n$, then by Lemma \ref{gmweights} the invariants in $\mathrm{H}^{m}_{\mathrm{dRB}}(A,\mathbb{Q}) \otimes \mathbb{Q}_{\mathrm{dRB}}(n)$ under the induced action of $\gmmath \subset\gdrbmath(A)$ are $\{0\}$.
\end{proof}

\subsection{A Non-De Rham-Betti Group Criterion for Algebraic Tori}
For a simple abelian fourfold $A$ with complex multiplication by $E$, we present a simple computational criterion which will help us rule out the possibility of certain (connected) subtorus of $\mathrm{U}_{E}$ being $\gdrbhmath(A)$. We will be using notations and constructions from Section \ref{algtori}. We begin with the following setup. 
\begin{setup}\label{torusinvsetup}
Fix a simple abelian fourfold $A$ with complex multiplication by $E$. Denote $\betti$ by $V$. By Lemma \ref{neweigenspacetranslate} the morphism of $\mathbb{Q}$-algebras \begin{equation}\label{endoinclusion}E \xhookrightarrow{} \mathrm{End}(V)\end{equation} induces a homomorphism between $\mathbb{Q}$-algebraic groups $$\mathrm{U}_E \xhookrightarrow{} \resgmmath\xhookrightarrow{}\mathrm{GL}(V)$$ Given a $\mathbb{Q}$-algebraic subtorus $T$ of $\mathrm{U}_E$, by the basic correspondence between $\mathbb{Q}$-algebraic tori and their character groups (see Theorem \ref{chartorus}), we have the following short exact sequence of $\absgalois$ modules 
\begin{equation}
\begin{tikzcd}\label{charseq}
0 \arrow[r] & \Delta \arrow[r, "i"] & X^{*}(\mathrm{U}_{E}) \arrow[r, "j"] & X^{*}(T) \arrow[r] & 0
\end{tikzcd}
\end{equation} where $\Delta$ is the kernel of the surjective homomorphism $X^{*}(\mathrm{U}_{E})\twoheadrightarrow X^{*}(T)$.
Moreover, if $T$ is a connected algebraic torus, then $X^{*}(T)$ is a free abelian group. 
Then the following paragraph and Lemma \ref{torusinvariants} describe how the position of $\Delta$ inside $X^{*}(\mathrm{U}_{E})$ gives information about invariant tensors of $T(\qbar)$ under the representation $$T \xhookrightarrow{}
\mathrm{U}_{E}\xhookrightarrow{}\mathrm{GL}(V)$$
Recall from Lemma \ref{neweigenspacetranslate}, we have the eigenspace decomposition induced by formula (\ref{endoinclusion}) $$V\otimes_{\mathbb{Q}}\qbar=\oplus_{\sigma\in\mathrm{Hom}(E,\qbar)}V_{\sigma}$$ where $V_{\sigma}$ is a one-dimensional $\qbar$-vector space and an element $e\in E$ acts on $V_{\sigma}$ via scalar multiplication by $\sigma(e)$. We fix a $\qbar$-basis $v_{\sigma}$ for each $V_{\sigma}$. 

Moreover, coming from the geometry of the CM abelian variety $A$, we have a CM-type on $E$ i.e. a set of embeddings $S$ from $E$ to $\overline{\mathbb{Q}}$ such that $S\coprod\overline{S}=\mathrm{Hom}(E,\overline{\mathbb{Q}})$ (see Definition \ref{cmtypedefn}). We now label elements of $S$ and $\overline{S}$ by $(\sigma_1,\sigma_2,\sigma_3,\sigma_4)$ and $(\sigma_5:=\overline{\sigma_1},\sigma_6:=\overline{\sigma_2},\sigma_7:=\overline{\sigma_4},\sigma_8:=\overline{\sigma_4})$ respectively. Also label the corresponding basis in $$X^{*}(\resgmmath)\cong \mathbb{Z}^{\mathrm{Hom}(E,\overline{\mathbb{Q}})}$$ (see Definition \ref{chardefn} from Section \ref{algtori}) by $(e_{\sigma_1},e_{\sigma_2},e_{\sigma_3},e_{\sigma_4})$ and $(e_{\sigma_5},e_{\sigma_6},e_{\sigma_7},e_{\sigma_8})$ respectively. Denote the image of elements $(e_{\sigma_1},e_{\sigma_2},e_{\sigma_3},e_{\sigma_4})$ under the surjection from $X^{*}(\resgmmath)$ to $X^{*}(\mathrm{U}_{E})$ by $(f_{\sigma_1},f_{\sigma_2},f_{\sigma_3},f_{\sigma_4})$. They then form a $\mathbb{Z}$-basis of $X^{*}(\mathrm{U}_{E})$. This is because the kernel of the canonical surjective map $$X^{*}(\resgmmath)\twoheadrightarrow X^{*}(\mathrm{U}_{E})$$ is spanned by $\{e_{\sigma_{i}}+e_{\sigma_{i+4}}|i\in\{1,2,3,4\}\}$ by Example \ref{uexample}.  We are now ready to state the following lemma promised at the beginning of this setup
\end{setup}
\begin{lemma}\label{torusinvariants}
Using notations from Setup \ref{torusinvsetup}, suppose $$n_1f_{\sigma_1}+n_2f_{\sigma_2}+n_3f_{\sigma_3}+n_4f_{\sigma_4} \in i(\Delta), n_{i} \in \mathbb{Z}$$ Then $v':=v_{\sigma_1}^{\otimes n_1} \otimes v_{\sigma_2}^{\otimes n_2} \otimes v_{\sigma_3}^{\otimes n_3} \otimes v_{\sigma_4}^{\otimes n_4}$ is an invariant tensor under the induced action of $T({\qbar})$ on $V':=(V^{\otimes n_1} \otimes V^{\otimes n_2} \otimes V^{\otimes n_3} \otimes V^{\otimes n_4}) \otimes \overline{\mathbb{Q}}$. We adopt the convention that wherever $n<0$, $V^{\otimes n}:=V^{*\otimes (-n)}$ and $v_{\sigma_{i}}^{\otimes n}:=v_{\sigma_{i}}^{*\otimes (-n)}$, where $v_{\sigma_{i}}^{*}$ is the dual of $v_{\sigma_{i}}$.

 \end{lemma}

\begin{proof}
Suppose $e_{\chi}\in V\otimes_{\mathbb{Q}}\qbar$ is an eigenvector associated to a character $\chi\in X^{*}(\resgmmath)$. Then $e_{\chi}^{*}$, the dual of $e_{\chi}$, is an eigenvector associated with the character $\chi^{-1}$. Therefore given any element $$(z_{\sigma_1},z_{\sigma_2},z_{\sigma_3},z_{\sigma_4},z_{\sigma_1}^{-1},z_{\sigma_2}^{-1},z_{\sigma_3}^{-1},z_{\sigma_4}^{-1})\in\mathrm{U}_{E}(\qbar)\subset\resgmmath(\qbar)=\prod_{\sigma_{i}\in\mathrm{Hom}(E,\qbar)}\qbar^{\times}$$ it acts on $v'$ by scalar multiplication by $z_{\sigma_1}^{n_1}z_{\sigma_2}^{n_2}z_{\sigma_3}^{n_3}z_{\sigma_4}^{n_4}$. 
Hence $v'$ is an eigenvector of the induced action of $\mathrm{U}_{E}(\qbar)$ on $V'$ with eigencharacter $\alpha:=n_1f_{\sigma_1}+n_2f_{\sigma_2}+n_3f_{\sigma_3}+n_4f_{\sigma_4}$. Since the representation of $T({\qbar})$ on $V'$ factors through the inclusion $T({\qbar}) \xhookrightarrow{} \mathrm{U}_{E}(\qbar)$, $v'$ is also an eigenvector of $\alpha|_{T({\qbar})}$. Then because $\alpha$ lies in the kernel of the surjective map $X^{*}(\mathrm{U}_{E}) \twoheadrightarrow X^{*}(T)$, we have that for any $t \in T({\qbar})$, $\alpha(t)=1$. Therefore the induced action of $T({\qbar})$ fixes the tensor $v'$. 
\end{proof}

\begin{corollary} \label{divisortest}
We adopt notations from Setup \ref{torusinvsetup}. If $i(\Delta)$ contains an element of the form $n_1f_{\sigma_1}+n_2f_{\sigma_2}+n_3f_{\sigma_3}+n_4f_{\sigma_4}$ where two of the elements in $\{n_{1},n_{2},n_{3},n_{4}\}$ are equal to zero and the other two elements in $\{n_{1},n_{2},n_{3},n_{4}\}$ belong to $\{-1,1\}$ then $T$ cannot be $\gdrbhmath(A)$.
\end{corollary}

\begin{proof}
Suppose $T =\gdrbhmath(A)$ and suppose that $i(\Delta)$ contains elements of the form mentioned in the lemma. First we assume that $$n_{i}f_{\sigma_{i}}+n_{j}f_{\sigma_{j}}\in i(\Delta)$$ with $n_{i}=1,n_{j}=-1$ and $i\neq j$. Then according to Lemma \ref{invarintslemma} and Lemma \ref{torusinvariants} we would have that $v_{\sigma_{i}} \otimes v^{-1}_{\sigma_{j}}\in(V \otimes V^{*} \otimes \qbar)^{T(\qbar)} = \mathrm{End}(V_{\qbar})^{T(\qbar)} = \mathrm{End}_{\gdrbhmath(A)}(V) \otimes \qbar =\mathrm{End}_{\gdrbmath(A)}(V) \otimes \qbar = \mathrm{End}_{\mathrm{MT}(A)}(V) \otimes \qbar$ where the third equality follows from Lemma \ref{gdrbhinv} and the last equality follows from Theorem \ref{mainfact} and Corollary \ref{gdrbhinv}. Since $A$ is a simple abelian variety with complex multiplication, we have that $\mathrm{End}_{\mathrm{MT}(A)}(V)=E$ and therefore $\mathrm{End}_{\mathrm{MT}(A)}(V) \otimes \qbar$ is precisely equal to the $\qbar$-span of $\{v_{\sigma_{i}} \otimes v^{-1}_{\sigma_{i}}|i \in \{1,\dots,8\}\}$. Thus $v_{\sigma_{i}} \otimes v^{-1}_{\sigma_{j}}(i\neq j)$ cannot lie in it, which gives a contradiction.

Now suppose $$n_{i}f_{\sigma_{i}}+n_{j}f_{\sigma_{j}}\in i(\Delta)$$ with $n_{i}=n_{j}=1$ and $i\neq j$. Then by Lemma \ref{torusinvariants}, $v_{\sigma_{i}} \wedge v_{\sigma_{j}}$ would lie in $$(\wedge^2V_{\qbar})^{T(\qbar)} = (\wedge^2 V)^{T} \otimes_{\mathbb{Q}}\qbar=(\wedge^2 V)^{\gdrbhmath(A)} \otimes_{\mathbb{Q}}\qbar$$ However, by Lemma \ref{gdrbhinv} and Theorem \ref{mainfact} we have that \begin{align*}&(\wedge^2 V \otimes \mathbb{Q}_{\mathrm{dRB}}(1))^{\gdrbmath(A)}=(\wedge^2 V)^{\gdrbhmath(A)} \otimes \mathbb{Q}_{\mathrm{dRB}}(1)\\&=(\wedge^2 V \otimes \mathbb{Q}(1))^{\mathrm{MT}(A)}=(\wedge^2 V)^{\mathrm{Hdg}(A)} \otimes \mathbb{Q}(1)\end{align*} From which we deduce that $(\wedge^2 V)^{\gdrbhmath(A)} \otimes_{\mathbb{Q}}\qbar=(\wedge^2 V)^{\mathrm{Hdg}(A)} \otimes_{\mathbb{Q}} \qbar$. But by Lemma \ref{rosaticmcoincides}, the Picard rank of a simple CM abelian fourfold is 4 and recall that we have fixed the CM-type $S$ to be precisely $\{\sigma_1,\sigma_2,\sigma_3,\sigma_4\}$. Therefore $(\wedge^2 V)^{\mathrm{Hdg}(A)}\otimes_{\mathbb{Q}}\qbar$ is precisely the $\qbar$-span of $\{v_{\sigma_{i}} \wedge v_{\overline{\sigma_{i}}}|1 \leq i \leq 4\}=\{v_{\sigma_{i}} \wedge v_{\sigma_{i+4}}|1 \leq i \leq 4\}$.
Hence for $i,j\in \{1,2,3,4\},i\neq j$, $v_{\sigma_{i}} \wedge v_{\sigma_{j}}$ never appears in $(\wedge^2 V)^{\mathrm{Hdg}(A)} \otimes_{\mathbb{Q}}\qbar$. Again we have a contradiction.
\end{proof}
\subsection{Classification of One and Two-Dimensional Subtori of \texorpdfstring{$\mathrm{U}_{E}$}{UE}}\label{longcomp}
Fix an arbitrary simple CM abelian fourfold $A$ with complex multiplication by $E$. In this section we will study subtori of $\mathrm{U}_{E}$ which are of dimension 1 and 2. We will rule out most of them as a candidate for $\gdrbhmath(A)$ (see Corollary \ref{biggerthan-one}, Lemma \ref{nolength4cycle}, Proposition \ref{noklein4} and Lemma \ref{A_4notpossible}). The main tools are the Galois theoretical analysis of the short exact sequence (\ref{charseq}) from Setup \ref{torusinvsetup} and the criterion from Corollary \ref{divisortest}.

Most of the contents in this section will be elementary computations. Roughly speaking, the main principle guiding these computations is that certain symmetry underlying the Galois action on $i(\Delta)$ from short exact sequence (\ref{charseq}) forces $i(\Delta)$ to contain elements that violate Corollary \ref{divisortest}. Eventually the key conclusion that inevitably leads to the next section is Key Remark \ref{keyremarktotakeaway}.

The reasons for the long analysis of one-dimensional and two-dimensional subtori of $\mathrm{U}_{E}$ in this section are partially explained in Remark \ref{ribetinequalitu}. 

\begin{setup}\label{setupcontinued}
We continue with Setup \ref{torusinvsetup} from the previous section. Moreover, we are going to use the labeled basis of $X^{*}(\resgmmath)$ and $X^{*}(\mathrm{U}_{E})$ from Example \ref{firstexample} and Example \ref{uexample}. To facilitate the discussion, we still denote the Galois closure of $E/\mathbb{Q}$ in $\qbar$ by $L$. 

 Since  $\galoisL$ acts on $\mathrm{Hom}(E,L)$ by permutation and recall that we have labeled a basis of $X^{*}(\resgmmath)$ by elements in $\mathrm{Hom}(E,L)$, the image of $\galoisL$ in $\mathrm{Aut}_{\mathbb{Z}}(X^{*}(\resgmmath))$ consists of permutation matrices. Denote the image by $G'$. By Lemma \ref{comjugationcommuting}, the complex conjugation in $\galoisL$ commutes with every other element of $\galoisL$. By Example \ref{uexample}, the kernel of the surjection from $X^{*}(\resgmmath)$ to $X^{*}(\mathrm{U}_{E})$ is spanned by $\{e_{\sigma_{i}}+e_{\sigma_{i+4}}|i\in\{1,2,3,4\}\}$. Therefore we have a well defined group homomorphism from $G'$ to $\mathrm{Aut}_{\mathbb{Z}}(X^{*}(\mathrm{U}_{E}))$. Denote the image of $G'$ in $\mathrm{Aut}_{\mathbb{Z}}(X^{*}(\mathrm{U}_{E}))$ by $G''$. With respect to the basis $(f_{\sigma_1},f_{\sigma_2},f_{\sigma_3},f_{\sigma_4})$ of $X^{*}(\mathrm{U}_{E})$ constructed in Setup \ref{torusinvsetup}, elements of $G''$ are generalized permutations matrices where both 1 and -1 are allowed as possible non-zero entries.  Also note that $G' \cong G''$. For each element of $G''$, flipping -1 entry to 1, we obtain a forgetful group homomorphism $F: G'' \rightarrow \mathrm{S}_4$, the symmetric group on four elements. We  summarize the above construction in the following lemma.
 \begin{lemma} \label{galois_permutation}
     We have the following sequence of group homomorphisms:
\begin{equation*}
\begin{tikzcd}
\galoisL \arrow[r] & G' \arrow[r, "\cong"'] \arrow[d, hook]       & G'' \arrow[r,"F"] \arrow[d, hook]                    & H \arrow[d, hook] \\
                   & \mathrm{Aut}_{\mathbb{Z}}(X^{*}(\resgmmath)) & \mathrm{Aut}_{\mathbb{Z}}(X^{*}(\mathrm{U}_{E})) & \mathrm{S}_4     
\end{tikzcd}   
\end{equation*}    
and H is a transitive subgroup of $\mathrm{S}_4$.
\end{lemma}

 \begin{proof}
     We only need to verify that $H$ is indeed a transitive subgroup. But this is because the action of $\galoisL$ on $\mathrm{Hom}(E,L)$ is already transitive. 
 \end{proof}
\end{setup}
\begin{remark}
    Tracing through the construction in Section \ref{algtori}, the group $\mathrm{Gal}(L/\mathbb{Q})$ acts on $X^{*}(\resgmmath)$ on the left. We keep this convention throughout this whole subsection.
\end{remark}

Immediately from Corollary \ref{divisortest} and Setup \ref{torusinvsetup} and Setup \ref{setupcontinued} we have the following Corollary. 
\begin{corollary}\label{biggerthan-one}
The algebraic group $\gdrbhmath(A)$ of a simple CM abelian fourfold $A$ cannot be a one-dimensional algebraic torus.   
\end{corollary}

\begin{proof}
Recall by construction, $\gdrbhmath(A)$ is a connected subtorus of $\mathrm{U}_{E}$. Suppose it is one-dimensional and denote it by $T$. Then $X^{*}(T)$ is a rank one free $\mathbb{Z}$-module and therefore any $g \in G:=\galoisL$ acts on it as scalar multiplication by 1 or -1. Note that $\phi: X^{*}(\mathrm{U}_{E}) \rightarrow X^{*}(T)$ is a $G$-equivariant surjective map, we therefore have $\phi(g \circ e)=g \circ \phi(e)=\pm\phi(e)$ for any $e\in X^{*}(\mathrm{U}_{E})$. Also by Lemma \ref{galois_permutation}, given any two elements $e_{i}, e_{j}$ from the set of basis $(f_{\sigma_1},f_{\sigma_2},f_{\sigma_3},f_{\sigma_4})$ of $X^{*}(\mathrm{U}_{E})$ fixed in Setup \ref{torusinvsetup}, there exists $g \in G$ such that either $g \circ e_{i}=e_{j}$ or $g \circ e_{i}=-e_{j}$.  Hence for any pair $e_{i},e_{j}\in (f_{\sigma_1},f_{\sigma_2},f_{\sigma_3},f_{\sigma_4})$ with $i\neq j$, either $e_{i}-e_{j} $ or $e_{i}+e_{j}$ lies in $ i(\Delta)=\mathrm{ker}(\phi)$. But both cases violate Corollary \ref{divisortest}. Hence the dimension of $\gdrbhmath(A)$ cannot be 1.   
\end{proof}

To study two-dimensional $\mathbb{Q}$-algebraic subtori of $\mathrm{U}_{E}$, we will perform a detailed analysis of the short exact sequence from Lemma \ref{galois_permutation}, based on the property of the group $G''$ and $H$. Recall from Lemma \ref{galois_permutation} that $H$ is a transitive subgroup of $\mathrm{S}_4$. Below we have a classification of all transitive subgroups of $\mathrm{S}_4$:

\begin{lemma}[Ex 3.51 of \cite{rotman2012introduction}]\label{transitiveclassification}
The possible transitive subgroups of $\mathrm{S}_4$ are:  $\mathrm{S}_4$; $\mathrm{A}_4$; $V:=\{\mathrm{id},(12)(34),(13)(24),(14)(23)\}$; $\langle g \rangle \cong\mathbb{Z}/4\mathbb{Z}$; $\langle g,a \rangle\cong\mathrm{D}_4$(where $a\in V$ and $a \neq id_{V}$ and $g$ is an element of order 4).
     
 \end{lemma}

Hence if $H$ is one of the transitive subgroups of $\mathrm{S}_4$, then it satisfies at least one of the following conditions 
\begin{enumerate}
    \item $H$ contains an element of order 4 in $\mathrm{S}_4$.
    \item $H$ is equal to the Klein four-group $V$. 
    \item $H$ is equal to $\mathrm{A}_4$.
\end{enumerate}
Using the information that $i(\Delta)$ is stable under the action of $F^{-1}(H)=G''$, we are going to analyze properties of elements in $i(\Delta)$ case by case based on the above classification of $H$. This is done in Lemma \ref{nolength4cycle}, Proposition \ref{noklein4} and Lemma \ref{A_4notpossible}. We first record a convenient lemma to analyze properties of submodules of a free module.

\begin{lemma}\label{integralelement}
    Given $i:\Delta \xhookrightarrow{} X$ an inclusion of finite rank free $\mathbb{Z}$-modules, suppose $X/i(\Delta)$ is torsion free. Then we have $i(\Delta)=i(\Delta)\otimes_{\mathbb{Z}}\mathbb{Q}\cap X$.
\end{lemma}

\begin{proof}
Since $X/i(\Delta)$ is a free $\mathbb{Z}$-module, we have the following decomposition of abelian groups $$X=i(\Delta)\oplus N$$ for some free submodule $N$ of $X^{*}(\mathrm{U}_{E})$. This implies that $$X\otimes_{\mathbb{Z}}\mathbb{Q}=i(\Delta)\otimes_{\mathbb{Z}}\mathbb{Q}\oplus N\otimes_{\mathbb{Z}}\mathbb{Q}$$ Therefore \begin{equation*}\label{integrallocation}i(\Delta)=i(\Delta)\otimes_{\mathbb{Z}}\mathbb{Q}\cap X\end{equation*}   
\end{proof}

\begin{lemma} \label{nolength4cycle}
  We adopt notations from Setup \ref{setupcontinued} and Lemma \ref{galois_permutation}. Assuming that $T=\gdrbhmath(A)$ is a dimension 2 (connected) subtorus of $\mathrm{U}_{E}$, it is impossible for the subgroup $H$ to contain a cycle of order 4.
\end{lemma}

\begin{proof}
 Suppose $H$ contains a cycle of order 4. Then $H$ contains the cycle $g=(1432)$ or the cycle $(1324)$ or the cycle $(1243)$. Suppose the first case. Denote $(1432)$ by $g$. Then the preimage $F^{-1}(g)$ of $g$ under the group homomorphism $G'' \xrightarrow{F} \mathrm{S}_4$ has the possibility that $-1$ appears $k$ times among its entries where $0\leq k \leq 4$. But the element $-\mathrm{id} \in G''$, it suffices to consider the cases where $F^{-1}(g)$ contains a generalized permutation matrix whose entries have exactly zero, one or two $-1$s.  
\begin{enumerate}
\item \label{computations case_1M_1} First suppose $F^{-1}(g)$ contains the usual permutation matrix with no $-1$ among its entries, i.e. $$M_1:=\begin{pmatrix} 0 & 1 & 0 & 0\\0 & 0 & 1 & 0\\0 & 0 & 0 & 1\\1 & 0 & 0 & 0 \end{pmatrix} \in F^{-1}(g)$$ We will study rank two submodules $i(\Delta)$ of $X^{*}(\mathrm{U}_{E})$ stable under the action of matrix $M_1$ by extending scalars. Then $i(\Delta) \otimes_{\mathbb{Z}}\mathbb{Q}$ is a two dimensional $\mathbb{Q}$-vector subspace of $X^{*}(\mathrm{U}_{E}) \otimes_{Z} \mathbb{Q}$ stable under $M_1$. Diagonalizing $M_1$ we obtain that it has four distinct eigenvalues with the following eigenbasis: $\lambda_1=-1,v_1=(-1,1,-1,1); \lambda_2=i,v_2=(i,-1,-i,1); \lambda_3=-i,v_3=(-i,-1,i,1); \lambda_4=1,v_4=(1,1,1,1)$. Note that the eigenvalues are distinct hence $i(\Delta) \otimes_{\mathbb{Z}} \qbar$ is spanned by two eigenvectors of different eigenvalues. But $i(\Delta) \otimes_{\mathbb{Z}} \qbar$ is stable under the action of $\mathrm{Gal}(\mathbb{Q}(i)/\mathbb{Q})$. Therefore $i(\Delta) \otimes_{\mathbb{Z}} \mathbb{Q}$ is either equal to \begin{equation}\label{situ1}\mathrm{span}_{\mathbb{Q}}(\{v_{1},v_{4}\})=\mathrm{span}_{\mathbb{Q}}(\{(0,1,0,1),(1,0,1,0)\})\end{equation} or equal to \begin{equation}\label{situ2}\mathrm{span}_{\mathbb{Q}}(\{ v_2+v_3,\frac{v_2-v_{3}}{2i}\})=\mathrm{span}_{\mathbb{Q}}(\{(0,-1,0,1),(1,0,-1,0)\})\end{equation} Note that by construction $\gdrbhmath(A)$ is a connected torus, therefore $i(\Delta)$ satisfies the condition that $X^{*}(\mathrm{U}_{\mathrm{E}})/i(\Delta)$ is torsion free. Recall that we have fixed an integral basis for $X^{*}(\mathrm{U}_{\mathrm{E}})$ in Setup \ref{torusinvsetup}, hence by Lemma \ref{integralelement} elements in $i(\Delta)$ are of the form $(x,y,x,y)$ with $x,y \in \mathbb{Z}$. Therefore in situation (\ref{situ1}) we have that $(1,0,1,0) \in i(\Delta)$. Similarly in situation (\ref{situ2}), we have that both $(0,-1,0,1)$ and $(1,0,-1,0)$ lie in $i(\Delta)$. And both situations contradict Corollary \ref{divisortest}.
\item \label{computations for Case1M2}
Suppose $F^{-1}(g)$ contains a generalized permutation matrix with exactly one $-1$ among its entries i.e. $$M_{2}:=\begin{pmatrix} 0 & 1 & 0 & 0\\0 & 0 & 1 & 0\\0 & 0 & 0 & 1\\-1 & 0 & 0 & 0\end{pmatrix} \in F^{-1}(g)$$ We follow the same recipe as above. Diagonalizing $M_2$ we obtain that it has four distinct eigenvalues $\lambda_1=e^{\frac{\pi i}{4}}=\frac{\sqrt{2}}{2}+\frac{\sqrt{2}}{2}i, \lambda_2=e^{\frac{7\pi i}{4}}=\frac{\sqrt{2}}{2}-\frac{\sqrt{2}}{2}i, \lambda_3=-\lambda_2=-\frac{\sqrt{2}}{2}+\frac{\sqrt{2}}{2}i, \lambda_4=-\lambda_1=-\frac{\sqrt{2}}{2}-\frac{\sqrt{2}}{2}i$ with the corresponding eigenbasis: $v_1=(\lambda_4,-i,\lambda_2,1), v_2=(\lambda_3,i,\lambda_1,1), v_3=(\lambda_2,i,\lambda_4,1), v_4=(\lambda_1,-i,\lambda_3,1)$. Note that in this case the eigenvalues are also distinct. Hence $i(\Delta) \otimes_{\mathbb{Z}} \qbar$ is spanned by two eigenvectors associated with two distinct eigenvalues. Note the eigenspace decomposition descend to $X^{*}(\mathrm{U}_{E}) \otimes_{\mathbb{Z}}\mathbb{Q}(\frac{\sqrt{2}}{2},i)$. Denote $\mathbb{Q}(\frac{\sqrt{2}}{2},i)$ by $F$. Therefore $i(\Delta) \otimes_{\mathbb{Z}} F$ is simultaneously stable under the action of $M_2$ and under the action of the Galois group $\mathrm{Gal}(F/\mathbb{Q})$.  The Galois group $\mathrm{Gal}(F/\mathbb{Q}) \cong \mathbb{Z}/2 \times \mathbb{Z}/2$ has generators: $\alpha: \frac{\sqrt{2}}{2} \rightarrow -\frac{\sqrt{2}}{2}$ and $\beta:i \rightarrow -i$. Then the action of the Galois group on the eigenvectors is as follows: $\alpha \circ v_1=v_4, \beta \circ v_1=v_2; \alpha \circ v_2=v_3,  \beta \circ v_2=v_1; \alpha \circ v_3=v_2, \beta \circ v_3=v_4; \alpha \circ v_4=v_1, \beta \circ v_4=v_3$. Therefore no two dimensional subspaces of $X^{*}(\mathrm{U}_{E}) \otimes F$ stable under $M_2$ can descend to $\mathbb{Q}$. Other generalized permutation matrices contained in $F^{-1}(g)$ with exactly one $-1$ among its entries have similar eigenvalues and eigenspace decompositions. Therefore they can be excluded in a similar manner.

\item \label{computations for case1m3}
Suppose $F^{-1}(g)$ contains a generalized permutation matrix with exactly two $-1$s among its entries i.e. $$M_3:=\begin{pmatrix} 0 & 1 & 0 & 0\\0 & 0 & 1 & 0\\0 & 0 & 0 & -1\\-1 & 0 & 0 & 0\end{pmatrix} \in F^{-1}(g)$$ Diagonalizing $M_3$, it has the following distinct eigenvalues and eigenvectors:
$\lambda_1=-1,v_1=(1,-1,1,1); \lambda_2=i,v_2=(-i,1,i,1); \lambda_3=-i,v_3=(i,1,-i,1); \lambda_4=1,v_4=(-1,-1,-1,1)$. By reasons similar to the first case, $i(\Delta) \otimes_{\mathbb{Z}} \mathbb{Q}$ is either equal to $$\mathrm{span}_{\mathbb{Q}}(\{\frac{v_{1}+v_{4}}{2}=(0,-1,0,1), \frac{v_{1}-v_{4}}{2}=(1,0,1,0)\})$$ or equal to $$\mathrm{span}_{\mathbb{Q}}(\{\frac{v_2+v_3}{2}=(0,1,0,1),\frac{v_2-v_{3}}{2i}=(-1,0,1,0)\})$$ Again because $X^{*}(T) \cong X^{*}(\mathrm{U}_{E})/\Delta$ is torsion free and because of Lemma \ref{integralelement}, we have that either $$\{(0,-1,0,1),(1,0,1,0)\} \in i(\Delta)$$ or $$\{(0,1,0,1),(1,0,-1,0)\} \in i(\Delta)$$ Both cases contradict Corollary \ref{divisortest}. Other generalized permutation matrices contained in $F^{-1}(g)$ with exactly two $-1$s among their entries have similar eigenvalues and eigenspace decompositions. Hence they can be excluded similarly.
\end{enumerate}
By symmetry, the case where $(1324)\in H$ or $(1243)\in H$ can be dealt with in a similar way.
\end{proof}
\begin{proposition} \label{noklein4}
We use notations from Setup \ref{setupcontinued} and Lemma \ref{galois_permutation}. Suppose $T=\gdrbhmath(A)$ is a two-dimensional (connected) algebraic torus and the image $H$ of $G''$ in $\mathrm{S}_4$ is isomorphic to the Klein four-group. Then $G''$ is isomorphic to the dihedral group $\mathrm{D}_4$.
\end{proposition}
We make the following preparations for the proof of Proposition \ref{noklein4}.
\begin{setup}\label{klein4setup}
Using the same notation from Lemma \ref{galois_permutation}, we have that $$H=\{\mathrm{id},(12)(34),(13)(24),(14)(23)\}$$ and we denote $g_1=(12)(34)$, $g_2=(14)(23)$ and $g_3=(13)(24)$. Since the positions of the three order 2 elements in $H$ are symmetric and the element $-1$ lies in $G'' \subset \mathrm{Aut}_{\mathbb{Z}}(X^{*}(\mathrm{U}_{E}))$, it suffices to consider submodules $i(\Delta)\subset X^{*}(\mathrm{U}_{E})$ stable under the $G''$s matching one of the following descriptions, which are not mutually exclusive.

\begin{enumerate}\label{bigcases}
    \item\label{bigcase1} One of the matrices in $F^{-1}(g_1) \subset G''$ under the forgetful homomorphism $F:G''\rightarrow H$  contains exactly one $-1$ among its entries
    \item\label{bigcase2} One of the matrices in $F^{-1}(g_1) \subset G''$ is an ordinary permutation matrix
    \item\label{bigcase3} Every matrix in $F^{-1}(g_1)$ contains exactly two $-1$s among their entries 
\end{enumerate}
\end{setup}
Then we have the following lemma dealing with Case (\ref{bigcase1}) from Setup \ref{klein4setup}.
\begin{lemma}\label{bigcase1lemma}
 We keep the same notations of Proposition \ref{noklein4} and Setup \ref{klein4setup}. Furthermore, if one of the matrices in $F^{-1}(g_1) \subset G''$ contains exactly one $-1$ among its entries, then $\gdrbhmath(A)$ cannot be an (connected) algebraic torus of dimensional two.
\end{lemma}
\begin{proof}
Suppose the matrix $M:=\begin{pmatrix} 0 & -1 & 0 & 0\\1 & 0 & 0 & 0\\0 & 0 & 0 & 1\\0 & 0 & 1 & 0 \end{pmatrix} \in F^{-1}(g_1)$. Then $i(\Delta)$ from the $\mathrm{Gal}(\qbar/\mathbb{Q})$-equivariant short exact sequence (\ref{charseq}) is a rank two free $\mathbb{Z}$-module stable under the action of $M$. Now we perform a computation similar to Lemma \ref{nolength4cycle}. Diagonalizing the matrix $M$, it has four distinct eigenvalues $\lambda_1=i,\lambda_2=-i,\lambda_3=1,\lambda_4=-1$ and the corresponding eigenvectors are $v_1=(1,-i,0,0),v_2=(1,i,0,0),v_3=(0,0,1,1),v_4=(0,0,-1,1)$. Hence by a similar analysis to Lemma \ref{nolength4cycle} either $(0,0,1,1)$ or $(1,1,0,0)$ is contained in $i(\Delta)$ and both cases violate Corollary \ref{divisortest}. Because other generalized permutation matrices which contain exactly one $-1$ whose image under $F$ is equal to $g_1$ have similar eigenvalues and eigenspace decompositions to $M$, they can be excluded in a similar manner. Hence $\gdrbhmath(A)$ cannot be a two dimensional algebraic subtorus of $\mathrm{U}_{E}$.     
\end{proof}
The following lemma deals with Case (\ref{bigcase2}) from Setup \ref{klein4setup}.
\begin{lemma}\label{p_0inpreimagelemma}
We keep the assumptions of Proposition \ref{noklein4} and Setup \ref{klein4setup}. Furthermore, if one of the matrices in $F^{-1}(g_1) \subset G''$ is an ordinary permutation matrix and if $\gdrbhmath(A)$ is a (connected) two-dimensional subtorus of $\mathrm{U}_{E}$, then $G''$ is isomorphic to $\mathrm{D}_4:=\langle a,x |a^4=1;x^2=1;axa=x\rangle$.
\end{lemma}
\begin{proof}
 We claim it suffices to study rank two submodules of $X^{*}(\mathrm{U}_{E})$ stable under the action of a ordinary permutation matrix $$P_0:=\begin{pmatrix} 0 & 1 & 0 & 0\\1 & 0 & 0 & 0\\0 & 0 & 0 & 1\\0 & 0 & 1 & 0 \end{pmatrix}\in F^{-1}(g_1)$$ and preserved by at least one element from the following list of possible preimages of $g_2$ under the forgetful map $F$. 
\[\left
(
\begin{pmatrix} 0 & 0 & 0 & 1\\0 & 0 & 1 & 0\\0 & 1 & 0 & 0\\1 & 0 & 0 & 0 \end{pmatrix},
\begin{pmatrix} 0 & 0 & 0 & 1\\0 & 0 & -1 & 0\\0 & 1 & 0 & 0\\-1 & 0 & 0 & 0 \end{pmatrix},
\begin{pmatrix} 0 & 0 & 0 & 1\\0 & 0 & -1 & 0\\0 & -1 & 0 & 0\\1 & 0 & 0 & 0 \end{pmatrix},
\begin{pmatrix} 0 & 0 & 0 & -1\\0 & 0 & -1 & 0\\0 & 1 & 0 & 0\\1 & 0 & 0 & 0 \end{pmatrix}
\right)
\] This is because we can use Lemma \ref{bigcase1lemma} to exclude the scenario where one of the generalized permutation matrices in $F^{-1}(g_2)$ contains exactly one $-1$ among its entries or exactly three $-1$s among its entries. Label the above list of matrices by $(P_1,P_2,P_3,P_4)$. 

We first write down all possible $\mathbb{Q}$-linear extensions of rank 2 submodules of $X^{*}(\mathrm{U}_{E})$ stable under $P_{0}$. Diagonalizing $P_{0}$, viewed as a $\mathbb{Q}$-linear transformation on $X^{*}({\mathrm{U}_{E}}) \otimes_{\mathbb{Z}}\mathbb{Q}$, its eigenvalues and eigenvectors are: $\lambda_1=-1, v_1=(0,0,-1,1); \lambda_2=-1, v_2=(-1,1,0,0); \lambda_3=1, v_3=(0,0,1,1); \lambda_4=1, v_4=(1,1,0,0)$. Since $i(\Delta) \otimes_{\mathbb{Z}} \mathbb{Q}$ is stable under $P_0$, $P_0|_{i(\Delta) \otimes_{\mathbb{Z}} \mathbb{Q}}$ remains semisimple. Hence its set of eigenvalues on $i(\Delta) \otimes_{\mathbb{Z}} \mathbb{Q}$ is either $\{-1,1\}$ or $\{1\}$ or $\{-1\}$. Therefore possible basis for $i(\Delta) \otimes_{\mathbb{Z}} \mathbb{Q}$ are $\{x_1v_1+x_2v_2,y_1v_3+y_2v_4\}=\{(-x_1,x_1,-x_2,x_2),(y_1,y_1,y_2,y_2)\}(x_{i},y_{i} \in \mathbb{Q})$ or $\{v_1,v_2\}$ or $\{v_3,v_4\}$. Note that the $\mathbb{Q}$-span of $\{(-x_1,x_1,-x_2,x_2),(y_1,y_1,y_2,y_2)\}$ is equal to the $\mathbb{Q}$-span of $\{(-x'_1,x'_1,-x'_2,x'_2),(y'_1,y'_1,y'_2,y'_2)\}$ for some $x'_{i},y'_{i}\in\mathbb{Z}$. Thus for potential candidates of $i(\Delta)\otimes_{\mathbb{Z}}\mathbb{Q}$ which are stable under $P_0$, it suffices to consider $$\mathrm{span}_{\mathbb{Q}}(\{(-x_1,x_1,-x_2,x_2),(y_1,y_1,y_2,y_2)\})(x_{i},y_{i} \in \mathbb{Z})$$ or $$\mathrm{span}_{\mathbb{Q}}(\{(0,0,-x,x),(-y,y,0,0)\})(x,y \in \mathbb{Z})$$ or $$\mathrm{span}_{\mathbb{Q}}(\{(0,0,x,x),(y,y,0,0)\})(x,y \in \mathbb{Z})$$ In the latter two cases, because $X^{*}(\mathrm{U}_{E})/i(\Delta)$ is torsion free and $i(\Delta)=i(\Delta)\otimes_{\mathbb{Z}}\mathbb{Q}\cap X^{*}(\mathrm{U}_{E})$ by Lemma \ref{integralelement}, the submodule $i(\Delta)$ would in fact be spanned by $\{(0,0,-1,1),(-1,1,0,0)\}$ or $\{(0,0,1,1),(1,1,0,0)\}$. This contradicts Corollary \ref{divisortest}. Hence it suffices to consider the scenario where $i(\Delta)\otimes_{\mathbb{Z}}\mathbb{Q}$ is spanned by $$\{x_1v_1+x_2v_2,y_1v_3+y_2v_4\}=\{(-x_1,x_1,-x_2,x_2),(y_1,y_1,y_2,y_2)\}(x_{i},y_{i} \in \mathbb{Z})$$ 

By the analysis in the beginning of the proof, we now consider the scenario where $i(\Delta)$ is preserved by $P_0$ and one of the elements from $\{P_1,P_2,P_3,P_4\}$. We start by summarizing the results for the relevant computations done below. Suppose $P_0, P_1$ are contained in $G''$. By computations done in (\ref{p_0p_1}), one can conclude such $i(\Delta)$ stable under $G''$ contradicts Corollary \ref{divisortest}. Suppose $P_0, P_2$ are contained in $G''$. By Case (\ref{p_0p_2}), submodules spanned by $$\{(-x_1,x_1,-x_2,x_2),(x_2,x_2,x_1,x_1)\}$$ where $x_1,x_2\in\mathbb{Z}$ are stable under the action of $P_0$ and $P_2$. Thus upon close inspections, for suitably chosen $x_1,x_2\in\mathbb{Z}$, such modules cannot be ruled out using Corollary \ref{divisortest}. Similarly, if $P_0, P_3$ are contained in $G''$, then by Case (\ref{p_0p_3}), submodules spanned by $$\{(-x_1,x_1,-x_2,x_2),(x_2,x_2,-x_1,-x_1)\}$$ where $x_1,x_2\in\mathbb{Z}$ are stable under the action of $P_0$ and $P_3$. For suitably chosen $x_1,x_2\in\mathbb{Z}$, they cannot be ruled out using Corollary \ref{divisortest} either. By the isomorphism (\ref{p0p2d4}) and isomorphism (\ref{p0p3d4}), the group generated by $P_0,P_2$ or $P_0,P_3$ is isomorphic to $\mathrm{D}_4$, the dihedral group with 8 elements. Finally, if $P_0, P_4$ are contained in $G''$, then by Case (\ref{p_0p_4}), there does not exist a rank two submodule of $X^{*}(\mathrm{U}_{E})$ stable under both $P_0,P_4$, thus giving a contradiction.

In the two cases where $P_0, P_2$ or $P_0, P_3$ are contained in $G''$, we will further consider branch cases where $F^{-1}(g_3)$ contains elements other than $\pm P_0P_2$ or $\pm P_0P_3$. Using Lemma \ref{bigcase1lemma}, we need not consider the case where $F^{-1}(g_3)$ contains a matrix which has exactly one $-1$ or three $-1$s among its entries.

In the case where $P_0\in F^{-1}(g_1)$ and $P_2\in F^{-1}(g_2)$, label elements in $F^{-1}(g_3)$ which differ from $\pm P_0P_2$ and contain exactly two $-1$s or no $-1$ among its entries by $$(Q_{1},Q_2,Q_3):=(\begin{pmatrix} 0 & 0 & -1 & 0\\0 & 0 & 0 & -1\\1 & 0 & 0 & 0\\0 & 1 & 0 & 0\end{pmatrix}, \begin{pmatrix} 0 & 0 & -1 & 0\\0 & 0 & 0 & 1\\1 & 0 & 0 & 0\\0 & -1 & 0 &0\end{pmatrix},\begin{pmatrix} 0 & 0 & 1 & 0\\0 & 0 & 0 & 1\\1 & 0 & 0 & 0\\0 & 1 & 0 &0\end{pmatrix})$$

In the case where $P_0\in F^{-1}(g_1)$ and $P_3\in F^{-1}(g_2)$, label elements in $F^{-1}(g_3)$ which differ from $\pm P_0P_3$ and contain exactly two $-1$s or no $-1$ among its entries by $$(Q'_{1},Q'_2,Q'_3):=(\begin{pmatrix} 0 & 0 & -1 & 0\\0 & 0 & 0 & -1\\1 & 0 & 0 & 0\\0 & 1 & 0 & 0\end{pmatrix}, \begin{pmatrix} 0 & 0 & -1 & 0\\0 & 0 & 0 & 1\\-1 & 0 & 0 & 0\\0 & 1 & 0 &0\end{pmatrix},\begin{pmatrix} 0 & 0 & 1 & 0\\0 & 0 & 0 & 1\\1 & 0 & 0 & 0\\0 & 1 & 0 &0\end{pmatrix})$$
The computations of these branch cases are summarized in the following table.
\begin{table}[H]
    \centering
    \begin{tabular}{|c|c|c|}
    \hline
    Elements in $G''$  & Properties of $i(\Delta)$ & Can we rule out such $T$ \\
    \hline
    $P_0, P_2, Q_1$ & Can only be of rank 0 & Yes, See (\ref{p_0p_2q_1})\\
    \hline
    $P_0, P_2, Q_2$ & Contradicts Corollary \ref{divisortest} & Yes, See (\ref{p_0p_2q_2})\\
    \hline
    $P_0, P_2, Q_3$ & Similar to (\ref{p_0p_1})  & Yes, See (\ref{p_0p_2q_3}) \\
    \hline  
    $P_0, P_3, Q'_1$ & Can only be of rank 0 & Yes, See (\ref{p_0p_3q'_1})\\
    \hline  
    $P_0, P_3, Q'_2$ & Contradicts Corollary \ref{divisortest} & Yes, See (\ref{p_0p_3q'_2})\\
    \hline
    $P_0, P_3, Q'_3$ & Similar to (\ref{p_0p_1}) & Yes, See (\ref{p_0p_3q'_3})\\
    \hline
    \end{tabular}
    \caption{Branch Cases of (\ref{p_0p_2}) and (\ref{p_0p_3})}
    \label{p_0table2}
\end{table}
Below are the detailed computations of the above discussion.

\begin{enumerate}
    \item\label{p_0p_1} Suppose the vector space $\mathrm{span}_{\mathbb{Q}}(\{(-x_1,x_1,-x_2,x_2),(y_1,y_1,y_2,y_2)\})$ preserved by $P_0$ is also preserved by the matrix \[P_{1}=
\begin{pmatrix} 0 & 0 & 0 & 1\\0 & 0 & 1 & 0\\0 & 1 & 0 & 0\\1 & 0 & 0 & 0 \end{pmatrix}\] Then $P_1\circ(-x_1,x_1,-x_2,x_2)=(x_2,-x_2,x_1,-x_1); P_1\circ(y_1,y_1,y_2,y_2)=(y_2,y_2,y_1,y_1)$ and hence there exist $\alpha, \beta, \alpha^{'}, \beta^{'} \in \mathbb{Q}$ such that  
\begin{equation*}
\begin{split}
-\alpha x_1+\beta y_1 =x_2;  -\alpha^{'} x_1+\beta^{'} y_1 =y_2 \\
\alpha x_1+\beta y_1 =-x_2;   \alpha^{'} x_1+\beta^{'} y_1 =y_2 \\
-\alpha x_2+\beta y_2 =x_1;   -\alpha^{'} x_2+\beta^{'} y_2 =y_1\\
\alpha x_2+\beta y_2 =-x_1;    \alpha^{'} x_2+\beta^{'} y_2 =y_1
\end{split}
\end{equation*}

Then the left column implies that: $x_1^2=x_2^2$ and the right column implies that $y_1^2=y_2^2$. Hence we have $x_1=\pm x_2$ and $y_1=\pm y_2$. Since $i(\Delta)=i(\Delta)\otimes_{\mathbb{Z}}\mathbb{Q}\cap X^{*}(\mathrm{U}_{E})$ by Lemma \ref{integralelement} this implies $i(\Delta)$ contains $(0,2x_1y_1,0,2x_1y_1)$ or $(0,2x_1y_1,-2x_1y_1,0)$ or $(0,2x_1y_1,2x_1y_1,0)$ or $(0,2x_1y_1,0,-2x_1y_1)$. But $ X^{*}(\mathrm{U}_{E})/i(\Delta)$ is torsion free, this implies $(0,1,0,1)$ or $(0,1,-1,0)$ or $(0,1,1,0)$ or $(0,1,0,-1)\in i(\Delta)$. This gives a contradiction to Corollary \ref{divisortest}. 
\item\label{p_0p_2} Suppose $\mathrm{span}_{\mathbb{Q}}(\{(-x_1,x_1,-x_2,x_2),(y_1,y_1,y_2,y_2)\})$ is preserved by the matrix \[P_{2}=
\begin{pmatrix} 0 & 0 & 0 & 1\\0 & 0 & -1 & 0\\0 & 1 & 0 & 0\\-1 & 0 & 0 & 0 \end{pmatrix}\]Then $P_2\circ(-x_1,x_1,-x_2,x_2)=(x_2,x_2,x_1,x_1)$ and $P_2\circ(y_1,y_1,y_2,y_2)=(y_2,-y_2,y_1,-y_1)$. Thus there exist $\alpha, \beta, \alpha^{'}, \beta^{'} \in \mathbb{Q}$ such that
\begin{equation*}
\begin{split}
-\alpha x_1+\beta y_1 =-x_2;  -\alpha^{'} x_1+\beta^{'} y_1 =-y_2 \\
\alpha x_1+\beta y_1 =-x_2;   \alpha^{'} x_1+\beta^{'} y_1 =y_2 \\
-\alpha x_2+\beta y_2 =-x_1;   -\alpha^{'} x_2+\beta^{'} y_2 =-y_1\\
\alpha x_2+\beta y_2 =-x_1;    \alpha^{'} x_2+\beta^{'} y_2 =y_1
\end{split}
\end{equation*} the above simplifies to 
\begin{equation*}
\begin{split}
\beta y_1 =-x_2;  \alpha^{'} x_1=y_2 \\
\beta y_2 =-x_1;  \alpha^{'} x_2=y_1              \\  
\end{split}
\end{equation*}
 Therefore every two-dimensional vector subspace of $X^{*}(\mathrm{U}_{E})\otimes_{\mathbb{Z}}\mathbb{Q}$ stable under the action of $P_0$ and $P_2$ is of the form $$\mathrm{span}_{\mathbb{Q}}(\{(-x_1,x_1,-x_2,x_2),(y_1,y_1,y_2,y_2)\})$$ where $x_{j},y_{j}$s satisfy the relation \begin{equation}\label{firstkeyrelation}x_1y_1=x_2y_2\end{equation}
In particular if we let $y_1=x_2$, $y_2=x_1$, one can check that there exist $x_1,x_2\in\mathbb{Z}$ such that $X^{*}(\mathrm{U}_{E})/i(\Delta)$ is torsion free and no ``divisorial" elements from Corollary \ref{divisortest} appear in $i(\Delta)$.

Suppose $F^{-1}(g_1)=\pm P_0$ and $F^{-1}(g_2)=\pm P_2$ and $F^{-1}(g_3)=\pm P_0P_2$(note that $P_0P_2=-P_2P_0$) i.e. $P_0 \in F^{-1}(g_1)$ and $P_2 \in F^{-1}(g_2)$ and $\mathrm{Ker}(F)=\{\pm 1\}$. Then simply from Corollary \ref{divisortest}, one cannot exclude the scenario where $$G''=\{\mathrm{id}, -\mathrm{id}, P_0, -P_0, P_2, -P_2, P_0P_2,-P_0P_2\}$$ and $\gdrbhmath(A)$ is a two-dimensional subtorus of $\mathrm{U}_{E}$ such that $$i(\Delta)=\mathrm{Ker}(X^{*}(\mathrm{U}_{E})\twoheadrightarrow X^{*}(\gdrbhmath(A)))$$ is of the form specified above.

However, in this case the group $G''$ is indeed isomorphic to the dihedral group $\mathrm{D}_4=\langle a,x |a^4=1;x^2=1;axa=x\rangle$ via the group isomorphism \begin{equation}\label{p0p2d4}x=P_{0}=\begin{pmatrix} 0 & 1 & 0 & 0\\1 & 0 & 0 & 0\\0 & 0 & 0 & 1\\0 & 0 & 1 & 0 \end{pmatrix}; a=P_2=\begin{pmatrix} 0 & 0 & 0 & 1\\0 & 0 & -1 & 0\\0 & 1 & 0 & 0\\-1 & 0 & 0 & 0 \end{pmatrix}\end{equation}

If, on the other hand, $F^{-1}(g_3)$ contains elements other than $\pm P_0P_2$ i.e. $P_0 \in F^{-1}(g_1)$ and $P_2 \in F^{-1}(g_2)$ but $|\mathrm{Ker}(F)|>2$. In this setting, $G''$ will not be isomorphic to $\mathrm{D}_4$. However, the presence of extra elements in $F^{-1}(g_3)$ will put more conditions on 
$i(\Delta)$. It suffices to consider the branch cases where one of the following elements is contained in $F^{-1}(g_3)$
\begin{equation}\label{q_1q_2q_3}
    (Q_{1},Q_2,Q_3):=(\begin{pmatrix} 0 & 0 & -1 & 0\\0 & 0 & 0 & -1\\1 & 0 & 0 & 0\\0 & 1 & 0 & 0\end{pmatrix}, \begin{pmatrix} 0 & 0 & -1 & 0\\0 & 0 & 0 & 1\\1 & 0 & 0 & 0\\0 & -1 & 0 &0\end{pmatrix},\begin{pmatrix} 0 & 0 & 1 & 0\\0 & 0 & 0 & 1\\1 & 0 & 0 & 0\\0 & 1 & 0 &0\end{pmatrix})
\end{equation}
This is because by the computations in Lemma \ref{bigcase1lemma}, we have ruled out the case where the preimage of one of the $g_{i}$s contains a generalized permutation matrix with exactly one $-1$ or three $-1$s among its entries.

The detailed computations of these branch cases are listed below.
\begin{enumerate}
\item\label{p_0p_2q_1} Suppose $F^{-1}(g_3)$ contains $$Q_{1}:=\begin{pmatrix} 0 & 0 & -1 & 0\\0 & 0 & 0 & -1\\1 & 0 & 0 & 0\\0 & 1 & 0 & 0\end{pmatrix}$$ then $Q_{1}$ has to preserve the vector space $$\mathrm{span}_{\mathbb{Q}}(\{(-x_1,x_1,-x_2,x_2),(y_1,y_1,y_2,y_2)\})$$ Therefore, there exist $\alpha,\beta,\alpha',\beta' \in \mathbb{Q}$ such that \begin{equation*}\label{toplefttwominus1}
\begin{split}
-\alpha x_1+\beta y_1 =x_2;  -\alpha^{'} x_1+\beta^{'} y_1 =-y_2 \\
\alpha x_1+\beta y_1 =-x_2;   \alpha^{'} x_1+\beta^{'} y_1 =-y_2 \\
-\alpha x_2+\beta y_2 =-x_1;   -\alpha^{'} x_2+\beta^{'} y_2 =y_1\\
\alpha x_2+\beta y_2 =x_1;    \alpha^{'} x_2+\beta^{'} y_2 =y_1
\end{split}
\end{equation*} from which we deduce that \begin{equation*}
\begin{split}
\\\alpha  x_1 =-x_2;  \beta ^{'} y_1=-y_2 \\
\alpha  x_2 =x_1;  \beta ^{'} y_2=y_1              \\ \\   
\end{split}
\end{equation*} Since $x_1,x_2,y_1,y_2\in\mathbb{Q}$, this implies that $i(\Delta)=\{0\}$, which is a contradiction. 

\item\label{p_0p_2q_2} If $F^{-1}(g_3)$ contains $$Q_{2}:=\begin{pmatrix} 0 & 0 & -1 & 0\\0 & 0 & 0 & 1\\1 & 0 & 0 & 0\\0 & -1 & 0 &0\end{pmatrix}$$ then $Q_{2}$ has to preserve the vector space $$\mathrm{span}_{\mathbb{Q}}(\{(-x_1,x_1,-x_2,x_2),(y_1,y_1,y_2,y_2)\})$$ Hence there exist $\alpha,\beta,\alpha',\beta' \in \mathbb{Q}$ such that \begin{equation*}
\begin{split}
-\alpha x_1+\beta y_1 =x_2;  -\alpha^{'} x_1+\beta^{'} y_1 =-y_2 \\
\alpha x_1+\beta y_1 =x_2;   \alpha^{'} x_1+\beta^{'} y_1 =y_2 \\
-\alpha x_2+\beta y_2 =-x_1;   -\alpha^{'} x_2+\beta^{'} y_2 =y_1\\
\alpha x_2+\beta y_2 =-x_1;    \alpha^{'} x_2+\beta^{'} y_2 =-y_1
\end{split}
\end{equation*} from which we deduce that \begin{equation*}
\begin{split}
\\\beta  y_1 =x_2;  \alpha^{'} x_1=y_2 \\
\beta  y_2 =-x_1;  \alpha^{'} x_2=-y_1              \\ \\   
\end{split}
\end{equation*} We then deduce that $x_1y_1=-x_2y_2$. But recall that $i(\Delta)_{\mathbb{Q}}$ is stable under both $P_0$ and $P_2$ which implies that $x_1y_1=x_2y_2$ (see formula (\ref{firstkeyrelation})). Hence this implies that $x_1y_1=x_2y_2=0$. Since we require that $i(\Delta)$ is of rank two, this implies that either $x_1=y_2=0$ or $x_2=y_1=0$. Hence by Lemma \ref{integralelement} and the requirement that $X^{*}({\mathrm{U}_{E}})/i(\Delta)$ is torsion free, we deduce that $i(\Delta)$ either contains $(0,0,-1,1)$ or contains $(-1,1,0,0)$, which gives a contradiction to Corollary \ref{divisortest}. 
\item\label{p_0p_2q_3} If $F^{-1}(g_3)$ contains $$Q_{3}:=\begin{pmatrix} 0 & 0 & 1 & 0\\0 & 0 & 0 & 1\\1 & 0 & 0 & 0\\0 & 1 & 0 &0\end{pmatrix}$$ then the submodule $i(\Delta)$ will be stable under the action of $P_0$ and $Q_3$, both of which contain no -1 among their entries. However, recall that previously we have excluded the case where the submodule $i(\Delta)$ is stable under the action of $P_0$ and $P_1$, both of which contain no -1 among their entries. Because the position of the order two elements in $H$ are symmetric, this case can also be excluded.

\end{enumerate}
Hence we have excluded all cases where $F^{-1}(g_3)$ contains elements other than $\pm P_0P_2$ or $\pm P_2P_0$.
To sum it up, suppose $H$ is isomorphic to the Klein four-group in $\mathrm{S}_4$ and $P_0\in F^{-1}(g_1)$ and $P_2\in F^{-1}(g_2)$, then $\gdrbhmath(A)$ cannot be a two-dimensional algebraic torus unless $G''$ is isomorphic to $\mathrm{D}_4$.

\item\label{p_0p_3} Now suppose the vector space $\mathrm{span}_{\mathbb{Q}}(\{(-x_1,x_1,-x_2,x_2),(y_1,y_1,y_2,y_2)\})$ preserved by $P_0$ is also preserved by the matrix \[P_{3}=
\begin{pmatrix} 0 & 0 & 0 & -1\\0 & 0 & 1 & 0\\0 & 1 & 0 & 0\\-1 & 0 & 0 & 0 \end{pmatrix}\] Then $P_3\circ(-x_1,x_1,-x_2,x_2)=(-x_2,-x_2,x_1,x_1); P_3\circ(y_1,y_1,y_2,y_2)=(-y_2,y_2,y_1,-y_1)$ and hence there exist $\alpha, \beta, \alpha^{'}, \beta^{'} \in \mathbb{Q}$ such that  
\begin{equation}
\begin{split}
-\alpha x_1+\beta y_1 =-x_2;  -\alpha^{'} x_1+\beta^{'} y_1 =-y_2 \\
\alpha x_1+\beta y_1 =-x_2;   \alpha^{'} x_1+\beta^{'} y_1 =y_2 \\
-\alpha x_2+\beta y_2 =x_1;   -\alpha^{'} x_2+\beta^{'} y_2 =y_1\\
\alpha x_2+\beta y_2 =x_1;    \alpha^{'} x_2+\beta^{'} y_2 =-y_1
\end{split}
\end{equation} from which we deduce that 
\begin{equation*}
    \begin{split}
        \beta y_1=-x_2; \alpha'x_1=y_2\\
        \beta y_2=x_1;  \alpha'x_2=-y_1
    \end{split}
\end{equation*}

Therefore $\mathrm{span}_{\mathbb{Q}}(\{(-x_1,x_1,-x_2,x_2),(y_1,y_1,y_2,y_2)\})$ is preserved by $P_3$ if and only if 
\begin{equation}\label{p3}
    x_1y_1=-x_2y_2
\end{equation} 

Similar to Case (\ref{p_0p_2}), if $F^{-1}(g_1)=\pm P_0$, $F^{-1}(g_2)=\pm P_3$ and $F^{-1}(g_3)=\pm P_0P_3$, then $G''$ in this case is equal to $$\{\mathrm{id}, -\mathrm{id}, P_0, -P_0, P_3, -P_3, P_0P_3,-P_0P_3\}$$ It is isomorphic to $\mathrm{D}_4$ via the group isomorphism \begin{equation}\label{p0p3d4}
    a=P_3P_0;x=P_0
\end{equation}

Now suppose $F^{-1}(g_3)$ contains elements other than $\pm P_0P_3$ i.e. the case where $P_0 \in F^{-1}(g_1)$, $P_3 \in F^{-1}(g_2)$ but $|\mathrm{Ker}(F)|>2$. Then analogous to the above Case (\ref{p_0p_2q_1}),Case (\ref{p_0p_2q_2}) and Case (\ref{p_0p_2q_3}) it suffices to consider the following cases. 
\begin{enumerate}
    \item\label{p_0p_3q'_1} Suppose $F^{-1}(g_3)$ contains $$Q'_{1}:=Q_{1}=\begin{pmatrix} 0 & 0 & -1 & 0\\0 & 0 & 0 & -1\\1 & 0 & 0 & 0\\0 & 1 & 0 & 0\end{pmatrix}$$ then $Q'_{1}$ has to preserve the vector space $$\mathrm{span}_{\mathbb{Q}}(\{(-x_1,x_1,-x_2,x_2),(y_1,y_1,y_2,y_2)\})$$ Then by the same computation as Case (\ref{toplefttwominus1}), there does not exist a two-dimensional $\mathbb{Q}$-vector space stable under $P_0,P_3,Q'_{1}$. 
    
\item\label{p_0p_3q'_2} Next suppose $F^{-1}(g_3)$ contains $$Q'_{2}:=\begin{pmatrix} 0 & 0 & -1 & 0\\0 & 0 & 0 & 1\\-1 & 0 & 0 & 0\\0 & 1 & 0 & 0\end{pmatrix}$$ Then since $Q'_{2}$ has to preserve the vector space $$\mathrm{span}_{\mathbb{Q}}(\{(-x_1,x_1,-x_2,x_2),(y_1,y_1,y_2,y_2)\})$$ there exist $\alpha,\beta,\alpha',\beta' \in \mathbb{Q}$ such that \begin{equation*}
\begin{split}
-\alpha x_1+\beta y_1 =x_2;  -\alpha^{'} x_1+\beta^{'} y_1 =-y_2 \\
\alpha x_1+\beta y_1 =x_2;   \alpha^{'} x_1+\beta^{'} y_1 =y_2 \\
-\alpha x_2+\beta y_2 =x_1;   -\alpha^{'} x_2+\beta^{'} y_2 =-y_1\\
\alpha x_2+\beta y_2 =x_1;    \alpha^{'} x_2+\beta^{'} y_2 =y_1
\end{split}
\end{equation*} from which we deduce that \begin{equation*}
\begin{split}
\\\beta  y_1 =x_2;  \alpha^{'} x_1=y_2 \\
\beta  y_2 =x_1;  \alpha^{'} x_2=y_1              \\ \\   
\end{split}
\end{equation*} which implies that $x_1y_1=x_2y_2$. But recall that $i(\Delta)_{\mathbb{Q}}$ has to be preserved by $P_0$ and $P_3$, and we required that $x_1y_1=-x_2y_2$ (see formula (\ref{p3})). Then by the same reasoning as Case (\ref{p_0p_2q_2}), this gives a contradiction to Corollary \ref{divisortest}. 
\item\label{p_0p_3q'_3} If $F^{-1}(g_3)$ contains $$Q'_{3}:=Q_3=\begin{pmatrix} 0 & 0 & 1 & 0\\0 & 0 & 0 & 1\\1 & 0 & 0 & 0\\0 & 1 & 0 &0\end{pmatrix}$$ then $i(\Delta)$ is stable under the action of $P_0$ and $Q_3$ simultaneously, both of which are ordinary permutation matrices. But then we can directly apply the reasoning from Case (\ref{p_0p_2q_3}) to exclude this case.
\end{enumerate}
To sum it up, suppose $H$ is isomorphic to the Klein four-group in $\mathrm{S}_4$ and $P_0\in F^{-1}(g_1)$ and $P_3\in F^{-1}(g_3)$, then unless $G''=F^{-1}(H)$ is isomorphic to $\mathrm{D}_4$, $\gdrbhmath(A)$ cannot be a two-dimensional algebraic torus . 

\item\label{p_0p_4} Suppose that $\mathrm{span}_{\mathbb{Q}}(\{(-x_1,x_1,-x_2,x_2),(y_1,y_1,y_2,y_2)\})$ preserved by $P_0$ is stable under the action of the matrix \[P_{4}=
\begin{pmatrix} 0 & 0 & 0 & -1\\0 & 0 & -1 & 0\\0 & 1 & 0 & 0\\1 & 0 & 0 & 0 \end{pmatrix}\]Then $$P_4\circ(-x_1,x_1,-x_2,x_2) =(-x_2,x_2,x_1,-x_1)$$ and $$P_4\circ(y_1,y_1,y_2,y_2)=(-y_2,-y_2,y_1,y_1)$$ Hence there exist $\alpha, \beta, \alpha^{'}, \beta^{'} \in \mathbb{Q}$ such that
\begin{equation*}
\begin{split}
-\alpha x_1+\beta y_1 =-x_2;  -\alpha^{'} x_1+\beta^{'} y_1 =-y_2 \\
\alpha x_1+\beta y_1 =x_2;   \alpha^{'} x_1+\beta^{'} y_1 =-y_2 \\
-\alpha x_2+\beta y_2 =x_1;   -\alpha^{'} x_2+\beta^{'} y_2 =y_1\\
\alpha x_2+\beta y_2 =-x_1;    \alpha^{'} x_2+\beta^{'} y_2 =y_1
\end{split}
\end{equation*} from which we deduce that
\begin{equation*}
    \begin{split}
        \alpha x_1=x_2; \beta' y_1=-y_2\\
        \alpha x_2=-x_1;
        \beta' y_2=y_1\\
    \end{split}
\end{equation*} Because $x_1,x_2,y_1,y_2\in\mathbb{Q}$ the only possible solution is that $x_1=x_2=y_1=y_2=0$ which is a contradiction. 

\end{enumerate}   
\end{proof}
The following lemma deals with Case (\ref{bigcase3}) from Setup \ref{klein4setup}.
\begin{lemma}\label{bigcase3lemma}
 We keep the assumptions of Proposition \ref{noklein4} and Setup \ref{klein4setup}. If every matrix in $F^{-1}(g_1)$ contains exactly two $-1$s among their entries, then $\gdrbhmath(A)$ cannot be a two-dimensional (connected) algebraic torus.
\end{lemma}
\begin{proof}
Suppose each element in $F^{-1}(g_1)$ is a generalized permutation matrix with exactly two $-1$s among its entries. Recall that in Lemma \ref{bigcase1lemma} we have ruled out the scenario where $F^{-1}(g_1)$ contains a matrix with exactly \textit{one} $-1$ among its entries and in Lemma \ref{p_0inpreimagelemma}, we have considered all situations where $F^{-1}(g_1)$ contains an ordinary permutation matrix. Then due to the symmetry of the position of $g_1,g_2,g_3$ in $H$ we can assume that every element in $F^{-1}(g_2)$ and $F^{-1}(g_3)$ is a generalized permutation matrices with exactly \textit{two} $-1$s among their entries. Label the following elements which could potentially lie in $F^{-1}(g_1)$ $$(\begin{pmatrix} 0 & -1 & 0 & 0\\-1 & 0 & 0 & 0\\0 & 0 & 0 & 1\\0 & 0 & 1 & 0 \end{pmatrix}, \begin{pmatrix} 0 & -1 & 0 & 0\\1 & 0 & 0 & 0\\0 & 0 & 0 & -1\\0 & 0 & 1 & 0 \end{pmatrix}, \begin{pmatrix} 0 & -1 & 0 & 0\\1 & 0 & 0 & 0\\0 & 0 & 0 & 1\\0 & 0 & -1 & 0 \end{pmatrix})$$ by $(P_0',P_1',P_2')$. Also label the following elements which could potentially lie in $F^{-1}(g_2)$ $$(\begin{pmatrix} 0 & 0 & 0 & 1\\0 & 0 & -1 & 0\\0 & -1 & 0 & 0\\1 & 0 & 0 & 0 \end{pmatrix}, \begin{pmatrix} 0 & 0 & 0 & 1\\0 & 0 & -1 & 0\\0 & 1 & 0 & 0\\-1 & 0 & 0 & 0 \end{pmatrix}, \begin{pmatrix} 0 & 0 & 0 & -1\\0 & 0 & -1 & 0\\0 & 1 & 0 & 0\\1 & 0 & 0 & 0 \end{pmatrix})$$ by $(Q_0',Q_1',Q_2')$. The computations are summarized in the following Table \ref{tableeachtwo-1}.

\begin{table}[H]
    \centering
    \begin{tabular}{|c|c|c|}
    \hline
    Elements in $G''$  & Properties of $i(\Delta)$ & Can we rule out such $T$ \\
    \hline
    $P_{0}', Q_{0}'$ & Contradicts Corollary \ref{divisortest} & Yes, See (\ref{p'_0q'_0})\\
    \hline
    $P_{0}', Q_{1}'$ & Can only be of rank 0 & Yes, See (\ref{p'_0q'_1})\\
    \hline
    $P_{1}', Q_{0}'$ & Can only be of rank 0 & Yes, See (\ref{p'_1q'_0})\\
    \hline
    $P_1',Q_2'$ &Can only be of rank 0& Yes, see (\ref{p'_1q'_2})\\
    \hline
    $P_2',Q_2'$ &Contradicts Corollary \ref{divisortest} & Yes, see (\ref{p'_2q'_2})\\
    \hline
    $P_2',Q_1'$ &Can only be of rank 0& Yes, see (\ref{p'_2q'_1})\\
    \hline
    \end{tabular}
    \caption{Results for Lemma \ref{bigcase3lemma}}
    \label{tableeachtwo-1}
\end{table}
Below are the detailed computations summarized in Table \ref{tableeachtwo-1}.
\begin{enumerate}
\item\label{p'_0q'_0} The preimage of $g_1$ contains $P_{0}'$ and the preimage of $g_2$ contains $Q_{0}'$ i.e.  
$$P_{0}'=\begin{pmatrix} 0 & -1 & 0 & 0\\-1 & 0 & 0 & 0\\0 & 0 & 0 & 1\\0 & 0 & 1 & 0 \end{pmatrix}\in F^{-1}(g_1); Q_{0}'=\begin{pmatrix} 0 & 0 & 0 & 1\\0 & 0 & -1 & 0\\0 & -1 & 0 & 0\\1 & 0 & 0 & 0 \end{pmatrix}\in F^{-1}(g_2)$$

Diagonalizing $P_0'$, it has eigenvalue $\lambda_1=1$ whose eigenspace is spanned by $(0,0,1,1)$ and $(-1,1,0,0)$ and eigenvalue $\lambda_2=-1$ whose eigenspace is spanned by $(0,0,-1,1)$ and $(1,1,0,0)$.  Recall that $i(\Delta)_{\mathbb{Q}}$ is the kernel of the surjective homomorphism $X^{*}(\mathrm{U}_{E})\rightarrow X^{*}(\gdrbhmath(A))$. Then by the same argument for Lemma \ref{p_0inpreimagelemma}, if $i(\Delta)_{\mathbb{Q}}$ is equal to the eigenspace associated with $\lambda_1$ or $\lambda_2$, then one can check this is a contradiction to Corollary \ref{divisortest}. Hence $i(\Delta)_{\mathbb{Q}}$ is equal to \begin{equation}\label{p_0'space}\mathrm{span}_{\mathbb{Q}}(\{(-x_1,x_1,x_2,x_2),(y_1,y_1,-y_2,y_2)\})\end{equation} for some $x_{i},y_{i} \in \mathbb{Z}$. The condition that  $i(\Delta)_{\mathbb{Q}}$ is stable under the action of $Q_{0}'$ indicates that there exist $\alpha, \beta, \alpha^{'}, \beta^{'} \in \mathbb{Q}$ such that
\begin{equation*}
\begin{split}
-\alpha x_1+\beta y_1 =x_2;  -\alpha^{'} x_1+\beta^{'} y_1 =y_2 \\
\alpha x_1+\beta y_1 =-x_2;   \alpha^{'} x_1+\beta^{'} y_1 =y_2 \\
\alpha x_2-\beta y_2 =-x_1;   \alpha^{'} x_2-\beta^{'} y_2 =-y_1\\
\alpha x_2+\beta y_2 =-x_1;    \alpha^{'} x_2+\beta^{'} y_2 =y_1
\end{split}
\end{equation*} from which we deduce that
\begin{equation*}
    \begin{split}
        \alpha x_1=-x_2; \beta' y_1=y_2\\
        \alpha x_2=-x_1;
        \beta' y_2=y_1\\
    \end{split}
\end{equation*} and we obtain that $x_1^2=x_2^2$ and $y_1^2=y_2^2$. But then the same argument as (\ref{p_0p_1}) from the proof of Lemma \ref{bigcase1lemma} can be applied to rule out this case.
\item\label{p'_0q'_1} The preimage of $g_1$ contains $P_{0}'$ and the preimage of $g_2$ contains $Q_{1}'$. i.e. $$P_{0}'=\begin{pmatrix} 0 & -1 & 0 & 0\\-1 & 0 & 0 & 0\\0 & 0 & 0 & 1\\0 & 0 & 1 & 0 \end{pmatrix}\in F^{-1}(g_1); Q_{1}'=\begin{pmatrix} 0 & 0 & 0 & 1\\0 & 0 & -1 & 0\\0 & 1 & 0 & 0\\-1 & 0 & 0 & 0 \end{pmatrix}\in F^{-1}(g_2)$$ Then by computations done in (\ref{p'_0q'_0}), it suffices to consider subspaces of the form $$\mathrm{span}_{\mathbb{Q}}(\{(-x_1,x_1,x_2,x_2),(y_1,y_1,-y_2,y_2)\})$$ from formula (\ref{p_0'space}) which are stable under the action $Q_1'$. Thus there exist $\alpha, \beta, \alpha^{'}, \beta^{'} \in \mathbb{Q}$ such that
\begin{equation*}
\begin{split}
-\alpha x_1+\beta y_1 =x_2;  -\alpha^{'} x_1+\beta^{'} y_1 =y_2 \\
\alpha x_1+\beta y_1 =-x_2;   \alpha^{'} x_1+\beta^{'} y_1 =y_2 \\
\alpha x_2-\beta y_2 =x_1;   \alpha^{'} x_2-\beta^{'} y_2 =y_1\\
\alpha x_2+\beta y_2 =x_1;    \alpha^{'} x_2+\beta^{'} y_2 =-y_1
\end{split}
\end{equation*} from which we deduce that
\begin{equation*}
    \begin{split}
        \alpha x_1=-x_2; \beta' y_1=y_2\\
        \alpha x_2=x_1;
        \beta' y_2=-y_1\\
    \end{split}
\end{equation*} Because $x_{i},y_{i}\in\mathbb{Q}$, this holds only when $x_1=x_2=y_1=y_2=0$, which gives a contradiction. 
\item\label{p'_1q'_0} The preimage of $g_1$ contains $P_{1}'$ and the preimage of $g_2$ contains $Q_{0}'$ i.e. $$P_{1}'=\begin{pmatrix} 0 & -1 & 0 & 0\\1 & 0 & 0 & 0\\0 & 0 & 0 & -1\\0 & 0 & 1 & 0 \end{pmatrix}\in F^{-1}(g_1); Q_{0}'=\begin{pmatrix} 0 & 0 & 0 & 1\\0 & 0 & -1 & 0\\0 & -1 & 0 & 0\\1 & 0 & 0 & 0 \end{pmatrix} \in F^{-1}(g_2)$$
For this case, $Q_{0}'$ has the following eigenspace decomposition: $\lambda_1=1$ and its corresponding eigenspace is spanned by $(1,0,0,1)$ and $(0,-1,1,0)$ and $\lambda_2=-1$ whose eigenspace is spanned by $(-1,0,0,1)$ and $(0,1,1,0)$. Hence for the same reason as (\ref{p'_0q'_0}), for candidates of $i(\Delta)_{\mathbb{Q}}$, it suffices to consider two-dimensional vector subspaces of the form $$\mathrm{span}_{\mathbb{Q}}(\{(x_1,-x_2,x_2,x_1),(-y_1,y_2,y_2,y_1)\})$$ where $x_i,y_i\in \mathbb{Z}$. Then since $i(\Delta)_{\mathbb{Q}}$ is stable under the action of $P_1'$, there exist $\alpha, \beta, \alpha^{'}, \beta^{'} \in \mathbb{Q}$ such that
\begin{equation*}
\begin{split}
\alpha x_1-\beta y_1 =x_2;  \alpha^{'} x_1-\beta^{'} y_1 =-y_2 \\
-\alpha x_2+\beta y_2 =x_1;   -\alpha^{'} x_2+\beta^{'} y_2 =-y_1 \\
\alpha x_2+\beta y_2 =-x_1;   \alpha^{'} x_2+\beta^{'} y_2 =-y_1\\
\alpha x_1+\beta y_1 =x_2;    \alpha^{'} x_1+\beta^{'} y_1 =y_2
\end{split}
\end{equation*} from which we deduce that
\begin{equation*}
    \begin{split}
        \alpha x_1=x_2; \beta' y_1=y_2\\
        \alpha x_2=-x_1;
        \beta' y_2=-y_1\\
    \end{split}
\end{equation*} which only holds when $x_1=x_2=y_1=y_2=0$ and this gives a contradiction. 
\item\label{p'_1q'_2} Suppose $i(\Delta)_{\mathbb{Q}}$ is stable under the action of matrices $$P_1'=\begin{pmatrix} 0 & -1 & 0 & 0\\1 & 0 & 0 & 0\\0 & 0 & 0 & -1\\0 & 0 & 1 & 0 \end{pmatrix}\in F^{-1}(g_1); Q_2'=\begin{pmatrix} 0 & 0 & 0 & -1\\0 & 0 & -1 & 0\\0 & 1 & 0 & 0\\1 & 0 & 0 & 0 \end{pmatrix}\in F^{-1}(g_2)$$ Diagonalizing $Q_2'$, it has eigenvalue $\lambda_1=i$ whose eigenspace is equal to $$\mathrm{span}_{\mathbb{Q}(i)}(\{(i,0,0,1),(0,i,1,0)\})$$ and eigenvalue $\lambda_2=-i$ whose eigenspace is equal to $$\mathrm{span}_{\mathbb{Q}(i)}(\{(-i,0,0,1),(0,-i,1,0)\})$$ The two eigenspaces cannot descend to $\mathbb{Q}$. Therefore it suffices to consider $\mathbb{Q}(i)$-vector subspaces of the form $$\mathrm{span}_{\mathbb{Q}(i)}(\{(ix_1,ix_2,x_2,x_1),(-iy_1,-iy_2,y_2,y_1)\}),(x_1,x_2,y_1,y_2\in \mathbb{Q}(i))$$ as candidates for $i(\Delta)\otimes_{\mathbb{Z}}\mathbb{Q}(i)$. Such $\mathbb{Q}(i)$-vector subspaces are stable under the action of $P_1'$, i.e. there exist $\alpha,\beta,\alpha',\beta'\in \mathbb{Q}(i)$ such that the following equations hold
\begin{equation*}
\begin{split}
      \alpha ix_1-i\beta y_1=-ix_2; \alpha' ix_1-i\beta' y_1=iy_2\\
      \alpha i x_2-i\beta y_2=i x_1; \alpha' i x_2-i\beta'y_2=-iy_1\\
      \alpha x_2+\beta y_2=-x_1; \alpha' x_2+\beta' y_2=-y_1\\
      \alpha x_1+\beta y_1=x_2; \alpha' x_1+\beta' y_1=y_2
    \end{split}
\end{equation*} from which we deduce that \begin{equation*}
\begin{split}
      \alpha x_1=0; \beta'y_2=0\\
      \alpha x_2=0; \beta'y_1=0\\
      \beta y_2=-x_1; \alpha'x_2=-y_1\\
      \beta y_1=x_2; \alpha'x_1=y_2
    \end{split}
\end{equation*}
The only situation where this can hold such that $i(\Delta)_{\mathbb{Q}(i)}$ is a two-dimensional $\mathbb{Q}(i)$-vector space is when $\alpha=\beta'=0$. In this case we have that $x_1y_1=-x_2y_2$. We claim that the $\mathbb{Q}(i)$-vector subspace $i(\Delta)_{\mathbb{Q}(i)}$ spanned by $$\{(ix_1,ix_2,x_2,x_1),(ix_2,-ix_1,x_1,-x_2)\}$$ where $x_1,x_2\in \mathbb{Q}(i)$ cannot descend to $\mathbb{Q}$. This is because if it were defined over $\mathbb{Q}$ then there would exist $m,n\in \mathbb{Q}(i)$ such that the following equations hold 
\begin{equation*}
\begin{split}
      imx_1+inx_2=-i\overline{x_1}\\
      imx_2-inx_1=-i\overline{x_2}\\
      mx_2+nx_1=\overline{x_2}\\
      mx_1-nx_2=\overline{x_1}
    \end{split}
\end{equation*} 
From which we deduce that $mx_1=mx_2=0$ and $nx_1=\overline{x_2}, -nx_2=\overline{x_1}$. But this implies that $\overline{n}\overline{x_1}=x_2$ and $-n\overline{n}\overline{x_1}=-nx_2=\overline{x_1}$. Since $n\overline{n}\geq0$, this implies that $x_1=x_2=0$. But then this contradicts that $i(\Delta)_{\mathbb{Q}(i)}$ is a two-dimensional vector space.
\item\label{p'_2q'_2} Suppose $i(\Delta)_{\mathbb{Q}}$ is stable under the action of matrix $$P_2'=\begin{pmatrix} 0 & -1 & 0 & 0\\1 & 0 & 0 & 0\\0 & 0 & 0 & 1\\0 & 0 & -1 & 0 \end{pmatrix}\in F^{-1}(g_1); Q_2'=\begin{pmatrix} 0 & 0 & 0 & -1\\0 & 0 & -1 & 0\\0 & 1 & 0 & 0\\1 & 0 & 0 & 0 \end{pmatrix}\in F^{-1}(g_2)$$ By (\ref{p'_1q'_2}), it suffices to consider $\mathbb{Q}(i)$-vector subspaces of the form $$\mathrm{span}_{\mathbb{Q}(i)}(\{(ix_1,ix_2,x_2,x_1),(-iy_1,-iy_2,y_2,y_1)\}),x_1,x_2,y_1,y_2\in \mathbb{Q}(i)$$ which are further stable under the action of $P_2'$, i.e. there exist $\alpha,\beta,\alpha',\beta'\in \mathbb{Q}(i)$ such that the following equations hold
\begin{equation*}
\begin{split}
      \alpha ix_1-i\beta y_1=-ix_2; \alpha' ix_1-i\beta' y_1=iy_2\\
      \alpha i x_2-i\beta y_2=i x_1; \alpha' i x_2-i\beta'y_2=-iy_1\\
      \alpha x_2+\beta y_2=x_1; \alpha' x_2+\beta' y_2=y_1\\
      \alpha x_1+\beta y_1=-x_2; \alpha' x_1+\beta' y_1=-y_2
    \end{split}
\end{equation*} which implies that \begin{equation*}
\begin{split}
      \beta y_2=0; \alpha'x_2=0\\
      \beta y_1=0; \alpha'x_1=0\\
      \alpha x_2=x_1; \beta'y_1=-y_2\\
      \alpha x_1=-x_2; \beta'y_2=y_1
    \end{split}
\end{equation*}
from which we deduce that \begin{equation*}
    x_1^2+x_2^2=0; y_1^2+y_2^2=0
\end{equation*} This implies that $x_2=\pm ix_1$ and $y_2=\pm iy_1$.

First suppose $x_2=ix_1$ and $y_2=iy_1$, then \begin{equation*}
\begin{split}
i(\Delta)_{\mathbb{Q}(i)}&=\mathrm{span}_{\mathbb{Q}(i)}(\{(ix_1,ix_2,x_2,x_1),(-iy_1,-iy_2,y_2,y_1)\})\\
&=\mathrm{span}_{\mathbb{Q}(i)}(\{(ix_1,-x_1,ix_1,x_1),(-iy_1,y_1,iy_1,y_1)\})\\
&=\mathrm{span}_{\mathbb{Q}(i)}(\{(0,0,i,1),(-i,1,0,0)\})
\end{split}
\end{equation*} which clearly cannot descend to $\mathbb{Q}$.

If $x_2=ix_1$ and $y_2=-iy_1$, then \begin{equation*}
\begin{split}
i(\Delta)_{\mathbb{Q}(i)}&=\mathrm{span}_{\mathbb{Q}(i)}(\{(ix_1,-x_1,ix_1,x_1),(-iy_1,-y_1,-iy_1,y_1)\})\\
&=\mathrm{span}_{\mathbb{Q}(i)}(\{(0,-1,0,1),(1,0,1,0)\})
\end{split}
\end{equation*} Because $X^{*}(\mathrm{U}_{E})/i(\Delta)$ is torsion free, this implies that $(0,-1,0,1)\in i(\Delta)$. This contradicts Corollary \ref{divisortest}.

When $x_2=-ix_1$ and $y_2=iy_1$, we have \begin{equation*}
\begin{split}
i(\Delta)_{\mathbb{Q}(i)}&=\mathrm{span}_{\mathbb{Q}(i)}(\{(ix_1,x_1,-ix_1,x_1),(-iy_1,y_1,iy_1,y_1)\})\\
&=\mathrm{span}_{\mathbb{Q}(i)}(\{(0,1,0,1),(-1,0,1,0)\})
\end{split}
\end{equation*} which contradicts Corollary \ref{divisortest} as well.

Finally if $x_2=-ix_1$ and $y_2=-iy_1$, then \begin{equation*}
\begin{split}
i(\Delta)_{\mathbb{Q}(i)}&=\mathrm{span}_{\mathbb{Q}(i)}(\{(ix_1,x_1,-ix_1,x_1),(-iy_1,-y_1,-iy_1,y_1)\})\\
&=\mathrm{span}_{\mathbb{Q}(i)}(\{(0,0,-i,1),(i
,1,0,0)\})
\end{split}
\end{equation*} which does not descend to $\mathbb{Q}$.
\item\label{p'_2q'_1} Finally suppose $i(\Delta)_{\mathbb{Q}}$ is stable under the action of matrices $$P_2'=\begin{pmatrix} 0 & -1 & 0 & 0\\1 & 0 & 0 & 0\\0 & 0 & 0 & 1\\0 & 0 & -1 & 0 \end{pmatrix}\in F^{-1}(g_1); Q_1'=\begin{pmatrix} 0 & 0 & 0 & 1\\0 & 0 & -1 & 0\\0 & 1 & 0 & 0\\-1 & 0 & 0 & 0 \end{pmatrix}\in F^{-1}(g_2)$$ Now $P_2'$ has eigenvalue $\lambda_1=i$ with eigenbasis $\{(0,0,-i,1),(i,1,0,0)\}$ and eigenvalue $\lambda_2=-i$ with eigenbasis $\{(0,0,i,1), (-i,1,0,0)\}$. These two eigenspaces themselves do not descend to $\mathbb{Q}$. Therefore it suffices to consider $\mathbb{Q}(i)$-vector subspaces spanned by $$\{(ix_2,x_2,-ix_1,x_1),(-iy_2,y_2,iy_1,y_1)\}$$ as candidates for $i(\Delta)_{\mathbb{Q}(i)}$ where $x_1,x_2,y_1,y_2\in \mathbb{Q}(i)$. Now such $\mathbb{Q}(i)$-vector spaces have to be stable under the action of $Q_1'$, hence there exist $\alpha,\beta,\alpha',\beta'\in \mathbb{Q}(i)$ such that the following equations hold
\begin{equation*}
    \begin{split}
      \alpha ix_2-i\beta y_2=x_1; \alpha' ix_2-i\beta' y_2=y_1\\
      \alpha x_2+\beta y_2=i x_1;\alpha' x_2+\beta' y_2=-iy_1\\
      -\alpha ix_1+i\beta y_1=x_2; -\alpha' ix_1+i\beta' y_1=y_2\\
      \alpha x_1+\beta y_1=-ix_2; \alpha' x_1+\beta' y_1=iy_2
    \end{split}
\end{equation*}
from which we deduce that $$\alpha=0,\beta y_1=-ix_2,\beta y_2=ix_1$$ Therefore we have that $x_1y_1=-x_2y_2$. We claim that the $\mathbb{Q}(i)$-vector subspace $i(\Delta) \otimes_{\mathbb{Z}} \mathbb{Q}(i)$ spanned by $$\{(ix_2,x_2,-ix_1,x_1),(-ix_1,x_1,-ix_2,-x_2)\}, x_1,x_2\in \mathbb{Q}(i)$$ cannot descend to $\mathbb{Q}$. This is because if it were defined over $\mathbb{Q}$ there would exist $m,n\in \mathbb{Q}(i)$ such that the following equations hold 
\begin{equation*}
\begin{split}
imx_2-inx_1=-i\overline{x_2}\\
mx_2+nx_1=\overline{x_2}\\
-imx_1-inx_2=i\overline{x_1}\\
mx_1-nx_2=\overline{x_1}
    \end{split}
\end{equation*} from which we deduce that $m=0;nx_1=\overline{x_2};nx_2=-\overline{x_1}$. This implies that $x_1=x_2=0$ which gives a contradiction.
\end{enumerate}
By the above results (see Table (\ref{tableeachtwo-1}) for the summary), we can rule out the possibility of Case (\ref{bigcase3}) from Setup \ref{klein4setup} where $F^{-1}(g_1),F^{-1}(g_2),F^{-1}(g_3)$ contain only the generalized permutation matrices with exactly two -1's among their entries.   
\end{proof}

\begin{proof}[Proof of Proposition \ref{noklein4}]
 The three cases from Setup \ref{klein4setup} have been considered by Lemma \ref{bigcase1lemma}, Lemma \ref{p_0inpreimagelemma} and Lemma \ref{bigcase3lemma} and we can therefore finish the proof of Proposition \ref{noklein4}.   
\end{proof}
The final lemma of this section deals with the scenario where $H$, the image of $G''$ in $\mathrm{S}_4$ under $F$, is isomorphic to $\mathrm{A}_4$.
\begin{lemma} \label{A_4notpossible}
We use notations from Setup \ref{setupcontinued} and Lemma \ref{galois_permutation}. If $T=\gdrbhmath(A)$ is a two-dimensional (connected) algebraic torus, then the image $H$ of $G''$ in $\mathrm{S}_4$ cannot be $\mathrm{A}_4$.  
\end{lemma}
\begin{proof}
    Note that $H=\mathrm{A}_4$ contains the Klein four-group $$V:=(\mathrm{id},(12)(34),(13)(24),(14)(23))=(\mathrm{id}, g_1, g_2, g_3)$$ and all length 3 cycles in $\mathrm{S}_4$. Therefore any submodule $i(\Delta)$ from the short exact sequence (\ref{charseq}) stable under the action of $G''=F^{-1}(H)$ is also stable under $F^{-1}(V)$. Then by Proposition \ref{noklein4} and the proof of Lemma \ref{p_0inpreimagelemma}, it suffices to examine cases where for some $i \neq i'$, $F^{-1}(g_{i})$ contains an ordinary permutation matrix and $F^{-1}(g_{i'})$ contains a generalized permutation matrix with exactly two $-1$s among its entries. By the computations done in Lemma \ref{p_0inpreimagelemma} and symmetry of the positions of $g_1,g_2,g_3$, it suffices to consider the following two scenario. 
    \begin{enumerate}
        \item $i(\Delta)$ is stable under the action of $$P_0=\begin{pmatrix} 0 & 1 & 0 & 0\\1 & 0 & 0 & 0\\0 & 0 & 0 & 1\\0 & 0 & 1 & 0 \end{pmatrix}\in F^{-1}(g_1), Q:=\begin{pmatrix} 0 & 0 & 0 & 1\\0 & 0 & -1 & 0\\0 & 1 & 0 & 0\\-1 & 0 & 0 & 0 \end{pmatrix}\in F^{-1}(g_2), R\in F^{-1}((123))$$
        \item $i(\Delta)$ is stable under the action of $$P_0, Q':= \begin{pmatrix} 0 & 0 & 0 & -1\\0 & 0 & 1 & 0\\0 & 1 & 0 & 0\\-1 & 0 & 0 & 0 \end{pmatrix}\in F^{-1}(g_2), R\in F^{-1}((123))$$ 
    \end{enumerate}

\begin{enumerate}
\item\label{ccccase1}Suppose $i(\Delta)$ is stable under the action of $P_0$ and $Q$. According to (\ref{p_0p_2}) in the proof of Lemma \ref{p_0inpreimagelemma}, it suffices to consider the following set of two dimensional $\mathbb{Q}$-vector spaces stable under $P_0$ and $Q$ as candidates for $i(\Delta)_{\mathbb{Q}}$. $$\mathcal{S}:=\{\mathrm{span}_{\mathbb{Q}}(\{(-x_1,x_1,-x_2,x_2),(y_1,y_1,y_2,y_2)\})|x_{i},y_{j}\in \mathbb{Z},x_1y_1=x_2y_2\}$$ As expected, the extra symmetry brought by an element in $F^{-1}((123))$ obstructs the existence of a two-dimensional vector space simultaneous stable under $P_0,Q,R\in F^{-1}((123))$. The results are summarized in the following table.
\begin{table}[H]
    \centering
    \begin{tabular}{|c|c|c|}
    \hline
    Elements in $G''$  & Properties of $i(\Delta)$ & Can we rule out such $T$ \\
    \hline
    $P_0, Q, R_0$ & Cannot be of rank 2 & Yes, See (\ref{P_0QR_0})\\
    \hline
    $P_0, Q, R_1$ & Cannot be of rank 2 & Yes, See (\ref{P_0QR_1})\\
    \hline
    $P_0, Q, R_2$ & Cannot be of rank 2 & Yes, See (\ref{P_0QR_2})\\
    \hline
    $P_0, Q, R_3$ & Cannot be of rank 2 & Yes, See (\ref{P_0QR_3})\\
    \hline
    $P_0, Q, R_4$ & Cannot be of rank 2 & Yes, See (\ref{P_0QR_4})\\
    \hline
    $P_0, Q, R_5$ & Cannot be of rank 2 & Yes, See (\ref{P_0QR_5})\\
    \hline
    $P_0, Q, R_6$ & Cannot be of rank 2 & Yes, See (\ref{P_0QR_6})\\
    \hline
    $P_0, Q, R_7$ & Cannot be of rank 2 & Yes, See (\ref{P_0QR_7})\\
    \hline
    \end{tabular}
    \caption{Results for Case (\ref{ccccase1}) of Proposition \ref{A_4notpossible}}
    \label{a4table}
\end{table}
Below are the detailed computations summarized in Table \ref{a4table}.
\begin{enumerate}
\item\label{P_0QR_0} Given a two-dimensional $\mathbb{Q}$-vector space from the set $\mathcal{S}$ which is further stable under the action of $$R_0:=\begin{pmatrix} 0 & 0 & 1 & 0\\1 & 0 & 0 & 0\\0 & 1 & 0 & 0\\0 & 0 & 0 & 1 \end{pmatrix}\in F^{-1}((123))$$ There exist $\alpha,\beta,\alpha',\beta' \in \mathbb{Q}$ such that \begin{equation}\label{eqs}
\begin{split}
-\alpha x_1+\beta y_1 =-x_2;  -\alpha^{'} x_1+\beta^{'} y_1 =y_2 \\
\alpha x_1+\beta y_1 =-x_1;   \alpha^{'} x_1+\beta^{'} y_1 =y_1 \\
-\alpha x_2+\beta y_2 =x_1;   -\alpha^{'} x_2+\beta^{'} y_2 =y_1\\
\alpha x_2+\beta y_2 =x_2;    \alpha^{'} x_2+\beta^{'} y_2 =y_2
\end{split}
\end{equation} from which we deduce that \begin{equation}
\begin{split}
&\alpha(x_1-x_2)+\beta(y_1+y_2)=0;\alpha(x_2-x_1)+\beta(y_1+y_2)=0\\
&\alpha'(x_1+x_2)+\beta'(y_1-y_2)=0;\alpha'(x_1+x_2)+\beta'(y_2-y_1)=0\\
\end{split}
\end{equation} 
And we deduce that 
\begin{equation}\label{directfromeqs}
\alpha(x_1-x_2)=\beta(y_1+y_2)=0;\alpha'(x_2+x_1)=\beta'(y_1-y_2)=0
\end{equation}
From (\ref{eqs}), because $i(\Delta)$ is of rank two, we cannot have $\alpha=\beta=0$ or $\alpha'=\beta'=0$. Hence either $x_1=x_2,y_1=y_2$ or $x_1=-x_2,y_1=-y_2$. Suppose the first case, then $\alpha'=\beta=0$. Swapping $x_1=x_2$ into (\ref{eqs}), we obtain that $\alpha x_1=-x_1;\alpha x_2=x_2$ thus giving a contradiction. Suppose the second case, then $\alpha=\beta'=0$. Swapping $x_1=-x_2,y_1=-y_2$ into (\ref{eqs}), we obtain that $\beta y_2=x_1;\beta y_2=x_2=-x_1$. Again this contradicts that $i(\Delta)$ is of rank two.

Hence we can conclude that there does not exist a two-dimensional $\mathbb{Q}$-vector space stable under the action of $P_0,Q,R_0$.
\item \label{P_0QR_1} Given a two-dimensional vector space from the set $\mathcal{S}$ which is further stable under the action of $$R_1:=\begin{pmatrix} 0 & 0 & 1 & 0\\-1 & 0 & 0 & 0\\0 & 1 & 0 & 0\\0 & 0 & 0 & 1 \end{pmatrix}\in F^{-1}((123))$$
There then exist $\alpha,\beta,\alpha',\beta' \in \mathbb{Q}$ such that \begin{equation*}
\begin{split}
-\alpha x_1+\beta y_1 =-x_2;  -\alpha^{'} x_1+\beta^{'} y_1 =y_2 \\
\alpha x_1+\beta y_1 =x_1;   \alpha^{'} x_1+\beta^{'} y_1 =-y_1 \\
-\alpha x_2+\beta y_2 =x_1;   -\alpha^{'} x_2+\beta^{'} y_2 =y_1\\
\alpha x_2+\beta y_2 =x_2;    \alpha^{'} x_2+\beta^{'} y_2 =y_2
\end{split}
\end{equation*} from which we deduce that 
\begin{equation*}
2\alpha x_1=x_1+x_2;2\alpha x_2=x_2-x_1
\end{equation*}  
which implies that 
\begin{equation*}
    x_1^2+x_2^2=0
\end{equation*}
Since $x_1,x_2\in\mathbb{Q}$, we conclude that $x_1=x_2=0$. This is a contradiction to $i(\Delta)$ being of rank 2.
\item \label{P_0QR_2}Suppose we are given a two-dimensional vector space from the set $\mathcal{S}$ which is further stable under the action of $$R_2:=\begin{pmatrix} 0 & 0 & 1 & 0\\1 & 0 & 0 & 0\\0 & -1 & 0 & 0\\0 & 0 & 0 & 1 \end{pmatrix}\in F^{-1}((123))$$
There then exist $\alpha,\beta,\alpha',\beta' \in \mathbb{Q}$ such that \begin{equation*}
\begin{split}
-\alpha x_1+\beta y_1 =-x_2;  -\alpha^{'} x_1+\beta^{'} y_1 =y_2 \\
\alpha x_1+\beta y_1 =-x_1;   \alpha^{'} x_1+\beta^{'} y_1 =y_1 \\
-\alpha x_2+\beta y_2 =-x_1;   -\alpha^{'} x_2+\beta^{'} y_2 =-y_1\\
\alpha x_2+\beta y_2 =x_2;    \alpha^{'} x_2+\beta^{'} y_2 =y_2
\end{split}
\end{equation*} from which we deduce that \begin{equation*}
2\beta'y_1=y_1+y_2; 2\beta'y_2=y_2-y_1
\end{equation*}  
which further implies that 
\begin{equation*}
    y_1=y_2=0
\end{equation*}
but then this is a contradiction to the requirement that $i(\Delta)$ is a rank two submodule.
\item \label{P_0QR_3}Suppose we are given a two-dimensional vector space from the set $\mathcal{S}$ which is further stable under the action of $$R_3:=\begin{pmatrix} 0 & 0 & -1 & 0\\1 & 0 & 0 & 0\\0 & 1 & 0 & 0\\0 & 0 & 0 & 1 \end{pmatrix}\in F^{-1}((123))$$ There then exist $\alpha,\beta,\alpha',\beta' \in \mathbb{Q}$ such that 
\begin{equation*}
\begin{split}
-\alpha x_1+\beta y_1 =x_2;  -\alpha^{'} x_1+\beta^{'} y_1 =-y_2 \\
\alpha x_1+\beta y_1 =-x_1;   \alpha^{'} x_1+\beta^{'} y_1 =y_1 \\
-\alpha x_2+\beta y_2 =x_1;   -\alpha^{'} x_2+\beta^{'} y_2 =y_1\\
\alpha x_2+\beta y_2 =x_2;    \alpha^{'} x_2+\beta^{'} y_2 =y_2
\end{split}
\end{equation*}
from which we deduce that \begin{equation*}
2\beta'y_1=y_1-y_2;2\beta'y_2=y_1+y_2
\end{equation*}  
which further implies that 
\begin{equation*}
    y_1=y_2=0
\end{equation*}
but then this is a contradiction to the requirement that $i(\Delta)$ is a rank two submodule.

\item \label{P_0QR_4}Suppose we are given a two-dimensional vector space from the set $\mathcal{S}$ which is further stable under the action of $$R_4=\begin{pmatrix} 0 & 0 & 1 & 0\\1 & 0 & 0 & 0\\0 & 1 & 0 & 0\\0 & 0 & 0 & -1 \end{pmatrix}\in F^{-1}((123))$$
There then exist $\alpha,\beta,\alpha',\beta' \in \mathbb{Q}$ such that \begin{equation*}
\begin{split}
-\alpha x_1+\beta y_1 =-x_2;  -\alpha^{'} x_1+\beta^{'} y_1 =y_2 \\
\alpha x_1+\beta y_1 =-x_1;   \alpha^{'} x_1+\beta^{'} y_1 =y_1 \\
-\alpha x_2+\beta y_2 =x_1;   -\alpha^{'} x_2+\beta^{'} y_2 =y_1\\
\alpha x_2+\beta y_2 =-x_2;    \alpha^{'} x_2+\beta^{'} y_2 =-y_2
\end{split}
\end{equation*} from which we deduce that 
\begin{equation*}
    2\alpha x_1=x_2-x_1; 2\alpha x_2=-x_2-x_1
\end{equation*}
which implies that 
\begin{equation*}
    x_1=x_2=0
\end{equation*} and this contradicts that $i(\Delta)$ is a rank two submodule.

\item\label{P_0QR_5} Suppose we are given a two-dimensional vector space from the set $\mathcal{S}$ which is further stable under the action of $$R_5=\begin{pmatrix} 0 & 0 & 1 & 0\\-1 & 0 & 0 & 0\\0 & -1 & 0 & 0\\0 & 0 & 0 & 1 \end{pmatrix}\in F^{-1}((123))$$
There then exist $\alpha,\beta,\alpha',\beta' \in \mathbb{Q}$ such that \begin{equation}\label{eqss}
\begin{split}
-\alpha x_1+\beta y_1 =-x_2;  -\alpha^{'} x_1+\beta^{'} y_1 =y_2 \\
\alpha x_1+\beta y_1 =x_1;   \alpha^{'} x_1+\beta^{'} y_1 =-y_1 \\
-\alpha x_2+\beta y_2 =-x_1;   -\alpha^{'} x_2+\beta^{'} y_2 =-y_1\\
\alpha x_2+\beta y_2 =x_2;    \alpha^{'} x_2+\beta^{'} y_2 =y_2
\end{split}
\end{equation} from which we deduce that \begin{equation*}
\begin{split}
\\\alpha(x_2-x_1)=\beta(y_1+y_2)=0 ;  \alpha'(x_1+x_2)=\beta'(y_1-y_2)=0    
\end{split}
\end{equation*}  
Then for reasons similar to Case (\ref{P_0QR_0}), it suffices to consider $x_2=x_1,y_1=y_2$ or $x_1+x_2=0,y_1+y_2=0$. Suppose the first case, then $\beta=\alpha'=0$. Swapping $y_1=y_2$ into (\ref{eqss}), we obtain that $\beta'y_1=-y_1,\beta'y_2=y_2$. This contradicts that $i(\Delta)$ is of rank two. Suppose the second case, then $\beta'=\alpha=0$. Swapping $y_1=-y_2$ into (\ref{eqss}), we obtain that $-\alpha'x_1=y_2,\alpha'x_1=-y_1=y_2$. This implies that $y_1=y_2=0$. Hence no two-dimensional $\mathbb{Q}$-vector space is stable under $P_0,Q,R_5$.
\item\label{P_0QR_6} Suppose we are given a two-dimensional vector space from the set $\mathcal{S}$ which is further stable under the action of $$R_6:=\begin{pmatrix} 0 & 0 & -1 & 0\\-1 & 0 & 0 & 0\\0 & 1 & 0 & 0\\0 & 0 & 0 & 1 \end{pmatrix}\in F^{-1}((123))$$
There then exist $\alpha,\beta,\alpha',\beta' \in \mathbb{Q}$ such that \begin{equation}\label{eqsss}
\begin{split}
-\alpha x_1+\beta y_1 =x_2;  -\alpha^{'} x_1+\beta^{'} y_1 =-y_2 \\
\alpha x_1+\beta y_1 =x_1;   \alpha^{'} x_1+\beta^{'} y_1 =-y_1 \\
-\alpha x_2+\beta y_2 =x_1;   -\alpha^{'} x_2+\beta^{'} y_2 =y_1\\
\alpha x_2+\beta y_2 =x_2;    \alpha^{'} x_2+\beta^{'} y_2 =y_2
\end{split}
\end{equation} from which we deduce that \begin{equation*}
\begin{split}
\\\alpha(x_2+x_1)=\beta(y_1-y_2)=0 ;  \alpha'(x_1-x_2)=\beta'(y_1+y_2)=0    
\end{split}
\end{equation*} 
For reasons similar to above, it suffices to consider $x_2=x_1,y_1=y_2$ or $x_1+x_2=0,y_1+y_2=0$. Suppose the first case, then $\beta'=\alpha=0$. Swapping $x_1=x_2,y_1=y_2$ into (\ref{eqsss}), we obtain that $\alpha'x_1=-y_1,\alpha'x_2=\alpha'x_1=y_2=y_1$. Thus this contradicts that $i(\Delta)$ is of rank two. Suppose the second case, then $\beta=\alpha'=0$. Swapping $x_1=-x_2$ into (\ref{eqsss}), we obtain that $\beta y_1=x_2,\beta y_1=x_1=-x_2$. Again this gives a contradiction. Hence no two-dimensional $\mathbb{Q}$-vector space is stable under $P_0,Q,R_6$.

\item\label{P_0QR_7} Suppose we are given a two-dimensional vector space from the set $\mathcal{S}$ which is further stable under the action of $$R_7:=\begin{pmatrix} 0 & 0 & 1 & 0\\-1 & 0 & 0 & 0\\0 & 1 & 0 & 0\\0 & 0 & 0 & -1 \end{pmatrix}\in F^{-1}((123))$$
There then exist $\alpha,\beta,\alpha',\beta' \in \mathbb{Q}$ such that \begin{equation}\label{eqssss}
\begin{split}
-\alpha x_1+\beta y_1 =-x_2;  -\alpha^{'} x_1+\beta^{'} y_1 =y_2 \\
\alpha x_1+\beta y_1 =x_1;   \alpha^{'} x_1+\beta^{'} y_1 =-y_1 \\
-\alpha x_2+\beta y_2 =x_1;   -\alpha^{'} x_2+\beta^{'} y_2 =y_1\\
\alpha x_2+\beta y_2 =-x_2;    \alpha^{'} x_2+\beta^{'} y_2 =-y_2
\end{split}
\end{equation} from which we deduce that \begin{equation*}
\begin{split}
\\\alpha(x_2+x_1)=\beta(y_1-y_2)=0 ;  \alpha'(x_1-x_2)=\beta'(y_1+y_2)=0    
\end{split}
\end{equation*}  
For reasons similar to above, it suffices to consider $x_1=x_2,y_1=y_2$ or $x_1+x_2=0,y_1+y_2=0$. Suppose the first case, then $\alpha=\beta'=0$. Swapping $x_1=x_2$ into (\ref{eqssss}), we obtain that $\beta y_1=-x_2=-x_1,\beta y_1=x_1$ which contradicts that $i(\Delta)$ is of rank two. Suppose the second case, then $\beta=\alpha'=0$. Swapping $x_1=-x_2$ into (\ref{eqssss}), we obtain that $\alpha x_1=x_1,\alpha x_2=-x_2$. Again this gives a contradiction. Hence no two-dimensional $\mathbb{Q}$-vector space is stable under $P_0,Q,R_7$.
\end{enumerate}
\item  Suppose $i(\Delta)$ is stable under $$P_0= \begin{pmatrix} 0 & 1 & 0 & 0\\1 & 0 & 0 & 0\\0 & 0 & 0 & 1\\0 & 0 & 1 & 0 \end{pmatrix}; Q'=\begin{pmatrix} 0 & 0 & 0 & -1\\0 & 0 & 1 & 0\\0 & 1 & 0 & 0\\-1 & 0 & 0 & 0 \end{pmatrix}$$ Now according to the computations done in (\ref{p_0p_3}) of the proof of Lemma \ref{p_0inpreimagelemma}, it suffices to consider the following set of two dimensional $\mathbb{Q}$-vector spaces stable under $P_0$ and $Q'$ as candidates for $i(\Delta)_{\mathbb{Q}}$. $$\mathcal{S}'=\{\mathrm{span}_{\mathbb{Q}}(\{(-x_1,x_1,-x_2,x_2),(y_1,y_1,y_2,y_2)\})|x_{i},y_{j}\in \mathbb{Z},x_1y_1=-x_2y_2\}$$ In Case (\ref{ccccase1}), we have ruled out all vector spaces in $$\mathcal{S}=\{\mathrm{span}_{\mathbb{Q}}(\{(-x_1,x_1,-x_2,x_2),(y_1,y_1,y_2,y_2)\})|x_{i},y_{j}\in \mathbb{Z},x_1y_1=x_2y_2\}$$ stable under the action of any $R\in F^{-1}((123))$ as candidates for $i(\Delta)_{\mathbb{Q}}$. A simple inspection of the computations done in Case (\ref{ccccase1}) reveals that the condition $x_1y_1=x_2y_2$ is not used at all. Hence the arguments from Case (\ref{ccccase1}) can be applied directly to rule out elements in $\mathcal{S}'$ which are stable under any $R\in F^{-1}((123))$ as candidates for $i(\Delta)_{\mathbb{Q}}$. 
\end{enumerate} 
\end{proof}

\begin{keyremark}\label{keyremarktotakeaway}
Let $A$ be a simple CM abelian fourfold. Suppose $\gdrbhmath(A)$ is a two-dimensional subtorus of $\mathrm{U}_E$, by the combined efforts of Lemma \ref{nolength4cycle}, Proposition \ref{noklein4} and Lemma \ref{A_4notpossible}, it remains to consider the case where $G''\cong \mathrm{D}_4$ (see Setup \ref{setupcontinued} for the notation).
\end{keyremark}

\subsection{Non-Simplicity of Certain CM Abelian Fourfolds}\label{cmnonsimplicity}
The goal of this section is to fill the caveat mentioned in Key Remark \ref{keyremarktotakeaway} by proving the following lemma. 
\begin{lemma}
\label{nosuchsimpleAV}
Let $E$ be a CM-field of degree 8 over $\mathbb{Q}$. Denote the Galois closure of $E/\mathbb{Q}$ inside $\qbar$ by $L$. Suppose that the image of the group homomorphism (see Example \ref{firstexample}) $$\phi: \mathrm{Gal}(L/\mathbb{Q}) \rightarrow \mathrm{Aut}_{\mathbb{Z}}(X^{*}(\resgmmath))$$ induced by the extension $L/E/\mathbb{Q}$ is isomorphic to the dihedral group $$\mathrm{D}_4=\langle a,x |a^4=1;x^2=1;axa=x\rangle$$
Then there does not exist a simple CM abelian fourfold whose endomorphism field is $E$.
\end{lemma}
\begin{remark}\label{alreadygalois}
In fact, the field $E$ from the above Lemma \ref{nosuchsimpleAV} is a Galois extension of $\mathbb{Q}$ whose Galois group is isomorphic to $\mathrm{D}_4$. By Example \ref{firstexample}, the group homomorphism $\phi$ in the lemma is in fact induced by the homomorphism $\phi': \mathrm{Gal}(L/\mathbb{Q}) \rightarrow \mathrm{Aut}_{\mathrm{set}}(\mathrm{Hom}(E,L))$. Because $L$ is the Galois closure of $E/\mathbb{Q}$ insider $\qbar$, we have that $\phi'$ is an injective homomorphism.  Hence if the image of $\phi'$ has order 8, then the group $\mathrm{Gal}(L/\mathbb{Q})$ also has order 8, which implies that $E$ is already a Galois extension over $\mathbb{Q}$ whose Galois group is isomorphic to $\mathrm{D}_4$.
\end{remark}
By virtue of Remark \ref{alreadygalois}, for the proof of Lemma \ref{nosuchsimpleAV}, we will work in the following setup.
\begin{setup}\label{Esubfield}
 Let $E$ be a Galois CM-field of degree 8 with \begin{equation*}\label{setupnosuchsimpleav}
G:=\mathrm{Gal}(E/\mathbb{Q})=\mathrm{D}_4    
\end{equation*} Denote the subgroup of $G$ whose elements are $\{1,x\}$ by $G_1$, and the subgroup of $G$ whose elements are $\{1,ax\}$ by $G_{2}$. Let $$K_1=E^{G_1}$$ and $$K_2=E^{G_2}$$ then they are quartic subfields of $E$. Moreover, since $a^2$ is the only order 2 element in $\mathrm{D}_4$ that commutes with every other element, it is the complex conjugation on $E$ by Lemma \ref{comjugationcommuting}. Hence the maximal totally real subfield $F$ of $E$ is given by $E^{\{1,a^2\}}$. By the Galois correspondence, $K_{i}$ is not a totally real subfield of $E$. Therefore both $K_1$ and $K_2$ are in fact CM-subfields of $E$ by Lemma \ref{subofcm}.

We now denote the image of one of the embeddings $i_0: E \rightarrow \qbar$ as $L$. Then $L \cong E$ is a Galois extension over $\mathbb{Q}$ and therefore the image of any other embedding  $E \rightarrow \qbar$ is also $L$. Also any embedding of $K_i$s into $\qbar$ factors through $L\xhookrightarrow{} \qbar$. Now since $E/\mathbb{Q}$ is also a Galois extension, $\mathrm{Gal}(E/\mathbb{Q})$ acts simple transitively on the right on $\mathrm{Hom}(E,L)$ where an embedding $i \in \mathrm{Hom}(E,L)$ is sent to $i \circ g \in \mathrm{Hom}(E,L)$ for $g \in \mathrm{Gal}(E/\mathbb{Q})$. Then upon fixing an embedding $i_{0}:E \xhookrightarrow{}L$, we can and will identify $\mathrm{Hom}(E,L)$ with $D_4$ as a set via the map: $$\mathrm{D}_4 \rightarrow \mathrm{Hom}(E,L): g \rightarrow i_0 \circ g$$ Then in particular, the eigenspace decomposition with respect to the action of $E$ on $\betti$ can be written as 
\begin{equation*}
   \mathrm{H}^1(A,\mathbb{Q}) \otimes_{\mathbb{Q}} L=V_{\mathrm{id}} \oplus V_{a} \oplus V_{a^2} \oplus V_{a^3} \oplus V_{x} \oplus V_{ax} \oplus V_{a^2 x} \oplus V_{a^3 x}
\end{equation*}
 For $v \in V_{g}, g\in \mathrm{Gal}(E/\mathbb{Q})$, we have $e \circ v=i_0(g(e))v$ for every $e \in E$. Now since $K_1=E^{\{1,x\}}$, for any $i \in \mathrm{Hom}(E,L)$, we have $i|_{K_1}=i \circ x|_{K_1}$. Hence with respect to the action of $K_1 \xhookrightarrow{} E \xhookrightarrow{}\mathrm{End}(\mathrm{H}^1(A,\mathbb{Q}))$, we have the following eigenspace decomposition
 \begin{equation}\label{eigenk1}
\mathrm{H}^1(A,\mathbb{Q}) \otimes_{\mathbb{Q}} L=(V_{\mathrm{id}} \oplus V_{x}) \oplus(V_{a} \oplus V_{ax}) \oplus (V_{a^2} \oplus V_{a^2 x}) \oplus (V_{a^3} \oplus V_{a^3 x})
\end{equation} 
Since $K_2=E^{\{1,ax\}}$, for any $i \in \mathrm{Hom}(E,L)$, we have $i|_{K_2}=i \circ ax|_{K_2}$. We have the following eigenspace decomposition for the action of $K_2$ on $\mathrm{H}^1(A,\mathbb{Q})$
\begin{equation}\label{eigenk2}
\mathrm{H}^1(A,\mathbb{Q}) \otimes_{\mathbb{Q}} L=(V_{\mathrm{id}} \oplus V_{ax}) \oplus( V_{a} \oplus V_{a^2 x}) \oplus (V_{a^2} \oplus V_{a^3 x}) \oplus (V_{a^3} \oplus V_{x})
\end{equation} 
\end{setup}
\begin{remark}
    Here it is crucial that $E/\mathbb{Q}$ is a Galois extension by Remark \ref{alreadygalois}. Otherwise it is difficult to obtain ``sensible"(the meaning of which will be clear below) choice of $K_1$ and $K_2$ inside $E$.
\end{remark}
Now we are ready to state the following lemma inspired by \cite[Formula 7.5.1]{mz-4-folds} controlling the set of multiplicities of $K_{i}$s (see Definition \ref{multiplicitydefn}) with respect to $$K_{i} \xhookrightarrow{} E=\mathrm{End}_{\mathrm{Hdg}}(\mathrm{H}^1(A,\mathbb{Q}))$$ It is also a variant of \cite[Proposition 14]{shimura1963analytic}.
\begin{lemma}\label{multiplicity-quartic-CM}
We work in the same setup as above. For both quartic CM subfields $K_1$ and $K_2$, with respect to the embedding $K_{i} \xhookrightarrow{} E=\mathrm{End}_{\mathrm{Hdg}}(\mathrm{H}^1(A,\mathbb{Q}))$, the only possible set of multiplicities (see Definition \ref{multiplicitydefn}) associated with $\mathrm{Hom}(K_{i},\mathbb{C})$ is $\{2,0,1,1\}$. 
\end{lemma}

\begin{proof}
Denote $\mathrm{H}^1(A,\mathbb{Q})$ by $V$. Suppose for $K\in \{K_1,K_2\}$ the set of multiplicities associated with $\mathrm{Hom}(K,\mathbb{C})$ is $\{1,1,1,1\}$. Then the $\mathbb{C}$-linear extension of the Weil structure (See Definition \ref{weilstructure}) satisfies $$\wedge^2_{K}V \otimes_{\mathbb{Q}} \mathbb{C} \subseteq \mathrm{H}^{1,1}(A,\mathbb{C})$$ See Remark \ref{weilhodgenumber} for the computation. Hence every class in $\wedge^2_{K}V\otimes\mathbb{Q}(1)$ would be a Hodge class. Now recall the construction of the Weil structure (see Definition \ref{weilstructure}) and use decomposition (\ref{eigenk1}). Then the Weil structure $\wedge^2_{K_1}V \otimes_{\mathbb{Q}} \mathbb{C} $ inside $\mathrm{H}^2(A,\mathbb{Q}) \otimes_{\mathbb{Q}} \mathbb{C}$ is equal to \begin{equation*}
   (V_{\mathrm{id}} \otimes V_{x}) \oplus(V_{\mathrm{a}} \otimes V_{ax}) \oplus (V_{\mathrm{a^2}} \otimes V_{a^2 x}) \oplus (V_{a^3} \otimes V_{a^3 x})
\end{equation*} Using decomposition (\ref{eigenk2}), the Weil structure $\wedge^2_{K_2}V \otimes_{\mathbb{Q}} \mathbb{C} $ inside $\mathrm{H}^2(A,\mathbb{Q}) \otimes_{\mathbb{Q}} \mathbb{C}$ is 
equal to 
\begin{equation*}
(V_{\mathrm{id}} \otimes V_{ax}) \oplus( V_{\mathrm{a}} \otimes V_{a^2 x}) \oplus (V_{a^2} \otimes V_{a^3 x}) \oplus (V_{a^3} \otimes V_{x})
\end{equation*} 
Recall the complex conjugation of the CM-field $E$ is $a^2$, the $\mathbb{C}$-linear expansion of Hodge classes inside $\mathrm{H}^2(A,\mathbb{Q}) \otimes_{\mathbb{Q}} \mathbb{C}$ contains
\begin{equation*}
W''=(V_{\mathrm{id}} \otimes V_{a^2}) \oplus( V_{\mathrm{a}} \otimes V_{a^3}) \oplus (V_{x} \otimes V_{a^2 x}) \oplus (V_{ax} \otimes V_{a^3 x})
\end{equation*} This is because elements in $W''$ are already fixed by the $\mathbb{C}$-points of $\mathrm{U}_{E}$ of which $\mathrm{Hdg}(A)$ is a subgroup.  But we have assumed that $A$ is a simple CM abelian fourfold. Therefore by Lemma \ref{rosaticmcoincides}, the dimension of the space of Hodge class in $\mathrm{H}^2(A,\mathbb{Q})$ is precisely 4. Hence the $\mathbb{C}$-linear expansion of Hodge classes inside $\mathrm{H}^2(A,\mathbb{Q}) \otimes_{\mathbb{Q}} \mathbb{C}$ is precisely the subspace $W''$. But $$W'' \cap \wedge^2_{K_i}\mathrm{H}^1(A,\mathbb{Q}) \otimes_{\mathbb{Q}} \mathbb{C}=\{0\}(i=1,2)$$ Hence this gives a contradiction. 

On the other hand, suppose for $K \in \{K_1,K_2\}$, the set of multiplicities associated with $\mathrm{Hom}(K,\mathbb{C})$ is $\{2,0,2,0\}$. We claim that the CM-type $(E,\Phi_{A})$, which is determined by the structure of the abelian variety $A$, restricts to a sub CM-type on $K$ via the natural embedding $K \xhookrightarrow{} E$. Because the degree of $K$ is half of the degree of $E$, for any pair $\rho,\rho' \in \mathrm{Hom}(E,L)$ such that $\rho|_{K}=\rho'|_{K}$, the eigenspace associated with $\rho|_{K}$ is precisely $V_{\rho} \oplus V_{\rho'}$. And by assumption $\rho|_{K}$ has multiplicity 2 or 0, hence $V_{\rho} \oplus V_{\rho'}$ either lies entirely in $\mathrm{H}^{1,0}(A)$ or $\mathrm{H}^{0,1}(A)$. Therefore we have $\Phi_{A}(\rho)=\Phi_{A}(\rho')$. Since $K_{i}$s are CM subfields of $E$, we have that $(K,\Phi_{A}|_{K})$ is a sub CM-type of $(E,\Phi_{A})$. But $A$ is a simple CM abelian fourfold, and according to \cite[Proposition 1.3.13]{milne2006complex} the CM-type associated with a simple CM-abelian variety is primitive. This gives a contradiction.   
    
\end{proof}
Another variant of \cite[Proposition 14]{shimura1963analytic} is the following lemma which plays a key role in the sequel. We record it here for the sake of continuity of exposition.
\begin{lemma}[Formula 7.5.1, \cite{mz-4-folds}]\label{deg4keylemma}
Suppose $K$ is a degree 4 CM-field and $A$ is a simple abelian fourfold with $K=\mathrm{End}_{\mathrm{Hdg}}(\betti)$. The set of multiplicities associated with $\mathrm{Hom}(K,\qbar)$ is $\{2,0,1,1\}$.  
\end{lemma}
\begin{proof}
We attach a proof for the convenience of the reader. Denote $\betti$ by $V$. By Lemma \ref{neweigenspacetranslate}, we have the decomposition $$V\otimes_{\mathbb{Q}}\qbar=V_{\sigma}\oplus V_{\overline{\sigma}}\oplus V_{\tau}\oplus V_{\overline{\tau}}$$ and the dimension of each summand is two. First we suppose the set of multiplicities associated with $\mathrm{Hom}(K,\qbar)$ is $\{1,1,1,1\}$. Then similar to the previous lemma, we may consider the Weil structure $\wedge^2_{K}V$ viewed inside $\mathrm{H}^2(A,\mathbb{Q})$ as a sub Hodge structure. By Remark \ref{weilhodgenumber} and the assumption that the set of multiplicities associated with $\mathrm{Hom}(K,\qbar)$ is $\{1,1,1,1\}$, every element in the four-dimensional $\mathbb{Q}$-vector space $\wedge^2_{K}V$ is a $(1,1)$ class. But because $K=\mathrm{End}_{\mathrm{Hdg}}(V)$ which is a degree 4 CM-field, by Lemma \ref{rosaticmcoincides}, the dimension of the space of $(1,1)$ classes in $\mathrm{H}^2(A,\mathbb{Q})$ is 2. Hence this gives a contradiction.

Next suppose the set of multiplicities associated with $\mathrm{Hom}(K,\qbar)$ is $\{2,0,2,0\}$. Without loss of generality, we may assume that $V_{\sigma}\subset\mathrm{H}^{1,0}$ and $V_{\tau}\subset\mathrm{H}^{1,0}$. Then we have that $V_{\sigmabar}\subset\mathrm{H}^{0,1}$ and $V_{\overline{\tau}}\subset\mathrm{H}^{0,1}$. Consider the subspace $$W=V_{\sigma}\otimes V_{\overline{\sigma}} \oplus V_{\tau}\otimes V_{\overline{\tau}} \subset \wedge^2V\otimes_{\mathbb{Q}}\qbar=\mathrm{H}^2(A,\mathbb{Q})\otimes_{\mathbb{Q}}\qbar$$ Then one can see that $W$ is in fact stable under the action of $\absgalois$, hence descends to an 8-dimensional $\mathbb{Q}$-vector subspace of $\wedge^2V$. Moreover, one can see that every class in $W$ is in fact a $(1,1)$ class by our assumption that $V_{\sigma}\subset\mathrm{H}^{1,0}$ and $V_{\tau}\subset\mathrm{H}^{1,0}$. Again, this poses a contradiction to the fact that the space of $(1,1)$ classes in $\mathrm{H}^2(A,\mathbb{Q})$ is two dimensional.
\end{proof}
\begin{proof}[Proof of Lemma \ref{nosuchsimpleAV}]We adopt notations from Setup \ref{Esubfield}. We will list all possible CM-types on $E$ such that the multiplicity condition from Lemma \ref{multiplicity-quartic-CM} is satisfied for elements in $\mathrm{Hom}(K_1,\mathbb{C})$. We then show that for each of these CM-types the multiplicity condition from Lemma \ref{multiplicity-quartic-CM} fails for $\mathrm{Hom}(K_2,\mathbb{C})$. By symmetry, it suffices to consider all CM-types such that $V_{\mathrm{id}}$ lies in $\mathrm{H}^{1,0}$. And we list them below.
\begin{enumerate}
\item The first possible CM-type on $E$ where the set of multiplicities associated with $\mathrm{Hom}(K_1,\mathbb{C})$ is $\{2,0,1,1\}$ is written in the table below.
\begin{table}[H]
    \centering
    \begin{tabular}{ccccc}
        $\mathrm{H}^{1,0}$ & $V_{\mathrm{id}}$ & $V_{x}$  & $V_{a}$ & $V_{a^3x}$\\
        $\mathrm{H}^{0,1}$ & $V_{a^2}$ & $V_{a^2x}$ & $V_{a^3}$ & $V_{ax}$\\
    \end{tabular}
    \caption{First possible CM-type on $E$ constrained by $K_1$}
    \label{1stcmtype}
\end{table}
But then checking formula (\ref{eigenk2}), every element in $\mathrm{Hom}(K_2,\mathbb{C})$ would have multiplicity 1, which contradicts Lemma \ref{multiplicity-quartic-CM}.
\item Below is the table for the second possible CM-type where the set of multiplicities associated with $\mathrm{Hom}(K_1,\mathbb{C})$ is $\{2,0,1,1\}$
\begin{table}[H]
    \centering
    \begin{tabular}{ccccc}
        $\mathrm{H}^{1,0}$ & $V_{\mathrm{id}}$ & $V_{x}$  & $V_{a^3}$ & $V_{ax}$\\
        $\mathrm{H}^{0,1}$ & $V_{a^2}$ & $V_{a^2x}$ & $V_{a}$ & $V_{a^3x}$\\
    \end{tabular}
    \caption{Second possible CM-type on $E$ constrained by $K_1$}
    \label{2ndcmtype}
\end{table}

If this were the CM-type, then checking formula (\ref{eigenk2}), every element in $\mathrm{Hom}(K_2,\mathbb{C})$ would have multiplicity 0 or 2, which also contradicts Lemma \ref{multiplicity-quartic-CM}.

\item Below is the table for the third possible CM-type on $E$ where the set of multiplicities associated with $\mathrm{Hom}(K_1,\mathbb{C})$ is $\{2,0,1,1\}$
\begin{table}[H]
    \centering
    \begin{tabular}{ccccc}
        $\mathrm{H}^{1,0}$ & $V_{\mathrm{id}}$ & $V_{a^2x}$  & $V_{a}$ & $V_{ax}$\\
        $\mathrm{H}^{0,1}$ & $V_{a^2}$ & $V_{x}$ & $V_{a^3}$ & $V_{a^3x}$\\
    \end{tabular}
    \caption{Third possible CM-type on $E$ constrained by $K_1$}
    \label{3rdcmtype}
\end{table} Then checking formula (\ref{eigenk2}), every element in $\mathrm{Hom}(K_2,\mathbb{C})$ would have multiplicity 0 or 2. This contradicts Lemma \ref{multiplicity-quartic-CM}.
\item Below is the table for the fourth possible CM-type on $E$ for which the set of multiplicities associated with $\mathrm{Hom}(K_1,\mathbb{C})$ is $\{2,0,1,1\}$
\begin{table}[H]
    \centering
    \begin{tabular}{ccccc}
        $\mathrm{H}^{1,0}$ & $V_{\mathrm{id}}$ & $V_{a^2x}$  & $V_{a^3}$ & $V_{a^3x}$\\
        $\mathrm{H}^{0,1}$ & $V_{a^2}$ & $V_{x}$ & $V_{a}$ & $V_{ax}$\\
    \end{tabular}
    \caption{Fourth possible CM-type on $E$ constrained by $K_1$}
    \label{4thcmtype}
\end{table}

Then checking formula (\ref{eigenk2}), every element in $\mathrm{Hom}(K_2,\mathbb{C})$ would have multiplicity 1, contradicting Lemma \ref{multiplicity-quartic-CM}. 
\end{enumerate}
\end{proof}

\begin{remark}
There is an alternative proof of Lemma \ref{nosuchsimpleAV} using Proposition 1.9 of \cite{milne2006complex}. One can directly show that there does not exist a primitive CM-type on $E$ where $E$ is a degree 8 Galois CM-field with $\mathrm{Gal}(E/\mathbb{Q})=\mathrm{D}_4$. 
\end{remark}

\subsection{Proof of Theorem \ref{mainthmcm}}
In this section we finish the proof that for a simple CM abelian fourfold $A$, its de Rham-Betti group coincide with its Mumford-Tate group. First we state the theorem from \cite{mz-4-folds} which computes the Mumford-Tate group of simple CM abelian fourfolds.

\begin{theorem}[Theorem 7.6, \cite{mz-4-folds}]\label{mzcmthm}
  Given a simple CM abelian fourfold $A$ with $\mathrm{End}^{0}(A)=E$, where $E$ is a degree 8 CM-field (not necessary Galois), we have the following two possibilities
\begin{enumerate}
\item There exists $k\xhookrightarrow{}E$ where $k$ is an imaginary quadratic field such that, associated with the embedding $k\xhookrightarrow{}E\xhookrightarrow{} \mathrm{End}_{\mathrm{Hdg}}(\betti)$, each element in $\mathrm{Hom}(k,\mathbb{C})$ has multiplicity 2. Then $\mathrm{Hdg}(A)=\mathrm{SU}_{E/k}$ (see Definition \ref{sudefn}).
\item $E$ does not contain an imaginary quadratic field $k$ such that each element in $\mathrm{Hom}(k,\mathbb{C})$ has multiplicity 2. Then $\mathrm{Hdg}(A)=\mathrm{U}_{E}$.
\end{enumerate} 
\end{theorem}

\begin{definition}\label{sudefn}
    Given a CM-field $E$ and a quadratic imaginary subfield $k\xhookrightarrow{}E$ we define $\mathrm{SU}_{E/k}$ as the $\mathbb{Q}$-subtorus of $\resgmmath$ whose $\mathbb{Q}$-points satisfy the condition  $\mathrm{SU}_{E/k}(\mathbb{Q})=\{x \in E^{\times}|\mathrm{det}_{k}(E \xrightarrow{\cdot x}E)=1\}$. 
\end{definition}
By 7.2 of \cite{mz-4-folds}, $\mathrm{SU}_{E/k}$ is a connected subtorus of $\mathrm{U}_{E}$ of codimension 1 (one can also see this by the proof of Lemma \ref{weilstructurefixer}).

The goal of this section is to prove the following theorem.

\begin{theorem} \label{mainthmcm}
    For any simple CM abelian fourfold $A$ defined over $\qbar$, we have that $\gdrbhmath(A)=\mathrm{Hdg}(A)$. Then by virtue of Proposition \ref{gm in CM gdrb}, we have that $\mathrm{MT}(A)=\gdrbmath(A)$.
\end{theorem}

Firstly, the previous sections already give a lower bound on the dimension of $\gdrbhmath(A)$.
\begin{proposition} \label{biggerthan2}
   For any simple CM abelian fourfold $A$ defined over $\qbar$, $\mathrm{dim}(\gdrbhmath(A)) > 2$.
\end{proposition}
\begin{proof}[Proof of Proposition \ref{biggerthan2}]
This follows from the classification in Lemma \ref{transitiveclassification} and computations in Lemma \ref{nolength4cycle}, Proposition \ref{noklein4} and Lemma \ref{A_4notpossible}. And the caveat mentioned in Remark \ref{keyremarktotakeaway} is filled in by Lemma \ref{nosuchsimpleAV}.
\end{proof}
\begin{remark}\label{ribetinequalitu}
There is a more general analogous statement about lower bound of the dimension of the Mumford-Tate group of simple CM abelian varieties, which is called Ribet's inequality in folklore (see \cite[3.5]{ribet1980division}). The statement is that the dimension of the Mumford-Tate group associated with an irreducible CM-type $(E,S)$ with $\mathrm{deg}(E/\mathbb{Q})=2d$ is bigger than or equal to the number $2+\mathrm{log}_2(d)$. And it is generalized in \cite[Theorem 1.1]{orr2013lowerboundsranksmumfordtate}, which states that for a simple abelian variety $A$ of dimension $d$ whose endomorphism algebra is commutative, we have that $\mathrm{dim}(\mathrm{MT}(A))\geq 2+\mathrm{log}_2(d)$. The proof for both the original inequality and the generalized version used the definition of Mumford-Tate groups as the smallest algebraic defined over $\mathbb{Q}$ whose $\mathbb{R}$-points contain the image of the Deligne torus (see Definition \ref{delignetorusmt}). However, due to the limited knowledge of the transcendence of periods or location of de Rham-Betti classes, it is not clear how to establish such an inequality about the dimension of the de Rham-Betti group of an abelian variety. And it is this lack of knowledge that leads to the lengthy computations in Section \ref{longcomp}. 
\end{remark}
One final ingredient needed in the proof of Theorem \ref{mainthmcm} is a piece of precise information about the comparison isomorphism of the Weil structure associated with $k\xhookrightarrow{}\mathrm{\betti}$ with $k$ a quadratic imaginary field (see Section \ref{weilstruc} for a brief introduction about Weil structures). And this is provided by \cite{gross1978periods}.
\begin{lemma}[Theorem 3, \cite{gross1978periods}]\label{gross}
For a simple abelian variety $A$ defined over $\qbar$ of dimension $n$, suppose we have $k\xhookrightarrow{} \mathrm{End}^{\circ}(A)$ where $k$ is an imaginary quadratic field. Suppose the set of multiplicities associated with $\mathrm{Hom}(k,\qbar)$ (see Definition \ref{multiplicitydefn}) is $\{p,n-p\}$, then the comparison isomorphism of the $\qbar$-linear extension of the Weil de Rham-Betti structure (compare with Remark \ref{explicityweildrb}) $$\wedge_{k}^{n}\mathrm{H}^1_{\mathrm{dRB}}(A,\mathbb{Q})\otimes \qbar:=(\wedge_{k}^{n}\mathrm{H}^1_{\mathrm{B}}(A,\mathbb{Q})\otimes_{\mathbb{Q}} \overline{\mathbb{Q}}, \wedge_{k \otimes \overline{\mathbb{Q}}}^{n}\mathrm{H}^1_{\mathrm{dR}}(A/\qbar),\rho_{m})$$ can be diagonalized to\begin{equation}\label{grosscomparison}\begin{pmatrix} b_{k}^{p} (2\pi i/b_{k})^{n-p} & 0 \\0 & b_{k}^{
n-p} (2\pi i/b_{k})^{p} \end{pmatrix}\end{equation} where $b_{k}$ is the Chowla-Selberg constant associated with the quadratic imaginary field $k$.
\end{lemma}

\begin{corollary}\label{grosscorollary}
We adopt the setup and notation of the above lemma and suppose $n=4$. If $p=1$, then none of the elements in $\rho_{m}(\wedge_{k}^4\mathrm{H}^1_{\mathrm{B}}(A,\mathbb{Q}))$ falls inside $$\wedge_{k \otimes \overline{\mathbb{Q}}}^4\mathrm{H}^1_{\mathrm{dR}}(A,\qbar) \otimes \qbar(2\pi i)^2$$ If $p=2$, every element in $\rho_{m}(\wedge_{k}^4\mathrm{H}^1_{\mathrm{B}}(A,\mathbb{Q}))$ falls inside $$\wedge_{k \otimes \overline{\mathbb{Q}}}^4\mathrm{H}^1_{\mathrm{dR}}(A,\qbar) \otimes \qbar(2\pi i)^2$$  
\end{corollary}

\begin{proof}
 Denote the choice of $\qbar$-basis of the Betti cohomology group in formula (\ref{grosscomparison}) by $e_1,e_2$ and denote the choice of $\qbar$-basis of the de Rham cohomology group by $f_1,f_2$. Then for any element $\alpha=q_1e_1+q_2e_2 \in \wedge_{k}^4\mathrm{H}^1_{\mathrm{B}}(A,\mathbb{Q})\subset\wedge_{k}^4\mathrm{H}^1_{\mathrm{B}}(A,\mathbb{Q})\otimes \qbar (q_1,q_2\in \qbar)$ we have $$\rho_{m}(\alpha)=q_1b_{k}^{p} (2\pi i/b_{k})^{4-p}f_1+q_2b_{k}^{
4-p} (2\pi i/b_{k})^{p}f_2$$ By the renowned theorem of Chudnovski \cite{chudnovsky1980algebraic}, we have that $b_{k}$ is a transcendental number and it is $\qbar$-\textit{algebraically} independent from $2\pi i$. Then by a simple inspection of the above $\rho_{m}(\alpha)$, we obtain the statement in the corollary.
 \end{proof}
\begin{remark}
Note that this observation is analogous to the statement from \cite{mz-4-folds} that the Hodge structure $\wedge_{k}^4\mathrm{H}^1_{\mathrm{B}}(A,\mathbb{Q}) \otimes \mathbb{Q}(2)$ consists entirely of Hodge classes when the multiplicity of each element of $\mathrm{Hom}(k,\mathbb{C})$ is equal to $2$ and it contains no Hodge class when the set of multiplicities associated to $\mathrm{Hom}(k,\mathbb{C})$ is $\{1,3\}$ (See also Remark \ref{weilhodgenumber}). 
\end{remark}
For the proof of Theorem \ref{mainthmcm}, by virtue of Proposition \ref{biggerthan2}, it suffices to study all subtori of $\mathrm{U}_{E}$ whose dimension is bigger than or equal to 3. We will be using the classification of such tori from Key Lemma 7.3 in \cite{mz-4-folds}, which we record below.
\begin{lemma}[Key Lemma 7.3, \cite{mz-4-folds}]\label{keylemmamz}
   Let $E$ be a CM-field. Suppose $T$ is a (connected) codimension 1 subtorus of $\mathrm{U}_{E}$. Then there exists a quadratic imaginary field $k\xhookrightarrow{}E$ such that $T=\mathrm{SU}_{E/k}$ (see Definition \ref{sudefn}).
\end{lemma}

\begin{proof}[Proof of Theorem \ref{mainthmcm}]
Equipped with Proposition \ref{biggerthan2} and Corollary \ref{grosscorollary}, we are now ready to apply the argument from 7.6 of \cite{mz-4-folds} verbatim. First suppose the endomorphism field $E$ of the simple CM abelian fourfold $A$ does not contain a quadratic imaginary subfield $k$, then according to Lemma \ref{keylemmamz}, there is no subtorus of codimension 1 inside $\mathrm{U}_{E}$. Therefore, $\gdrbhmath(A)$ is equal to $\mathrm{U}_{E}$ since we know its dimension is bigger than 2 by Proposition \ref{biggerthan2}. Therefore in this case $\gdrbhmath(A)=\mathrm{Hdg}(A)=\mathrm{U}_{E}$.

Now suppose $E$ does contain a quadratic imaginary subfield. First suppose for each quadratic imaginary subfield $k\xhookrightarrow{}E$, the set of multiplicities associated with $\mathrm{Hom}(k,\qbar)$ is $\{1,3\}$. Then by Corollary \ref{grosscorollary}, there is no de Rham-Betti class in $\wedge_{k}^4\mathrm{H}^1_{\mathrm{dRB}}(A,\mathbb{Q})\otimes \mathbb{Q}_{\mathrm{dRB}}(2)$. Hence by Corollary \ref{gdrbhinv}, we have that $(\wedge_{k}^4\mathrm{H}^1_{\mathrm{dRB}}(A,\mathbb{Q}))^{\gdrbhmath(A)}=\{0\}$. Now suppose $\gdrbhmath(A)$ is a codimension 1 subtorus of $\mathrm{U}_{E}$. Then by Lemma \ref{keylemmamz}, it is equal to $\mathrm{SU}_{E/k_{0}}$ for some $k_{0}\xhookrightarrow{}E$. Then according to Lemma \ref{weilstructurefixer}, every element in $\wedge_{k_{0}}^4\mathrm{H}^1(A,\mathbb{Q})$ would be fixed by $\mathrm{SU}_{E/k_{0}}$, which is a contradiction. Therefore $\gdrbhmath(A)=\mathrm{U}_{E}=\mathrm{Hdg}(A)$.

Suppose there exists $k_{0}\xhookrightarrow{}E$ with $k_{0}$ a quadratic imaginary field such that the set of multiplicities associated with $\mathrm{Hom}(k_{0},\qbar)$ is $\{2,2\}$. By Corollary \ref{grosscorollary}, every element in $\wedge_{k}^4\mathrm{H}^1_{\mathrm{dRB}}(A,\mathbb{Q})\otimes \mathbb{Q}_{\mathrm{dRB}}(2)$ is a dRB class. Hence by Corollary \ref{gdrbhinv}, we have that $$(\wedge_{k_0}^4\mathrm{H}^1_{\mathrm{dRB}}(A,\mathbb{Q}))^{\gdrbhmath(A)}=\wedge_{k_0}^4\mathrm{H}^1_{\mathrm{dRB}}(A,\mathbb{Q})$$ Therefore we can deduce that $$\gdrbhmath(A)\subset\mathrm{U}_{E} \cap \mathrm{SL}_{k_{0}}(\mathrm{H}^1(A,\mathbb{Q}))=\mathrm{SU}_{E/k_{0}}$$ But we have by Proposition \ref{biggerthan2} that $\mathrm{dim}(\gdrbhmath(A))\geq 3$ therefore in this case $\gdrbhmath(A)=\mathrm{SU}_{E/k_{0}}=\mathrm{Hdg}(A)$.

\end{proof}

\section{Centre of De Rham-Betti Groups of Simple Type IV Abelian Varieties}\label{chaptercentre}
In this section, we adopt a strategy similar to Section \ref{cmgdrbsection} to prove that the centre of the de Rham-Betti Lie algebra of a simple abelian fourfold of type IV coincides with the centre of its Mumford-Tate Lie algebra. Similar to Section \ref{cmgdrbsection}, we start by showing the existence of homotheties in the centre of the de Rham-Betti group of a simple abelian variety of type IV (see Definition \ref{type4defn}). Moreover, in Section \ref{seconedimobj}, we discuss some interesting corollaries regarding the classification of one-dimensional objects in the Tannakian category generated by the de Rham-Betti structure of a type IV abelian variety. We stress that Theorem \ref{reductive} which states that the de Rham-Betti group of an abelian variety is reductive is used in a crucial way.
\subsection{Existence of Homotheties in the Centre}
We begin this section with the following well-known lemma.
\begin{lemma} \label{centremakessense}
   Given a reductive Lie algebra $L$ defined over $\mathbb{Q}$, we have the following decomposition $$L=\mathrm{Z}(L) \bigoplus L^{\mathrm{ss}}$$ where $\mathrm{Z}(L)$ is the centre of $L$ and $L^{\mathrm{ss}}$=$[L,L]$, the semisimple part of $L$.
\end{lemma}

\begin{proof}
    By definition, $L_{\qbar}:=L\otimes_{\mathbb{Q}}\qbar$ is a reductive Lie algebra defined over $\qbar$ and therefore $L_{\qbar}=\mathrm{Z}(L_{\qbar}) \bigoplus L_{\qbar}^{\mathrm{ss}}$ (for example see Ex 6.5 of \cite{humphreys2012introduction}). But notice that the Lie bracket on $L_{\qbar}$ is in fact defined in $\mathbb{Q}$, hence the decomposition descends to $\mathbb{Q}$.
\end{proof}
In the sequel we will denote the Lie algebra of $\gdrbmath(A)$ by $\liegdrbmath(A)$. In the case of abelian varieties defined over $\qbar$, by Theorem \ref{reductive}, we have that $\gdrbmath(A)$ is a reductive algebraic group. Hence $\liegdrbmath(A)$ is a reductive Lie algebra and we will denote by $\liegdrbssmath(A)$ its semisimple part i.e. $\liegdrbssmath(A)=[\liegdrbmath(A),\liegdrbmath(A)]$. 
\begin{lemma}\label{centreincentre}
    For any abelian variety $A$ defined over $\qbar$, the centre of its de Rham-Betti group $\mathrm{Z}(\gdrbmath(A))$ naturally lives in the centre of its Mumford-Tate group $\mathrm{Z}(\mathrm{MT}(A))$. Also, $\mathrm{Z}(\gdrbmath(A))$ naturally lives in the centre of the endomorphism algebra of $A$.
\end{lemma}

\begin{proof}
By definition we have $$\mathrm{Z}(\mathrm{MT}(A))=\mathrm{End}_{\mathrm{Hdg}}(\mathrm{H}^1(A,\mathbb{Q})) \cap \mathrm{MT}(A)$$ and  $$\mathrm{Z}(\gdrbmath(A))=\mathrm{End}_{\mathrm{dRB}}(\mathrm{H}^1(A,\mathbb{Q})) \cap \gdrbmath(A)$$ Recall that we have the natural inclusion of algebraic groups $$\gdrbmath(A) \xhookrightarrow{} \mathrm{MT}(A)$$ Moreover, by Theorem \ref{mainfact} we have that $$\mathrm{End}_{\mathrm{Hdg}}(\mathrm{H}^1(A,\mathbb{Q})) = \mathrm{End}_{\mathrm{dRB}}(\mathrm{H}^1(A,\mathbb{Q}))$$ Therefore, the inclusion $\mathrm{Z}(\gdrbmath(A))\xhookrightarrow{}\mathrm{Z}(\mathrm{MT}(A))$ follows.  On the other hand, we have that $$\mathrm{Z}(\gdrbmath(A)) \xhookrightarrow{} \mathrm{End}_{\mathrm{dRB}}(\betti)$$ By definition, elements in $\mathrm{End}_{\mathrm{dRB}}(\betti)$ commute with elements in $$\mathrm{Z}(\gdrbmath(A)) \xhookrightarrow{}
\gdrbmath(A)$$ Hence the second statement in the lemma follows.
\end{proof}
In this section we will also use a well known convenient lemma from algebraic group theory.
\begin{lemma}\label{surjectiontotorus}
    Suppose a reductive group $G$ admits a surjective homomorphism $f: G\rightarrow H$ where $H$ is connected commutative algebraic group. Assume both $G$ and $H$ are over a field of characteristics zero. Then the restriction of $f$ to $\mathrm{Z}(G)$ is a also a surjective homomorphism of algebraic groups.
\end{lemma}
\begin{proof}
Because $f$ is a surjective homomorphism, we have $$\mathrm{Lie}(f): L:=\mathrm{Lie}(G) \twoheadrightarrow \mathrm{Lie}(H)$$
Because $G$ is a reductive algebraic group, we have that  $L$ is a reductive Lie algebra and therefore $L=\mathrm{Z}(L)\oplus L^{\mathrm{ss}}=\mathrm{Lie}(\mathrm{Z}(G))\oplus [L,L]$. Moreover because $H$ is a commutative algebraic group and $[L,L]$ is a semisimple Lie algebra, the image of $\mathrm{Lie}(f)|_{[L,L]}$ is $\{0\}$. We then obtain that $\mathrm{Lie}(f)|_{\mathrm{Lie}(\mathrm{Z}(G))}$ surjects to $\mathrm{Lie}(H)$. By Proposition 3.25 of \cite{milne2011algebraic}, we have that $H^{\circ}\subset f(Z(G))$. Since $H$ is assumed to be a connected algebraic group, we obtain that $$f|_{\mathrm{Z}(G)}:\mathrm{Z}(G) \twoheadrightarrow H$$
\end{proof}
We are now ready to apply the above observations to a simple abelian variety of type IV defined over $\qbar$, which we now define.

\begin{definition}\label{type4defn}
We say $A$ is a \textit{simple abelian variety of type IV} if the endomorphism algebra of the abelian variety $A$ is a division algebra of type IV in the Albert classification. See Theorem 2, Chapter 19 of \cite{mumford1970abelian} for the precise definition of a type IV division algebra. 
\end{definition}
\begin{example}
    A simple CM abelian variety is a special case of simple abelian variety of type IV.
\end{example}
\begin{remark}
In this section, we are going to use that the centre of a type IV division algebra is a CM-field. Moreover, given $A$ a simple abelian variety of type IV, the Rosati involution coming from a polarization form on $A$, when restricted to $E=\mathrm{Z}(\mathrm{End}^{\circ}(A))$, is equal to the complex conjugation on the CM-field $E$. Both properties can be found in Chapter 19 of \cite{mumford1970abelian}.
\end{remark}
We now prove a proposition similar in spirit to Proposition \ref{gm in CM gdrb}.
\begin{proposition} \label{gmintype4}
For any simple abelian variety $A$ of type IV defined over $\qbar$, there exists a natural homomorphism of $\gmmath$ into $\gdrbmath(A)$ such that the following diagram commutes
$$
\begin{tikzcd}
\gmmath \arrow[r] \arrow[rd] & \gdrbmath(A) \arrow[d] \\
                             & \mathrm{MT}(A)        
\end{tikzcd}
$$ The homomorphism from $\gmmath$ to $\mathrm{MT}(A)$ in the diagram is the canonical one determined by the weight of the Hodge structure i.e. on $\mathbb{Q}$-points, it sends $t\in\gmmath$ to scalar multiplication by $t^{-1}$ in $\mathrm{GL}(\betti)$.
\end{proposition}

\begin{proof}
The idea of the proof is to show that we have a natural inclusion $\gmmath \xhookrightarrow{} \mathrm{Z}(\gdrbmath(A))$. Denote the endomorphism algebra of $A$ by $R$ and denote the centre of $R$ by $E$. By Definition \ref{type4defn}, $R$ is a division algebra of type IV and $E$ is a CM-field. To apply the method from Section \ref{gmincmsection}, we aim to establish the connection between $\mathrm{Z}(\gdrbmath(A))$ and $\resgmmath$.

Denote $\betti$ by $V$. Then we have a natural inclusion $R \xhookrightarrow{} \mathrm{End}(V)$ and its image commutes with $\mathrm{MT}(A)$.  Moreover, any polarization form $$\phi: V \times V \rightarrow \mathbb{Q}(-1)$$ is preserved by $\mathrm{MT}(A)$ up to multiplication by scalars. To be more precise, we have the following inclusion of algebraic groups $$\mathrm{MT}(A) \xhookrightarrow{}\mathrm{GSp}_{R}(V,\phi) \xhookrightarrow{} \mathrm{GL}(V)$$ where $\mathrm{GSp}_{R}(V,\phi)$ is the algebraic group defined over $\mathbb{Q}$ whose $\mathbb{Q}$-points are $$\{g \in \mathrm{GL}(V)\cap\mathrm{End}_{R}(V)|\exists \chi(g) \in \mathbb{Q}^{\times}\forall v,v' \in V\phi(g\circ v,g\circ v')=\chi(g)\phi(v,v')\}$$ where $\chi: \mathrm{GSp}_{R}(V,\phi) \rightarrow \gmmath$ is in fact a morphism of algebraic groups defined over $\mathbb{Q}$.
Now by definition we have the natural inclusions $R \subset \mathrm{End}_{\mathrm{GSp}_{R}(V,\phi)}(V)$ and $\mathrm{MT}(A) \xhookrightarrow{}\mathrm{GSp}_{R}(V,\phi)$. By the assumption in the proposition we have that $$\mathrm{End}_{\mathrm{MT}(A)}(V)=R$$ Hence we have $R=\mathrm{End}_{\mathrm{GSp}_{R}(V,\phi)}(V)$. Therefore we have 
$$\mathrm{End}_{\gdrbmath(A)}(V)=\mathrm{End}_{\mathrm{MT}(A)}(V)=\mathrm{End}_{\mathrm{GSp}_{R}(V,\phi)}(V)=R$$
We also have the chain of inclusions $$\gdrbmath(A) \xhookrightarrow{} \mathrm{MT}(A) \xhookrightarrow{} \mathrm{GSp}_{R}(V,\phi)$$ Note that by definition we have
$Z(G)=\mathrm{End}_{G}(V) \cap G$ for any algebraic group $G \xhookrightarrow{} \mathrm{GL}(V)$. Then applying Lemma \ref{centreincentre}, we have the following homomorphisms relating the centre of these algebraic groups \begin{equation}\label{chainofcentres}\mathrm{Z}(\gdrbmath(A)) \xhookrightarrow{} \mathrm{Z}(\mathrm{MT}(A)) \xhookrightarrow{} \mathrm{Z}(\mathrm{GSp}_{R}(V,\phi)) \xhookrightarrow{} \mathrm{Res}_{E/\mathbb{Q}}\gmmath\end{equation}

Next we fix a polarization form $\phi$ such that the Rosati involution associated with $\phi$ on $R$ preserves $E$ and coincides with the complex conjugation on $E$. By \cite[Lemma 9.2]{kottwitz1992points}, such $\phi$ always exists. We claim the following diagram is commutative 
\begin{equation}\label{typeivkeydiagram}
\begin{tikzcd}
\Phi^{-1}(\mathrm{Z}(\gdrbmath(A)))^{\circ} \arrow[d, hook] \arrow[rr,two heads]   &  & \mathrm{Z}(\gdrbmath(A))^{\circ} \arrow[d, hook] \arrow[rr, "\mathrm{det}",two heads]      &                                          & \gmmath \arrow[d, Rightarrow, maps to] \\
\gmmath \times \mathrm{Z}(\mathrm{Hdg}(A)) \arrow[d, hook] \arrow[rr,"\Phi^{'}",two heads] &  & \mathrm{Z}(\mathrm{MT}(A)) \arrow[d, hook] \arrow[rr,"\mathrm{det}",two heads]    &                                          & \gmmath \arrow[d, Rightarrow]          \\
\gmmath\times\mathrm{U}_{E} \arrow[rr, "\Phi"]                   &  & {\mathrm{Z}(\mathrm{GSp}_{R}(V,\phi))} \arrow[rr, "\mathrm{det}", two heads] & & \gmmath                               
\end{tikzcd}
\end{equation} The homomorphism $\Phi$ in the above diagram on $\mathbb{Q}$-points sends $(t,g)\in \gmmath \times \mathrm{U}_{E}$ to $t^{-1}g\in\mathrm{GL}(V)$. Recall we have chosen the polarization form $\phi$ such that its Rosati involution restricted to $E$ coincides with the complex conjugation on $E$. Then for any element $g\in \mathrm{U}_{E}\subset R$ and any $v,v'\in V$, the polarization form $\phi$ satisfies that \begin{equation}\label{eqqs}\phi(gv,gv')=\phi(v,\overline{g}gv')=\phi(v,v')\end{equation} Hence tracing definitions carefully, one can verify that the image of $\Phi$ indeed lies in $\mathrm{Z}(\mathrm{GSp}_{R}(V,\phi))$.  We now verify that $\mathrm{Z}(\mathrm{Hdg}(A))\subset\mathrm{U}_{E}$. The reason is as follows. By Lemma \ref{centreincentre}, we have that $\mathrm{Z}(\mathrm{Hdg}(A))\subset \resgmmath\subset R$. By formula (\ref{eqqs}), for any element $g\in \mathrm{Z}(\mathrm{Hdg}(A))$ and any $v,v'\in V$, we have that the polarization form $\phi$ satisfies that $$\phi(gv,gv')=\phi(v,v')=\phi(v,\overline{g}gv')$$ where the first equality holds because $\phi$ is preserved by $\mathrm{Hdg}(A)$. Since $\phi$ is a non-degenerate form, we deduce that $\overline{g}g=1$. This implies that $\mathrm{Z}(\mathrm{Hdg}(A))\subset\mathrm{U}_{E}$. The map $\Phi'$ comes from the restriction of the canonical isogeny $\gmmath \times \mathrm{Hdg}(A)\rightarrow\mathrm{MT}(A)$. Since the image of $\gmmath \times \{\mathrm{id}\}$ inside $\mathrm{MT}(A)$ is the set of homotheties, the surjectivity of $\Phi'$ follows. As for the surjectivity of the two determinant maps, denote the dimension of $A$ by $n$. Then we have that $\wedge^{2n}_{\mathbb{Q}}\mathrm{H}^{1}_{\mathrm{dRB}}(A,\mathbb{Q})\cong\mathbb{Q}_{\mathrm{dRB}}(-n)$ from which we obtain a surjective morphism of algebraic groups $\mathrm{det}:\gdrbmath(A)\rightarrow\gmmath$. By Lemma \ref{surjectiontotorus}, we obtain that $\mathrm{det}|_{\mathrm{Z}(\gdrbmath(A))}$ is surjective. Then using a similar argument, we have that $\mathrm{det}|_{\mathrm{Z}(\mathrm{MT}(A))}$ is a surjective homomorphism of algebraic groups.

Denote by $\mathrm{pr}_2$ the projection map from $\gmmath \times \mathrm{Z}(\mathrm{Hdg}(A))$ to $\mathrm{Z}(\mathrm{Hdg}(A))$ and denote by $\mathrm{pr}_1$ the projection map to $\gmmath$.
Let $$T'=\mathrm{pr}_2(\Phi^{'-1}(\mathrm{Z}(\gdrbmath(A)))^{\circ})$$ Then $T'$ is a connected subtorus of $\mathrm{Z}(\mathrm{Hdg}(A))\subset\mathrm{U}_{E}$. Now from the above diagram (\ref{typeivkeydiagram}) and a similar argument to Lemma \ref{pr1surj}, we obtain two surjective homomorphisms of algebraic groups $\mathrm{pr}_1: \Phi^{'-1}(\mathrm{Z}(\gdrbmath(A)))^{\circ}\twoheadrightarrow \gmmath$ and $\mathrm{pr}_2: \Phi^{'-1}(\mathrm{Z}(\gdrbmath(A)))^{\circ}\twoheadrightarrow T'$. Then using Corollary \ref{gminside}, we have that $$\Phi^{'-1}(\mathrm{Z}(\gdrbmath(A)))^{\circ}\cap \gmmath \times \{\mathrm{id}\}=\gmmath \times \{\mathrm{id}\} \xhookrightarrow{} \gmmath \times \mathrm{U}_{E}$$ Therefore we have obtained the commutative diagram $$\begin{tikzcd}
\gmmath \arrow[r] \arrow[rd] & \mathrm{Z}(\gdrbmath(A)) \arrow[d] \\
                             & \mathrm{Z}(\mathrm{MT}(A))        
\end{tikzcd}$$ which implies the statement of the proposition.

\end{proof}

\subsection{Centre of De Rham-Betti Groups of Simple Type IV Abelian Fourfolds}
By the Albert classification, for any simple type IV abelian fourfold which is not of CM-type, its endomorphism algebra is either a degree 4 CM-field or a degree 2 CM-field (see \cite[Section 2.9]{moonen1999hodge} for example). When the endomorphism algebra is a degree 2 CM-field i.e. a quadratic imaginary field, we will further distinguish between the next two cases. 
\begin{definition}\label{defnweiltype}
 A simple abelian fourfold whose endomorphism algebra $E$ is a quadratic imaginary field is called an \textit{abelian fourfold of Weil type} if each element of $\mathrm{Hom}(E,\mathbb{C})$ has multiplicity 2 (see Definition \ref{multiplicitydefn}). If, on the other hand, the set of multiplicities associated with $\mathrm{Hom}(E,\mathbb{C})$ is $\{1,3\}$, then it is called an \textit{abelian fourfold of anti-Weil type}.
\end{definition}
\begin{remark}
By Proposition 14 of \cite{shimura1963analytic}, for a simple abelian fourfold whose endomorphism algebra $E$ is a quadratic imaginary field, the set of multiplicities associated with $\mathrm{Hom}(E,\mathbb{C})$ is either $\{2\}$ or $\{1,3\}$. 
\end{remark}
\begin{definition}\label{liegdrbdefn}
Given an abelian variety $A$ defined over $\qbar$, we will call the Lie algebra of the de Rham-Betti group $\gdrbmath(A)$ the \textit{de Rham-Betti Lie algebra of $A$}. We denote it by $\liegdrbmath(A)$. Similarly, we denote the Mumford-Tate Lie algebra of $A$ by $\mathrm{mt}(A)$ and the Hodge Lie algebra of $A$ by $\mathrm{hdg}(A)$. 
\end{definition}
In this section, we will show that for a simple type IV abelian fourfold defined over $\qbar$, the centre of its de Rham-Betti Lie algebra coincides with the centre of its Mumford-Tate Lie algebra. 

We start by summarizing results about the Lie algebra of the Hodge group of a simple type IV abelian fourfold. The reference is \cite[Section 2.6]{mz-4-folds}. Given a simple type IV abelian fourfold $A$ with $E=\mathrm{End}_{\mathrm{Hdg}}(\mathrm{H}^1(A,\mathbb{Q}))$, we fix a polarization $$\phi: \mathrm{H}^1(A,\mathbb{Q}) \times  \mathrm{H}^1(A,\mathbb{Q}) \rightarrow \mathbb{Q}(-1)$$ By Lemma \ref{rosaticmcoincides}, the Rosati involution on $E$ associated with $\phi$ is equal to the complex conjugation on $E$ (compare with Lemma 9.2 in \cite{kottwitz1992points}). Moreover, if we fix an $\alpha \in E$ such that $\overline{\alpha}=-\alpha$, there is a unique $E$-hermitian form $$\psi: \mathrm{H}^1(A,\mathbb{Q}) \times  \mathrm{H}^1(A,\mathbb{Q}) \rightarrow E \otimes_{\mathbb{Q}}\mathbb{Q}(-1)$$ such that $\mathrm{Tr}_{E/\mathbb{Q}}(\alpha\psi)=\phi$. Then the results from \cite{mz-4-folds} are summarized below:

\begin{theorem}[Table 1 in p.578 of \cite{mz-4-folds}]
\label{deg2mz4}
Let $A$ be a simple type IV abelian fourfold and denote $\mathrm{End}^{\circ}(A)$ by $E$. 
\begin{itemize}
    \item If $\mathrm{deg}(E)=4$, then $\mathrm{hdg}(A)$ is equal to: $$u_{E}(V,\psi):=\{M \in \mathrm{End}_{E}(\mathrm{H}^1(A,\mathbb{Q}))|\psi(Mv,w)+\psi(v,Mw)=0, \forall v,w \in \mathrm{H}^1
    (A,\mathbb{Q})\}$$ The centre of $\mathrm{hdg}(A)$ can be identified as the subspace $\{x \in E|x+\overline{x}=0\}$.
    \item If $A$ is an anti-Weil type abelian fourfold then $$\mathrm{hdg}(A)=u_{E}(V,\psi)$$ as well. The centre of $\mathrm{hdg}(A)$ can also be identified as $\{x \in E|x+\overline{x}=0\}$. 
    \item If $A$ is a Weil type abelian fourfold, then $$\mathrm{hdg}(A)=su_{E}(V,\psi):=\{M \in u_{E}(V,\psi)|\mathrm{Tr}_{E}(M)=0\}$$ The centre of $\mathrm{hdg}(A)$ in this case is trivial.
\end{itemize}
 
\end{theorem}

We will determine the centre of the de Rham-Betti Lie algebras of simple type IV abelian fourfolds using results about de Rham-Betti groups of certain Weil structures (see Section \ref{weilstruc}). This is done in Lemma \ref{13group} and Lemma \ref{deg4weil}.
\begin{lemma} \label{13group}
Let $A$ be a simple anti-Weil type abelian fourfold define over $\qbar$ with endomorphism field $E$, a quadratic imaginary field. Then the de Rham-Betti group of the Weil structure $$W_{\mathrm{Weil}}:=\wedge_{E}^4\mathrm{H}^1_{\mathrm{dRB}}(A,\mathbb{Q})$$ is isomorphic to $\resgmmath$.
\end{lemma}

\begin{proof}
This is a dimensional argument. By Lemma \ref{drbgroupdim}, for an arbitrary de Rham-Betti structure $(W_{\mathrm{B}},W_{\mathrm{dR}},\rho_{m})$, the dimension of its de Rham-Betti group is bigger than or equal to $\mathrm{trdeg}_{\qbar}\qbar(\alpha_{ij})$, where $\alpha_{ij}$s are entries of the matrix of the comparison isomorphism after choosing a basis for $W_{\mathrm{B}},W_{\mathrm{dR}}$ respectively. By Lemma \ref{gross} the matrix of the comparison isomorphism of $W_{\mathrm{Weil}} \otimes_{\mathbb{Q}} \qbar := (\wedge_{E}^4\mathrm{H}^1(A,\mathbb{Q}) \otimes_{\mathbb{Q}} \qbar, \wedge_{E \otimes \overline{\mathbb{Q}}}^4\mathrm{H}^1_{\mathrm{dR}}(A/\qbar), \rho_{m})$ can be diagonalized as \[\begin{pmatrix} b_{k}^{1} (2\pi i/b_{k})^{3} & 0 \\0 & b_{k}^{3} (2\pi i/b_{k})^{1} \end{pmatrix}\] And because $b_{k}$ is algebraically independent from $2\pi i$ by \cite{chudnovsky1980algebraic}, we have $$\mathrm{dim}(\gdrbmath(W_{\mathrm{Weil}})) \geq 2$$ By Lemma \ref{morphismofweilstruc}, we have a natural morphism of $\mathbb{Q}$-algebras $E \xhookrightarrow{} \mathrm{End}_{\mathrm{dRB}}(W_{\mathrm{Weil}})$. This induces a morphism of $\mathbb{Q}$-algebraic groups: $$\resgmmath \xhookrightarrow{} \mathrm{GL}(W_{\mathrm{Weil}})$$ We abuse the notation by denoting the Betti part of the Weil structure by $W_{\mathrm{Weil}}$ as well. Because $\mathrm{deg}(E/\mathbb{Q})=\mathrm{dim}_{\mathbb{Q}}W_{\mathrm{Weil}}$, this realizes $\resgmmath$ as a maximal tori inside $\mathrm{GL}(W_{\mathrm{Weil}})$. Therefore we have $$\gdrbmath(W_{\mathrm{Weil}}) \xhookrightarrow{} \resgmmath$$ But note that $\mathrm{dim}(\resgmmath)=2$, therefore $\gdrbmath(W_{\mathrm{Weil}})=\resgmmath$.
\end{proof}
Now suppose $A$ is a simple type IV abelian fourfold defined over $\qbar$ with endomorphism field $E$, where $E$ is a degree 4 CM-field. Then we will study the following de Rham-Betti structure \begin{equation}\label{deg4drbweildefn}
   W_{\mathrm{Weil}}^{\mathrm{dRB}}:=\wedge_{E}^2\mathrm{H}^1_{\mathrm{dRB}}(A,\mathbb{Q})
\end{equation} which is viewed as a dRB  substructure of
$\wedge_{\mathbb{Q}}^2\mathrm{H}^1_{\mathrm{dRB}}(A,\mathbb{Q})$. Then the computation of the de Rham-Betti group associated with $W_{\mathrm{Weil}}^{\mathrm{dRB}}$ requires more preparations (Lemma \ref{generalizedcm}, Lemma \ref{phicompatiblewithE}, Lemma \ref{evendegreeweildrb} as well as Lemma \ref{deg4keylemma} from Section \ref{cmgdrbsection}). The method, however, is similar to computing the de Rham-Betti group of a CM abelian variety, once the issue of weights is handled carefully.

Recall that for a simple type IV abelian variety whose endomorphism algebra is a CM-field $E$, we have by Lemma \ref{rosaticmcoincides} the Rosati involution on $E$ associated with any polarization form $\phi$ is given by the complex conjugation. This has the following linear algebraic consequence which will also be used frequently in the sequel. 
\begin{lemma} \label{phicompatiblewithE}
Let $A$ be a simple abelian variety whose endomorphism algebra is isomorphic to a CM-field $E$. Recall the equidimensional eigenspace decomposition $\mathrm{H}^1(A,\mathbb{Q}) \otimes \qbar=\oplus_{\sigma \in \mathrm{Hom}(E,\qbar)}V_{\sigma}$ with respect to $E\xhookrightarrow{}\mathrm{End}(\betti)$ from Lemma \ref{neweigenspacetranslate}. Then the $\qbar$-linear extension of the polarization form $\phi$ 
$$\phi_{\qbar}: \mathrm{H}^1(A,\mathbb{Q}) \otimes \qbar \times  \mathrm{H}^1(A,\mathbb{Q}) \otimes \qbar \rightarrow \qbar(2\pi i)$$ satisfies that $\phi_{\qbar}(V_{\sigma},V_{\sigma'})=0$ if $\sigma' \neq \overline{\sigma}$. Moreover, for each $\sigma\in \mathrm{Hom}(E,\qbar)$, we can find a $\qbar$-basis $\{e_{i}\}$ for $V_{\sigma}$ and a $\qbar$-basis $\{f_{j}\}$ for $V_{\overline{\sigma}}$ such that $\phi_{\qbar}(e_{i},f_{j})=0$ if $i \neq j$ and $\phi_{\qbar}(e_{i},f_{j})=1$ if $i=j$.
\end{lemma}
\begin{proof}
For the first statement, notice that for any $e \in E$, we have $$\phi_{\qbar}(ev,w)=\phi_{\qbar}(v,\overline{e}w)$$ for any $v,w \in \mathrm{H}^1(A,\mathbb{Q}) \otimes \qbar$. Therefore when $v\in V_{\sigma},w\in V_{\sigma'}$, we have $$\sigma(e)\phi_{\qbar}(v,w)=\sigma'(\overline{e})\phi_{\qbar}(v,w)=\overline{\sigma'(e)}\phi_{\qbar}(v,w)$$ for any $e\in E$. The second equality holds because $E$ is a CM-field.   But if $\sigma' \neq \overline{\sigma}$, this implies that $\phi_{\qbar}(v,w)=0$. For the second statement, notice that $\phi$ is a nondegenerate bilinear form, therefore the map $V_{\sigma}\rightarrow V_{\overline{\sigma}}^{*}:v\rightarrow \phi_{\qbar}(v,-)$ is an isomorphism of $\qbar$-vector spaces. Hence we obtain the desired basis $\{e_{i}\}$ for $V_{\sigma}$ and $\{f_{j}\}$ for $V_{\overline{\sigma}}$ satisfying the condition in the lemma.
\end{proof}

From the above lemma we deduce the following.

\begin{lemma} \label{evendegreeweildrb}
We keep the same notations as above. The de Rham-Betti structure $\mathbb{Q}_{\mathrm{dRB}}(-2)$ is an object in the Tannakian subcategory $\langle W_{\mathrm{Weil}}^{\mathrm{dRB}} \rangle^{\otimes}$.
\end{lemma}
\begin{proof}
Recall that we have a non-zero morphism of dRB structures induced by a polarization form
    $$\phi: \mathrm{H}^1_{\mathrm{dRB}}(A,\mathbb{Q}) \otimes \mathrm{H}^1_{\mathrm{dRB}}(A,\mathbb{Q}) \rightarrow \mathbb{Q}_{\mathrm{dRB}}(-1)$$ This induces a symmetric bilinear form $$\phi': \mathrm{H}^2(A,\mathbb{Q}) \times \mathrm{H}^2(A,\mathbb{Q})  \rightarrow \mathbb{Q}(-2)$$ where on basic tensors $v_1\wedge v_2\in\mathrm{H}^2(A,\mathbb{Q})$ and $v_3\wedge v_4\in\mathrm{H}^2(A,\mathbb{Q})$ we have $\phi'(v_1\wedge v_2,v_3\wedge v_4)=\phi(v_1,v_3)\phi(v_2,v_4)-\phi(v_1,v_4)\phi(v_2,v_3)$. We now claim that $\phi'$ induces a morphism of de Rham-Betti structures $\mathrm{H}^2_{\mathrm{dRB}}
    (A,\mathbb{Q}) \otimes \mathrm{H}^2_{\mathrm{dRB}}(A,\mathbb{Q})  \rightarrow \mathbb{Q}_{\mathrm{dRB}}(-2)$. Note that $\phi$ is a morphism of dRB structures, hence for any $v,v'\in\betti$ and any $g\in\gdrbmath(A)$, we have $\phi(g\circ v,g\circ v')=\chi(g)\phi(v,v')$ where $\chi$ is the character of $\gdrbmath(A)$ associated with the de Rham-Betti object $\mathbb{Q}_{\mathrm{dRB}}(-1)\in \mathrm{ob}\langle\mathrm{H}^1_{\mathrm{dRB}}(A,\mathbb{Q})\rangle^{\otimes}$. Therefore given any $g\in\gdrbmath(A)$ and arbitrary basic tensors $v_1\wedge v_2\in \mathrm{H}^2(A,\mathbb{Q})$ and $v_3\wedge v_4\in \mathrm{H}^2(A,\mathbb{Q})$, we have $\phi'(g\circ v_1\wedge v_2,g\circ v_3\wedge v_4)=\phi'(g\circ v_1\wedge g\circ v_2,g\circ v_3\wedge g\circ v_4)=\phi(g\circ v_1,g\circ v_3)\phi(g\circ v_2,g\circ v_4)-\phi(g\circ v_1,g\circ v_4)\phi(g\circ v_2,g\circ v_3)=\chi^2(g)(\phi(v_1,v_3)\phi(v_2,v_4)-\phi(v_1,v_4)\phi(v_2,v_3))$, which implies that $\phi'$ is indeed equivariant with respect to the action of $\gdrbmath(A)$. Restricting to the sub-dRB structure $W_{\mathrm{Weil}}^{\mathrm{dRB}} \xhookrightarrow{} \mathrm{H}^2_{\mathrm{dRB}}(A,\mathbb{Q})$, we obtain $$\phi'|_{W_{\mathrm{Weil}}^{\mathrm{dRB}}}:W_{\mathrm{Weil}}^{\mathrm{dRB}} \otimes W_{\mathrm{Weil}}^{\mathrm{dRB}} \rightarrow \mathbb{Q}_{\mathrm{dRB}}(-2)$$ It then suffices to show that it is a non-zero morphism. Extending scalars to $\qbar$, we obtain the following $\qbar$-bilinear form $$\phi'|_{W_{\mathrm{Weil}}^{\mathrm{dRB}},\qbar}: \bigoplus_{\sigma \in \mathrm{Hom}(E,\qbar)}\wedge^2V_{\sigma} \times \bigoplus_{\sigma \in \mathrm{Hom}(E,\qbar)}\wedge^2V_{\sigma} \rightarrow \qbar(-2)$$ Then for any $\sigma_0\in\mathrm{Hom}(E,\qbar)$, on the $\qbar$-subspace $\wedge^2V_{\sigma_0} \otimes \wedge^2V_{\overline{\sigma_0}}$, using the two basis $\{e_1,e_2\}$ and $\{f_1,f_2\}$ found in Lemma \ref{phicompatiblewithE} for $V_{\sigma_0}$ and $V_{\overline{\sigma_0}}$, we find that $\phi'(e_1\wedge e_2,f_1 \wedge f_2)=\phi(e_1,f_1)\phi(e_2,f_2)-\phi(e_1,f_2)\phi(e_2,f_1)=1$. Therefore indeed $\phi'|_{W_{\mathrm{Weil}}^{\mathrm{dRB}}}$ is a non-zero morphism of dRB structures. Thus $\phi'$ realizes  $\mathbb{Q}_{\mathrm{dRB}}(-2)$ as a quotient de Rham-Betti structure of $W_{\mathrm{Weil}}^{\mathrm{dRB}} \otimes W_{\mathrm{Weil}}^{\mathrm{dRB}}$.
\end{proof}
 
We can now state and prove the main result about the de Rham-Betti group of the Weil structure $W_{\mathrm{Weil}}^{\mathrm{dRB}}$ defined by formula (\ref{deg4drbweildefn}).
\begin{lemma} \label{deg4weil}
The de Rham-Betti group of the Weil structure $W_{\mathrm{Weil}}^{\mathrm{dRB}}$ is isogenous to $\gmmath \times \mathrm{U}_{E}$.
\end{lemma}

\begin{proof}
The proof strategy is the same as Theorem \ref{mainthmcm}. Denote the Weil Hodge structure $W_{\mathrm{Weil}}^{\mathrm{Hodge}}=\wedge_{E}^2\mathrm{H}^1(A,\mathbb{Q}))$ by $\mathcal{W}_1$ and denote the Weil de Rham-Betti structure  $W_{\mathrm{Weil}}^{\mathrm{dRB}}$ by $\mathcal{W}_2$. Notice we have the following commutative diagram describing the relation between 
the two algebraic groups $\gdrbmath(\mathcal{W}_2) $ and $\mathrm{MT}(\mathcal{W}_1)$. 
$$\begin{tikzcd}
\gdrbmath(\mathcal{W}_2) \arrow[r, hook] \arrow[d, hook]    & \mathrm{MT}(\mathcal{W}_1) \arrow[ld, hook] \\
{\mathrm{GL}(\wedge_{E}^2\mathrm{H}^1(A,\mathbb{Q}))} &                 \end{tikzcd}$$ The rest of the proof will be divided into several steps.  
\begin{enumerate}
    \item\label{hdgroup} We first show that $T':=\mathrm{Hdg}
    (\mathcal{W}_1)= \mathrm{U}_{E}$. 
    \item\label{gminweil} Then we show that there is an inclusion of $\gmmath$ into $\gdrbmath(\mathcal{W}_2)$ as homotheties and define $\gdrbhmath(\mathcal{W}_2)$ which can be viewed as an algebraic subgroup of $\mathrm{Hdg}(\mathcal{W}_1)$. The construction is similar to the construction of $\gdrbhmath(A)$.
    \item\label{drbgstep} Finally we will show that $\gdrbhmath(\mathcal{W}_2)=\mathrm{Hdg}
    (\mathcal{W}_1)$ using the computation of periods by Gross in \cite{gross1978periods} i.e. Lemma \ref{gross}.
\end{enumerate}
\textbf{Step (\ref{hdgroup}):} By Lemma \ref{morphismofweilstruc}, we have a morphism of $\mathbb{Q}$-algebras $$E \xhookrightarrow{} \mathrm{End}(\wedge_{E}^2\mathrm{H}^1(A,\mathbb{Q}))$$ whose image preserves the Hodge structure. Then because $\mathrm{deg}(E/\mathbb{Q})=\mathrm{dim}(\wedge_{E}^2\mathrm{H}^1(A,\mathbb{Q}))$ we can use Lemma \ref{generalizedcm} to deduce that $T' \subset \mathrm {U}_{E}$.

The following analysis is similar to the method of determining the Hodge group of simple CM abelian fourfolds from \cite{mz-4-folds}. Suppose $\mathrm{dim}(T')=0$, then because Hodge groups are connected, $T'$ would have to be the trivial group. This means that the weight 2 sub-Hodge structure $\mathcal{W}_1$ viewed inside $\wedge_{\mathbb{Q}}^2\mathrm{H}^1(A,\mathbb{Q})$ consists purely of $(1,1)$ class. But the set of multiplicities for embeddings $\mathrm{Hom}(E, \mathbb{C})=\{\sigma,\overline{\sigma},\tau,\overline{\tau}\}$ is $\{2,0,1,1\}$ by Lemma \ref{deg4keylemma}.  Therefore by Lemma \ref{qbarhodgedecomp} $$\mathcal{W}_1 \otimes_{\mathbb{Q}}\mathbb{C} = (\mathcal{W}_1 \otimes_{\mathbb{Q}}\qbar)\otimes_{\qbar}\mathbb{C}=\wedge^2V_{\sigma}\otimes_{\qbar}\mathbb{C} \oplus \wedge^2V_{\overline{\sigma}}\otimes_{\qbar}\mathbb{C}\oplus\wedge^2V_{\tau}\otimes_{\qbar}\mathbb{C} \oplus \wedge^2V_{\overline{\tau}}\otimes_{\qbar}\mathbb{C} \subset \mathrm{H}^{2,0} \oplus \mathrm{H}^{0,2} \oplus \mathrm{H}^{1,1}$$ In particular $(\mathcal{W}_1 \otimes_{\mathbb{Q}}\mathbb{C})^{2,0}\neq \{0\}$. Hence $T'$ cannot be the trivial group.

Suppose $\mathrm{dim}(T')=1$. By Lemma \ref{keylemmamz}, all codimension 1 (connected) subtorus of $\mathrm{U}_{E}$ is of the form $\mathrm{SU}_{E/k}$ (see Definition \ref{sudefn}) where $k\xhookrightarrow{}E$ is a quadratic imaginary subfield of $E$. Suppose $\mathrm{SU}_{E/k}$ is the Hodge group of $\mathcal{W}_1$ for some quadratic CM subfield $k\xhookrightarrow{}E$. Applying Weil structure construction to $k\xhookrightarrow{} E\xhookrightarrow{}\mathrm{End}(\mathcal{W}_1)$ then $\wedge_{k}^2\mathcal{W}_1$ would entirely consist of $(2,2)$ classes. According to Lemma \ref{Weilstructurelayer}, we have $$\wedge_{k}^2\mathcal{W}_1= \wedge_{k}^4\mathrm{H}^1(A,\mathbb{Q})$$ as Hodge structures. Next we claim the only possible set of multiplicities for embeddings of $k$ into $\mathbb{C}$ is $\{1,3\}$. This is because the only possible set of multiplicities for embeddings $\mathrm{Hom}(E, \mathbb{C})=\{\sigma,\overline{\sigma},\tau,\overline{\tau}\}$ is $\{2,0,1,1\}$ by Lemma \ref{deg4keylemma} and either $\sigma|_{k}=\tau|_{k}$ or $\sigma|_{k}=\overline{\tau}|_{k}$. Hence upon tensoring with $\mathbb{C}$, we obtain that $$\wedge_{k}^2\mathcal{W}_1 \otimes_{\mathbb{Q}} \mathbb{C} \cong \wedge_{k}^4\mathrm{H}^1(A,\mathbb{Q}) \otimes_{\mathbb{Q}} \mathbb{C} \xhookrightarrow{} \mathrm{H}^{1,3}(A) \oplus \mathrm{H}^{3,1}(A)$$ Therefore $\wedge_{k}^2\mathcal{W}_1$ cannot contain any (2,2) classes. Hence the Hodge group of $\mathcal{W}_1$ cannot be  $\mathrm{SU}_{E/k}$ for any quadratic imaginary subfield $k\xhookrightarrow{}E$.

Since dim$(\mathrm{U}_E)$=2, we have that the Hodge group of $\mathcal{W}_1$ is equal to $\mathrm{U}_E$ itself. And we have finished proving Step (\ref{hdgroup}). \\
\textbf{Step (\ref{gminweil})} 
By Lemma \ref{weilhodge} $\mathcal{W}_2$ is a sub-dRB structure of $\mathrm{H}^2_{\mathrm{dRB}}(A,\mathbb{Q})\cong\wedge^2\mathrm{H}^1_{\mathrm{dRB}}(A,\mathbb{Q})$. Therefore $\gdrbmath(\mathcal{W}_2)\subset\mathrm{GL}(\wedge_{E}^2\mathrm{H}^1(A,\mathbb{Q}))$ is the image under the natural restriction of the induced action of $\gdrbmath(A)$ on $\wedge^2\mathrm{H}^1(A,\mathbb{Q})$. By Proposition \ref{gmintype4}, $\gdrbmath(A)\xhookrightarrow{}\mathrm{GL}(\betti)$ contains the homotheties. Therefore we obtain a natural homomorphism $\gmmath\xhookrightarrow{}\gdrbmath(\mathcal{W}_2)$ where $t\in\gmmath$ is sent to the scalar multiplication by $t^{-2}$ on $\wedge_{E}^2\mathrm{H}^1(A,\mathbb{Q})$. Because $\gdrbmath(\mathcal{W}_2)$ is an algebraic group, we obtain an inclusion of $\gmmath$ in $\gdrbmath(\mathcal{W}_2)$ as homotheties in $\mathrm{GL}(\wedge_{E}^2\mathrm{H}^1(A,\mathbb{Q}))$.

In Step (\ref{hdgroup}), we have obtained that the Hodge group of the weight two Hodge structure $\mathcal{W}_1$ is $\mathrm{U}_{E}$. By Lemma \ref{hodgetimesgm}, this induces the following commutative diagram $$\begin{tikzcd}
\gmmath \times \mathrm{U}_{E} \arrow[rr, "\Psi_1", two heads]                                  &  & \mathrm{MT}(\mathcal{W}_1)\\
\Psi^{-1}(\gdrbmath(\mathcal{W}_2))^{\circ} \arrow[rr, two heads] \arrow[u, hook] &  & \gdrbmath(\mathcal{W}_2) \arrow[u, hook]       
\end{tikzcd}$$ where the group homomorphism $\Psi:\gmmath \times \mathrm{U}_{E} \rightarrow \mathrm{MT}(\mathcal{W}_1)$ is given by $$(t,\alpha)\rightarrow t^{-1}\alpha$$ By Lemma \ref{existenceofhomothetycor}, one sees that $\Psi$ is indeed an isogeny of algebraic groups. By the discussion above we also have $\Psi(\gmmath\times\{\mathrm{id}\})\subset\gdrbmath(\mathcal{W}_2)$. We denote $$\gdrbhmath(\mathcal{W}_2):=\mathrm{pr}_2(\Psi^{-1}(\gdrbmath(\mathcal{W}_2))^{\circ})$$ where $\mathrm{pr}_2:\gmmath \times \mathrm{U}_{E}\rightarrow \mathrm{U}_E$ is the projection map to the second factor. Then analogous to Proposition \ref{gm in CM gdrb}, because $\gdrbmath(\mathcal{W}_2)$ is a connected algebraic group, we have the following commutative diagram $$\begin{tikzcd}
\gmmath \times \gdrbhmath(\mathcal{W}_2) \arrow[r, "\Psi_2",  two heads] \arrow[d, hook] & \gdrbmath(\mathcal{W}_2) \arrow[d, hook] \\
\gmmath \times \mathrm{Hdg}(\mathcal{W}_1) \arrow[r, "\Psi_1", two heads]               & \mathrm{MT}(\mathcal{W}_1)              
\end{tikzcd}$$ We conclude Step (\ref{gminweil}) with the explicit description of the action of the various algebraic groups on the one-dimensional de Rham-Betti structure $\mathbb{Q}_{\mathrm{dRB}}(-2)$. Recall in the proof of Lemma \ref{evendegreeweildrb}, $\mathbb{Q}_{\mathrm{dRB}}(-2)$ is realized as a quotient of the dRB structure $\mathcal{W}_2\otimes\mathcal{W}_2$. Hence one can see that $t \in \gmmath \xhookrightarrow{} \gdrbmath(\mathcal{W}_2)$ acts on $\mathbb{Q}_{\mathrm{dRB}}(-2)$ as scalar multiplication by $t^{-2}$.  Also $\gdrbhmath(\mathcal{W}_2)$, which by construction is a connected subtorus of $\mathrm{U}_{E}$, acts trivially on the dRB structure $\mathbb{Q}_{\mathrm{dRB}}(-2)$ by Lemma \ref{nosurjectiontogm}. \\ 
\textbf{Step (\ref{drbgstep}):} We will show that $\gdrbhmath(\mathcal{W}_2)=\mathrm{U}_{E}$. Now suppose $$\mathrm{dim}(\gdrbhmath(\mathcal{W}_2))<2$$  
Therefore by Lemma \ref{keylemmamz}, either $$\gdrbhmath(\mathcal{W}_2)=\{\mathrm{id}\}$$ or $$\gdrbhmath(\mathcal{W}_2)=\mathrm{SU}_{E/k}$$ for some quadratic CM-subfield $k\xhookrightarrow{}E$.

Assuming the first case, then $\gdrbmath(\mathcal{W}_2)=\gmmath$. Hence every element in $\mathcal{W}_2\otimes\mathbb{Q}_{\mathrm{dRB}}(-1)\subset \mathrm{H}^2_{\mathrm{dRB}}(A,\mathbb{Q})\otimes\mathbb{Q}_{\mathrm{dRB}}(-1)$ is a de Rham-Betti class. However, de Rham-Betti classes and Hodge classes coincide in degree two cohomology groups for abelian varieties by Theorem \ref{mainfact} and we have shown in Step (\ref{hdgroup}) that $\mathcal{W}_1\otimes\mathbb{Q}(-1)$ does not entirely consist of Hodge classes. This poses a contradiction.

Now assume the second case. Then every element in $\wedge_{k}^2\mathcal{W}_2$ would be fixed by $\gdrbhmath(\mathcal{W}_2)$ by Lemma \ref{weilstructurefixer}. Hence according to the analysis at the end of Step (\ref{gminweil}), we have 

\begin{equation*}
\begin{split}
(\wedge_{k}^2\mathcal{W}_2 \otimes \mathbb{Q}_{\mathrm{dRB}}(2))^{\gdrbmath(\mathcal{W}_2)}&=(\wedge_{k}^2\mathcal{W}_2 \otimes \mathbb{Q}_{\mathrm{dRB}}(2))^{\Psi_2(\gmmath\times\gdrbhmath(\mathcal{W}_2))}\\&=(\wedge_{k}^2\mathcal{W}_2)^{\gdrbhmath(\mathcal{W}_2)}\otimes\mathbb{Q}_{\mathrm{dRB}}(2)
\end{split}
\end{equation*}Therefore every element in the de Rham-Betti structure $\wedge_{k}^2 \mathcal{W}_2 \otimes \mathbb{Q}_{\mathrm{dRB}}(2)$ would be a de Rham-Betti class. However according to Lemma \ref{Weilstructurelayer}, we have the isomorphism of dRB structures $$\wedge_{k}^2\mathcal{W}_2 \otimes \mathbb{Q}_{\mathrm{dRB}}(2) \cong \wedge_{k}^4\mathrm{H}^1_{\mathrm{dRB}}(A,\mathbb{Q}) \otimes \mathbb{Q}_{\mathrm{dRB}}(2)$$ which implies that the RHS of the above isomorphism consists entirely of dRB classes. Now by the analysis in Step (\ref{hdgroup}) the set of multiplicities for embeddings $\mathrm{Hom}(k,\qbar)$ with respect to $k \rightarrow E \rightarrow \mathrm{End}(\mathrm{H}^1(A,\mathbb{Q}))$  is $\{1,3\}$. Hence by Lemma \ref{gross}, the matrix of the comparison isomorphism of $\wedge_{k}^4\mathrm{H}^1_{\mathrm{dRB}}(A,\mathbb{Q})\otimes_{\mathbb{Q}} \qbar := (\wedge_{k}^4\mathrm{H}^1(A,\mathbb{Q}) \otimes_{\mathbb{Q}}\qbar, \wedge_{k \otimes_{\mathbb{Q}}\qbar}^4\mathrm{H}^1_{\mathrm{dR}}(A/\qbar), \rho_{m})$ can be diagonalized to \[\begin{pmatrix} b_{k}^{1} (2\pi i/b_{k})^{3} & 0 \\0 & b_{k}^{3} (2\pi i/b_{k})^{1} \end{pmatrix}\] Because $b_{k}$ is algebraically independent from $2\pi i$ by \cite{chudnovsky1980algebraic}, and because of Lemma \ref{grosscorollary}, one can see that $\wedge_{k}^2\mathcal{W}_2 \otimes_{\mathbb{Q}} \mathbb{Q}_{\mathrm{dRB}}(2) =\wedge_{k}^4\mathrm{H}^1_{\mathrm{dRB}}(A,\mathbb{Q}) \otimes_{\mathbb{Q}} \mathbb{Q}_{\mathrm{dRB}}(2)$ \textit{cannot} contain any dRB classes. Hence this implies that $\mathrm{dim}(\gdrbhmath(\mathcal{W}_2))=2$. But this leads to the desired equality $$\gdrbhmath(\mathcal{W}_2)=\mathrm{U}_E=\mathrm{Hdg}(\mathcal{W}_1)$$

\end{proof}

\begin{proposition} \label{twocentressame}
    For any simple type IV abelian fourfold $A$ defined over $\qbar$, the center of its de Rham-Betti Lie algebra is equal to the center of its Mumford-Tate Lie algebra. 
\end{proposition}

\begin{proof}[Proof of Proposition \ref{twocentressame}]
First suppose $A$ is a Weil type abelian fourfold. Then by Theorem \ref{deg2mz4}, we have that $\mathrm{Z}(\mathrm{MT}(A))=\gmmath$. But from Proposition \ref{gmintype4} we have $\gmmath \xhookrightarrow{} \mathrm{Z}(\gdrbmath(A))$. Then applying Lemma \ref{centreincentre} which states that $\mathrm{Z}(\gdrbmath(A)) \xhookrightarrow{} \mathrm{Z}(\mathrm{MT}(A))$, we obtain $\mathrm{Z}(\mathrm{MT}(A))=\gmmath=\mathrm{Z}(\gdrbmath(A))$. In particular, we deduce that the center of the Mumford-Tate Lie algebra is equal to the center of the de Rham-Betti Lie algebra.

Next suppose $A$ is an anti-Weil type abelian fourfold, then by Lemma \ref{13group}, we have that $\gdrbmath(W_{\mathrm{Weil}})=\resgmmath$. But notice that the Tannakian subcategory $\langle W_{\mathrm{Weil}}\rangle^{\otimes}$ is a natural subcategory of $\langle(\mathrm{H}^1_{\mathrm{dRB}}(A,\mathbb{Q})\rangle^{\otimes}$, therefore we obtain a natural surjective homomorphism of algebraic groups $$\gdrbmath(A) \twoheadrightarrow \gdrbmath(W_{\mathrm{Weil}})$$ But because  $\gdrbmath(A)$ is a reductive group, $\mathrm{Z}(\gdrbmath(A))$ admits a surjection to $\resgmmath$ by Lemma \ref{surjectiontotorus}, which implies the dimension of the centre of $\gdrbmath(A)$ is bigger than or equal to 2. By Lemma \ref{centreincentre}, we have that $\mathrm{Z}(\gdrbmath(A))\xhookrightarrow{} \mathrm{Z}(\mathrm{MT}(A))$ and by Theorem \ref{deg2mz4}, we have that $\mathrm{dim}(\mathrm{Z}(\mathrm{MT}(A)))=2$. Hence this implies that the centre of the de Rham-Betti Lie algebra is equal to the centre of the Mumford-Tate Lie algebra.

The last case where $\mathrm{deg}(E)=4$ can be resolved in a similar fashion. We have computed in Lemma \ref{deg4weil} that $\gdrbmath(W_{\mathrm{Weil}})$ is isogenous to $\gmmath \times \mathrm{U}_{E}$. Therefore the dimension of $\mathrm{Z}(\gdrbmath(A))$ is at least 3. And by Theorem \ref{deg2mz4}, we have that the centre of the Mumford-Tate Lie algebra of $A$ has dimension 3. Therefore again by Lemma \ref{centreincentre} and Lemma \ref{surjectiontotorus} we deduce the centre of the de Rham-Betti Lie algebra coincides with the centre of the Mumford-Tate Lie algebra. 
\end{proof}

\subsection{One-Dimensional De Rham-Betti Structures}\label{seconedimobj}
In this section, we explore corollaries of Proposition \ref{gmintype4} in more general settings. In particular, we are going to study characters of the de Rham-Betti group of a simple abelian variety of type IV as well as properties of one-dimensional de Rham-Betti structures in the Tannakian category $\langle\mathrm{H}^1_{\mathrm{dRB}}(A,\mathbb{Q})\rangle^{\otimes}$. In Remark \ref{uncountablymany} from Section \ref{drbsect}, we mentioned that in the category of all de Rham-Betti structures, there are uncountably many isomorphism classes of one-dimensional dRB structures. However, recall that in the category of all $\mathbb{Q}$-Hodge structures, a one-dimensional Hodge structure is of the form $\mathbb{Q}(n)$ i.e. $\mathbb{Q}(n)\otimes_{\mathbb{Q}}\mathbb{C}$ has a single summand with bidegree $(-n,-n)$ for some $n\in\mathbb{Z}$. In parallel, we may form the following conjecture regarding one-dimensional de Rham-Betti structures arising from geometry. Recall from Definition \ref{drbweights} that we call a one-dimensional dRB structure $(V_{\mathrm{B}},V_{\mathrm{dR}},\rho_{m})$ as $\mathbb{Q}_{\mathrm{dRB}}(k)$, if we can find a $\mathbb{Q}$-basis for $V_{\mathrm{B}}$ and a $\qbar$-basis for $V_{\mathrm{dR}}$ such that $\rho_{m}$ is the scalar multiplication by $(2\pi i)^{k}$.
\begin{conjecture}\label{onedimconj}
For any abelian variety $A$ defined over $\qbar$, every one-dimensional de Rham-Betti structure in $\langle\mathrm{H}^1_{\mathrm{dRB}}(A,\mathbb{Q})\rangle^{\otimes}$ is of the form $\mathbb{Q}_{\mathrm{dRB}}(k)$ for some $k\in\mathbb{Z}$. 
\end{conjecture}
For type IV simple abelian varieties, we give a group theoretic interpretation of the above conjecture. 
\begin{lemma}\label{groupinterpret}
 Suppose $A$ is a simple abelian variety of type IV defined over $\qbar$. Then the following two statements are equivalent: \begin{enumerate}
     \item Every one-dimensional object in the Tannakian category $\langle\mathrm{H}^1_{\mathrm{dRB}}(A,\mathbb{Q})\rangle^{\otimes}$ is of the form $\mathbb{Q}_{\mathrm{dRB}}(k)$, where $k\in\mathbb{Z}$.
     \item $\gmmath \cap \gdrbhmath(A)$ is a finite group scheme of order 2.
 \end{enumerate}   
\end{lemma}
The proof of Lemma \ref{groupinterpret} relies on the following lemma analogous to the statement that any character of the Hodge group of an abelian variety is trivial. 
\begin{lemma}\label{chartrivial}
    Given $A$ a simple abelian variety of type IV defined over $\qbar$, then any character from $\gdrbhmath(A)$ to $\gmmath$ is trivial.
\end{lemma}
\begin{proof}
    We will be using notations from the proof of Proposition \ref{gmintype4}. By construction, $\gdrbhmath(A)$ is a connected algebraic group. Therefore its image under any character is either the entire $\gmmath$ or $\{\mathrm{id}\}$. Because $\gdrbmath(A)$ is a reductive algebraic group, by \cite[Theorem 22.125]{milne2011algebraic}, we have an isogeny $$\mathrm{Z}(\gdrbmath(A))\times\gdrbmath^{\mathrm{ad}}(A)\rightarrow\gdrbmath(A)$$ which sends $(x,y)\in\mathrm{Z}(\gdrbmath(A))\times\gdrbmath^{\mathrm{ad}}(A)$ to $xy\in\gdrbmath(A)$. Moreover, from Proposition \ref{gmintype4}, we obtain an isogeny $\gmmath\times T'\rightarrow\mathrm{Z}(\gdrbmath(A))$ which sends $(g,t')$ to $g^{-1}t'$. Recall $T'$ is defined in diagram (\ref{typeivkeydiagram}) as $\mathrm{pr}_2(\Phi^{'-1}(\mathrm{Z}(\gdrbmath(A))^{\circ})$. Composing these two isogenies we obtain the third isogeny \begin{equation}\label{multiplieisogeny}
      \gmmath\times T'\times\gdrbmath^{\mathrm{ad}}(A)\rightarrow\gdrbmath(A)
    \end{equation} Note that $T'\subset\mathrm{Hdg}(A)$ by construction and $\gdrbmath^{\mathrm{ad}}(A)\subset\mathrm{MT}^{\mathrm{ad}}(A)\subset\mathrm{Hdg}(A)$. Then from the construction of $\gdrbhmath(A)$ (see Definition \ref{defgdrbh} and diagram (\ref{gdrbhdiagram})), we may deduce that the morphism (\ref{multiplieisogeny}) induces an isogeny from $T'\times\gdrbmath^{\mathrm{ad}}(A)$ to $\gdrbhmath(A)$. If $\gdrbhmath(A)$ were to admit a surjective algebraic group homomorphism to $\gmmath$, this would imply that there is a surjection from $T'\times\gdrbmath^{\mathrm{ad}}(A)$ to $\gmmath$. By construction $T'=\mathrm{pr}_2(\Phi^{'-1}(\mathrm{Z}(\gdrbmath(A))^{\circ})$ is a connected subtorus of $\mathrm{U}_{E}$, where $E$ is the centre of $\mathrm{End}^{\circ}(A)$ and is a CM-field. Hence by Lemma \ref{nosurjectiontogm}, $T'$ has only the trivial character. Because $\gdrbmath^{\mathrm{ad}}(A)$ is a semisimple algebraic group, the image of any morphism from $\gdrbmath^{\mathrm{ad}}(A)$ to $\gmmath$ is zero-dimensional. Hence there does not exists a surjective group homomorphism from $T'\times\gdrbmath^{\mathrm{ad}}(A)$ to $\gmmath$.
\end{proof}

\begin{proof}[Proof of Lemma \ref{groupinterpret}]    Given a one-dimensional de Rham-Betti structure  $$L_{\alpha}\in\mathrm{ob}\langle\mathrm{H}^1_{\mathrm{dRB}}(A,\mathbb{Q})\rangle^{\otimes}$$ if $L_{\alpha}$ is not of the form $(\mathbb{Q},\qbar,\rho_{m}=\alpha\in\qbar^{\times})$, then by the Tannakian duality we obtain a surjective homomorphism of $\mathbb{Q}$-algebraic groups $$\gdrbmath(A)\twoheadrightarrow\gmmath$$
Any polarization form on $A$ induces a morphism of de Rham-Betti structures $$\phi: \mathrm{H}^{1}_{\mathrm{dRB}}(A,\mathbb{Q})\otimes \mathrm{H}^{1}_{\mathrm{dRB}}(A,\mathbb{Q}) \rightarrow \mathbb{Q}_{\mathrm{dRB}}(-1)$$ Therefore the one-dimensional de Rham-Betti structure $\mathbb{Q}_{\mathrm{dRB}}(-1)$ induces a character $$\chi_{\phi}:\gdrbmath(A) \rightarrow \gmmath$$ When restricted to $\gmmath \xhookrightarrow{} \gdrbmath(A)$, $\chi_{\phi}$ is an isogeny $\gmmath \rightarrow \gmmath$ of the form $t\rightarrow t^{-2}$. Moreover, by Lemma \ref{chartrivial}, we have that $\chi_{\phi}(\gdrbhmath(A))=\{\mathrm{id}\}$. This implies that the order of every element in the algebraic subgroup $\gmmath \cap \gdrbhmath(A)\subset \gmmath$ divides 2. Because every finite algebraic subgroup of $\gmmath$ is cyclic $\gmmath \cap \gdrbhmath(A)$ either equals to $\{\pm1\}\in\mathrm{GL}(V)$, or equals to $\{\mathrm{id}\}$.

Now suppose the first statement in the lemma is satisfied i.e. every one-dimensional object in the Tannakian category $\langle\mathrm{H}^1_{\mathrm{dRB}}(A,\mathbb{Q})\rangle^{\otimes}$ is of the form $\mathbb{Q}_{\mathrm{dRB}}(k)$, for some $k\in\mathbb{Z}$. Note that $\mathbb{Q}_{\mathrm{dRB}}(k)\cong \mathbb{Q}_{\mathrm{dRB}}(1)^{\otimes k}$. Then by Tannakian duality, every character of $\gdrbmath(A)$ when restricting to $\gmmath\xhookrightarrow{}\gdrbmath(A)$ is equal to $t\rightarrow t^{2k}$ where $k\in\mathbb{Z}$. Now assume that $\gmmath \cap \gdrbhmath(A)$ is equal to $\{\mathrm{id}\}$. Then there exists a character $$\chi_{\frac{1}{2}}:\gdrbmath(A)\rightarrow\gmmath$$ such that its restriction to $\gmmath\xhookrightarrow{}\gdrbmath(A)$ is equal to $t\rightarrow t^{-1}$ and $\chi_{\frac{1}{2}}(\gdrbhmath(A))=\{\mathrm{id}\}$. This gives a contradiction. This then implies that $\gmmath \cap \gdrbhmath(A)$ is equal to $\{\pm \mathrm{id}\}$.

On the other hand, suppose $\gmmath \cap \gdrbhmath(A)$ is equal to $\{\pm1\}\in\mathrm{GL}(V)$. Then given any one-dimensional dRB structure $$L_{\alpha} \in\mathrm{ob}\langle \mathrm{H}^1_{\mathrm{dRB}}(A,\mathbb{Q})\rangle^{\otimes}$$ the induced character $$\chi_{\alpha}:\gdrbmath(A)\rightarrow\gmmath$$ satisfies that its restriction to $\gmmath \xhookrightarrow{} \gdrbmath(A)$ is an isogeny of degree $m$ and $\chi_{\alpha}(\gdrbhmath(A))=\{\mathrm{id}\}$ by Lemma \ref{chartrivial}. By assumption, $m$ is an even number. Therefore by the Tannakian duality, we have that $L_{\alpha}$ is isomorphic to $\mathbb{Q}_{\mathrm{dRB}}(\frac{m}{2})$ as de Rham-Betti structures.
\end{proof}
Using Lemma \ref{groupinterpret}, Conjecture \ref{onedimconj} can be immediately verified in a small number of cases.
\begin{corollary}\label{onecaseofconj}
 Suppose $A$ is a simple abelian variety of type IV defined over $\qbar$. Moreover, suppose that $E=\mathrm{Z}(\mathrm{End}^{\circ}(A))$ is a CM-field of degree 2 or 4 and that $\mathrm{dim}(\mathrm{Z}(\gdrbmath(A)))\geq 2$. Then Conjecture \ref{onedimconj} is true in this case.  
\end{corollary}
\begin{proof}
We will use notations from the proof of Lemma \ref{chartrivial}. We will show that the element $-1\in \gmmath\cap T'$. In particular we will show that $-1\in T'(\qbar)$. Then since $T'$ is defined over $\mathbb{Q}$, this will imply that $-1\in \gmmath\cap T'$. By construction $T'$ is a connected subtorus of $\mathrm{U}_{E}$. Also by Proposition \ref{gmintype4}, we have an isogeny $\gmmath \times T'\rightarrow\mathrm{Z}(\gdrbmath(A))^{\circ}$. We first suppose that $\mathrm{deg}(E/\mathbb{Q})=2$. Then by the assumption in the corollary that $\mathrm{dim}(\mathrm{Z}(\gdrbmath(A)))\geq 2$, we have $T'=\mathrm{U}_{E}$. The image of $\mathrm{U}_{E}(\qbar)$ inside $\mathrm{GL}(V\otimes_{\mathbb{Q}}\qbar)$ is equal to $$\{\begin{pmatrix}
    t_{\sigma}|_{V_{\sigma}} & \\
                & t_{\overline{\sigma}}|_{V_{\overline{\sigma}}}
\end{pmatrix}|t_{\sigma},t_{\overline{\sigma}}\in\qbar^{\times},t_{\sigma}t_{\overline{\sigma}}=1,\{\sigma,\sigmabar\}=\mathrm{Hom}(E,\qbar)\}$$ 
From this description, it is clear that $-1\in\mathrm{U}_{E}$.

Now suppose $E$ is a degree 4 CM-field. Then by the assumption in the lemma that $$\mathrm{dim}(\mathrm{Z}(\gdrbmath(A)))\geq 2$$ we have that $T'$ is either $\mathrm{U}_{E}$ or a one-dimensional subtorus of $\mathrm{U}_{E}$. Suppose the first case, then the argument from above can be applied verbatim to show that $-1\in\mathrm{U}_{E}$. Suppose the latter is true, then by Lemma \ref{keylemmamz}, we have that $T'$ is of the form $\mathrm{SU}_{E/k}$ (see Definition \ref{sudefn}) where $k$ is a degree two CM-subfield of $E$. Then there exist $\sigma,\tau\in\mathrm{Hom}(E,\qbar)$ such that $\sigma|_{k}=\tau|_{k}$ and $\overline{\sigma}|_{k}=\overline{\tau}|_{k}$. Denote $n'=\frac{2\mathrm{dim}(A)}{\mathrm{deg}(E/\mathbb{Q})}$. By Lemma \ref{neweigenspacetranslate}, each summand in the eigenspace decomposition $$V_{\qbar}=V_{\sigma}\oplus V_{\sigmabar}\oplus V_{\tau}\oplus V_{\overline{\tau}}$$ has dimension $n'$. Then the image of $T'(\qbar)=\mathrm{SU}_{E/k}(\qbar)$ inside $\mathrm{GL}(V_{\qbar})$ is $$\{\begin{pmatrix}
    t_{\sigma}|_{V_{\sigma}} & & &\\
        & t_{\tau}|_{V_{\tau}} & & \\
        & & t_{\sigmabar}|_{V_{\sigmabar}} & \\
        & & & t_{\overline{\tau}}|_{V_{\overline{\tau}}}
\end{pmatrix}|(t_{\sigma}t_{\tau})^{n'}=1,(t_{\overline{\sigma}}t_{\overline{\tau}})^{n'}=1,t_{\sigma}t_{\sigmabar}=1,t_{\tau}t_{\overline{\tau}}=1\}$$ And hence $-1\in T'(\qbar)$.

\end{proof}
\begin{remark}
    The condition in the lemma that $\mathrm{dim}(\mathrm{Z}(\gdrbmath(A)))\geq 2$ is not vacuous. For example by Proposition \ref{twocentressame}, an abelian fourfold of anti-Weil type defined over $\qbar$ satisfies this condition. Nevertheless, we do not currently know how to compute the de Rham-Betti group of a simple anti-Weil type abelian fourfold defined over $\qbar$. See Section \ref{noideasubsect} for more details.
\end{remark}

\section{Setup for De Rham-Betti Groups of Simple Type IV Abelian Fourfolds}
In this section we will make the necessary preparations for computing the de Rham-Betti groups of simple type IV abelian fourfolds (see Definition \ref{type4defn}). As an immediate consequence for the setup, we will show that if $A$ is a simple Weil type abelian fourfold defined over $\qbar$ then its de Rham-Betti group coincides with its Mumford-Tate group (see Proposition \ref{weiltype}). By the classification table from \cite{mz-4-folds}, a simple type IV abelian fourfold, which is not CM, has endomorphism algebra which is either a quartic CM-field or a quadratic CM-field. The latter case is further distinguished by Weil type abelian fourfolds and anti-Weil type abelian fourfolds (see Definition \ref{defnweiltype}).

The discussions in this section mirror the technique of computing the Mumford-Tate groups of abelian varieties (for example \cite{mz-4-folds}). More details are added for the convenience of the reader: Proposition \ref{reprclassification} and Proposition \ref{weiltype} can be seen as an expansion of the proof for 7.4 (ii) of \cite{mz-4-folds}.

\subsection{Decomposition of De Rham-Betti Representations}
For an abelian variety $A$ defined over $\qbar$, denote $\betti$ by $V$ as usual. Recall that the de Rham-Betti group of $A$ admits a canonical representation of $\mathbb{Q}$-algebraic group
$$\gdrbmath(A) \rightarrow \mathrm{GL}(V)$$ We will analyze the induced morphism on Lie algebras
$$\liegdrbmath(A) \rightarrow \mathrm{gl}(V)$$ Recall from Definition \ref{liegdrbdefn} that we denote $\mathrm{Lie}\gdrbmath(A)$ by $\liegdrbmath(A)$.
To obtain results in present and later sections, the observation from \cite[Theorem 7.11]{kreutz2023rhambetti} that $\gdrbmath(A)$ is a reductive algebraic group is used in a crucial way. 

We start this section with a well-known simple lemma, which is a direct consequence of Proposition 3.38 and Corollary 3.39 of \cite{milne2014algebraic}.
\begin{lemma} \label{liealgebrainv}
Let $G$ be a connected algebraic group defined over $\mathbb{Q}$ and suppose that we are given a representation $G\rightarrow \mathrm{GL}(V)$ defined over $\mathbb{Q}$. This induces a representation of $\mathbb{Q}$-Lie algebra: $\mathrm{Lie}(G)\rightarrow \mathrm{gl}(V)$. Then an element $x\in V^{\otimes m} \otimes V^{*\otimes n}$ is $G$-invariant if and only if $x$ is annihilated by the induced action of $\mathrm{Lie}(G)$ on $V^{\otimes m} \otimes V^{*\otimes n}$.
\end{lemma}

In the sequel we will denote the set of elements in $V^{\otimes m} \otimes V^{*\otimes n}$ annihilated by $\mathrm{Lie}(G)$ by $(V^{\otimes m} \otimes V^{*\otimes n})^{\mathrm{Lie}(G)}$.
\begin{setup}
 \label{sectiondeg2}
In this paragraph, we fix $A$ a simple abelian fourfold defined over $\qbar$ with $\mathrm{End}^{\circ}(A)\cong K$ where $K$ is a quadratic CM-field. Recall by Lemma \ref{neweigenspacetranslate}, with respect to the embedding $$i:K\rightarrow \mathrm{End}(V)$$ we have the eigenspace decomposition \begin{equation}\label{keydecomp}
V \otimes \qbar=V_{\sigma} \oplus V_{\sigmabar}\end{equation} where $\{\sigma,\sigmabar\}=\mathrm{Hom}(K,\qbar)$. Moreover, we have $$\mathrm{dim}(V_{\sigma})=\mathrm{dim}(V_{\sigmabar})=4$$

\begin{lemma}\label{eigenspacesubrep}
The eigenspaces $V_{\sigma}$ and $V_{\sigmabar}$ from decomposition (\ref{keydecomp}) are in fact subrepresentations of the canonical representation of Lie algebras $$\eta: \liegdrbmath(A) _{\qbar}:=\liegdrbmath(A) \otimes_{\mathbb{Q}} \qbar \rightarrow \mathrm{gl}(V_{\qbar})$$ 
\end{lemma}
In this paragraph we will denote the restrictions of the Lie algebra representations $\liegdrbmath(A)_{\qbar} \rightarrow \mathrm{gl}(V_{\sigma})$ by $\eta_{\sigma}$ and $\liegdrbmath(A)_{\qbar} \rightarrow \mathrm{gl}(V_{\sigmabar})$ by $\eta_{\sigmabar}$.
\begin{proof}
Recall that by Theorem \ref{mainfact}, we have $\mathrm{End}_{\mathrm{Hdg}}(V)=\mathrm{End}_{\mathrm{dRB}}(V)=(V\otimes V^{*})^{\gdrbmath(A)}=K$. Therefore by Lemma \ref{invarintslemma}, $(V_{\qbar}\otimes V_{\qbar}^{*})^{\gdrbmath(A)_{\qbar}}=K\otimes_{\mathbb{Q}}\qbar$. Consider the canonical representation of the de Rham-Betti Lie algebra $$\eta: \liegdrbmath(A) \rightarrow \mathrm{gl}(V)$$ For any $k \in K\otimes_{\mathbb{Q}}\qbar \subset \mathrm{End}_{\qbar}(V_{\qbar},V_{\qbar})\cong V_{\qbar}\otimes V_{\qbar}^{*}$, by Lemma \ref{liealgebrainv}, with respect to the representation $$\eta_{\qbar}: \liegdrbmath(A)_{\qbar} \rightarrow \mathrm{gl}(V_{\qbar})$$ we have $\eta_{\qbar}(l) \circ k=0$. By Section 2.6.1 of \cite{humphreys2012introduction} for any $v \in V_{\qbar}$, we have $\eta_{\qbar}(l)\circ k(v)=k(\eta_{\qbar}(l)\circ v)$. Hence for any $v \in V_{\sigma}$, we have $k(\eta_{\qbar}(l)\circ v)=\eta_{\qbar}(l)\circ k(v)=\sigma(k)\eta_{\qbar}(l)\circ v$. Hence the eigenspace $V_{\sigma}$ is preserved by the image of $\eta_{\qbar}$. The case of $V_{\sigmabar}$ follows similarly.
\end{proof}
The following two lemmas provide more information about the two representations $\eta_{\sigma}$ and $\eta_{\sigmabar}$.
\begin{lemma}\label{irred}
The two representations $\eta_{\sigma}$ and $\eta_{\sigmabar}$ are absolutely irreducible. In fact their semisimplifications $\rho_{\sigma}:=\eta_{\sigma}^{\mathrm{ss}}: \liegdrbssmath(A)_{\qbar} \rightarrow \mathrm{gl}(V_{\sigma})$ and $\rho_{\sigmabar}:=\eta_{\sigmabar}^{\mathrm{ss}}$ are also absolutely irreducible.
\end{lemma}

\begin{proof}
We will show that $\rho_{\sigma}$ and $\rho_{\sigmabar}$ are absolutely irreducible. Then the absolute irreducibility of $\eta_{\sigma}$ and $\eta_{\sigmabar}$ will follow. Suppose $V_{\sigma}$ as a $\liegdrbssmath(A)_{\qbar}$ representation decomposes into $V_1 \oplus V_2$. We then consider the following set 
$$S:=\{\begin{pmatrix}
    a_1|_{V_1}&\\
    &a_2|_{V_2}\\
    &&a_3|_{V_{\sigmabar}}\end{pmatrix}|a_1,a_2,a_3\in\qbar\}\subset V_{\qbar}\otimes V_{\qbar}^{*}$$ where the notation $a_{i}|_{W}$ means scalar multiplication by $a_{i}$ on the subspace $W$. Then elements in $S$ would be annihilated by the image of $\eta^{\mathrm{ss}}: \liegdrbssmath(A)_{\qbar} \rightarrow \mathrm{gl}(V_{\qbar})$. Moreover, by Lemma \ref{13group} and Theorem \ref{deg2mz4}, the image of $\mathrm{Z}(\liegdrbmath(A))$ under $\eta$ is a subset of $\{\begin{pmatrix}
    a_{\sigma}|_{V_{\sigma}}&\\
&b_{\sigmabar}|_{V_{\sigmabar}}\end{pmatrix}|a_{\sigma},b_{\sigmabar}\in\qbar\}\subset \mathrm{gl}(V_{\sigma})\oplus\mathrm{gl}(V_{\sigmabar})$. Hence elements in the set $S$ are also annihilated by the centre of the de Rham-Betti Lie algebra. But because $\liegdrbmath(A)$ is reductive, we have that $\liegdrbmath(A)=\mathrm{Z}(\liegdrbmath(A))\oplus \liegdrbssmath(A)$(see Lemma \ref{centremakessense}), hence elements in $S$ would be annihilated by the entire de Rham-Betti Lie algebra. But by Lemma \ref{liealgebrainv} the set $(V_{\qbar}\otimes V^{*}_{\qbar})^{\liegdrbmath(A)_{\qbar}}$ is precisely equal to $\{\begin{pmatrix}
    a_{\sigma}|_{V_{\sigma}}&\\
&a_{\sigmabar}|_{V_{\sigmabar}}\end{pmatrix}|a_{\sigma},a_{\sigmabar}\in\qbar\}\subset \mathrm{gl}(V_{\sigma})\oplus\mathrm{gl}(V_{\sigmabar})$ which is strictly smaller than the set $S$. Thus we obtain the desired contradiction.
\end{proof}
According to Lemma \ref{phicompatiblewithE}, the $\qbar$-linear extension of a polarization form $$\phi_{\qbar}:  V_{\qbar} \times  V_{\qbar} \rightarrow \qbar(2\pi i)$$ restricts to the following nondegenerate pairing
\begin{equation}\label{pairingbetweeneigenspaces}
\phi_{\qbar}: V_{\sigma} \times V_{\sigmabar} \rightarrow \qbar(2\pi i)    
\end{equation} Then the following lemma asserts that this pairing induces an isomorphism of representations $\rho_{\sigma} \cong (\rho_{\sigmabar})^{*}$.
\begin{lemma} \label{isomdual}
There is an isomorphism of representations between $\rho_{\sigma}:\liegdrbssmath(A)_{\qbar} \rightarrow \mathrm{gl}(V_{\sigma})$ and the dual of $\rho_{\sigmabar}:\liegdrbssmath(A)_{\qbar} \rightarrow \mathrm{gl}(V_{\sigmabar})$ induced by the $\qbar$-bilinear form (\ref{pairingbetweeneigenspaces}).
\end{lemma}
\begin{proof}
The polarization form $\phi$ is preserved by the image of $\eta^{\mathrm{ss}}: \liegdrbssmath(A)_{\qbar} \rightarrow \mathrm{gl}(V_{\qbar})$ infinitesimally. This is because $\phi$ is preserved by the image of $\mathrm{hdg}(A)_{\qbar} \rightarrow \mathrm{gl}(V_{\qbar})$ infinitesimally and $$\liegdrbssmath(A)=[\liegdrbmath(A),\liegdrbmath(A)] \subset [\mathrm{mt}(A),\mathrm{mt}(A)] \subset \mathrm{hdg}(A)$$ Hence for any $l \in \liegdrbssmath(A)_{\qbar}$ and any $v_{\sigma} \in V_{\sigma}, v_{\sigmabar} \in V_{\sigmabar}$ we have the following $$\phi_{\qbar}(\rho_{\sigma}(l)\circ v_{\sigma}, v_{\sigmabar})+\phi_{\qbar}(v_{\sigma}, \rho_{\sigmabar}(l)\circ v_{\sigmabar})=0$$ Therefore the isomorphism of vector spaces $V_{\sigma} \rightarrow V_{\sigmabar}^{*}, v_{\sigma} \rightarrow \phi_{\qbar}(v_{\sigma},\underline{})$ is equivariant with respect to the action of 
$\liegdrbssmath(A)_{\qbar}$.
\end{proof}
\end{setup}
\begin{setup}\label{deg4setup}
    In this paragraph, we fix $A$ a simple abelian fourfold defined over $\qbar$ with $\mathrm{End}^{\circ}(A)\cong K$ where $K$ is a CM-field with $\mathrm{deg}(K/\mathbb{Q})=4$. We now set up relevant notations and prove some basic facts similar to Setup \ref{sectiondeg2}. Denote $\mathrm{H}^1(A,\mathbb{Q})$ by $V$. Label the embeddings of $K$ into $\qbar$ by $\sigma,\overline{\sigma},\tau,\overline{\tau}$ as in the proof of Lemma \ref{deg4weil}. Then according to Lemma \ref{neweigenspacetranslate} with respect to the morphism of $\mathbb{Q}$-algebras $$i:K\rightarrow \mathrm{End}(V)$$ we have the eigenspace decomposition \begin{equation}\label{deg4keydecomp}
V_{\qbar}=V_{\sigma} \oplus V_{\overline{\sigma}} \oplus V_{\tau} \oplus V_{\overline{\tau}}\end{equation} and moreover we have $$\mathrm{dim}(V_{\sigma})=\mathrm{dim}(V_{\tau})=\mathrm{dim}(V_{\overline{\sigma}})=\mathrm{dim}(V_{\overline{\tau}})=2$$
Similar to Lemma \ref{eigenspacesubrep}, we have the following lemma.
\begin{lemma}\label{deg4eigenspacesubrep}
The eigenspaces $V_{\sigma}, V_{\tau}, V_{\overline{\sigma}}, V_{\overline{\tau}}$ from decomposition (\ref{deg4keydecomp}) are in fact irreducible subrepresentations of the canonical representation of de Rham-Betti Lie algebra $$\eta: \liegdrbmath(A)_{\qbar} \rightarrow \mathrm{gl}(V_{\qbar})$$ and its semisimplification $$\eta^{\mathrm{ss}}: \liegdrbssmath(A)_{\qbar }\rightarrow \mathrm{gl}(V_{\qbar})$$ 
\end{lemma}
In this paragraph we will denote the subrepresentations of $\eta^{\mathrm{ss}}$ corresponding to $V_{\sigma}, V_{\tau}, V_{\overline{\sigma}}, V_{\overline{\tau}}$ by $$\rho_{\sigma}, \rho_{\tau}, \rho_{\overline{\sigma}}, \rho_{\overline{\tau}}$$
\begin{proofsketch}
We can apply the same reasoning for Lemma \ref{eigenspacesubrep} to show that $V_{\iota}$s ($\iota\in\{\sigma,\sigmabar,\tau,\overline{\tau}\}$) are subrepresentations of $\mu_{\qbar}$. Moreover by Lemma \ref{deg4weil}, the centre of $\liegdrbssmath(A)$ lies in the following set
 $$\{\begin{pmatrix}
    a_1|_{V_{\sigma}}&&&\\
    &-a_1|_{V_{\overline{\sigma}}}&&\\
    &&a_2|_{V_{\tau}}&\\ &&&-a_2|_{V_{\overline{\tau}}}\end{pmatrix}|a_1,a_2\in\qbar\}$$ Then the method of proving Lemma \ref{irred} can be applied verbatim to show that $\rho_{\iota}$s ($\iota\in\{\sigma,\sigmabar,\tau,\overline{\tau}\}$) are indeed irreducible representations.
   
\end{proofsketch}
Then similar to Lemma \ref{isomdual}, using Lemma \ref{phicompatiblewithE} we obtain the following lemma.
\begin{lemma}\label{deg4isomdual}
 We keep the same notations as above. Given a polarization form $\phi$ on $V$, for any two embeddings $\iota,\iota' \in \mathrm{Hom}(K,\qbar)$, we have that $\phi_{\qbar}(V_{\iota},V_{\iota'})=0$ unless $\iota'=\overline{\iota}$. Also we have the isomorphism of representations $\rho_{\overline{\sigma}}\cong \rho_{\sigma}^{*}$ and  $\rho_{\overline{\tau}}\cong \rho_{\tau}^{*}$ induced by the restriction of $\phi_{\qbar}$ on $V_{\sigma} \oplus V_{\overline{\sigma}}$ and $V_{\tau} \oplus V_{\overline{\tau}}$.
\end{lemma}

\begin{proof}
    It is identical to Lemma \ref{phicompatiblewithE} and Lemma \ref{isomdual}.
\end{proof}

\end{setup}

\subsection{Classification of Possible De Rham-Betti Representations}
Recall that we aim to study isomorphism classes of representations of semisimple Lie algebras defined over $\mathbb{Q}$ which could be candidates for the canonically defined $$\eta: \liegdrbssmath(A)\rightarrow \mathrm{gl}(V)$$ In this section we demonstrate two preliminary results for simple type IV abelian fourfolds.

First we consider the case where $A$ is a simple abelian fourfold defined over $\qbar$ whose endomorphism field $K$ is a degree 2 CM-field. We will be using notations and results from Setup \ref{sectiondeg2}. Note that $\mathrm{dim}(V_{\sigma})=\mathrm{dim}(V_{\sigmabar})=4$. Using this dimension constraint, we can already restrict the study of $\rho_{\sigma}$ and $\rho_{\sigmabar}$ to certain isomorphism classes of Lie algebra representations listed in the following proposition. We will be using notations from \cite{humphreys2012introduction}.
\begin{proposition} \label{reprclassification}
We use notations from Setup \ref{sectiondeg2}. In order to determine $\rho_{\sigma}:\liegdrbssmath(A)_{\qbar}\rightarrow\mathrm{gl}(V_{\sigma})$, it suffices to consider the following isomorphism classes of Lie algebra representations
\begin{enumerate}
    \item\label{sl2degk2} $L \otimes \qbar=\mathrm{sl}(2)$ and the representation $L \otimes \qbar\rightarrow\mathrm{gl}(V_{\sigma})$ is isomorphic to V(3)
    \item\label{sl2timessl2degk2} $L\otimes \qbar=\mathrm{sl}(2) \times \mathrm{sl}(2)$ and the representation $L \otimes \qbar\rightarrow\mathrm{gl}(V_{\sigma})$ is isomorphic to $V(1) \boxtimes V(1)$
    \item\label{sp4degk2}  $L \otimes \qbar=\mathrm{B}_2 \cong \mathrm{C}_2=\mathrm{sp}(4,\qbar)$ and the representation $L \otimes \qbar\rightarrow\mathrm{gl}(V_{\sigma})$ is isomorphic to $V(\lambda_2)$ where $\lambda_2$ is the fundamental weight associated with the short root $\alpha_2$ of $\mathrm{B}_2$. In fact this is the standard representation of $\mathrm{sp}(4,\qbar)$.
    \item $L \otimes \qbar=\mathrm{sl}(4)$. In this case the image of the representation $L \otimes \qbar\rightarrow\mathrm{gl}(V_{\sigma})$ is the entire $\mathrm{sl}(V_{\sigma})$.
\end{enumerate}
\end{proposition}

\begin{proof}
Firstly $\liegdrbssmath(A)_{\qbar}$ is a semisimple Lie algebra by definition and $\rho_{\sigma}$ is an irreducible representation by Lemma \ref{irred}. Moreover, it suffices to consider faithful representations for candidates of $\rho_{\sigma}$. Note that the image of $\rho_{\sigma}: \liegdrbssmath(A)_{\qbar} \rightarrow \mathrm{gl}(V_{\sigma})$ lies in $\mathrm{sl}(V_{\sigma})$. This is because $\liegdrbssmath(A)_{\qbar}$ is semisimple, therefore 
    $\rho_{\sigma}(\liegdrbssmath(A)_{\qbar})=\rho_{\sigma}([\liegdrbssmath(A)_{\qbar},\liegdrbssmath(A) _{\qbar}])=[\rho_{\sigma}(\liegdrbssmath(A)_{\qbar}),\rho_{\sigma}(\liegdrbssmath(A)_{\qbar}]\subset \mathrm{sl}(V_{\sigma})$.
Hence, for possible candidates of $\liegdrbssmath(A)_{\qbar}$, which we denote by $L \otimes \qbar$, it suffices to consider isomorphism classes of semisimple Lie algebras whose dimensions are smaller than or equal to the dimension of $\mathrm{sl}(V_{\sigma})$, which is equal to 15. \begin{enumerate}
    \item Suppose $L \otimes \qbar$ is of the type $\mathrm{A}_{l}$. Then it is isomorphic to one of $\mathrm{sl}(2)$, $\mathrm{sl}(3)$ or $\mathrm{sl}(4)$.
    \begin{enumerate}
        \item Suppose $L \otimes \qbar$ is isomorphic to $\mathrm{sl}(2)$, then by Section 2.7 of \cite{humphreys2012introduction}, the only 4-dimensional irreducible representation of $\mathrm{sl}(2)$ is $V(3)$.
        \item Suppose $L \otimes \qbar$ is isomorphic to $\mathrm{sl}(3)$ i.e. $\mathrm{A}_2$. Denote the two fundamental weights by $\lambda_1$ and $\lambda_2$. By Weyl's dimension formula, (see Example in Section 24.3 of \cite{humphreys2012introduction}), for irreducible representation of the form $V(m_1\lambda_1+m_2\lambda_2)$, we have $$\mathrm{dim}(V(m_1\lambda_1+m_2\lambda_2))=\frac{1}{2}(m_1+1)(m_2+1)(m_1+m_2+2)$$ where $m_1,m_2\in \mathbb{Z}_{\geq0}$. Then we need $$\frac{1}{2}(m_1+1)(m_2+1)(m_1+m_2+2)=4$$ But a small computation shows that this is impossible.
        \item Suppose $L \otimes \qbar$ is isomorphic to $\mathrm{sl}(4)$. Then because we have required $\rho_{\sigma}$ to be faithful, its image under $\rho_{\sigma}$ is the entire $\mathrm{sl}(V_{\sigma})$. 
    \end{enumerate}
    \item Suppose $L \otimes \qbar$ is of the type $\mathrm{B}_{l}$ or $\mathrm{C}_{l}$. Then by Section 1.2 of \cite{humphreys2012introduction}, dim$(\mathrm{B}_{l})$=dim$(\mathrm{C}_{l})=2l^2+l(l\geq2)$. Hence the only possible option is $l=2$, where they are both isomorphic to the simple Lie algebra $\mathrm{sp}(4)$ (see Chapter 11 and 14 of \cite{humphreys2012introduction}). Then denote the two fundamental weights of $\mathrm{B}_2$ by $\lambda_1$ and $\lambda_2$, where $\alpha_2$ is the short root. By Weyl's dimension formula (again, see Example in Section 24.3 of \cite{humphreys2012introduction}) for irreducible representation of the form $V(m_1\lambda_1+m_2\lambda_2)$, we have $$\mathrm{dim}(V(m_1\lambda_1+m_2\lambda_2))=\frac{1}{3!}(m_1+1)(m_2+1)(m_1+m_2+2)(2m_1+m_2+3)$$ and this is equal to 4 when $m_1=0,m_2=1$. Hence in this case $\rho_{\sigma}$ is given by $V(\lambda_2)$. To see this is the standard representation of $\mathrm{sp}(4,\qbar)$, note that according to the Weyl dimension formula above and the highest weight theory (for example see Chapter 20 and 21 of \cite{humphreys2012introduction}), the smallest dimension an irreducible representations of $\mathrm{sp}(4)$ can have is 4 and $V(\lambda_2)$ is the unique irreducible representation of dimension 4 up to isomorphism. Because $\mathrm{sp}(4,\qbar)$ is a simple Lie algebra, by Weyl's complete reducibility theorem (for example see Chapter 6 of \cite{humphreys2012introduction}), the standard representation $\mathrm{sp}(4,\qbar) \rightarrow \mathrm{gl}(4)$ has to be isomorphic to $V(\lambda_2)$.
    \item Suppose $L \otimes \qbar$ is isomorphic to $\mathrm{D}_{l}$. By the classification of Dynkin diagrams in 11.4 of \cite{humphreys2012introduction}, it suffices to consider $D_l$s where $l\geq4$. By Section 1.2 of \cite{humphreys2012introduction}, dim$(\mathrm{D}_{l})=2l^2-2>15$.  By the above discussion, we do not need to consider this possibility.
    \item Suppose $L \otimes \qbar$ is an exceptional simple Lie algebra. Then the only one whose dimension is smaller than $15$ is $\mathrm{G}_2$. Denote the two fundamental weights of $\mathrm{G}_2$ by $\lambda_1$ and $\lambda_2$, where $\alpha_1$ is the short root. Again by Weyl's dimension formula an irreducible representation of the form $V(m_1\lambda_1+m_2\lambda_2)$ has dimension equal to $\frac{1}{5!}(m_1+1)(m_2+1)(m_1+m_2+2)(m_1+2m_2+3)(m_1+3m_2+4)(2m_1+3m_2+5)$ where $m_1,m_2\in \mathbb{Z}_{\geq0}$. A short computation shows that this formula cannot be equal to 4.
    \item Finally we suppose $L\otimes \qbar$ is a direct sum of simple Lie algebras. By \cite[Part 1, Chapter 6, Exercise 7b]{serreliealgebra}, if $g=g_1\oplus g_2$ is a direct sum of two Lie algebras over $\qbar$, then a representation $W$ of $g$ is irreducible if and only it is the form $W_1 \boxtimes W_2$ where $W_1$ and $W_2$ are irreducible representations of $g_1$ and $g_2$. Note that in this case $\mathrm{dim}(W)=\mathrm{dim}(W_1)\mathrm{dim}(W_2)$. Because $\mathrm{dim}(V_{\sigma})=4$, it suffices to consider the case where $L \otimes \qbar=g_1\oplus g_2$ is the direct sum of two simple Lie algebras and $\rho_{\sigma} \cong W_1 \boxtimes W_2$ where $W_1$ and $W_2$ are two-dimensional irreducible representations of $g_1$ and $g_2$ respectively. Then using previous results, we deduce that the only possibility is $L \otimes \qbar=\mathrm{sl}(2)\times \mathrm{sl}(2)$ and the representation $L \otimes \qbar\rightarrow\mathrm{gl}(V_{\sigma})$ is given by $V(1) \boxtimes V(1)$.
\end{enumerate}
\end{proof}
Now we consider the case where $A$ is a simple abelian fourfold whose endomorphism field $K$ is a degree 4 CM-field. We will be using results and notations from Setup \ref{deg4setup}. Similar to Proposition \ref{reprclassification}, in the following Lemma, we will use the condition that $\mathrm{dim}(V_{\sigma})=\mathrm{dim}(V_{\tau})=2$ to constrain possibilities for isomorphism classes of representations $\eta^{\mathrm{ss}}_{\qbar}:\liegdrbssmath(A)_{\qbar} \rightarrow \mathrm{sl}(V_{\qbar})$. 

\begin{lemma}\label{deg4repclassification}
  The images of $\rho_{\sigma}:\liegdrbssmath(A) \rightarrow \mathrm{gl}(V_{\sigma})$ and $\rho_{\tau}: \liegdrbssmath(A) \rightarrow \mathrm{gl}(V_{\tau})$ are $\mathrm{sl}(V_{\sigma})$ and $\mathrm{sl}(V_{\tau})$. In fact, it suffices to consider the following two candidates for the isomorphism class of the representation $\liegdrbssmath(A)_{\qbar} \rightarrow \mathrm{gl}(V_{\qbar})$.
  \begin{enumerate}
      \item\label{case1} $\liegdrbssmath(A)_{\qbar}\cong\mathrm{sl}(2)$ and its image under $\rho_{\sigma} \oplus \rho_{\tau}$ in $\mathrm{sl}(V_{\sigma}) \times \mathrm{sl}(V_{\tau})$ is the graph of a Lie algebra isomorphism between $\mathrm{sl}(V_{\sigma})$ and $\mathrm{sl}(V_{\tau})$.
      \item $\liegdrbssmath(A)_{\qbar}\cong\mathrm{sl}(2) \times \mathrm{sl}(2)$ and its image under $\rho_{\sigma} \oplus \rho_{\tau}$ is the entire $\mathrm{sl}(V_{\sigma}) \times \mathrm{sl}(V_{\tau})$.
  \end{enumerate}
  
\end{lemma}
\begin{proof}
Because $\liegdrbssmath(A)_{\qbar}$ is a semisimple Lie algebra, by Corollary 5.3 of \cite{humphreys2012introduction}, its images under $\rho_{\sigma}$ and $\rho_{\tau}$ are semisimple Lie algebras inside $\mathrm{sl}(V_{\sigma})$ and $\mathrm{sl}(V_{\tau})$. But $\mathrm{dim}(V_{\sigma})= \mathrm{dim}(V_{\tau})=2$ hence the images are either $\{0\}$ or the entire $\mathrm{sl}(V_{\sigma})$ or $\mathrm{sl}(V_{\tau})$. By Lemma \ref{deg4isomdual}, $\rho_{\sigma}$ and $\rho_{\tau}$ are irreducible representations. Hence their images are the entire $\mathrm{sl}(V_{\sigma})$ and $\mathrm{sl}(V_{\tau})$.

Denoting $\rho_{\sigma} \oplus \rho_{\tau}$ by $\iota$, we have the following commutative diagram
    $$\begin{tikzcd}
                                                                                                                                          & \mathrm{sl}(V_{\tau})   &                                                                                                               \\
\liegdrbssmath(A)_{\qbar} \arrow[rr,"\iota"] \arrow[ru, "\rho_{\tau}=p_2\circ \iota", two heads] \arrow[rd, "\rho_{\sigma}=p_1\circ \iota"', two heads] &                         & \mathrm{sl}(V_{\sigma})\times\mathrm{sl}(V_{\tau}) \arrow[ld, "p_1", two heads] \arrow[lu, "p_2"', two heads] \\
                                                                                                                                          & \mathrm{sl}(V_{\sigma}) &                                                                                                              
\end{tikzcd}$$
By Goursat's Lemma \ref{goursat} and the fact that $\mathrm{sl}(2)$ is a simple Lie algebra, we have that $\iota(\liegdrbssmath(A)_{\qbar})$ is either $\mathrm{sl}(V_{\sigma}) \times \mathrm{sl}(V_{\tau})$ or the graph of a Lie algebra isomorphism between $\mathrm{sl}(V_{\sigma})$ and $\mathrm{sl}(V_{\tau})$. However, by Lemma \ref{deg4isomdual}, we have the isomorphism of representations $\rho_{\overline{\sigma}}\cong \rho_{\sigma}^{*}$ and  $\rho_{\overline{\tau}}\cong \rho_{\tau}^{*}$. Therefore the image of $\liegdrbssmath(A)_{\qbar}$ in $\mathrm{sl}(V_{\qbar})$ is completely determined by its restriction to $\mathrm{sl}(V_{\sigma}) \times \mathrm{sl}(V_{\tau})$. Hence $\liegdrbssmath(A)$ is either isomorphic to a $\qbar/\mathbb{Q}$ form of $\mathrm{sl}(2)$ or isomorphic to a $\qbar/\mathbb{Q}$ form of $\mathrm{sl}(2) \times \mathrm{sl}(2)$.
\end{proof}
\subsection{De Rham-Betti Groups of Simple Weil Type Abelian Fourfolds}
Using the classification result i.e. Proposition \ref{reprclassification} together with some basic invariant theory, we can already determine $\gdrbmath(A)$ where $A$ is a simple Weil-type abelian fourfold defined over $\qbar$. The method is the same as computing the Mumford-Tate group of a a simple Weil type abelian fourfold (\cite[Section 7.4]{mz-4-folds}). We fill in more details as well as the necessary computations. However, for a simple anti-Weil type abelian fourfold, situations are different. This is discussed in Key Remark \ref{whyweilproofdoesnotwork} which leads to the computations and discussions in Section \ref{antiweilsection}.

\begin{proposition}\label{weiltype}
Suppose $A$ is a simple Weil-type abelian fourfold defined over $\qbar$. Then $$\liegdrbssmath(A)_{\qbar}=\mathrm{sl}_4$$ In this case we have that $\gdrbmath(A)=\mathrm{MT}(A)$. 
\end{proposition}
\begin{proof}
We will be using notations from Setup \ref{sectiondeg2}. From Theorem \ref{deg2mz4} and Proposition \ref{gmintype4}, we know that $\mathrm{Z}(\gdrbmath(A))=\mathrm{Z}(\mathrm{MT}(A))=\gmmath$. By Section 7.4 of \cite{mz-4-folds}, we have that $\mathrm{mt}^{\mathrm{ss}}(A)_{\qbar}=\mathrm{sl}_4$. Hence if we can show that $\liegdrbssmath(A)_{\qbar} =\mathrm{sl}_4$ then because there is a natural inclusion $\liegdrbssmath(A)\subset\mathrm{mt}^{\mathrm{ss}}(A)$ defined over $\mathbb{Q}$, we will obtain that $\liegdrbssmath(A)=\mathrm{mt}^{\mathrm{ss}}(A)$ and thus $\liegdrbmath(A)=\mathrm{mt}(A)$. Because $\mathrm{MT}(A)$ is a connected algebraic group, we have that $\gdrbmath(A)=\mathrm{MT}(A)$. The rest of the proof is to show that $\liegdrbssmath(A)_{\qbar}=\mathrm{sl}_4$. 

Because $\mathrm{Z}(\gdrbmath(A))=\gmmath$, we have that $\mathrm{Lie}\gdrbhmath(A)=\liegdrbssmath(A)$. Therefore by Lemma \ref{liealgebrainv} and Corollary \ref{gdrbhinv}, we have  $$\mathrm{H}^2(A,\mathbb{Q})^{\liegdrbssmath(A)} \otimes \mathbb{Q}_{\mathrm{dRB}}(1)=(\mathrm{H}^2(A,\mathbb{Q}) \otimes \mathbb{Q}_{\mathrm{dRB}}(1))^{\gdrbmath(A)}$$ and hence \begin{equation}\label{deg2eq}
\mathrm{H}^2(A,\qbar)^{\liegdrbssmath(A)_{\qbar}} \otimes \qbar_{\mathrm{dRB}}(1)=(\mathrm{H}^2(A,\qbar) \otimes \qbar_{\mathrm{dRB}}(1))^{\gdrbmath(A)(\qbar)}    
\end{equation} For simplicity of the notation, we will drop the twists in the rest of the proof when no confusion arises. Now \begin{equation}\label{degree2cohdecomp}\mathrm{H}^2(A,\qbar)=\wedge^2(V_{\sigma} \oplus V_{\overline{\sigma}})=\wedge^2V_{\sigma} \oplus \wedge^2 V_{\sigmabar} \oplus V_{\sigma} \otimes V_{\overline{\sigma}}\end{equation} By Lemma \ref{isomdual}, the two irreducible representations $\rho_{\sigma}: \liegdrbssmath(A) _{\qbar} \rightarrow \mathrm{gl}(V_{\sigma})$ and $\rho_{\overline{\sigma}}: \liegdrbssmath(A)_{\qbar}\rightarrow \mathrm{gl}(V_{\overline{\sigma}})$ are dual to each other. Hence by Schur's lemma, in $V_{\sigma} \otimes V_{\overline{\sigma}}$ there is a unique element $l$ up to scalar which is annihilated by the action of $\liegdrbmath(A)_{\qbar}$. However, recall that $\mathrm{dim}(\mathrm{End}^{\circ}(A))=\mathrm{deg}(K/\mathbb{Q})=2$. Therefore by Lemma \ref{rosaticmcoincides}, we have $$\mathrm{dim}(\mathrm{H}^2(A,\mathbb{Q})^{\liegdrbssmath(A)})=\mathrm{dim}(\mathrm{H}^2(A,\mathbb{Q})^{\mathrm{Hdg}(A)})=1$$ Therefore $l$ is the unique element up to scalar in $\mathrm{H}^2(A,\qbar)$ annihilated by $\liegdrbssmath(A)_{\qbar}$.

Meanwhile, we also have that \begin{equation}\label{endoinv}
 \mathrm{End}_{\liegdrbssmath(A)_{\qbar}}(\mathrm{H}^1(A,\qbar))=\mathrm{End}_{\liegdrbmath(A)_{\qbar}}(\mathrm{H}^1(A,\qbar))=\mathrm{End}_{\mathrm{Hdg}(A)_{\qbar}}(\mathrm{H}^1(A,\qbar))   
\end{equation} Now $$\mathrm{End}(\mathrm{H}^1(A,\qbar))=\mathrm{End}(V_{\sigma},V_{\sigma})\oplus \mathrm{End}(V_{\sigmabar},V_{\sigmabar})\oplus \mathrm{End}(V_{\sigma},V_{\sigmabar})\oplus \mathrm{End}(V_{\sigmabar},V_{\sigma})$$ By formula (\ref{endoinv}), the elements in $\mathrm{End}(\mathrm{H}^1(A,\qbar))$ annihilated by $\liegdrbssmath(A)_{\qbar}$ are spanned by the set of scalar multiplications in $\mathrm{End}(V_{\sigma},V_{\sigma})$ and $\mathrm{End}(V_{\sigmabar},V_{\sigmabar})$.

We will show that if $\rho_{\sigma}$ is isomorphic to any one of the first three isomorphism classes of representations in Proposition \ref{reprclassification}, either there exists $l' \in \wedge^2V_{\sigma}\subset \mathrm{H}^2(A,\qbar)$ or $l''\in \mathrm{End}(V_{\sigma},V_{\sigmabar})=V_{\sigma}\otimes V_{\sigmabar}^{*}\cong_{\liegdrbssmath(A)_{\qbar}} V_{\sigma}\otimes V_{\sigma}$ annihilated by the image of $\rho_{\sigma}$. Then a contradiction is reached.
\begin{enumerate}
    \item Suppose $L:=\liegdrbssmath(A)_{\qbar}$ is isomorphic to $\mathrm{sl}(2)$ and $V_{\sigma}$ is isomorphic to $V(3)$. We will use notations and constructions from Chapter 7 of \cite{humphreys2012introduction}. Then $L$ is spanned by $h,x,y$ where $[h,x]=2x$,$[h,y]=-2y$ and $[x,y]=h$. And $V_{\sigma}=V(3)$ has a basis labeled by $\{v_0,v_1,v_2,v_3\}$ satisfying $h\circ v_{i}=(3-2i)v_{i}$ while $y\circ v_{i}=(i+1)v_{i+1}$ and $x\circ v_{i}=(4-i)v_{i-1}$. Then inside $\wedge^2V_{\sigma}$, the vector $3v_0\wedge v_3-v_1\wedge v_2$ is annihilated by the action of $\mathrm{sl}(2)$.
    \item Suppose $L:=\liegdrbssmath(A)_{\qbar}$ is isomorphic to $\mathrm{sl}(2) \times \mathrm{sl}(2)$ and $V_{\sigma}$ is isomorphic to $V(1) \boxtimes V(1)$. Then we denote $v_{1,1}:=v_{0}\otimes v_{0}'; v_{1,-1}:=v_{0}\otimes v_{1}';v_{-1,1}:=v_{1}\otimes v_{0}';v_{-1,-1}:=v_{1}\otimes v_{1}'$, where $(h,0) \circ v_{i,j}=iv_{i,j}$ and  $(0,h) \circ v_{i,j}=jv_{i,j}$. Moreover $(x,0) \circ v_{1,j}=0$; $(x,0) \circ v_{-1,j}=v_{1,j}$ and $(0,x) \circ v_{i,1}=0$; $(0,x) \circ v_{i,-1}=v_{i,1}$. Finally, $(y,0) \circ v_{1,j}=v_{-1,j}$; $(y,0) \circ v_{-1,j}=0$ and $(0,y) \circ v_{i,1}=v_{i,-1}$; $(0,y) \circ v_{i,-1}=0$. Then inside $V_{\sigma}\otimes V_{\sigma}$, the vector $v_{1,1}\otimes v_{-1,-1}+v_{-1,-1}\otimes v_{1,1}-v_{1,-1}\otimes v_{-1,1}-v_{-1,1}\otimes v_{1,-1}$ is annihilated by the action of $\mathrm{sl}(2) \times \mathrm{sl}(2)$.
    \item Suppose $L:=\liegdrbssmath(A)_{\qbar}$ is isomorphic to $\mathrm{sp}(4)$ and $V_{\sigma}$ is isomorphic to the standard representation of $\mathrm{sp}(4,\qbar)$. Fixing a basis $\{e_1,e_2,e_3,e_4\}$ of $V_{\sigma}$ such that the standard symplectic form on $V_{\sigma}$ is diagonalized as $s=\begin{pmatrix}
        0&0&1&0\\
        0&0&0&1\\
        -1&0&0&0\\
        0&-1&0&0
    \end{pmatrix}$. Then $$\mathrm{sp}(4,\qbar)=\{x \in \mathrm{gl}(V_{\sigma})|x^{t}s+sx=0\}$$ One can then check that the vector $e_1\wedge e_3+e_2\wedge e_4 \in \wedge^2V_{\sigma}$ is annihilate by $\mathrm{sp}(4,\qbar)$. 
\end{enumerate}
Therefore the representation $\rho_{\sigma}: \liegdrbssmath(A)_{\qbar} \rightarrow \mathrm{gl}(V_{\sigma})$ is isomorphic to the last possible scenario in Proposition \ref{reprclassification}, i.e. the standard representation of $\mathrm{sl}_4$.
\end{proof}
\begin{keyremark}\label{whyweilproofdoesnotwork}
The above proof cannot apply to a simple anti-Weil type abelian fourfold $A$. By Lemma \ref{twocentressame}, we have that the centre of the de Rham-Betti Lie algebra of $A$ is equal to the centre of its Mumford-Tate Lie algebra. Denote $\mathrm{Lie}\gdrbhmath(A)$ by $\liegdrbhmath(A)$. Then the centre of $\liegdrbhmath(A)$ is equal to the centre of $\mathrm{hdg}(A)$. By computations from Section \ref{chaptercentre}, we have that the image of $\mathrm{Z}(\liegdrbhmath(A)_{\qbar})$ inside $\mathrm{gl}(V_{\qbar})$ is
$$\{\begin{pmatrix}
 \alpha|_{V_{\sigma}}&\\
 &-\alpha|_{V_{\sigmabar}}
\end{pmatrix}|\alpha\in \qbar\}$$ 
Therefore for each isomorphism class of representation $$\rho: L \rightarrow \mathrm{gl}(V)$$ from Proposition \ref{reprclassification} we have that $$(\wedge^2 V_{\sigma})^{\mathrm{Z}(\liegdrbhmath(A)_{\qbar}) \oplus \rho(L)_{\qbar}}=(\wedge^2 V_{\sigmabar})^{\mathrm{Z}(\liegdrbhmath(A)_{\qbar}) \oplus \rho(L)_{\qbar}}=\{0\}$$ Moreover because each representation $\rho$ of $L$ from Proposition \ref{reprclassification} is absolutely irreducible and $\rho_{\sigma} \cong \rho_{\sigmabar}^{*}$ by Lemma \ref{isomdual}, we have that $(V_{\sigma}\otimes V_{\sigmabar})^{\mathrm{Z}(\liegdrbhmath(A))_{\qbar} \oplus \rho(L)_{\qbar}}$ is a one-dimensional $\qbar$-vector space, which is equal to $(V_{\sigma}\otimes V_{\sigmabar})^{\mathrm{Hdg}(A)_{\qbar}}$. Therefore by decomposition (\ref{degree2cohdecomp}) and Lemma \ref{invarintslemma}, we have that $$\mathrm{H}^2(A,\mathbb{Q})^{\mathrm{Z}(\liegdrbhmath(A))\oplus \rho(L)}=\mathrm{H}^2(A,\mathbb{Q})^{\mathrm{Hdg}(A)}$$ which is spanned by algebraic classes.
Similarly we also have that $$(V_{\sigma}\otimes V_{\sigma})^{\mathrm{Z}(\liegdrbhmath(A)_{\qbar}) \oplus \rho(L)_{\qbar}}=(V_{\sigmabar} \otimes V_{\sigmabar})^{\mathrm{Z}(\liegdrbhmath(A))_{\qbar} \oplus \rho(L)_{\qbar}}=\{0\}$$ Therefore we have that $$\mathrm{End}_{\mathrm{Z}(\liegdrbmath(A))\oplus \rho(L)}(\betti)=\mathrm{End}_{\mathrm{Hdg}}(\betti)$$ and both of sides of the above are spanned by algebraic classes.

Hence simply using the information of the location of de Rham-Betti classes in degree 2 cohomology groups, we cannot rule out any one of the first three cases from Proposition \ref{reprclassification} being isomorphic to $\liegdrbssmath(A)_{\qbar}$. We will discuss another way to partially deal with those cases in Section \ref{antiweilsection}.
\end{keyremark}

\section{De Rham-Betti Groups of Simple Type IV(2,1) Abelian Fourfolds}\label{deg4section}
The goal of this section is to prove the following theorem.
\begin{theorem} \label{deg4theorem}
    Let $A$ be a simple abelian fourfold defined over $\qbar$ such that $$\mathrm{End}_{\mathrm{Hdg}}(A)=K$$ where $K$ is a CM-field with $\mathrm{deg}(K/\mathbb{Q})=4$. Then we have $\gdrbmath(A)=\mathrm{MT}(A)$.
\end{theorem}

As usual we denote $\betti$ by $V$. Denote $\mathrm{Lie}\gdrbmath(A)$ by $\liegdrbmath(A)$. Recall that in Proposition \ref{twocentressame}, we have obtained that $\mathrm{Z}(\liegdrbmath(A))=\mathrm{Z}(\mathrm{mt}(A))$. Hence the natural strategy of proving Theorem \ref{deg4theorem} is to analyze all possible isomorphism classes of representations of semisimple Lie algebras $$L \rightarrow \mathrm{gl}(V)$$ defined over $\mathbb{Q}$ which can be potential candidates for the canonical representation $$\liegdrbssmath(A) \rightarrow \mathrm{gl}(V)$$ To be more precise, we will be using the classification result of Lemma \ref{deg4repclassification}. In particular, we are going to show that if the canonical representation $\liegdrbssmath(A) \rightarrow \mathrm{gl}(V)$ were isomorphic to Case (\ref{case1}) from Lemma \ref{deg4repclassification}, then the polarization form $\phi$, which is preserved by the de Rham-Betti group of $A$, cannot polarize any weight 1 Hodge structure on which $K$ acts with multiplicities $\{2,0,1,1\}$ (see Lemma \ref{deg4keylemma}) thus giving a contradiction.
\begin{remark} 
Our strategy of proving Theorem \ref{deg4theorem} is different from the method of computing the Mumford-Tate group of such abelian fourfolds from Section 7.5 of \cite{mz-4-folds}. To be more precise, the authors of \cite{mz-4-folds} used the definition of the Mumford-Tate group as the smallest $\mathbb{Q}$-algebraic group whose set of $\mathbb{R}$-points contains the image of the Deligne torus (see Definition \ref{delignetorusmt}) in a fundamental way. However, in the current state of art, de Rham-Betti groups are only defined in the Tannakian formalism (Definition \ref{drbgdefn}).
\end{remark}

\subsection{Forms of \texorpdfstring{$\mathrm{sl}(2)$}{sl(2)}}
We now focus our attention on Case (\ref{case1}) from Lemma \ref{deg4repclassification}. Recall that in this case we assume that $\liegdrbssmath(A) \otimes_{\mathbb{Q}} \qbar$ is isomorphic to $\mathrm{sl}(2)$. Then since $\liegdrbssmath(A)$ is defined over $\mathbb{Q}$ by Lemma \ref{centremakessense}, it is in particular a $\qbar/\mathbb{Q}$ form of $\mathrm{sl}(2)$. Such forms are classified by quaternion algebras over $\mathbb{Q}$, of which we now recall the definition.
\begin{definition}\label{quarterniondefn}
 A quaternion algebra $\mathbb{Q}(a,\lambda)$ is a four-dimensional $\mathbb{Q}$-algebra spanned by $\{1,i,j,k\}$ whose ring structure is given by $i^2=a,j^2=\lambda,ij=-ji=k$ where $a,\lambda\in\mathbb{Q}^{\times}$. We denote by $\mathbb{Q}(a,\lambda)^{\circ}$ the $\mathbb{Q}$-vector subspace spanned by $\{i,j,k\}$.
\end{definition}
We can also put a Lie algebra structure on $\mathbb{Q}(a,\lambda)$ by defining $[x,y]=xy-yx$, making $\mathbb{Q}(a,\lambda)^{\circ}$ a Lie subalgebra.
\begin{lemma}[III.1.4, \cite{serre1979galois}]\label{sl2qformlemma}
Every $\qbar/\mathbb{Q}$ form of $\mathrm{sl}(2,\qbar)$ is isomorphic to the Lie algebra $\mathbb{Q}(a,\lambda)^{\circ}$ for some $a,\lambda \in \mathbb{Q}^{\times}$. 
\end{lemma}
\begin{proof}
For the proof, see Section III.1 of \cite{serre1979galois}. Given a quaternion algebra $\mathbb{Q}(a,\lambda)$, we will fix an explicit isomorphism $\mathbb{Q}(a,\lambda)^{\circ} \otimes \qbar \cong \mathrm{sl}(2)$ as follows
\begin{equation}\label{sl2explicitiso}
  h=i\otimes\frac{1}{\sqrt{a}};x=\frac{1}{2\lambda}(j+k\otimes\frac{1}{\sqrt{a}});y=\frac{1}{2}(j-k\otimes\frac{1}{\sqrt{a}})
\end{equation} where we fix a square root of $a$ and denote it by $\sqrt{a}$. Note that $h,x,y$ indeed defines a $\mathrm{sl}(2)$ triple in the sense of \cite[Chapter 7]{humphreys2012introduction} i.e. $$[x,y]=h,[h,x]=2x,[h,y]=-2y$$
\end{proof}
Depending on $a>0$ or $a<0$, 
the complex conjugation fixes $\sqrt{a}$ or sends $\sqrt{a}$ to $-\sqrt{a}$.

\begin{corollary} \label{relations between y and x}
    Under the isomorphism (\ref{sl2explicitiso}), if $a<0$ then $\lambda\overline{x}=y$ and if $a>0$, then $\overline{x}=x$.
\end{corollary}
Suppose we are in the scenario of Case (\ref{case1}) from Lemma \ref{deg4repclassification}. Recall from Setup \ref{deg4setup} with respect to $K\xhookrightarrow{}\mathrm{End}(V)$ there is an eigenspace decomposition $$V\otimes_{\mathbb{Q}}\qbar=V_{\sigma}\oplus V_{\overline{\sigma}}\oplus V_{\tau}\oplus V_{\overline{\tau}}$$ where $\mathrm{Hom}(K,\qbar)=\{\sigma,\sigmabar,\tau,\overline{\tau}\}$. Moreover we are given a $\mathbb{Q}$-representation of Lie algebras $$\mu: \mathbb{Q}(a,\lambda)^{\circ} \rightarrow \mathrm{gl}(V)$$ such that its $\qbar$-linear extension decomposes into a direct sum of standard representations of $\mathrm{sl}(2)$
$$\mu_{\qbar}:\mathrm{sl}(2)\rightarrow \mathrm{gl}(V_{\sigma})\oplus \mathrm{gl}(V_{\overline{\sigma}})\oplus \mathrm{gl}(V_{\tau})\oplus \mathrm{gl}(V_{\overline{\tau}})$$ Also  $\overline{V_{\sigma}}=V_{\overline{\sigma}}$ and $\overline{V_{\tau}}=V_{\overline{\tau}}$. Denote the subrepresentation $\mathrm{sl}(2)\rightarrow\mathrm{gl}(V_{\iota})$ by $\rho_{\iota}$ where $\iota \in \{\sigma,\sigmabar,\tau,\overline{\tau}\}$. Then the complex conjugate of a weight vector in $V_{\iota}$  under the action of $h$ ($\iota \in \{\sigma,\sigmabar,\tau,\overline{\tau}\})$ in fact remains a weight vector in $V_{\overline{\iota}}$. This is made more precise in the following lemma.
\begin{lemma} \label{conjugate and weight vector V(1)}
 We keep the same notations as above. Let $v_{m} \in V_{\iota}$($\iota \in \{\sigma,\sigmabar,\tau,\overline{\tau}\}$) be an eigenvector of $\rho_{\iota}(h)$ of weight $m$($m \in \{-1,1\}$). If $a<0$, then $\overline{v_{m}} \in V_{\overline{\iota}}$ remains an eigenvector of $\rho_{\overline{\iota}}(h)$ but has weight $-m$. If $a>0$, then $\overline{v_{m}} \in V_{\overline{\iota}}$ remains an eigenvector of $\rho_{\overline{\iota}}(h)$ with the same weight $m$.
\end{lemma}
\begin{proof}
 Using isomorphism (\ref{sl2explicitiso}), we have that $h=i\otimes\frac{1}{\sqrt{a}}$ and therefore $\mu_{\qbar}(h)=\mu(i)\otimes \frac{1}{\sqrt{a}}$. Hence if $v_{m} \in V_{\iota}$ satisfies 
\begin{equation*}
  \rho_{\iota}(h) \circ v_{m}=\mu_{\qbar}(h) \circ v_{m}=mv_{m}  
\end{equation*} then
\begin{equation*}
   \mu(i)\otimes \frac{1}{\sqrt{a}} \circ v_{m}=mv_{m}
\end{equation*} Now suppose $a<0$. Notice that $\mu(i)$ has $\mathbb{Q}$-coefficients with respect to a $\mathbb{Q}$-basis of $V$, therefore applying complex conjugation to the above equation, $\mu(i)$ remains unchanged. Thus we have
\begin{equation*}
   \overline{\rho_{\iota}(h) \circ v_{m}}=\mu(i) \otimes \overline{\frac{1}{\sqrt{a}}} \circ \overline{v_{m}}=-\mu(i) \otimes \frac{1}{\sqrt{a}} \circ \overline{v_{m}}=-\rho_{\overline{\iota}}(h) \circ \overline{v_{m}}=m\overline{v_{m}}
\end{equation*}
Hence $\overline{v_{i}}$ remains an eigenvector of $h$ but has weight $-m$.
Next suppose that $a>0$, then applying complex conjugation, we obtain
\begin{equation*}
   \overline{\rho_{\iota}(h) \circ v_{m}}=\mu(i) \otimes \overline{\frac{1}{\sqrt{a}}} \circ \overline{v_{m}}=\mu(i) \otimes \frac{1}{\sqrt{a}} \circ \overline{v_{m}}=\rho_{\overline{\iota}}(h) \circ \overline{v_{m}}=m\overline{v_{m}}
\end{equation*} Hence the second statement of the lemma follows.
\end{proof}

\subsection{Values of Polarization Form}
We will be using notations from Setup \ref{deg4setup} and Lemma \ref{deg4repclassification} in this section. Observe that any polarization form on $\betti$ is preserved by the image of $\liegdrbssmath(A)$ infinitesimally. Therefore the hypothesis that $\liegdrbssmath(A)$ were isomorphic to $\mathbb{Q}(a,\lambda)^{\circ}$ with $a<0$ or were isomorphic to $\mathbb{Q}(a,\lambda)^{\circ}$ with $a>0$ would actually impose linear relations on values of the polarization form when evaluated at certain weight vectors. This is made precise in Lemma \ref{deg4Polarisation properties k imaginary} and Lemma \ref{deg4Polarisation properties k totallyreal}.

\begin{lemma} \label{deg4Polarisation properties k imaginary}  Suppose there exists an abelian fourfold $A$ defined over $\qbar$ with $\mathrm{End}^{\circ}(A)\cong K$ where $K$ is a quartic CM-field and furthermore suppose we are in the hypothetical scenario where $\liegdrbssmath(A) \cong \mathbb{Q}(a,\lambda)^{\circ}$ with $a<0$. Fix the $\mathrm{sl}(2)$ triple $(h,x,y)$ for $\liegdrbssmath(A) \otimes \qbar \cong \mathrm{sl}(2)$ as in isomorphism (\ref{sl2explicitiso}). Denote $\betti$ by $V$. Then for any polarization form $\phi:V \times V \rightarrow \mathbb{Q}$, we can find a basis $\{v_{-1},v_1\}$ for $V_{\sigma}$, $\{w_{-1},w_1\}$ for $V_{\overline{\sigma}}$, $\{v'_{-1},v'_{1}\}$ for $V_{\tau}$ and $\{w'_{-1},w'_{1}\}$ for $V_{\overline{\tau}}$ such that the following conditions are satisfied. \begin{enumerate}
    \item The $v_{m}$s, $w_{m}$s, $v'_{m}$s and $w'_{m}$s are of weights $m$
    \item $w_{m}=\overline{v_{-m}}$ and $w'_{m}=\overline{v'_{-m}}$
    \item We have that \begin{equation}\label{sameweight0}\phi_{\mathbb{C}}(v_{m},w_{m})=\phi_{\mathbb{C}}(v'_{m},w'_{m})=0\end{equation} for $m\in\{-1,1\}$. Also $V_{\sigma}$, $V_{\overline{\sigma}}$, $V_{\tau}$ and $V_{\overline{\tau}}$ are totally isotropic subspaces for $\phi_{\mathbb{C}}$.
    \item The values of $\phi_{\mathbb{C}}(v_{m},w_{m'})$ and $\phi_{\mathbb{C}}(v'_{m},w'_{m'})$ when $m+m'=0$ satisfy the following relations \begin{align*} 
 \phi_{\mathbb{C}}(v_{-1},w_1)=-\lambda\phi_{\mathbb{C}}(v_1,w_{-1})\\
 \phi_{\mathbb{C}}(v'_{-1},w'_1)=-\lambda\phi_{\mathbb{C}}(v'_1,w'_{-1})
\end{align*}
\end{enumerate} 
\end{lemma}
\begin{proof}
We will write down an explicit construction for the basis from the lemma. We fix $v_1$ a weight 1 vector under the action of $\rho_{\sigma}(h)$ in $V_{\sigma}$. And let \begin{equation}\label{expliciteigenvector}
v_{-1}=y\circ v_1, w_{-1}=\overline{v_1}\in V_{\overline{\sigma}}, w_1=\overline{v_{-1}} \in V_{\overline{\sigma}}    
\end{equation} while we also fix $v'_1$ a weight 1 vector under the action of $\rho_{\tau}(h)$ in $V_{\tau}$ and let $$v'_{-1}=y\circ v'_1, w'_{-1}=\overline{v'_1}\in V_{\overline{\tau}}, w'_1=\overline{v'_{-1}} \in V_{\overline{\tau}}$$ Then by Lemma \ref{conjugate and weight vector V(1)}, the first two conditions in the lemma are satisfied by construction. As for the third condition, note that any polarization form is preserved by the image of $\liegdrbssmath(A)$ infinitesimally. Hence in particular for any $v,v'\in V_{\qbar}$, we have that $$\phi(h\circ v,v')+\phi(v,h\circ v')=0$$ Hence formula (\ref{sameweight0}) from Condition 3) is also satisfied. Moreover, because the Rosati involution on $K$ associated with $\phi$ coincides with the complex conjugation on $K$ by Corollary \ref{rosaticmcoincides}, we can apply Lemma \ref{phicompatiblewithE} show that each eigenspace is totally isotropic with respect to $\phi_{\mathbb{C}}$. As for the last one, we note that $$\phi(y \circ v_1,w_1)+\phi(v_1,y\circ w_1)=0$$ and $$\phi(y \circ v'_1,w'_1)+\phi(v'_1,y\circ w'_1)=0$$ Combining this with the observation that $\lambda\overline{x}=y$ from Corollary \ref{relations between y and x} leads to the last relation in the lemma.
\end{proof}

\begin{lemma} \label{deg4Polarisation properties k totallyreal} 
Suppose there exists an abelian fourfold $A$ defined over $\qbar$ with $\mathrm{End}^{\circ}(A)\cong K$ where $K$ is a quartic CM-field and furthermore assume we are in the hypothetical scenario where $\liegdrbssmath(A) \cong \mathbb{Q}(a,\lambda)^{\circ}$ with $a>0$. Fix the $\mathrm{sl}(2)$ triple $(h,x,y)$ of $\liegdrbssmath \otimes_{\mathbb{Q}} \qbar \cong \mathrm{sl}(2)$ as in isomorphism (\ref{sl2explicitiso}). Then for any polarization form: $\phi:V \times V \rightarrow \mathbb{Q}$, we can find a basis $\{v_{-1},v_1\}$ for $V_{\sigma}$, $\{w_{-1},w_1\}$ for $V_{\overline{\sigma}}$, $\{v'_{-1},v'_{1}\}$ for $V_{\tau}$ and $\{w'_{-1},w'_{1}\}$ for $V_{\overline{\tau}}$ such that the following conditions are satisfied. \begin{enumerate}
    \item The $v_{m}$s, $w_{m}$s, $v'_{m}$s and $w'_{m}$s are of weights $m$
    \item $w_{m}=\overline{v_{m}}$ and $w'_{m}=\overline{v'_{m}}$
    \item We have that $$\phi_{\mathbb{C}}(v_{m},w_{m})=\phi_{\mathbb{C}}(v'_{m},w'_{m})=0$$ for $m\in\{-1,1\}$. Also $V_{\sigma}$, $V_{\overline{\sigma}}$, $V_{\tau}$ and $V_{\overline{\tau}}$ are totally isotropic subspaces for $\phi_{\mathbb{C}}$.
    \item The values of $\phi_{\mathbb{C}}(v_{m},w_{m'})$ and $\phi_{\mathbb{C}}(v'_{m},w'_{m'})$ when $m+m'=0$ satisfy the following relations \begin{align*} 
 \phi_{\mathbb{C}}(v_{-1},w_1)+\phi_{\mathbb{C}}(v_1,w_{-1})=0\\
 \phi_{\mathbb{C}}(v'_{-1},w'_1)+\phi_{\mathbb{C}}(v'_1,w'_{-1})=0
\end{align*}
\end{enumerate} 
\end{lemma}

\begin{proof}
The proof will be similar to that of Lemma \ref{deg4Polarisation properties k imaginary} and we will merely point out the difference. We fix $v_1$ a weight 1 vector under the action of $\rho_{\sigma}(h)$ in $V_{\sigma}$. And let \begin{equation}\label{construct}
 v_{-1}=y\circ v_1, w_{-1}=\overline{v_{-1}}\in V_{\overline{\sigma}}, w_1=\overline{v_{1}} \in V_{\overline{\sigma}} 
\end{equation} while we also fix  $v'_1$ a weight 1 vector under the action of $\rho_{\tau}(h)$ in $V_{\tau}$ and let $$v'_{-1}=y\circ v'_1, w'_{1}=\overline{v'_1}\in V_{\overline{\tau}}, w'_{-1}=\overline{v'_{-1}} \in V_{\overline{\tau}}$$ Then for reasons similar to the previous lemma, the first three conditions are satisfied by construction. As for the last condition, we note that $$\phi_{\mathbb{C}}(y \circ v_1,w_1)+\phi_{\mathbb{C}}(v_1,y\circ w_1)=0$$ and $$\phi_{\mathbb{C}}(y \circ v'_1,w'_1)+\phi_{\mathbb{C}}(v'_1,y\circ w'_1)=0$$ But from isomorphism (\ref{sl2explicitiso}) we also have that $\overline{y}=y$. This combined with the explicit construction (\ref{construct}) leads to the last relation in the lemma.   
\end{proof}
\subsection{Proof of Theorem \ref{deg4theorem}}\label{2011section}
\begin{proposition}\label{deg4nosuchabelian4fold}
There does not exist a simple abelian fourfold $A$ with $\mathrm{End}^{\circ}(A)\cong K$ and $K$ is a quartic CM-field such that $\liegdrbssmath(A) \otimes_{\mathbb{Q}} \qbar \cong \mathrm{sl}(2)$. In other words, Case (\ref{case1}) from Lemma \ref{deg4repclassification} does not occur.
\end{proposition}
We first write down in more detail the relation between the Hodge decomposition and the eigenspace decomposition. Recall that the set of multiplicities associated with $\mathrm{Hom}(K,\mathbb{C})$ is $\{2,0,1,1\}$ by Lemma \ref{deg4keylemma}. Without loss of generality, we let $m_{\sigma}=2, m_{\overline{\sigma}}=0,m_{\tau}=1,m_{\overline{\tau}}=1$. Therefore tensoring the eigenspace decomposition with $\mathbb{C}$ and by Definition \ref{multiplicitydefn} we have that \begin{equation}\label{deg4hodgedecomp}
    V^{1,0}=V_{\sigma}^{1,0}\oplus V_{\overline{\sigma}}^{1,0}\oplus V_{\tau}^{1,0} \oplus V_{\overline{\tau}}^{1,0}
\end{equation} with $\mathrm{dim}(V_{\sigma}^{1,0})=2,\mathrm{dim}(V_{\overline{\sigma}}^{1,0})=0, \mathrm{dim}(V_{\tau}^{1,0})=1, \mathrm{dim}(V_{\overline{\tau}}^{1,0})=1$ and 
\begin{equation}\label{deg4hodgedecomp'}
    V^{0,1}=V_{\sigma}^{0,1}\oplus V_{\overline{\sigma}}^{0,1}\oplus V_{\tau}^{0,1} \oplus V_{\overline{\tau}}^{0,1}\end{equation} with $\mathrm{dim}(V_{\sigma}^{0,1})=0,\mathrm{dim}(V_{\overline{\sigma}}^{0,1})=2, \mathrm{dim}(V_{\tau}^{0,1})=1, \mathrm{dim}(V_{\overline{\tau}}^{0,1})=1$.
Moreover, formula (\ref{deg4hodgedecomp}) and formula (\ref{deg4hodgedecomp'}) satisfy the following orthogonal property. 

\begin{lemma}
With respect to the $\mathbb{C}$-linear extension of the polarization form $\phi$ on $\betti \otimes \mathbb{C}$, we have that $V_{\overline{\tau}}^{0,1}=(V_{\tau}^{0,1})^{\perp}\cap V_{\overline{\tau}}$ and $V_{\overline{\tau}}^{1,0}=(V_{\tau}^{1,0})^{\perp}\cap V_{\overline{\tau}}$.    
\end{lemma}
\begin{proof}
Note that $\phi_{\mathbb{C}}$ preserves the Hodge structure, therefore for any $v,v'\in V^{1,0}$, we have that $\phi_{\mathbb{C}}(v,v')=0$. Thus $V_{\overline{\tau}}^{1,0}\subset(V_{\tau}^{1,0})^{\perp}\cap V_{\overline{\tau}}$. By Lemma \ref{phicompatiblewithE}, the subspaces $V_{\sigma}$, $V_{\sigmabar}$, $V_{\tau}$ are all orthogonal to $V_{\tau}$. But $\phi_{\mathbb{C}}$ is a non-degenerate form and $V_{\overline{\tau}}^{1,0}$ is a one-dimensional subspace of the two-dimensional vector space $V_{\overline{\tau}}$, therefore we obtain the equality $V_{\overline{\tau}}^{1,0}=(V_{\tau}^{1,0})^{\perp}\cap V_{\overline{\tau}}$.
\end{proof}
We are now ready to prove Proposition \ref{deg4nosuchabelian4fold}.
\begin{proof}[Proof of Proposition \ref{deg4nosuchabelian4fold}]
 
Assume such $A$ in the statement of the proposition exists. First suppose that $\liegdrbssmath(A)\cong\mathbb{Q}(a,\lambda)^{\circ}$ with $a<0$. Then there exists a weight 1 Hodge structure on $V$ satisfying decomposition (\ref{deg4hodgedecomp}) and (\ref{deg4hodgedecomp'}), polarized by a bilinear form $\phi$ satisfying the conditions from Lemma \ref{deg4Polarisation properties k imaginary}. Fix a square root of $-1$ and denote it by $\sqrt{-1}$. In particular, for any $v \in V^{1,0}-\{0\}$, the following inequality holds: \begin{equation}\label{deg4pos}
\phi_{\mathbb{C}}(v+\overline{v},\sqrt{-1}v-\sqrt{-1}\overline{v})=-2\sqrt{-1}\phi_{\mathbb{C}}(v,\overline{v})>0    
\end{equation} for example see Definition 3.1.6 of \cite{huybrechts2016lectures}.
Fix a basis for $V_{\sigma}$,$V_{\overline{\sigma}}$, $V_{\tau}$ and $V_{\overline{\tau}}$ as in Lemma \ref{deg4Polarisation properties k imaginary}. Recall that we have assumed in decomposition (\ref{deg4hodgedecomp}) that $\mathrm{dim}(V_{\sigma}^{1,0})=2$. Then for any $v_{\sigma}^{1,0}=x_{-1}v_{-1}+x_{1}v_{1}\in V_{\sigma}^{1,0}-\{0\}=V_{\sigma}\otimes_{\qbar}\mathbb{C}-\{0\}$ with $x_{-1},x_{1}\in\mathbb{C}$, the positivity condition of the polarization form is equivalent to
\begin{equation*}
-2\sqrt{-1}\phi_{\mathbb{C}}(x_{-1}v_{-1}+x_{1}v_{1},\overline{x_{-1}}w_1+\overline{x_1}w_{-1})>0
\end{equation*} for any $(x_{-1},x_{1})\in\mathbb{C}^2-\{0\}$. We now use linear relations between values of the polarization form from Lemma \ref{deg4Polarisation properties k imaginary} to simplify this expression.
 Denote $\sqrt{-1}\phi_{\mathbb{C}}(v_1,w_{-1})$ by $M$. This simplifies to
\begin{equation}
    -2M(x_1\overline{x_{1}}-\lambda x_{-1}\overline{x_{-1}})>0
\end{equation} for any $(x_1,x_{-1})\in \mathbb{C}^{2}-\{0\}$. Thus we deduce that $\lambda \in \mathbb{Q}_{<0}$.

Condition (\ref{deg4pos}) also needs to hold for elements in $V_{\tau}^{1,0}$ and $V_{\overline{\tau}}^{1,0}$. By decomposition (\ref{deg4hodgedecomp}), $\mathrm{dim}(V_{\tau}^{1,0})=1$. We now choose a $\mathbb{C}$-basis for $V_{\tau}^{1,0}$ and denote it by $v_{\tau}^{1,0}=y_{-1}v'_{-1}+y_{1}v'_{1}\in V_{\tau}\otimes_{\qbar}\mathbb{C}-\{0\}$ where $y_{-1},y_1\in\mathbb{C}$.
Then condition (\ref{deg4pos}) for elements in $V_{\tau}^{1,0}-\{0\}$ is equivalent to
\begin{equation*}
-2\sqrt{-1}\phi_{\mathbb{C}}(y_{-1}v'_{-1}+y_{1}v'_{1},\overline{y_{-1}}w'_1+\overline{y_1}w'_{-1})>0
\end{equation*}
Denoting $\sqrt{-1}\phi_{\mathbb{C}}(v'_1,w'_{-1})$ by $N$ and using relations from Lemma \ref{deg4Polarisation properties k imaginary}, this simplifies to \begin{equation}\label{deg4vtau10}
    -2N(y_1\overline{y_{1}}-\lambda y_{-1}\overline{y_{-1}})>0
\end{equation} 
Moreover, for any $y'_{-1}w'_{-1}+y'_{1}w'_{1}\in V_{\overline{\tau}}^{1,0}-\{0\}$, condition (\ref{deg4pos}) is equivalent to
\begin{equation*}
-2\sqrt{-1}\phi_{\mathbb{C}}(y'_{-1}w'_{-1}+y'_{1}w'_{1},\overline{y'_{-1}}v'_1+\overline{y'_1}v'_{-1})>0
\end{equation*}
which simplifies to 
\begin{equation}\label{vbartau10}
2N(-\lambda y'_1\overline{y'_{1}}+y'_{-1}\overline{y'_{-1}})>0
\end{equation}
Since $\lambda<0$, inequalities (\ref{deg4vtau10}) and (\ref{vbartau10}) give a clear contradiction.

Next suppose we are in the setup of Lemma \ref{deg4Polarisation properties k totallyreal}, i.e. $\liegdrbssmath(A)\cong\mathbb{Q}(a,\lambda)^{\circ}$ with $a>0$. Then for any non-zero $v_{\sigma}^{1,0}=x_{-1}v_{-1}+x_{1}v_{1}\in V_{\sigma}^{1,0}=V_{\sigma}\otimes_{\qbar}\mathbb{C}$ with $x_{-1},x_{1}\in\mathbb{C}$, condition (\ref{deg4pos}) is equivalent to
\begin{equation*}
-2\sqrt{-1}\phi_{\mathbb{C}}(x_{-1}v_{-1}+x_{1}v_{1},\overline{x_{-1}}w_{-1}+\overline{x_1}w_{1})>0
\end{equation*}
Using the last relations from Lemma \ref{deg4Polarisation properties k totallyreal}, and denoting $\sqrt{-1}\phi_{\mathbb{C}}(v_1,w_{-1})$ by $M'$, this simplifies to
\begin{equation}
    -2M'(x_{1}\overline{x_{-1}}-x_{-1}\overline{x_{1}})>0
\end{equation} for any $(x_1,x_{-1})\in \mathbb{C}^2-\{0\}$. But when $x_{-1},x_1\in\mathbb{R}^2-\{0\}$, this is equal to zero, which gives the desired contradiction.

\end{proof}

\begin{proof}[Proof of Theorem \ref{deg4theorem}]
 From Proposition \ref{twocentressame}, we obtained that $\mathrm{Z}(\mathrm{mt}(A))=\mathrm{Z}(\liegdrbmath(A))$. Now Proposition \ref{deg4nosuchabelian4fold} has ruled out Case (\ref{case1}) from Proposition \ref{deg4repclassification}. Therefore we can conclude that $\liegdrbssmath(A) \otimes_{\mathbb{Q}} \qbar=\mathrm{sl}(2)\times\mathrm{sl}(2)$. Thus by (7.5) of \cite{mz-4-folds}, one can see that $\liegdrbssmath(A) \otimes_{\mathbb{Q}} \qbar=\mathrm{mt}(A)^{\mathrm{ss}} \otimes_{\mathbb{Q}} \qbar$. Recall we have the inclusion $\liegdrbmath(A) \xhookrightarrow{} \mathrm{mt}(A)$, we then obtain that $\liegdrbmath(A)=\mathrm{mt}(A)$. Because $\mathrm{MT}(A)$ is a connected algebraic group, we can conclude that $\gdrbmath(A)=\mathrm{MT}(A)$.
\end{proof}

\section{De Rham-Betti Lie Algebras of Simple Anti-Weil Type Abelian Fourfolds}\label{antiweilsection}
In this section, unless otherwise specified, $A$ is a simple abelian fourfold of anti-Weil type (see Definition \ref{defnweiltype}) defined over $\qbar$. We will be studying possible candidates for the isomorphism class of the representation $$\liegdrbssmath(A) \rightarrow \mathrm{gl}(\betti)$$ based on the preliminary classification results from Proposition \ref{reprclassification}. In the first section, we will show in Proposition \ref{nosl2qform} that $\liegdrbssmath(A)$ cannot be isomorphic to a $\qbar/\mathbb{Q}$-form of $\mathrm{sl}(2)$. The proof method is almost the same as Section \ref{deg4section}, albeit the computation is different.

The focal point of this section, however, is Section \ref{noideasubsect}. We show that using the method for Proposition \ref{nosl2qform}, we cannot rule out $\liegdrbssmath(A)$ being isomorphic to certain $\qbar/\mathbb{Q}$ form of $\mathrm{sl}(2)\times\mathrm{sl}(2)$. The precise obstruction is pointed out in Proposition \ref{existenceofsl2timessl2}.

\subsection{\texorpdfstring{$\liegdrbssmath(A)$}{liegdrbssmath(A)} not Isomorphic to any \texorpdfstring{$\qbar/\mathbb{Q}$}{Qbar/Q}-Form of \texorpdfstring{$\mathrm{sl}(2)$}{sl(2)}}

The goal of this section is to prove the following result, the method of which is the same as the proof for Theorem \ref{deg4theorem}.
\begin{proposition} \label{nosl2qform}
 There does not exist a simple abelian fourfold $A$ of anti-Weil type such that $\liegdrbssmath(A)$ is isomorphic to a $\qbar/\mathbb{Q}$ form of $\mathrm{sl}(2)$.
\end{proposition}
We will be using notations from Setup \ref{sectiondeg2} and Proposition \ref{reprclassification}. We denote $\betti$ by $V$. Then recall from Setup \ref{sectiondeg2} that, with respect to the morphism of $\mathbb{Q}$-algebras $K\xhookrightarrow{}\mathrm{End}(V)$, we have the eigenspace decomposition $$V_{\qbar}=V_{\sigma}\oplus V_{\sigmabar}$$ Suppose $\liegdrbssmath(A) \otimes_{\mathbb{Q}} \qbar$ is isomorphic to $\mathrm{sl}(2)$. Then by Proposition \ref{reprclassification} the representation $$\eta_{\qbar}: \liegdrbssmath(A) \otimes \qbar \rightarrow \mathrm{gl}(V\otimes \qbar)$$ splits as $\rho_{\sigma}: \mathrm{sl}(2) \rightarrow \mathrm{gl}(V_{\sigma})$ and  $\rho_{\sigmabar}:\mathrm{sl}(2) \rightarrow \mathrm{gl}(V_{\sigmabar})$, both of which are isomorphic to $V(3)$.

We first give a simple lemma arising from the above construction, which is a variant of Lemma \ref{conjugate and weight vector V(1)}. Given a $\qbar/\mathbb{Q}$ form of $\mathrm{sl}(2)$, we will be using notations and conventions from Lemma \ref{sl2qformlemma} and the fixed isomorphism (\ref{sl2explicitiso}): $\mathbb{Q}(a,\lambda)^{\circ} \otimes \qbar\cong \mathrm{sl}(2,\qbar)$.
\begin{lemma} \label{conjugate and weight vector}
Suppose we have a Lie algebra representation $$\eta: \mathbb{Q}(a,\lambda)^{\circ} \rightarrow \mathrm{gl}(V)$$ defined over $\mathbb{Q}$ which upon tensoring with $\qbar$ splits in the form of Case (\ref{sl2degk2}) from Proposition \ref{reprclassification}, i.e. $\eta_{\qbar}=\rho_{\sigma}\oplus\rho_{\sigmabar}$ and both $\rho_{\sigma}$ and $\rho_{\sigmabar}$ are isomorphic to $V(3)$. Let $v_{m} \in V_{\sigma}$(or $w_{m} \in V_{\sigmabar}$) be an eigenvector of $\rho_{\sigma}(h)$(or $\rho_{\sigmabar}(h)$) of weight $3-2m$ ($m \in \{0,1,2,3\}$). If $a<0$, then $\overline{v_{m}} \in V_{\sigmabar}$(or $\overline{w_{m}} \in V_{\sigma}$) remains an eigenvector of $\rho_{\sigmabar}(h)$(or $\rho_{\sigma}(h)$) but has weight $2m-3$. If $a>0$, then $\overline{v_{m}} \in V_{\sigmabar}$(or $\overline{w_{m}} \in V_{\sigma}$) remains an eigenvector of $\rho_{\sigmabar}(h)$(or $\rho_{\sigma}(h)$) with the same weight $3-2m$.
\end{lemma}

\begin{proof}
Using isomorphism (\ref{sl2explicitiso}), we have that $h=i\otimes\frac{1}{\sqrt{a}}$ and therefore $\rho_{\sigma}(h)=\eta(i)\otimes \frac{1}{\sqrt{a}}$. Hence if $v_{m} \in V_{\sigma}$ satisfies 
\begin{equation*}
  \rho_{\sigma}(h) \circ v_{m}=(3-2m)v_{m}  
\end{equation*} then
\begin{equation*}
   (\eta(i)\otimes \frac{1}{\sqrt{a}})\circ v_{m}=(3-2m)v_{m}
\end{equation*} Now suppose $a<0$. Notice that $\eta(i)$ has $\mathbb{Q}$-coefficients with respect to a $\mathbb{Q}$-basis of $V$. Therefore upon applying complex conjugation to the above equation, $\eta(i)$ remains unchanged. Also note that $\overline{v_{m}}\in V_{\sigmabar}$. Therefore we have:
\begin{equation*}
   \overline{\rho_{\sigma}(h) \circ v_{m}}=(\eta(i) \otimes \overline{\frac{1}{\sqrt{a}}}) \circ \overline{v_{m}}=-\eta(i) \otimes \frac{1}{\sqrt{a}} \circ \overline{v_{m}}=-\rho_{\sigmabar}(h) \circ \overline{v_{m}}=(3-2m)\overline{v_{m}}
\end{equation*}
Hence $\overline{v_{m}}$ has weight $2m-3$ and the same holds for $w_{m}$s in $V_{\sigmabar}$.
Next suppose that $a>0$. Then applying complex conjugation, we get:
\begin{equation*}
   \overline{\rho_{\sigma}(h) \circ v_{m}}=\eta(i) \otimes \overline{\frac{1}{\sqrt{a}}} \circ \overline{v_{m}}=\eta(i) \otimes \frac{1}{\sqrt{a}} \circ \overline{v_{m}}=\rho_{\sigmabar}(h) \circ \overline{v_{m}}=(3-2m)\overline{v_{m}}
\end{equation*} Hence the second statement of the lemma follows.
\end{proof}
In the rest of this section, we further extend scalars from $\qbar$ to $\mathbb{C}$ and denote $V_{\sigma}\otimes_{\qbar}\mathbb{C}$ by $V_{\sigma}$ and $V_{\sigmabar}\otimes_{\qbar}\mathbb{C}$ by $V_{\sigmabar}$ as well. The following two lemmas are variants of Lemma \ref{deg4Polarisation properties k imaginary} and Lemma \ref{deg4Polarisation properties k totallyreal}.
\begin{lemma} \label{Polarisation properties k imaginary} 
Suppose there exists an abelian fourfold $A$ satisfying the conditions in Proposition \ref{nosl2qform}, i.e. $\liegdrbssmath(A) \cong \mathbb{Q}(a,\lambda)^{\circ}$. We furthermore assume that $a<0$. Fix the $\mathrm{sl}(2)$ triple $(h,x,y)$ of $\liegdrbssmath(A) \otimes \qbar \cong \mathrm{sl}(2)$ as in isomorphism (\ref{sl2explicitiso}). Then for any polarization form: $\phi:V \times V \rightarrow \mathbb{Q}$ we can find a basis $\{v_0,v_1,v_2,v_3\}$ for $V_{\sigma}$ and a basis for $\{w_0,w_1,w_2,w_3\}$ for $V_{\sigmabar}$ such that the following conditions are satisfied. \begin{enumerate}
    \item\label{condition1} The $v_{m}$s and $w_{m}$s are of weights $3-2m$ i.e. are eigenvectors of the $\rho_{\sigma}(h)$ or $\rho_{\sigmabar}(h)$ with eigenvalues $3-2m$
    \item\label{condition2} $w_{3-m}=\overline{v_{m}}$ 
    \item\label{condition3} $\phi_{\mathbb{C}}(v_{m},w_{m'})=0$ if $m+m'\neq 3$. Furthermore $V_{\sigma}$ and $V_{\sigmabar}$ are totally isotropic subspaces for $\phi_{\mathbb{C}}$.
    \item\label{condition4} The values of $\phi_{\mathbb{C}}(v_{m},w_{m'})$ satisfy the following relations when $m+m'=3$ \begin{align*} \label{phivalueforimaginaryk}
 \phi_{\mathbb{C}}(v_1,w_2)=-3\lambda\phi_{\mathbb{C}}(v_0,w_3)\\
 \phi_{\mathbb{C}}(v_2,w_1)=-4\lambda\phi_{\mathbb{C}}(v_1,w_2)\\
 \phi_{\mathbb{C}}(v_3,w_0)=-3\lambda\phi_{\mathbb{C}}(v_2,w_1)
\end{align*}
\end{enumerate} 
\end{lemma}
\begin{proof}
By Lemma \ref{rosaticmcoincides}, for any polarization form $\phi$ on $\betti$ we have that the Rosati involution on $K$ induced by $\phi$ coincides with the complex conjugation on $K \xhookrightarrow{}\mathrm{End}(V)$. Therefore Lemma \ref{phicompatiblewithE} can be handily applied. We will explicit write down a basis $\{v_0,v_1,v_2,v_3\}$ for $V_{\sigma}$ and a basis $\{w_0,w_1,w_2,w_3\}$ for $V_{\sigmabar}$ and then verify that they satisfy the conditions in the lemma. 

We start by fixing an arbitrary weight 3 vector $v_0$ in $V_{\sigma}$. Then let
\begin{align*}
&v_1=\rho_{\sigma}(y) \circ v_0, v_2=\rho_{\sigma}(y) \circ v_1, v_3=\rho_{\sigma}(y) \circ v_2\\
&w_0=\overline{v_3}, w_1=\overline{v_2}, w_2=\overline{v_1}, w_3=\overline{v_0}
\end{align*}  
Then $v_{m}$s are of weights $3-2m$ by the Lie algebra relation $[h,y]=-2y$ and $w_{m}$s are of weights $3-2m$ by virtue of Lemma \ref{conjugate and weight vector}. Hence condition (\ref{condition1}) and condition (\ref{condition2}) from above are satisfied by construction.

As for condition (\ref{condition3}) and (\ref{condition4}), note that $\phi$ being a polarization form, is preserved by the image of $\eta: \liegdrbssmath(A) \rightarrow \mathrm{gl}(V)$ infinitesimally. Thus for any $v,w \in V\otimes \mathbb{C}$ and for any $l \in \liegdrbssmath(A) \otimes \mathbb{C} \cong \mathrm{sl}(2,\mathbb{C})$, the following relation holds: 
\begin{equation}\label{polarisationpreservinggdrb}
    \phi_{\mathbb{C}}(\eta_{\mathbb{C}}(l)\circ v,w)+\phi_{\mathbb{C}}(v,\eta_{\mathbb{C}}(l)\circ w)=0
\end{equation} Denote the restriction of $\eta_{\mathbb{C}}$ on $V_{\sigma}$ and $V_{\sigmabar}$ by $\rho_{\sigma}$ and $\rho_{\sigmabar}$. For any weight vector $v_{m'}$ in $V_{\sigma}$ of weight $3-2m'$  and $w_{m}$ in $V_{\sigmabar}$ of weight $3-2m$, if $m'+m \neq 3$, we have $$\phi_{\mathbb{C}}(
\rho_{\sigma}(h)\circ v_{m'},w_{m})+\phi_{\mathbb{C}}(v_{m'},\rho_{\sigmabar}(h)\circ w_{m})=\alpha\phi_{\mathbb{C}}(v_{m'},w_{m})=0$$ where $\alpha$ is a non-zero integer. Hence $\psi_{\mathbb{C}}(v_{m'},w_{m})=0$. The subspace $V_{\sigma}$ and $V_{\sigmabar}$ are isotropic by Lemma \ref{phicompatiblewithE}. Hence condition (\ref{condition3}) is satisfied. 

As for condition (\ref{condition4}), by Corollary \ref{relations between y and x} we have $\lambda\overline{x}=y$. Using the Lie bracket  relations among $h,x,y$ we can also deduce the following relations:
\begin{align*}
  \rho_{\sigma}(y)\circ v_0=v_1; \rho_{\sigmabar}(y)\circ w_2=3\lambda w_3\\
  \rho_{\sigma}(y)\circ v_1=v_2;\rho_{\sigmabar}(y)\circ w_1=4\lambda w_2\\
  \rho_{\sigma}(y)\circ v_2=v_3; \rho_{\sigmabar}(y)\circ w_0=3\lambda w_1\\
\end{align*} 
We now swap the above relations in the following formula $$\phi_{\mathbb{C}}(\rho_{\sigma}(y)\circ v_{m},w_{2-m})+\phi_{\mathbb{C}}(v_{m'},\rho_{\sigmabar}(y)\circ w_{2-m'})=0$$ where $m,m' \in \{0,1,2\}$. And we obtain that $\phi_{\mathbb{C}}(v_{m},w_{3-m})$s satisfy the last relations in the lemma: 
\begin{align*} \label{phivalueforimaginaryk}
 \phi_{\mathbb{C}}(v_1,w_2)+3\lambda\phi_{\mathbb{C}}(v_0,w_3)=0\\
 \phi_{\mathbb{C}}(v_2,w_1)+4\lambda\phi_{\mathbb{C}}(v_1,w_2)=0\\
 \phi_{\mathbb{C}}(v_3,w_0)+3\lambda\phi_{\mathbb{C}}(v_2,w_1)=0
\end{align*}
\end{proof}

The following lemma is similar to the above but we assume $a>0$.
\begin{lemma} \label{polarization values k real}
Suppose there exists an abelian fourfold $A$ satisfying the conditions in Proposition \ref{nosl2qform}, i.e. $\liegdrbssmath(A) \cong \mathbb{Q}(a,\lambda)^{\circ}$. We assume that $a>0$. Fix the $\mathrm{sl}(2)$ triple $(h,x,y)$ of $\liegdrbssmath(A) \otimes \qbar \cong \mathrm{sl}(2)$ as in isomorphism (\ref{sl2explicitiso}). Then for any polarization form: $\phi:V \times V \rightarrow \mathbb{Q}$ we can find a basis $\{v_0,v_1,v_2,v_3\}$ for $V_{\sigma}$ and $\{w_0,w_1,w_2,w_3\}$ for $V_{\sigmabar}$ such that the following conditions are satisfied. \begin{enumerate}
    \item The $v_{m}$s and $w_{m}$s are of weights $3-2m$
    \item $w_{m}=\overline{v_{m}}$ 
    \item $\phi_{\mathbb{C}}(v_{m},w_{m'})=0$ if $m+m'\neq 3$ for $m,m' \in \{0,1,2,3\}$. Also $V_{\sigma}$ and $V_{\sigmabar}$ are totally isotropic subspaces for $\phi_{\mathbb{C}}$.
    \item The values of $\phi_{\mathbb{C}}(v_{m},w_{m'})$ satisfy the following relations when $m+m'=3$ 
    \begin{align*} 
   \phi_{\mathbb{C}}(v_1,w_2)+\phi_{\mathbb{C}}(v_0,w_3)=0\\
   \phi_{\mathbb{C}}(v_2,w_1)+\phi_{\mathbb{C}}(v_1,w_2)=0\\
   \phi_{\mathbb{C}}(v_3,w_0)+\phi_{\mathbb{C}}(v_2,w_1)=0
\end{align*}
\end{enumerate} 
\end{lemma}

\begin{proof}
The construction and the proof will be almost identical to the above Lemma \ref{Polarisation properties k imaginary}, and we will merely point out the difference. We fix a nonzero weight 3 vector $v_0$ in $V_{\sigma}$ and make the following choice of basis for $V_{\sigma}$ and $V_{\sigmabar}$: \begin{align*}
&v_1=\rho_{\sigma}(y) \circ v_0, v_2=\rho_{\sigma}(y) \circ v_1, v_3=\rho_{\sigma}(y) \circ v_2\\
&w_0=\overline{v_0}, w_1=\overline{v_1}, w_2=\overline{v_2}, w_3=\overline{v_3}
\end{align*} Then by Lemma \ref{conjugate and weight vector}, and an identical argument in the proof of Lemma \ref{Polarisation properties k imaginary}, the first 3 properties in the above Lemma are automatically satisfied. As for the last one: note that because $a>0$, according to formula (\ref{sl2explicitiso}) we have $\overline{y}=y$. Therefore we also have $w_{m+1}=\rho_{\sigmabar}(y)\circ w_{m}$ for $m \in \{0,1,2\}$. Then combining these relations and the condition $$\mu_{\mathbb{C}}: \phi_{\mathbb{C}}(\rho_{\sigma}(y)\circ v_{m},w_{2-m})+\phi_{\mathbb{C}}(v_{m'},\rho_{\sigmabar}(y)\circ w_{2-m'})=0$$ we obtain the last relations in the lemma.
\end{proof}
\begin{remark}
To obtain the linear relations among the values of the polarization form when evaluated at weight vectors from the above two lemmas, it is crucial that the representation $\eta$ is defined over $\mathbb{Q}$.
\end{remark}
The next lemma gives more detailed information about the Hodge structure of an anti-Weil type abelian fourfold.
\begin{lemma}\label{hodgeandphiandK}
Recall that the Hodge structure of a simple anti-Weil type abelian fourfold $A$ can be written as follows $$V^{1,0}=V_{\sigma}^{1,0}\oplus V_{\sigmabar}^{1,0},V^{0,1}=V_{\sigma}^{0,1}\oplus V_{\sigmabar}^{0,1}$$ In addition it satisfies the following orthogonal property with respect to any polarization form $\phi$ on $V$: $$V_{\sigmabar}^{1,0}=(V_{\sigma}^{1,0})^{\perp}\cap V_{\sigmabar}, V_{\sigmabar}^{0,1}=(V_{\sigma}^{0,1})^{\perp} \cap V_{\sigmabar}$$ 
\end{lemma}
\begin{proof}
Assume that $m_{\sigma}=1$ i.e. $\mathrm{dim}(V_{\sigma}^{1,0})=1$. Then because $A$ is of anti-Weil type, we have that dim$V_{\sigmabar}^{1,0}$=3. By Lemma \ref{phicompatiblewithE}, we have that dim$((V_{\sigma}^{1,0})^{\perp}\cap V_{\sigmabar})=3$. Since $\phi$ is a morphism of Hodge structures, we have that $V_{\sigmabar}^{1,0} \subset (V_{\sigma}^{1,0})^{\perp}\cap V_{\sigmabar}$. Therefore we actually have the equality $V_{\sigmabar}^{1,0}=(V_{\sigma}^{1,0})^{\perp}\cap V_{\sigmabar}$. The other statement in the lemma follows similarly.
\end{proof}
\begin{proof}[Proof of Proposition \ref{nosl2qform}]
\label{sl2proof}
Suppose that there exists an anti-Weil type abelian fourfold $A$ satisfying the condition in Proposition \ref{nosl2qform} i.e. $\liegdrbssmath(A)\cong\mathbb{Q}(a,\lambda)^{\circ}$. Then depending on $a<0$ or $a>0$, the values of any polarization form on $\betti$ have to satisfy relations specified in Lemma \ref{Polarisation properties k imaginary} or Lemma \ref{polarization values k real}. The strategy of the proof of Proposition \ref{nosl2qform} is to show that the positivity condition of the polarization form actually fails, which will give the desired contradiction.

First we suppose that $a<0$. Because $A$ is an anti-Weil type abelian fourfold, we have $$V^{1,0}=V_{\sigma}^{1,0} \oplus V_{\sigmabar}^{1,0}$$ and we may assume that $\mathrm{dim}(V_{\sigma}^{1,0})=1$ and $\mathrm{dim}(V_{\sigmabar}^{1,0})=3$. Fix $\phi$ a polarization form on $V$. Since $\phi$ polarizes the weight one Hodge structure, fixing a square root $\sqrt{-1}$ of $-1$, for any $v \in V^{1,0}-\{0\}$ we have the following positivity condition \begin{equation}\label{positivity}
\phi_{\mathbb{C}}(v+\overline{v},\sqrt{-1}v-\sqrt{-1}\overline{v})=-\sqrt{-1}\phi_{\mathbb{C}}(v,\overline{v})+\sqrt{-1}\phi_{\mathbb{C}}(\overline{v},v)=-2\sqrt{-1}\phi_{\mathbb{C}}(v,\overline{v})>0\end{equation} 
We will show that the positivity condition (\ref{positivity}) cannot simultaneously hold for every possible $v_{\sigma}^{1,0}\in V_{\sigma}^{1,0}$ and $w\in V_{\sigmabar}^{1,0}-\{0\}$.
We fix a basis $\{v_0,v_1,v_2,v_3\}$ for $V_{\sigma}$ and a basis $\{w_0,w_1,w_2,w_3\}$ for $V_{\sigmabar}$ as in Lemma \ref{Polarisation properties k imaginary}.
Then for any 
$v_{\sigma}^{1,0}=x_0v_0+x_1v_1+x_2v_2+x_3v_3 \in V_{\sigma}^{1,0}-\{0\}$, where $x_0,x_1,x_2,x_3\in \mathbb{C}$ we have that 
\begin{align*}
&-2\sqrt{-1}\phi_{\mathbb{C}}(x_0v_0+x_1v_1+x_2v_2+x_3v_3,\overline{x_0}\overline{v_0}+\overline{x_1}\overline{v_1}+\overline{x_2}\overline{v_2}+
\overline{x_3}\overline{v_3})>0
\end{align*} 
By Lemma \ref{Polarisation properties k imaginary} and Lemma \ref{hodgeandphiandK}, the vector space $V_{\sigmabar}^{1,0}$ is equal to \begin{equation}\label{vtau10equ1}
\{y_{0}w_{0}+y_{1}w_{1}+y_{2}w_{2}+y_{3}w_{3} \in V_{\sigmabar}|x_0y_3-3\lambda x_1y_2+12\lambda^2 x_2y_1-36\lambda^3 x_3y_0=0\}    
\end{equation} 
Now we are going to use Lemma \ref{Polarisation properties k imaginary} to simplify the positivity condition (\ref{positivity}) for both $v_{\sigma}^{1,0}$ and $w\in V_{\sigmabar}^{1,0}$.

The positivity condition for $v_{\sigma}^{1,0}-\{0\}$ can be expressed as 
\begin{equation*}
-2\sqrt{-1}(\phi_{\mathbb{C}}(v_0,w_3)x_0\overline{x_0}+\phi_{\mathbb{C}}(v_1,w_2)x_1\overline{x_1}+\phi_{\mathbb{C}}(v_2,w_1)x_2\overline{x_2}+\phi_{\mathbb{C}}(v_3,w_0)x_3\overline{x_3})>0
\end{equation*}
While inequality (\ref{positivity}) for any $w=\Sigma y_{i}w_{i} \in V_{\sigmabar}^{1,0}-\{0\}$ is expressed as:
\begin{align*}
&-2\sqrt{-1}\phi_{\mathbb{C}}(y_0w_0+y_1w_1+y_2w_2+y_3w_3,\overline{y_0}\overline{w_0}+\overline{y_1}\overline{w_1}+\overline{y_2}\overline{w_2}+
\overline{y_3}\overline{w_3})\\
&=-2\sqrt{-1}(\phi_{\mathbb{C}}(w_0,v_3)y_0\overline{y_0}+\phi_{\mathbb{C}}(w_1,v_2)y_1\overline{y_1}+\phi_{\mathbb{C}}(w_2,v_1)y_2\overline{y_2}+\phi_{\mathbb{C}}(w_3,v_0)y_3\overline{y_3})>0
\end{align*}

Denoting $-2\sqrt{-1}\phi_{\mathbb{C}}(v_0,w_3)$ as $M$, then according to Lemma \ref{Polarisation properties k imaginary} for $v_{\sigma}^{1,0}$ the positivity condition reads:
\begin{equation}\label{postivityforx1}
 M(x_0\overline{x_0}-3\lambda x_1\overline{x_1}+12\lambda^2  x_2\overline{x_2}-36\lambda^3x_3\overline{x_3})>0   
\end{equation}

In particular we deduce that $M \in \mathbb{R}$.
As for $w=\Sigma y_{i}w_{i} \in V_{\sigmabar}^{1,0}-\{0\}$ the positivity condition reads:
\begin{equation}\label{postivityfory1}
-M(-36\lambda^3y_0\overline{y_0}+12\lambda^2y_1\overline{y_1}-3\lambda y_2\overline{y_2}+y_3\overline{y_3})>0    
\end{equation}
  If $\lambda<0$, then because $M(x_0\overline{x_0}-3\lambda x_1\overline{x_1}+12\lambda^2  x_2\overline{x_2}-36\lambda^3x_3\overline{x_3})>0$, we deduce that $M>0$. But this is a contradiction, since then $-M(-36\lambda^3y_0\overline{y_0}+12\lambda^2y_1\overline{y_1}-3\lambda y_2\overline{y_2}+y_3\overline{y_3})<0$.

Next we investigate the case where $\lambda>0$. Letting $y_0=y_2=0$, formula (\ref{postivityfory1}) simplifies to $$-M(12\lambda^2y_1\overline{y_1}+y_3\overline{y_3})>0$$ for any $(y_1,y_3)\neq(0,0)$ satisfying that $x_0y_3+12\lambda^2 x_2y_1=0$. But this implies that $M<0$. On the other hand, letting $y_1=y_3=0$, formula (\ref{postivityfory1}) simplifies to $$-M(-36\lambda^3y_0\overline{y_0}-3\lambda y_2\overline{y_2})>0$$ for any $(y_0,y_2)\neq(0,0)$ satisfying that $3\lambda x_1y_2+36\lambda^3 x_3y_0=0$ which implies that $M>0$. And thus we have obtained the desired contradiction.

Hence we have proven there cannot exist an abelian fourfold of anti-Weil type such that the semisimple part of its de Rham-Betti Lie algebra is isomorphic to $\mathbb{Q}(a,\lambda)^{\circ}$ with $a<0$.

Now we suppose that $a>0$.
Fix a basis $\{v_0,v_1,v_2,v_3\}$ for $V_{\sigma}$ and a basis $\{w_0,w_1,w_2,w_3\}$ for $V_{\sigmabar}$ as in Lemma \ref{polarization values k real}. The positivity condition (\ref{positivity}) for elements in
$v_{\sigma}^{1,0}=\Sigma x_{i}v_{i} \in V_{\sigma}^{1,0}-\{0\}$ where $x_{i} \in \mathbb{C}$ is
\begin{align*}
&-2\sqrt{-1}\phi_{\mathbb{C}}(x_0v_0+x_1v_1+x_2v_2+x_3v_3,\overline{x_0}\overline{v_0}+\overline{x_1}\overline{v_1}+\overline{x_2}\overline{v_2}+
\overline{x_3}\overline{v_3})>0
\end{align*} Letting $M'=\sqrt{-1}\phi_{\mathbb{C}}(v_1,w_2)$ and using the relations among values of polarization form in Lemma \ref{polarization values k real}, the above equation simplifies to:
\begin{equation}\label{positivityx2}
2M'(x_0\overline{x_3}-x_1\overline{x_2}+x_2\overline{x_1}-x_3\overline{x_0})>0    
\end{equation}

In particular, the above implies that $M'$ is a non-zero totally imaginary number. Moreover, since $M' \neq 0$, we have by Lemma \ref{hodgeandphiandK} \begin{equation}\label{vtau10real}
 V_{\sigmabar}^{1,0}=\{w' \in V_{\sigmabar}|\phi_{\mathbb{C}}(v_{\sigma}^{1,0},w')=0\}=\{w'=\Sigma y_{j}w_{j}|x_0y_3-x_1y_2+x_2y_1-x_3y_0=0\}   
\end{equation} And the positivity condition for elements in $V_{\sigmabar}^{1,0}-\{0\}$ is 
\begin{align*}
&-2\sqrt{-1}\phi_{\mathbb{C}}(y_0w_0+y_1w_1+y_2w_2+y_3w_3,\overline{y_0}\overline{w_0}+\overline{y_1}\overline{w_1}+\overline{y_2}\overline{w_2}+
\overline{y_3}\overline{w_3})>0
\end{align*}
Using relations from Lemma \ref{polarization values k real}, this is equivalent to
\begin{equation}\label{positivityy2}
 2M'(y_0\overline{y_3}-y_1\overline{y_2}+y_2\overline{y_1}-y_3\overline{y_0})>0   
\end{equation}

We claim that the expression $y_0\overline{y_3}-y_1\overline{y_2}+y_2\overline{y_1}-y_3\overline{y_0}$ always admits a nontrivial zero on the set of $(y_0,y_1,y_2,y_3)$s satisfying formula (\ref{vtau10real}).
Let $y_0=y_1=0$, then $y_0\overline{y_3}-y_1\overline{y_2}+y_2\overline{y_1}-y_3\overline{y_0}=0$. Also in this case the condition (\ref{vtau10real}) is simplified to $x_0y_3-x_1y_2=0$. Hence no matter the value of $x_0$ and $x_1$, we can always find a pair $(y_2,y_3) \neq (0,0)$ satisfying this linear relation. And this gives a contradiction to the positivity condition (\ref{positivityy2}).

Hence we have proven there cannot exist an abelian fourfold of anti-Weil type such that the semisimple part of its de Rham-Betti Lie algebra is isomorphic to $\mathbb{Q}(a,\lambda)^{\circ}$ with $a>0$. This finishes the proof of Proposition \ref{nosl2qform}.
\end{proof}


\subsection{Discussion of the Remaining Cases from Proposition \ref{reprclassification}}\label{noideasubsect}

Suppose $A$ is a simple anti-Weil type abelian fourfold defined over $\qbar$. Recall that in Proposition \ref{nosl2qform}, we ruled out the hypothesis where $\liegdrbssmath(A)$ is isomorphic to a $\qbar/\mathbb{Q}$-form of $\mathrm{sl}(2)$. In this section, however, we will demonstrate that the hypothesis that $\liegdrbssmath(A)$ is isomorphic to certain $\qbar/\mathbb{Q}$-form of $\mathrm{sl}(2)\times\mathrm{sl}(2)$ does not contradict any known properties of the de Rham-Betti structure of $A$. See Proposition \ref{existenceofsl2timessl2} and Remark \ref{existenceexplained} for the precise statement. Thus the method of proving Theorem \ref{deg4theorem} and Proposition \ref{nosl2qform} does not work in this setting.
\begin{setup}\label{sl2timessl2setup}
 In this paragraph, we will focus on a particular type of $\qbar/\mathbb{Q}$-form $E(a,b)^{\circ}$ (see Definition \ref{Equaterdefn}) of $\mathrm{sl}(2)\times\mathrm{sl}(2)$ and fix an explicit isomorphism \begin{equation*}
E(a,b)^{\circ}\otimes_{\mathbb{Q}}\qbar\cong\mathrm{sl}(2)\times\mathrm{sl}(2)
 \end{equation*}
Let $E=\mathbb{Q}(\sqrt{D})$ be a degree two field extension over $\mathbb{Q}$ and $a,b\in E^{\times}$.
\begin{definition}\label{Equaterdefn}
   The quaternion algebra $E(a,b)^{\circ}$ is a four-dimensional $E$-algebra with the additive basis $\{1,i,j,k\}$ and the multiplication rule is given by $i^2=a,j^2=b,ij=-ji=k$. We denote by $E(a,b)^{\circ}$ the $E$-vector subspace spanned by $i,j,k$.
 \end{definition}
Denote $E(a,b)$ by $\mathcal{Q}$. We define a Lie bracket on $\mathcal{Q}$ by letting $[x,y]=xy-yx$. Then the $E$-vector subspace $\mathcal{Q}^{\circ}=E(a,b)^{\circ}$ is a Lie subalgebra. Also note that the multiplication by $\sqrt{D}$ induces a $\mathbb{Q}$-linear transformation on $\mathcal{Q}$, which we denote by $J$. Also we fix a square root of $a$ in $\qbar$ and denote it by $\sqrt{a}$. Then we can write down an explicit isomorphism $\mathcal{Q}^{\circ} \otimes_{\mathbb{Q}} \overline{\mathbb{Q}} \cong \mathrm{sl}(2) \times \mathrm{sl}(2)$ between $\qbar$-Lie algebras in the following way 
\begin{equation}\label{isomofsl2timessl2qform1st}
\begin{split}
    &(h,0)=\frac{J i\otimes\frac
    {1}{\sqrt{a}}+i\otimes\frac
    {\sqrt{D}}{\sqrt{a}}}{2\sqrt{D}};(0,h)=\frac{J i\otimes\frac
    {1}{\sqrt{a}}-i\otimes\frac
    {\sqrt{D}}{\sqrt{a}}}{2\sqrt{D}};\\
    &(x,0)=\frac{\frac{J}{2b}(j+\frac{1}{\sqrt
    {a}}k)+\frac{\sqrt{D}}{2b}(j+\frac{1}{\sqrt
    {a}}k)}{2\sqrt{D}};(0,x)=\frac{\frac{J}{2b}(j-\frac{1}{\sqrt
    {a}}k)-\frac{\sqrt{D}}{2b}(j-\frac{1}{\sqrt
    {a}}k)}{2\sqrt{D}};\\
    &(y,0)=\frac{\frac{J}{2}(j-\frac{1}{\sqrt
    {a}}k)+\frac{\sqrt{D}}{2}(j-\frac{1}{\sqrt
    {a}}k)}{2\sqrt{D}};(0,y)=\frac{\frac{J}{2}(j+\frac{1}{\sqrt
    {a}}k)-\frac{\sqrt{D}}{2}(j+\frac{1}{\sqrt
    {a}}k)}{2\sqrt{D}}\\
\end{split}   
\end{equation}
\begin{remark}
    When $E$ is a totally real quadratic field, one can apply the same method of proof as Proposition \ref{nosl2qform} to show that there does not exist a simple anti-Weil type abelian fourfold $A$ such that $\liegdrbssmath(A)\cong E(a,b)^{\circ}$ for any $a,b\in E^{\times}$.
\end{remark}
In this section, we will be interested in quaternion algebras of the form $E(a,1)$ where $E$ is a quadratic imaginary field and $a\in \mathbb{Q}_{<0}$.
\end{setup}

\begin{proposition}\label{existenceofsl2timessl2}
We keep the same notations as Setup \ref{sl2timessl2setup}. Suppose we are given an imaginary quadratic field $K$ and an embedding of $\mathbb{Q}$-algebras $\iota: K \xhookrightarrow{} \mathrm{End}(V)$, where $V$ is an 8-dimensional $\mathbb{Q}$-vector space. Then there exist a quadratic imaginary field $E$, a scalar $a\in \mathbb{Q}_{<0}$ and a simple anti-Weil type abelian fourfold $A$ with $\betti=V$ and $\mathrm{End}_{\mathrm{Hdg}}{V}=K$. Moreover, there exists a faithful representation of Lie algebras defined over $\mathbb{Q}$ 
    \begin{equation*}
        \mu: E(a,1)^{\circ} \rightarrow \mathrm{gl}(V)
    \end{equation*} such that the following conditions hold 
    \begin{enumerate}
\item\label{item1}  Consider the eigenspace decomposition $V_{\qbar}=V_{\sigma} \oplus V_{\sigmabar}$ induced by $\iota: K \xhookrightarrow{} \mathrm{End}(V)$ where $\mathrm{Hom}(K,\qbar)=\{\sigma,\sigmabar\}$. Then $V_{\sigma}$ and $V_{\sigmabar}$ are irreducible subrepresentations of $\mu_{\qbar}:\mathrm{sl}(2)\times\mathrm{sl}(2)\rightarrow\mathrm{gl}(V_{\qbar})$. Moreover, both $V_{\sigma}$ and $V_{\sigmabar}$ are isomorphic to $V(1) \boxtimes V(1)$.
\item $V$ is an irreducible $\mathbb{Q}$-representation under $$\tilde{\mu}: \mathrm{Z}(\mathrm{mt}(A))\oplus E(a,1)^{\circ} \rightarrow \mathrm{gl}(V)$$
\item\label{item3} Denote the representation $\mathrm{Z}(\mathrm{hdg}(A))\oplus E(a,1)^{\circ} \rightarrow \mathrm{gl}(V)$ by $\tilde{\mu}^{\mathrm{h}}$. Then at least one polarization form on $A$ is preserved by the image of $\tilde{\mu}^{\mathrm{h}}$ infinitesimally. 
\end{enumerate}  
\end{proposition}
\begin{remark}\label{existenceexplained}
Suppose the simple anti-Weil type abelian fourfold $A$ found in the above proposition is defined over $\qbar$ and denote the image of $\mu$ inside $\mathrm{gl}(V)$ from the above proposition by $\tilde{L}$. By Lemma \ref{twocentressame} we have that $\mathrm{Z}(\liegdrbmath(A))=\mathrm{Z}(\mathrm{mt}(A))$. 
Then by the discussion in Remark \ref{whyweilproofdoesnotwork} (and implicitly Proposition \ref{gmintype4}) the hypothesis that $$\liegdrbmath(A)=\mathrm{Z}(\liegdrbmath(A)) \oplus \tilde{L}$$ will not contradict any known property of the de Rham-Betti structure of the abelian fourfold $A$. In other words, we have that \begin{enumerate}
\item $\betti$ is an irreducible $\mathbb{Q}$-representation under $$\tilde{\mu}: \mathrm{Z}(\liegdrbmath(A))\oplus E(a,1)^{\circ} \rightarrow \mathrm{gl}(\betti)$$
\item The cohomological degree two invariants agree with algebraic cycle classes i.e. $$(\mathrm{End}(\mathrm{H}^1(A,\mathbb{Q}))^{\mathrm{Z}(\liegdrbmath(A)) \oplus \tilde{L}}=\mathrm{End}_{\mathrm{Hdg}}(A)$$ and $$\mathrm{H}^2(A,\mathbb{Q})^{\mathrm{Z}(\liegdrbhmath(A)) \oplus \tilde{L}}=\mathrm{span}_{\mathbb{Q}}(\{L_{\phi}\})$$ where $L_{\phi}$ is the divisor class associated with the polarization form $\phi$.
    \item None of the elements in the Weil structure $\wedge_{K}^4\mathrm{H}^1(A,\mathbb{Q})\otimes\mathbb{Q}_{\mathrm{dRB}}(2)$ is annihilated by the action of $\mathrm{Z}(\liegdrbmath(A)) \oplus \tilde{L}$.
    \item None of the elements in $\mathrm{H}^{1}(A,\mathbb{Q})^{\otimes m} \otimes \mathrm{H}^{1}(A,\mathbb{Q})^{*\otimes n} \otimes \mathbb{Q}_{\mathrm{dRB}}(i')$ is annihilated by the action of  $\mathrm{Z}(\liegdrbmath(A)) \oplus \tilde{L}$ if $m-n \neq 2i'$.
\end{enumerate} 
\end{remark}
The proof of Proposition \ref{existenceofsl2timessl2} will be divided into Lemma \ref{reprexists}, Lemma \ref{polarizationformexists}, Lemma \ref{familyexists} and Lemma \ref{simplexists}. We first demonstrate how to construct a Lie algebra representation satisfying Property (\ref{item1}) of Proposition \ref{existenceofsl2timessl2}.
\begin{lemma}\label{reprexists}
Given an imaginary quadratic field $K$ and an embedding of $\mathbb{Q}$-algebras $\iota: K \xhookrightarrow{} \mathrm{End}(V)$, where $V$ is an 8-dimensional $\mathbb{Q}$-vector space. Then there exist a quadratic imaginary field $E$, a scalar $a\in \mathbb{Q}_{<0}$ as well as a representation of $\mathbb{Q}$-Lie algebras
$$\mu:E(a,1)^{\circ} \rightarrow \mathrm{gl}(V)$$ such that $\mu_{\qbar}$ splits as in property (\ref{item1}) from Proposition \ref{existenceofsl2timessl2}. Moreover, the representation of $\mathbb{Q}$-Lie algebras $$\tilde{\mu}:=\iota+\mu: K\oplus E(a,1)^{\circ} \rightarrow \mathrm{gl}(V)$$ is irreducible. 
\end{lemma}
Before the proof of lemma, we set up the notations which will be used throughout the rest of this section.  
\begin{setup}\label{basisandgalois}
Denote $K=\mathbb{Q}(\sqrt{D'})$ where $D'\in\mathbb{Q}_{<0}$. Choose $a,D\in \mathbb{Q}_{<0}$ such that the field composite $F$ in $\qbar$ of $\mathbb{Q}(\sqrt{a}),E:=\mathbb{Q}(\sqrt{D}),K$ is a Galois extension of degree 8 with its Galois group $\mathrm{Gal}(F/\mathbb{Q})$ isomorphic to $\mathbb{Z}/(2) \times \mathbb{Z}/(2) \times \mathbb{Z}/(2)$. Note that in this case $F=\mathbb{Q}(\sqrt{a},\sqrt{D},\sqrt{D'})$. By a close inspection, the isomorphism (\ref{isomofsl2timessl2qform1st}): $E(a,1)^{\circ} \otimes_{\mathbb{Q}} \qbar\cong \mathrm{sl}(2,\qbar) \times \mathrm{sl}(2,\qbar)$ from Setup \ref{sl2timessl2setup} actually descends to $F$. We can therefore fix an isomorphism of Lie algebras $$E(a,1)^{\circ} \otimes_{\mathbb{Q}} F\cong \mathrm{sl}(2,F) \times \mathrm{sl}(2,F)$$ in the same way as formula (\ref{isomofsl2timessl2qform1st}).

We label the generators of $\mathrm{Gal}(F/\mathbb{Q})$ by $g_1,g_2,g_3$. Their actions on $F/\mathbb{Q}$ are explicitly written down in Table \ref{galoisactiononfields}. Note that because $a<0$, for any embedding $F \xhookrightarrow{} \mathbb{C}$, the image of the complex conjugation in $\mathrm{Gal}(F/\mathbb{Q})$ is equal to $(g_1,g_2,g_3)$.

\begin{table}[ht]
    \centering
    \begin{tabular}{|c|c|c|c|}
    \hline
    & $\sqrt{D'}$  & $\sqrt{D}$  & $\sqrt{a}$ \\
    \hline
    $g_1$ & $\sqrt{D'}$ & $-\sqrt{D}$ & $\sqrt{a}$\\
    \hline
    $g_2$ & $-\sqrt{D'}$  & $\sqrt{D}$  & $\sqrt{a}$\\
    \hline
    $g_3$ & $\sqrt{D'}$ & $\sqrt{D}$ & $-\sqrt{a}$ \\
    \hline
    \end{tabular}
    \caption{Action of $\mathrm{Gal}(F/\mathbb{Q})$ on $F/\mathbb{Q}$}
    \label{galoisactiononfields}
\end{table}

With respect to the embedding of $\mathbb{Q}$-algebras $\iota: K \xhookrightarrow{} \mathrm{End}(V)$, by Lemma \ref{neweigenspacetranslate}, we have the eigenspace decomposition $$V \otimes_{\mathbb{Q}} F=V_{\sigma} \oplus V_{\sigmabar}$$ where $\sigma,\sigmabar\in \mathrm{Hom}(K,F)$ and $\mathrm{dim}(V_{\sigma})=\mathrm{dim}(V_{\sigmabar})=4$. Note that $K$ is a quadratic extension over $\mathbb{Q}$, hence the above decomposition descends to $K$. We then fix a $K$-basis for $V_{\sigma}$ and denote them by $\{f_1,f_2,f_3,f_4\}$. Then $\{g_2 \circ f_1,g_2 \circ f_2,g_2 \circ f_3,g_2 \circ f_4\}$ is a $K$-basis for $V_{\sigmabar}$.
   
\end{setup}
\begin{proof}[Proof of Lemma \ref{reprexists}]
We will be using notations from Setup \ref{basisandgalois}. We start by constructing an $F$-linear representation $$\mu_{F}: E(a,1)^{\circ} \otimes_{\mathbb{Q}} F \rightarrow \mathrm{gl}(V \otimes_{\mathbb{Q}} F)$$ which splits according to the requirement specified in the lemma. Let
\begin{equation}\label{vsigmabasis}
  v_{1,1}=-f_1+\sqrt{aD}f_2; v_{-1,-1}=f_1+\sqrt{aD}f_2; v_{1,-1}=-f_3+\sqrt{a}f_4; v_{-1,1}=f_3+\sqrt{a}f_4
\end{equation} be an $F$-basis of $V_{\sigma}$. And let \begin{equation}\label{vsigmabarbasis}
  w_{1,1}=g_2\circ v_{1,1}; w_{-1,-1}=g_2\circ v_{-1,-1};
  w_{1,-1}=g_2\circ v_{1,-1}; w_{-1,1}=g_2\circ v_{-1,1} 
\end{equation} be an $F$-basis of $V_{\sigmabar}$.

Then the action of $\mu_{F}$ on $V_{\sigma}$ is defined in Table \ref{reprconstruc}. And the action of $\mu_{F}$ on $V_{\sigmabar}$ is defined in an analogous way in Table \ref{reprconstruct'}.
\begin{table}[h]
    \centering
    \begin{tabular}{|c|c|c|c|c|}
    \hline
    $\mu_{F}$  & $v_{1,1}$  & $v_{-1,-1}$  & $v_{1,-1}$ & $v_{-1,1}$  \\
    \hline
    $(h,0)$ & $v_{1,1}$ & $-v_{-1,-1}$ & $v_{1,-1}$ & $-v_{-1,1}$ \\
    \hline
    $(0,h)$ & $v_{1,1}$  & $-v_{-1,-1}$  & $-v_{1,-1}$             & $v_{-1,1}$\\
    \hline
    $(y,0)$ & $v_{-1,1}$ & $0$ & $v_{-1,-1}$  & $0$\\
    \hline
    $(0,y)$ & $v_{1,-1}$ & $0$ & $0$  & $v_{-1,-1}$\\
    \hline
    $(x,0)$ & $0$ & $v_{1,-1}$ & $0$  & $v_{1,1}$\\
    \hline
    $(0,x)$ & $0$ & $v_{-1,1}$ & $v_{1,1}$  & $0$\\
    \hline
    \end{tabular}
    \caption{Action of $\mathrm{sl}(2)\times\mathrm{sl}(2)$ on $V_{\sigma}$}
    \label{reprconstruc}
\end{table}

\begin{table}[ht]
    \centering
    \begin{tabular}{|c|c|c|c|c|}
    \hline
    $\mu_{F}$  & $w_{1,1}$  & $w_{-1,-1}$  & $w_{1,-1}$ & $w_{-1,1}$  \\
    \hline
    $(h,0)$ & $w_{1,1}$ & $-w_{-1,-1}$ & $w_{1,-1}$ & $-w_{-1,1}$ \\
    \hline
    $(0,h)$ & $w_{1,1}$  & $-w_{-1,-1}$  & $-w_{1,-1}$             & $w_{-1,1}$\\
    \hline
    $(y,0)$ & $w_{-1,1}$ & $0$ & $w_{-1,-1}$  & $0$\\
    \hline
    $(0,y)$ & $w_{1,-1}$ & $0$ & $0$  & $w_{-1,-1}$\\
    \hline
    $(x,0)$ & $0$ & $w_{1,-1}$ & $0$  & $w_{1,1}$\\
    \hline
    $(0,x)$ & $0$ & $w_{-1,1}$ & $w_{1,1}$  & $0$\\
    \hline
    \end{tabular}
    \caption{Action of $\mathrm{sl}(2)\times\mathrm{sl}(2)$ on $V_{\sigmabar}$}
    \label{reprconstruct'}
\end{table} Then one can check that $\mu_{F}$ indeed defines a morphism of Lie algebras $\mathrm{sl}(2)\times \mathrm{sl}(2) \rightarrow \mathrm{gl}(V_{\sigma}) \oplus \mathrm{gl}(V_{\sigmabar})$. Moreover, both $V_{\sigma}$ and $V_{\sigmabar}$ are isomorphic to $V(1) \boxtimes V(1)$.

Now we verify that $\mu_{F}$ descends to $\mathbb{Q}$. We first write down in Table \ref{galoisoneigenspace} the action of $\mathrm{Gal}(F/\mathbb{Q})$ on the fixed basis for $V_{\sigma}$ and $V_{\overline{\sigma}}$ explicitly.
\begin{table}[ht]
\begin{tabular}{|c|c|c|c|c|}
\hline
    & $v_{1,1}$  & $v_{-1,-1}$  & $v_{1,-1}$ & $v_{-1,1}$  \\
\hline
$g_1$ & $-v_{-1,-1}$ & $-v_{1,1}$ & $v_{1,-1}$ & $v_{-1,1}$ \\
\hline
$g_2$ & $w_{1,1}$  & $w_{-1,-1}$  & $w_{1,-1}$             & $w_{-1,1}$\\
\hline
$g_3$ & $-v_{-1,-1}$ & $-v_{1,1}$ & $-v_{-1,1}$  & $-v_{1,-1}$\\
\hline
\end{tabular}
\caption{Galois action on the eigenspace $V_{\sigma}$}
\label{galoisoneigenspace}
\end{table}

Since $\mathrm{Gal}(F/\mathbb{Q})$ is an abelian group, the Galois action on the basis $$\{w_{1,1}=g_2\circ v_{1,1},w_{-1,-1}=g_2\circ v_{-1,-1},w_{1,-1}=g_2\circ v_{1,-1},w_{-1,1}=g_2\circ v_{-1,1}\}$$ of $V_{\sigmabar}$ is determined by the Table \ref{galoisoneigenspace}.

On the other hand, using formula (\ref{isomofsl2timessl2qform1st}), the action of $g_{i} \in \mathrm{Gal}(F/\mathbb{Q})$ on the Lie algebra $E(a,1)^{\circ} \otimes_{\mathbb{Q}} F\cong \mathrm{sl}(2)\times \mathrm{sl}(2)$ is written down explicitly in Table \ref{galoisonliealgebra}.
\begin{table}[ht]
\begin{tabular}{|c|c|c|c|c|c|c|}
\hline
    & $(h,0)$  & $(0,h)$  & $(x,0)$ & $(0,x)$ & $(y,0)$ & $(0,y)$   \\
\hline
$g_1$ & $-(0,h)$ & $-(h,0)$ & $-(0,y)$ & $-(y,0)$ & $-(0,x)$ & $-(x,0)$ \\
\hline
$g_2$ & $(h,0)$  & $(0,h)$  & $(x,0)$             & $(0,x)$             & $(y,0)$   & $(0,y)$   \\
\hline
$g_3$ & $-(h,0)$ & $-(0,h)$ & $(y,0)$  & $(0,y)$ & $(x,0)$  & $(0,x)$\\
\hline
\end{tabular}
\caption{Galois action on $\mathrm{sl}(2)\times\mathrm{sl}(2)$}
\label{galoisonliealgebra}
\end{table}

To verify that $\mu_{F}$ descends to $\mathbb{Q}$, it suffices to check that for any generator $g$ of $\mathrm{Gal}(F/\mathbb{Q})$ and any element $l$ of $E(a,1)^{\circ} \otimes_{\mathbb{Q}} F$, we have 
\begin{equation*}
    g^{-1}\circ\mu_{F}(l)=\mu_{F}(g^{-1} \circ l)
\end{equation*}
The Galois group $\mathrm{Gal}(F/\mathbb{Q})$ acts on $T\in\mathrm{gl}(V \otimes_{\mathbb{Q}} F)$ via the rule $g \circ T: v \rightarrow g\circ T(g^{-1}\circ v)$. Therefore it suffices to verify that for any element $v \in V \otimes F$, the following holds:
\begin{equation}\label{galoisequ}
    g^{-1}\circ\mu_{F}(l)(g\circ v)=\mu_{F}(g^{-1} \circ l)(v) 
\end{equation}
Note that both sides of the above equality are $F$-linear transformations of the vector space $V_{F}$. Hence it suffices to verify the above equality for generators of $\mathrm{Gal}(F/\mathbb{Q})$, an $F$-basis of $E(a,1)^{\circ} \otimes_{\mathbb{Q}} F$ and an $F$-basis of $V_{F}$.

But now we can handily check the relations from the various tables above to obtain the desired Galois equivariance condition (\ref{galoisequ}). Hence $\mu$ is indeed defined over $\mathbb{Q}$.

Finally we show that the representation $$\tilde{\mu}:K\oplus E(a,1)^{\circ} \rightarrow \mathrm{gl}(V)$$ is irreducible. Suppose $V' \xhookrightarrow{} V$ is a $\mathbb{Q}$-subrepresentation of $\tilde{\mu}$. Then $V'_{\qbar}$ is simultaneously stable under the action of $K$ and the representation $$\mu_{\qbar}: \mathrm{sl}(2)\times \mathrm{sl}(2) \rightarrow \mathrm{gl}(V_{\qbar})$$ Therefore we have that $V'_{\qbar}\cong V(1) \boxtimes V(1)$. Then \begin{align*}V'_{\qbar}=\mathrm{span}_{\qbar}(\{&
v'_{1,1}=p_1v_{1,1}+q_1w_{1,1}, v'_{-1,1}=p_2v_{-1,1}+q_2w_{-1,1},\\
&v'_{1,-1}=p_3v_{1,-1}+q_3w_{1,-1},v'_{-1,-1}=p_4v_{-1,-1}+q_4w_{-1,-1}\})\end{align*}where $v'_{m,n}$s are eigenvectors of weights $(m,n)$ in $V'_{\qbar}$ and $p_l,q_l\in \qbar$. First consider the situation where there exists an $l \in \{1,2,3,4\}$ such that $p_l\neq0$ and $q_l\neq0$. Suppose $l=1$. Denote the image of $\sqrt{D'}\in K$ under $\iota: K \xhookrightarrow{} \mathrm{End}(V)$ by $J$. We then have that $J\circ v'_{1,1}=\sqrt{D'}p_1v_{1,1}-\sqrt{D'}q_1w_{1,1}$, which clearly does not lie in $V'_{\qbar}$. The case where $l=2,3,4$ follows similarly. Hence we have that for each $l$, either $p_l=0$ or $q_{l}=0$. But checking Table \ref{galoisoneigenspace} about the action of the Galois group on eigenvectors, no such $V'_{\qbar}$ can be stable under $\mathrm{Gal}(F/\mathbb{Q})$. Therefore $\tilde{\mu}$ is indeed an irreducible $\mathbb{Q}$-representation.   
\end{proof}

In the next lemma we construct a symplectic form on $V$ which will serve the role of the polarization form for the family of weight one Hodge structures on $V$ promised in Proposition \ref{existenceofsl2timessl2}.
\begin{lemma}\label{polarizationformexists}
We assume the notations from Setup \ref{basisandgalois} and Lemma \ref{reprexists}. For the representation $\mu$ constructed in Lemma \ref{reprexists}, there exists a symplectic form $\phi:V\times V \rightarrow \mathbb{Q}$ which is preserved by the image of $\mu$ infinitesimally. Moreover the Rosati involution associated with $\phi$ preserves $K$ and coincides with the complex conjugation on $K$.   
\end{lemma}
\begin{proof}
We proceed by first constructing an $F$-linear symplectic form on $V_{F}$ and then verify that it descends to $\mathbb{Q}$. Let \begin{equation}\label{phivalue}
M:=\phi_{F}(v_{-1,-1},w_{1,1})=\phi_{F}(v_{1,1},w_{-1,-1})=-\sqrt{D'}  
\end{equation} and \begin{equation}\label{phivalue'}
    -M=\phi_{F}(v_{-1,1},w_{1,-1})=\phi_{F}(v_{1,-1},w_{-1,1})=\sqrt{D'}\end{equation} Moreover, let $\phi_{F}(v_{m,n},w_{m',n'})=0$
 if $(m,n) \neq (-m',-n')$ and also let $V_{\sigma}$ and $V_{\sigmabar}$ be isotropic vector subspaces with respect to $\phi_{F}$.

First we claim that $\phi$ is defined over $\mathbb{Q}$. It suffices to check that for $v_{m,n}$ and $w_{m',n'}$ constructed in (\ref{vsigmabasis}) and (\ref{vsigmabarbasis}) in the proof of Lemma \ref{reprexists} and for any element $g \in G$, the following relation holds
\begin{equation*}
  \phi_{F}(g\circ v_{m,n},g\circ w_{m',n'})=g\circ\phi_{F}(v_{m,n},w_{m',n'}) 
\end{equation*} Using results in Table \ref{galoisoneigenspace} from the proof of Lemma \ref{reprexists} and noticing that $g_1\circ M=M,g_2\circ M=-M,g_3\circ M=M$,
one can see that this holds indeed.

By Table \ref{reprconstruc} and Table \ref{reprconstruct'} from the proof of Lemma \ref{reprexists} as well as equation (\ref{phivalue}) and (\ref{phivalue'}), the form $\phi_{F}$ is indeed preserved by the image of $E(a,1)^{\circ}\otimes_{\mathbb{Q}} F$ under $\mu_{F}$ infinitesimally. Moreover by the above construction, $V_{\sigma}$ and $V_{\sigmabar}$ are isotropic vector subspaces with respect to $\phi_{F}$. For any $k\in K$ satisfying that $k+\overline{k}=0$ and for any $v\in V_{\sigma}$ and any $v'\in V_{\sigmabar}$ we have that $$\phi_{F}(k\circ v,v')+\phi_{F}(v,k\circ v')=(\sigma(k)+\overline{\sigma}(k))\phi_{F}(v,v')=\sigma(k+\overline{k})\phi_{F}(v,v')=0$$ By Theorem \ref{deg2mz4}, $\mathrm{Z}(\mathrm{hdg}(A))$ is precisely equal to the set $\{k\in K|k+\overline{k}=0\}$. Therefore indeed $\phi_{F}$ is preserved by the image of $\tilde{\mu}^{h}$ infinitesimally.

Finally, for the same reason as above, for any $k\in K$ and any $v,v'\in V_{F}$ $$\phi_{F}(k\circ v,v')=\phi_{F}(v,\overline{k}\circ v')$$ Hence the Rosati involution associated with $\phi$ preserves $K\xhookrightarrow{}\mathrm{End}(V)$ and coincides with the complex conjugation on $K$.
   
\end{proof}
In the next lemma, we will show the existence of a family of weight one Hodge structures on $V$ polarized by the $\phi$ constructed in Lemma \ref{polarizationformexists}.
\begin{lemma}\label{familyexists}
Assume the notations from Setup \ref{basisandgalois}, Lemma \ref{reprexists} and Lemma \ref{polarizationformexists}. There exists a family of weight 1 Hodge structures on $V$, on which $K$ acts with multiplicities $(m_{\sigma},m_{\sigmabar})=(1,3)$. Furthermore, each member of this family is polarized by the $\phi$ constructed in Lemma \ref{polarizationformexists}. 
\end{lemma}
\begin{proof}
Denote by $\mathcal{S}$ the set of elements $(x_0,x_1,x_2,x_3)\in \mathbb{C}^4$ satisfying the following inequality \begin{equation}\label{postivex}
x_0\overline{x_3}+x_1\overline{x_1}+x_2\overline{x_2}+x_3\overline{x_0}<0
\end{equation}  For each $(x_0,x_1,x_2,x_3)\in \mathcal{S}$, let $$v_{\sigma}^{1,0}:=x_0v_{1,1}+x_1v_{1,-1}+x_2v_{-1,1}+x_3v_{-1,-1} \in V_{\sigma}\otimes_{\qbar}\mathbb{C}$$

Let $$V_{\sigmabar}^{1,0}:=\{w \in V_{\sigmabar}\otimes_{\qbar}\mathbb{C}|\phi_{\mathbb{C}}(v_{\sigma}^{1,0},w)=0\}$$ which is equal to the set
\begin{equation}\label{vtau10}
 \{y_{0}w_{1,1}+y_{1}w_{1,-1}+y_{2}w_{-1,1}+y_{3}w_{-1,-1} \in V_{\sigmabar}|x_{0}y_{3}-x_{1}y_{2}-x_{2}y_{1}+x_{3}y_{0}=0\}   
\end{equation} by formula (\ref{phivalue}) and formula (\ref{phivalue'}) from the proof of Lemma \ref{polarizationformexists}. Let $v_{\sigmabar}^{0,1}:=\overline{v_{\sigma}^{1,0}}$ and $V_{\sigma}^{0,1}=\overline{V_{\sigmabar}^{1,0}}$. Then associated to each element $(x_0,x_1,x_2,x_3)\in \mathcal{S}$, one can construct the following weight 1 Hodge structure
\begin{equation}\label{hodge13}
V^{1,0}=\mathbb{C}\langle v_{\sigma}^{1,0}\rangle\oplus V_{\sigmabar}^{1,0}; V^{0,1}=\mathbb{C}\langle v_{\sigmabar}^{0,1}\rangle\oplus V_{\sigma}^{0,1}   
\end{equation}
Then one can immediately check that $K\xhookrightarrow{}\mathrm{End}(V)$ preserves the Hodge structure (\ref{hodge13}) and moreover acts with multiplicities $(m_{\sigma},m_{\sigmabar})=(1,3)$. Abusing notations, we also denote the set of Hodge structures associated with the set of elements in $\mathcal{S}$ by $\mathcal{S}$.

Then we claim that each Hodge structure in $\mathcal{S}$ is polarized by $\phi$. To verify so, in the rest of proof, we fix an arbitrary Hodge structure in $\mathcal{S}$. By the explicit construction (\ref{hodge13}), we have that $\phi$ is indeed a morphism of Hodge structures. Fix a square root of $-1$ and denote it by $\sqrt{-1}$. It then suffices to verify that for any $v \in V^{1,0}-\{0\}$, the following positivity condition holds  $$-\sqrt{-1}\phi(v,\overline{v})+\sqrt{-1}\phi(\overline{v},v)=-2\sqrt{-1}\phi(v,\overline{v})>0$$ If $v=v_{\sigma}^{1,0}$, then the positivity condition is spelled as
\begin{equation}\label{lastinequality}
\begin{split}
&-2\sqrt{-1}\phi_{\mathbb{C}}(x_0v_{1,1}+x_1v_{1,-1}+x_2v_{-1,1}+x_3v_{-1,-1},\\
&\overline{x_0}\overline{v_{1,1}}+\overline{x_1}\overline{v_{1,-1}}+\overline{x_2}\overline{v_{-1,1}}+
\overline{x_3}\overline{v_{-1,-1}})>0
\end{split}
\end{equation} Recall by Setup \ref{basisandgalois}, the complex conjugation is equal to $(g_1,g_2,g_3)\in \mathrm{Gal}(F/\mathbb{Q})$. Therefore using Table \ref{galoisoneigenspace} from the proof of Lemma \ref{reprexists}, we have that 
$$\overline{v_{1,1}}=w_{1,1},\overline{v_{-1,-1}}=w_{-1,-1},\overline{v_{1,-1}}=-w_{-1,1},\overline{v_{-1,1}}=-w_{1,-1}$$
Using formula (\ref{phivalue}) and (\ref{phivalue'}) from the proof of Lemma \ref{polarizationformexists}, inequality (\ref{lastinequality}) simplifies to
\begin{equation*}
-2M\sqrt{-1}(x_0\overline{x_3}+x_1\overline{x_1}+x_2\overline{x_2}+x_3\overline{x_0})>0
\end{equation*} Recall that $M=-\sqrt{D'}$, hence this is consistent with the condition (\ref{postivex}) on $v_{\sigma}^{1,0}$.

Now we need to show that the positivity condition is also satisfied by elements in $V_{\sigmabar}^{1,0}$ i.e. we need to show that for any $y_0,y_1,y_2,y_3\in\mathbb{C}$ such that $y_{0}w_{1,1}+y_{1}w_{1,-1}+y_{2}w_{-1,1}+y_{3}w_{-1,-1}\in V_{\sigmabar}^{1,0}-\{0\}$ i.e. equation (\ref{vtau10}) holds and not all $y_{i}$s are equal to zero, the following inequality is satisfied
\begin{equation*}
\begin{split}
&-2\sqrt{-1}\phi_{\mathbb{C}}(y_{0}w_{1,1}+y_{1}w_{1,-1}+y_{2}w_{-1,1}+y_{3}w_{-1,-1},\\
&\overline{y_{0}}\overline{w_{1,1}}+\overline{y_{1}}\overline{w_{1,-1}}+\overline{y_{2}}\overline{w_{-1,1}}+
\overline{y_{3}}\overline{w_{-1,-1}})>0
\end{split}
\end{equation*}
By formula (\ref{phivalue}) and formula (\ref{phivalue'}) and constructions from Lemma \ref{polarizationformexists}, the above is equivalent to 
\begin{equation}\label{postivey}
2M\sqrt{-1}(y_0\overline{y_3}+y_1\overline{y_1}+y_2\overline{y_2}+y_3\overline{y_0})>0   
\end{equation}

For an element $(x_0,x_1,x_2,x_3)\in\mathcal{S}$, we have that \begin{equation*}
x_0\overline{x_3}+x_1\overline{x_1}+x_2\overline{x_2}+x_3\overline{x_0}<0   
\end{equation*} In particular, 
\begin{equation*}
x_0\overline{x_3}+x_3\overline{x_0}<0  
\end{equation*} and therefore $x_0 \neq 0$ and $x_3\neq0$.

First we consider the special case where $x_1=x_2=0$.  Then the $y_{i}$s satisfy that $x_0y_3+x_3y_0=0$ while $y_1,y_2$s can be arbitrary. Hence we have that \begin{align*}
&2M\sqrt{-1}(y_0\overline{y_3}+y_1\overline{y_1}+y_2\overline{y_2}+y_3\overline{y_0}) \\
&\geq 2M\sqrt{-1}(y_0\overline{y_3}+y_3\overline{y_0})\\   
&=2M\sqrt{-1}(-\frac{x_0\overline{x_3}+x_3\overline{x_0}}{x_0\overline{x_0}}y_0\overline{y_0})\geq0
\end{align*} The equality is only attained when $y_0=y_1=y_2=0$. But this would imply that $y_3=0$. Hence we are done with this case.

Now suppose that $x_1 = 0$ and $x_2 \neq 0$. Then according to formula (\ref{vtau10}) we have that $y_1=\frac{1}{x_2}(x_0y_3+x_3y_0)$ and the value of $y_2$ is arbitrary. Moreover, $y_0\overline{y_3}+y_1\overline{y_1}+y_2\overline{y_2}+y_3\overline{y_0} \geq y_0\overline{y_3}+y_1\overline{y_1}+y_3\overline{y_0}$ where the equality is attained when $y_2=0$. So it suffices to show that
\begin{align*}
&x_0\overline{x_0}y_3\overline{y_3}+x_3\overline{x_3}y_0\overline{y_0}+x_0\overline{x_3}y_3\overline{y_0}+x_3\overline{x_0}y_0\overline{y_3}+x_2\overline{x_2}y_0\overline{y_3}+x_2\overline{x_2}y_3\overline{y_0}>0
\end{align*}
Letting $\mathcal{M}=(x_0\overline{x_3}+x_2\overline{x_2})\overline{y_0}$ and write $y_3=x+\sqrt{-1}y$, the above equation simplifies to
\begin{align*}
x_0\overline{x_0}(x^2+y^2)+2(\mathfrak{Re}(\mathcal{M})x-\mathfrak{Im}(\mathcal{M})y)+x_3\overline{x_3}y_0\overline{y_0}>0
\end{align*}
It suffices to show the minimal of the above expression, viewed as a quadratic function in two variables $x$ and $y$, is greater than 0. And this is equivalent to the following constraint on $x_{i}$s
\begin{align*}
(x_1\overline{x_1}+x_2\overline{x_2})(x_0\overline{x_3}+x_2\overline{x_2}+x_3\overline{x_0})<0
\end{align*}
But this is consistent with the condition (\ref{postivex}) we have set for $x_{i}$s.

Finally suppose $x_1 \neq 0$, then according to formula (\ref{vtau10}) we have that $y_2=(-\frac{x_2}{x_1}y_1+\frac{x_0}{x_1}y_3+\frac{x_3}{x_1}y_0)$ and we would like to show that 
\begin{align*}
    &(1+\frac{x_2\overline{x_2}}{x_1\overline{x_1}})y_1\overline{y_1}+\frac{1}{x_1\overline{x_1}}(-x_2\overline{x_0}y_1\overline{y_3}-x_2\overline{x_3}y_1\overline{y_0}\\
    &-x_0\overline{x_2}y_3\overline{y_1}+x_0\overline{x_0}y_3\overline{y_3}+x_0\overline{x_3}y_3\overline{y_0}\\
    &-x_3\overline{x_2}y_0\overline{y_1}+x_3\overline{x_0}y_0\overline{y_3}+x_3\overline{x_3}y_0\overline{y_0})+y_0\overline{y_3}+y_3\overline{y_0}>0
\end{align*}
Rearranging term, this is equivalent to showing that
\begin{align*}
    &(1+\frac{x_2\overline{x_2}}{x_1\overline{x_1}})y_1\overline{y_1}+\frac{1}{x_1\overline{x_1}}(-x_2\overline{x_0}y_1\overline{y_3}-x_2\overline{x_3}y_1\overline{y_0}\\
    &-x_0\overline{x_2}y_3\overline{y_1}-x_3\overline{x_2}y_0\overline{y_1})+\frac{1}{x_1\overline{x_1}}(x_0\overline{x_0}y_3\overline{y_3}+x_0\overline{x_3}y_3\overline{y_0}\\
    &+x_3\overline{x_0}y_0\overline{y_3}+x_3\overline{x_3}y_0\overline{y_0})+y_0\overline{y_3}+y_3\overline{y_0}>0
\end{align*}
Writing $y_1=x+\sqrt{-1}y$ and letting $x_2\overline{x_0}\overline{y_3}+x_2\overline{x_3}\overline{y_0}$ equal to $\mathcal{N}$, this is equivalent to
\begin{align*}
    &(1+\frac{x_2\overline{x_2}}{x_1\overline{x_1}})(x^2+y^2)-\frac{2}{x_1\overline{x_1}}(\mathfrak{Re}(\mathcal{N})x-\mathfrak{Im}(\mathcal{N})y)+\frac{1}{x_1\overline{x_1}}(x_0\overline{x_0}y_3\overline{y_3}+x_0\overline{x_3}y_3\overline{y_0}\\
    &+x_3\overline{x_0}y_0\overline{y_3}+x_3\overline{x_3}y_0\overline{y_0})+y_0\overline{y_3}+y_3\overline{y_0}>0
\end{align*}
And the minimal value of the above expression viewed as a quadratic function in two variables $x$ and $y$ is equal to 
\begin{align*}
    &-\frac{1}{x_1\overline{x_1}(x_2\overline{x_2}+x_1\overline{x_1})}(\mathfrak{Re}^2(\mathcal{N})+\mathfrak{Im}^2(\mathcal{N}))+\frac{1}{x_1\overline{x_1}}(x_0\overline{x_0}y_3\overline{y_3}+x_0\overline{x_3}y_3\overline{y_0}\\
    &+x_3\overline{x_0}y_0\overline{y_3}+x_3\overline{x_3}y_0\overline{y_0})+y_0\overline{y_3}+y_3\overline{y_0}
\end{align*}
Rearranging the order of terms, it suffices to show that
\begin{align*}
&-\frac{x_0\overline{x_0}x_2\overline{x_2}y_3\overline{y_3}+x_2\overline{x_2}x_3\overline{x_0}y_0\overline{y_3}+x_2\overline{x_2}x_0\overline{x_3}y_3\overline{y_0}+x_2\overline{x_2}x_3\overline{x_3}y_0\overline{y_0}}{x_1\overline{x_1}(x_2\overline{x_2}+x_1\overline{x_1})}+\\
&\frac{1}{x_1\overline{x_1}}(x_0\overline{x_0}y_3\overline{y_3}+x_0\overline{x_3}y_3\overline{y_0}+x_3\overline{x_0}y_0\overline{y_3}+x_3\overline{x_3}y_0\overline{y_0})+y_0\overline{y_3}+y_3\overline{y_0}>0
\end{align*}
Simplifying terms this is equal to the following statement
\begin{align*}
&\frac{x_0\overline{x_0}y_3\overline{y_3}+x_3\overline{x_0}y_0\overline{y_3}+x_0\overline{x_3}y_3\overline{y_0}+x_3\overline{x_3}y_0\overline{y_0}}{(x_1\overline{x_1}+x_2\overline{x_2})}+y_0\overline{y_3}+y_3\overline{y_0}>0
\end{align*}
Rearranging terms, this is equivalent to showing that
\begin{equation*}
x_0\overline{x_0}y_3\overline{y_3}+x_0\overline{x_3}y_3\overline{y_0}+(x_1\overline{x_1}+x_2\overline{x_2})(y_3\overline{y_0}+y_0\overline{y_3})+x_3\overline{x_0}y_0\overline{y_3}+x_3\overline{x_3}y_0\overline{y_0}>0
\end{equation*}Letting $\mathcal{N}'=(x_0\overline{x_3}+x_1\overline{x_1}+x_2\overline{x_2})\overline{y_0}$ and writing $y_3=x+\sqrt{-1}y$, the above equation simplifies to
\begin{align*}
x_0\overline{x_0}(x^2+y^2)+2(\mathfrak{Re}(\mathcal{N}')x-\mathfrak{Im}(\mathcal{N}')y)+x_3\overline{x_3}y_0\overline{y_0}>0
\end{align*}

It suffices to show the minimal of the above expression, viewed as a quadratic function in two variables $x$ and $y$, is greater than 0. This is equivalent to the following requirement on $x_{i}$s
\begin{equation*}
(x_1\overline{x_1}+x_2\overline{x_2})(x_0\overline{x_3}+x_1\overline{x_1}+x_2\overline{x_2}+x_3\overline{x_0})<0
\end{equation*}
But this is consistent with the requirement (\ref{postivex}) we have set for $x_{i}$s.

Hence we have shown that for any $v_{\sigma}^{1,0}\in V_{\sigma}\otimes_{\qbar}\mathbb{C}$ satisfying condition (\ref{postivex}), the form $\phi$ when evaluated at an arbitrary $w\in V_{\sigmabar}^{1,0}-\{0\}$ also satisfies the positivity condition (\ref{postivey}). Therefore each Hodge structure in $\mathcal{S}$ is indeed polarized by $\phi$.   
\end{proof}
Finally we will show the existence of a simple anti-Weil type abelian fourfold in the family constructed in Lemma \ref{familyexists}.
\begin{lemma}\label{simplexists}
We assume the notations from Lemma \ref{familyexists} as well as its proof. There exists an abelian fourfold in the family constructed in Lemma \ref{familyexists} whose endomorphism field is precisely $K$ and therefore is simple and of anti-Weil type.
\end{lemma}
\begin{proof}
The idea of proof is to show that the locus of Hodge structures in $\mathcal{S}$ whose endomorphism algebra is strictly larger than $K$ is a countably union of lower dimensional subsets of $\mathcal{S}$. Then an element in $\mathcal{S}$ outside this locus corresponds to a simple anti-Weil type abelian fourfold. First consider the set of weight 1 Hodge structures on $V$ satisfying the following two conditions
\begin{enumerate}
    \item The image of $\iota:K\xhookrightarrow{}\mathrm{End}(V)$ falls inside the endomorphism algebra of each Hodge structure and acts with multiplicities $(m_{\sigma},m_{\sigmabar})=(1,3)$.
    \item The bilinear form $\phi$ in Lemma \ref{polarizationformexists} preserves each Hodge structure in this set.
\end{enumerate}
Denote this set by $\tilde{\mathcal{S}}$. Then each Hodge structure in $\tilde{\mathcal{S}}$ is completely determined by the choice of a vector $v_{\sigma}^{1,0}$ in $V_{\sigma}\otimes_{\qbar}\mathbb{C}$. Hence $\tilde{\mathcal{S}}$ is isomorphic to $\mathbb{P}^3$. Note that $\mathcal{S}$ constructed by formula (\ref{postivex}) and (\ref{hodge13}) from Lemma \ref{familyexists} is an analytic open subset of $\tilde{\mathcal{S}}$.

We first show that there is a positive dimensional locus of simple weight one Hodge structures in $\mathcal{S}$. Given $W\subset V$ an even dimensional $\mathbb{Q}$-vector subspace of $V$, denote by $\Gamma_{W}$ the set of weight 1 Hodge structures in $\tilde{\mathcal{S}}$ which contain $W$ as a sub-Hodge structure i.e. $\Gamma_{W}=\{\tilde{s}\in\tilde{\mathcal{S}}|\mathrm{dim}_{\mathbb{C}}(V^{1,0}_{\tilde{s}}\cap W_{\mathbb{C}})=\frac{1}{2}\mathrm{dim}(W)\}$. Since there are countably many even dimensional $\mathbb{Q}$-vector subspaces of $V$, there are countably many $\Gamma_{W}$s. Denote the union of all such $\Gamma_{W}$s by $\tilde{\mathcal{S}}_{\mathrm{split}}$. By construction $\tilde{\mathcal{S}}_{\mathrm{split}}$ consists of Hodge structures in $\tilde{\mathcal{S}}$ which are not simple. Then because $\tilde{\mathcal{S}}$ is isomorphic to $\mathbb{P}^3$, either $\tilde{\mathcal{S}}\cap \Gamma_{W}=\tilde{\mathcal{S}}$ or $\tilde{\mathcal{S}}\cap \Gamma_{W}$ is a subvariety of $\tilde{\mathcal{S}}$ of strictly lower dimension. Checking the construction of $\tilde{\mathcal{S}}$, one can see that there does not exist a $\mathbb{Q}$-vector subspace $W_0$ such that $\tilde{\mathcal{S}}=\Gamma_{W_0}$. Therefore for each $W$,  $\tilde{\mathcal{S}}\cap \Gamma_{W}$ is a subvariety of $\tilde{\mathcal{S}}$ with strictly lower dimension. Since $\mathcal{S}$ is an analytic open subset of $\tilde{\mathcal{S}}$, we have that for each $W$, $\mathcal{S}\cap \Gamma_{W}$ is an analytic subset of $\mathcal{S}$ of strictly lower dimension. Therefore $\mathcal{S}_{\mathrm{simple}}:=\mathcal{S}-\mathcal{S}\cap\tilde{\mathcal{S}}_{\mathrm{split}}$ is a positive dimensional open subset of $\mathcal{S}$.

Elements in $\mathcal{S}_{\mathrm{simple}}$ correspond to simple abelian fourfolds whose endomorphism algebra contains the deg 2 CM-field $K$ which acts by multiplicities $(m_{\sigma},m_{\sigmabar})=(1,3)$. We now show that there exists $s_0\in\mathcal{S}_{\mathrm{simple}}$ such that the corresponding abelian fourfold $A_{s_0}$ has endomorphism algebra precisely $K$. By the analysis on Page 561 of \cite{mz-4-folds}, for an element $s_0$ in $\mathcal{S}_{\mathrm{simple}}$, we have the following possibilities \begin{enumerate}
    \item\label{scene1} The endomorphism algebra of  $A_{s_0}$ is isomorphic to a quartic CM-field $\tilde{K}$.
    \item\label{scene2} The endomorphism algebra of  $A_{s_0}$ is isomorphic to a degree 8 CM-field $K'$. 
    \item\label{scene3} The endomorphism algebra of $A_{s_0}$ is isomorphic to a quaternion algebra $D$ over $\mathbb{Q}$ such that $D\otimes_{\mathbb{Q}}\mathbb{R}\cong\mathcal{H}$.
\end{enumerate}

Now we consider scenario (\ref{scene1}). Suppose we are given the following pair $(\tilde{K}/K,\tilde{\iota})$: a quartic CM-field $\tilde{K}$ and a morphism of $\mathbb{Q}$-algebras $\tilde{\iota}:\tilde{K}\xhookrightarrow{}\mathrm{End}(V)$ which in addition satisfies that $\iota: K\xhookrightarrow{}\mathrm{End}(V)$ factors through $\tilde{\iota}$. Denote $$\mathrm{Hom}(\tilde{K},\qbar)=\{\sigma_1,\sigma_2,\overline{\sigma_1},\overline{\sigma_2}\}$$ such that $\sigma_1|_{K}=\sigma_2|_{K}=\sigma$ and $\overline{\sigma_1}|_{K}=\overline{\sigma_2}|_{K}=\overline{\sigma}$. With respect to $\tilde{\iota}$ we have the eigenspace decompositions $$V_{\qbar}=V_{\sigma_1}\oplus V_{\sigma_2}\oplus V_{\overline{\sigma_1}}\oplus V_{\overline{\sigma_2}}$$ which satisfy the relations $V_{\sigma}=V_{\sigma_1}\oplus V_{\sigma_2}$ and $V_{\overline{\sigma}}=V_{\overline{\sigma_1}}\oplus V_{\overline{\sigma_2}}$. Next consider the set of weight one Hodge structures in $\tilde{\mathcal{S}}$ which admits a pair $(\tilde{K}/K,\tilde{\iota})$ satisfying the above conditions in its endomorphism algebra and denote this set by $\Gamma_{(\tilde{K}/K,\tilde{\iota})}$. Then by Lemma \ref{multiplicity-quartic-CM} and Lemma \ref{deg4keylemma}, for any polarized simple weight one Hodge structure  $s\in\Gamma_{(\tilde{K}/K,\tilde{\iota})}\cap\mathcal{S}_{\mathrm{simple}}$, $\tilde{K}$ acts on the simple abelian fourfold $A_{s}$ with multiplicities $\{2,0,1,1\}$. By the construction of $\mathcal{S}$, we have that $\mathrm{dim}(V_{\sigma}\otimes_{\qbar}\mathbb{C}\cap V_{s}^{1,0})=1$. Therefore the possible multiplicities for $\tilde{K}$ are $m_{\sigma_1}=0,m_{\sigma_2}=1$ or $m_{\sigma_1}=1,m_{\sigma_2}=0$. This implies that the $v_{\sigma}^{1,0}$ associated with the Hodge structure $s$ either lies in $V_{\sigma_1}\otimes_{\qbar}\mathbb{C}$ or lies in $V_{\sigma_2}\otimes_{\qbar}\mathbb{C}$. Recall that by formula (\ref{hodge13}) the Hodge structure is determined by $v_{\sigma}^{1,0}$. Therefore for each pair $(\tilde{K}/K,\tilde{\iota})$, the set $\Gamma_{(\tilde{K}/K,\tilde{\iota})}$ is cut out by $\qbar$-coefficient linear equations. Because $V_{\sigma_1}$ and $V_{\sigma_2}$ are proper subspaces of $V_{\sigma}$, the subset $\Gamma_{(\tilde{K}/K,\tilde{\iota})}$ is properly contained in $\tilde{\mathcal{S}}$. Denote the union of $\Gamma_{(\tilde{K}/K,\tilde{\iota})}$s by $\Gamma_{A}$. Since there are countably many pairs $(\tilde{K}/K,\tilde{\iota})$, by the same argument as above, we deduce that $\mathcal{S}_{\mathrm{simple}}-\Gamma_{A}$ is a positive dimensional subset of $\mathcal{S}_{\mathrm{simple}}$.

Scenario (\ref{scene2}) can be dealt with in basically the same way as above. Suppose we are given the following pair $(K'/K,\iota')$: a degree 8 CM-field $K'$ and a morphism of $\mathbb{Q}$-algebras $\iota':K'\xhookrightarrow{}\mathrm{End}(V)$ which in addition satisfies that $\iota: K\xhookrightarrow{}\mathrm{End}(V)$ factors through $\iota'$. Denote $$\mathrm{Hom}(K',\qbar)=\{\sigma_{i}|i\in\{1,2,3,4\}\}\coprod\{\overline{\sigma_i}|i\in\{1,2,3,4\}\}$$ such that $\sigma_i|_{K}=\sigma$ and $\overline{\sigma_i}|_{K}=\overline{\sigma}$. Then the eigenspace decomposition with respect to $\iota'$ $$V_{\qbar}=\oplus_{i\in\{1,2,3,4\}}(V_{\sigma_{i}}\oplus V_{\overline{\sigma_i}})$$ satisfies the following relations $V_{\sigma}=\oplus_{i\in\{1,2,3,4\}}V_{\sigma_i}$ and $V_{\overline{\sigma}}=\oplus_{i\in\{1,2,3,4\}}V_{\overline{\sigma_i}}$. Consider the set of Hodge structures in $\tilde{\mathcal{S}}$ which admits a pair $(K'/K,\iota')$ satisfying the above conditions in its endomorphism algebra and denote this set by $\Gamma_{(K'/K,\iota')}$. Then by Lemma \ref{famouscmtype} (see Section \ref{cmintrosection}), for each element $s'$ in $\Gamma_{(K'/K,\iota')}$, we have $V_{\sigma_{i}}\subset V_{s'}^{1,0}$ or $V_{\sigma_{i}}\subset V_{s'}^{0,1}$. Since $\Gamma_{(K'/K,\iota')}\subset\mathcal{S}$, we have $\mathrm{dim}(V_{\sigma}\otimes_{\qbar}\mathbb{C}\cap V_{s'}^{1,0})=1$. Therefore there exists precisely one $i_0\in\{1,2,3,4\}$ such that $V_{\sigma_{i_0}}\subset V_{s'}^{1,0}$ while the rest of $V_{\sigma_{i}}$s belong to $V_{s'}^{0,1}$. Hence for each $s'\in\Gamma_{(K'/K,\iota')}$, its corresponding $v_{\sigma}^{1,0}$ lies in $V_{\sigma_{i_0}}\otimes_{\qbar}\mathbb{C}$ for some $i_0\in\{1,2,3,4\}$. Therefore for each pair $(K',\iota')$, the set $\Gamma_{(K'/K,\iota')}\subset\tilde{\mathcal{S}}$ is cut out by $\qbar$-coefficient linear equations and is properly contained in $\tilde{\mathcal{S}}$. Denote the union of $\Gamma_{(K'/K,\iota')}$s by $\Gamma_{B}$. By the same argument as above, we deduce that $\mathcal{S}_{\mathrm{simple}}-\Gamma_{A}-\Gamma_{B}$ is a positive dimensional subset of $\mathcal{S}_{\mathrm{simple}}$.

The analysis for scenario (\ref{scene3}) is slightly different than above. Given any simple abelian fourfold $A$ whose endomorphism algebra is isomorphic to a quaternion algebra $D$, we claim the following decomposition of irreducible $\mathrm{Hdg}(A)({\qbar})$-representations \begin{equation}\label{splitting}V_{\qbar}=W_1\oplus W_2\end{equation} where $W_1$ and $W_2$ are isomorphic as $\mathrm{Hdg}(A)({\qbar})$-representations. The reason is as follows. Note that $D=\mathrm{End}^{\circ}(A)=\mathrm{End}_{\mathrm{Hdg}(A)}(V)$, hence $\mathrm{End}_{\mathrm{Hdg}(A)(\qbar)}(V_{\qbar})=D\otimes_{\mathbb{Q}}\qbar\cong M_2(\qbar)$ as $\qbar$-algebras. Because the maximal tori inside $M_2(\qbar)$ has dimension two, $V_{\qbar}=W_1\oplus W_2$ where $W_1$ and $W_2$ are \textit{irreducible} $\mathrm{Hdg}(A)({\qbar})$-representations. By Lemma \ref{qbarhodgedecomp} from Section \ref{cmintrosection} , this puts the following condition on the Hodge structure $V_{\mathbb{C}}=V^{1,0}\oplus V^{0,1}$ associated with $A$ \begin{equation}
\label{quaternionlocus}
\begin{split}
    &W_1\otimes_{\qbar}\mathbb{C}=W_1\otimes_{\qbar}\mathbb{C}\cap V^{1,0}\oplus W_1\otimes_{\qbar}\mathbb{C}\cap V^{0,1}\\
    &W_2\otimes_{\qbar}\mathbb{C}=W_2\otimes_{\qbar}\mathbb{C}\cap V^{1,0}\oplus W_2\otimes_{\qbar}\mathbb{C}\cap V^{0,1}
\end{split}
\end{equation} 
Moreover there is an isomorphism $\chi$ between $W_1$ and $W_2$ as $\mathrm{Hdg}(A)({\qbar})$-representations. Using the same argument from the proof of Lemma \ref{qbarhodgedecomp}, the $\mathbb{C}$-linear extension of $\chi$ satisfies the following condition \begin{equation}\label{generlizedhodgemorphism}
\begin{split}
&\chi_{\mathbb{C}}(W_1\otimes_{\qbar}\mathbb{C}\cap V^{1,0})\subset W_2\otimes_{\qbar}\mathbb{C}\cap V^{1,0}\\
&\chi_{\mathbb{C}}(W_1\otimes_{\qbar}\mathbb{C}\cap V^{0,1})\subset W_2\otimes_{\qbar}\mathbb{C}\cap V^{0,1}\end{split}\end{equation}

Inspired by the above paragraph, suppose we are given a triple $(W_1,W_2,\chi)$ where $W_1$ and $W_2$ are four-dimensional $\qbar$-vector subspaces of $V_{\qbar}$ such that $V_{\qbar}=W_1\oplus W_2$ and $\chi:W_1\rightarrow W_2$ is an invertible $\qbar$-linear map between $W_1$ and $W_2$. Consider the subset of Hodge structures in $\mathcal{S}$ satisfying formula (\ref{quaternionlocus}) and moreover $\chi$ is compatible with the ``Hodge decomposition'' on $W_1$ and $W_2$ i.e. formula (\ref{generlizedhodgemorphism}). Denote this set by $\Gamma_{(W_1,W_2,\chi)}$. Then a careful inspection of the construction of $\Gamma_{(W_1,W_2,\chi)}$ reveals that it is cut out in $\mathcal{S}$ by $\qbar$-coefficient polynomial equations.

Last but not least, we claim that for each $(W_1,W_2,\chi)$, $\mathcal{S}_{\mathrm{simple}}\cap\Gamma_{(W_1,W_2,\chi)}$ is indeed \textit{properly} contained in $\mathcal{S}_{\mathrm{simple}}$. It suffices to check the case where $\{W_1,W_2\}=\{V_{\sigma},V_{\sigmabar}\}$. But recall from the construction (\ref{hodge13}) that for an element $s\in\mathcal{S}_{\mathrm{simple}}$, $V_{\sigma}^{1,0}$ has dimension one while $V_{\sigmabar}^{1,0}$ has dimension three. Thus there cannot exist an isomorphism $\chi:V_{\sigma}\cong V_{\sigmabar}$ compatible the decomposition $V_{\sigma}\otimes_{\qbar}\mathbb{C}=V_{\sigma}^{1,0}\oplus V_{\sigma}^{0,1}$ and $V_{\sigmabar}\otimes_{\qbar}\mathbb{C}=V_{\sigmabar}^{1,0}\oplus V_{\sigmabar}^{0,1}$ associated with $s$. Therefore $\Gamma_{(V_{\sigma},V_{\sigmabar},\chi)}$ is in fact empty for any choice of $\chi$.

There are countably many triples of the form $(W_1,W_2,\chi)$ and denote the union of $\Gamma_{(W_1,W_2,\chi)}$s by $\Gamma_{C}$. Then the locus of simple abelian fourfolds in $\mathcal{S}_{\mathrm{simple}}$ whose endomorphism algebra satisfying scenario (\ref{scene3}) is contained $\Gamma_{C}$. By the same argument as above, we deduce that $\mathcal{S}_{\mathrm{simple}}-\Gamma_{A}-\Gamma_{B}-\Gamma_{C}$ is non-empty.

Then any abelian fourfold associated with a polarized weight one Hodge structure in $\mathcal{S}_{\mathrm{simple}}-\Gamma_{A}-\Gamma_{B}-\Gamma_{C}$ is simple of anti-Weil type.    
\end{proof}
\begin{proof}[Proof of Proposition \ref{existenceofsl2timessl2}]
Note that for a simple abelian fourfold $A$ of anti-Weil type with endomorphism field $K$, the centre of its Mumford-Tate group coincides with $K^{\times}$ by Theorem \ref{deg2mz4}.    
Then the proof is finished by the combined efforts of Lemma \ref{reprexists}, Lemma \ref{polarizationformexists}, Lemma \ref{familyexists} and Lemma \ref{simplexists}.    
\end{proof}
\begin{remark}
    It is not clear to the author how to show there exists an abelian fourfold defined over $\qbar$ satisfying the conditions listed in Proposition \ref{existenceofsl2timessl2}.
\end{remark}

\section{Family of De Rham-Betti Structures}

In this section we study a family of de Rham-Betti structures. We will take a Tannakian point of view, and define the de Rham-Betti group associated with such a family (see Definition \ref{drbgroupoffamily}). Moreover, we will give a characterization of the fixed tensors of such groups. Along the way, we will also show that Deligne's ``Principle B" holds for a family of de Rham-Betti structures coming from geometry (Proposition \ref{mainprincipleB}).

\subsection{Tannakian Formalism of De Rham-Betti Systems}

We first introduce the notion of a family of dRB structures following \cite{saito1997determinant}. Let $U$ be a smooth quasi-projective variety over $\qbar$, $X$ a smooth compactification of $U$ with $U=X-D$ where $D$ is a simple normal crossing divisor on $X$. 

\begin{definition}\label{defndrbsystem}
  A \textit{de Rham-Betti system} or \textit{dRB system} $\mathcal{V}$ on $U$ consists of the following triple of data: $(\localsystem, (\algebraicderham,\nabla), \rho_{m})$ where 
\begin{enumerate}
    \item $\localsystem$ is a locally constant sheaf of $\mathbb{Q}$-vector spaces on $U^{\mathrm{an}}$
    \item $(\algebraicderham,\nabla)$ is an algebraic vector bundle on $U$ together with an integrable connection $\nabla$ which has regular singularities along $D$
    \item $\rho_{m}$ is an isomorphism of holomorphic vector bundles $$\flatholomorphic:=\localsystem \otimes_{\mathbb{Q}}\mathcal{O}_{U^{\mathrm{an}}} \cong \algebraicderham \otimes_{\mathcal{O}_{U}} \mathcal{O}_{U^{\mathrm{an}}}$$
\end{enumerate}  
Moreover, under the comparison isomorphism $\rho_{m}$, the Gauss-Manin connection $\nabla^{\mathrm{GM}}: \flatholomorphic \rightarrow \flatholomorphic \otimes \Omega_{U^{\mathrm{an}}}^1$ and the analytification of $(\algebraicderham,\nabla)$ are compatible.
\end{definition}
According to \cite{saito1997determinant}, the category of dRB systems on $U$ has the natural structure of a tensor category. For the convenience of the reader, we will write down the morphisms, the tensor product structure of this category. Moreover we will describe the identity objects in it.
\begin{enumerate}
\item A morphism $\mathcal{F}$ between two de Rham-Betti systems $\mathcal{V}=(\localsystem, (\algebraicderham,\nabla), \rho_{m})$ and $\mathcal{V}'=(\localsystem', (\algebraicderham',\nabla'), \rho_{m}')$ consists of $(\mathcal{F}_{\mathrm{B}},\mathcal{F}_{\mathrm{dR}})$ where \begin{itemize}
    \item $\mathcal{F}_{\mathrm{B}}: \localsystem \rightarrow \localsystem'$ is a morphism of local systems on $U^{\mathrm{an}}$
    \item $\mathcal{F}_{\mathrm{dR}}: (\algebraicderham,\nabla) \rightarrow (\algebraicderham',\nabla')$ is a morphism of flat algebraic vector bundles defined over $\qbar$, i.e. the following diagram commutes 
$$\begin{tikzcd}
\algebraicderham \arrow[r, "\nabla"] \arrow[d, "\mathcal{F}_{\mathrm{dR}}"'] & \algebraicderham \otimes \Omega^1_{U} \arrow[d, "\mathcal{F}_{\mathrm{dR}}\otimes1"] \\
\algebraicderham' \arrow[r, "\nabla'"]                                       & \algebraicderham'\otimes \Omega_{U}^1                                                 
\end{tikzcd}$$
\item Moreover, $\mathcal{F}_{\mathrm{B}}$ and $\mathcal{F}_{\mathrm{dR}}$ are compatible with respect to the comparison isomorphism i.e. the following diagram commutes
$$\begin{tikzcd}
\localsystem \otimes_{\mathbb{Q}}\mathcal{O}_{U^{\mathrm{an}}} \arrow[d, "\mathcal{F}_{\mathrm{B}} \otimes 1"'] \arrow[r, "\rho_{m}"] & \algebraicderham \otimes_{\mathcal{O}_{U}}\mathcal{O}_{U^{\mathrm{an}}} \arrow[d, "\mathcal{F}_{\mathrm{dR}}\otimes1"] \\
\localsystem' \otimes_{\mathbb{Q}}\mathcal{O}_{U^{\mathrm{an}}} \arrow[r, "\rho_{m}'"] \arrow[r]                                      & \algebraicderham' \otimes_{\mathcal{O}_{U}}\mathcal{O}_{U^{\mathrm{an}}}                                              
\end{tikzcd}$$
\end{itemize}  
\item \label{tensorconstuction}Given two de Rham-Betti systems $$\mathcal{V}=(\localsystem, (\algebraicderham,\nabla), \rho_{m})$$ and $$\mathcal{V}'=(\localsystem', (\algebraicderham',\nabla'), \rho_{m}')$$ their tensor product $\mathcal{V} \otimes \mathcal{V}'$ is defined as $$(\localsystem \otimes_{\mathbb{Q}} \localsystem', (\algebraicderham \otimes_{\mathcal{O}_{U}} \algebraicderham',\nabla \otimes \mathrm{id}+\mathrm{id}\otimes \nabla'), \rho_{m} \otimes \rho_{m}')$$ where $$\nabla \otimes \mathrm{id}+\mathrm{id}\otimes \nabla': \algebraicderham \otimes \algebraicderham' \rightarrow \algebraicderham \otimes \algebraicderham'\otimes \Omega_{U}^1$$ sends elementary tensors $v \otimes v'$ to $v \otimes \nabla'(v')+\nabla(v) \otimes v'$. The second term $\nabla(v) \otimes v'$ lies in $\algebraicderham \otimes \algebraicderham'\otimes \Omega_{U}^1$ via the canonical isomorphism $\algebraicderham \otimes \algebraicderham'\otimes \Omega_{U}^1 \cong (\algebraicderham \otimes \Omega_{U}^1) \otimes \algebraicderham'$.
\item \label{identityofdrbsystems} We claim that the identity objects in the category of dRB systems have the following ``$\qbar$-algebraic extendable property". They are one-dimensional de Rham-Betti systems $$(\mathbb{L}_{\mathrm{B}},(\mathcal{L}_{\mathrm{dR}},\nabla),\rho_{m})$$ satisfying that $\mathbb{L}_{\mathrm{B}}\cong\underline{\mathbb{Q}}_{U^{\mathrm{an}}}$ and $\mathcal{L}_{\mathrm{dR}}\cong\mathcal{O}_{U}$ and additionally they satisfy the following property \begin{itemize}
    \item Let $\alpha_{\mathrm{B}}$ be a non-zero constant global section of $\mathbb{L}_{\mathrm{B}}$ on $U^{\mathrm{an}}$. Then $$\rho_{m}(\alpha_{\mathrm{B}}) \in \mathrm{H}^{0}(U^{\mathrm{an}},\mathcal{L}_{\mathrm{dR}}\otimes_{\mathcal{O}_{U}}\mathcal{O}_{U^{\mathrm{an}}})$$ descends to an invertible \textit{flat} global section $f\in\mathrm{H}^{0}(U,\mathcal{L}_{\mathrm{dR}})$. 
\end{itemize}
\end{enumerate}
The above claim is verified in the following lemma.
\begin{lemma}\label{idenitiysystem}
Every identity object in the category of de Rham-Betti systems on $U$ satisfies the property listed above in claim (\ref{identityofdrbsystems}).
\end{lemma}

\begin{proof}
By Proposition 1.3 of \cite{deligne2012tannakian}, any two identity objects in a tensor category are isomorphic. Hence given any identity de Rham Betti system $\mathcal{V}'=(\mathbb{L}_{\mathrm{B}},(\mathcal{L}_{\mathrm{dR}},\nabla'), \rho_{m}')$ in the tensor category of de Rham-Betti systems on $U$, it is isomorphic to $\mathcal{V}_{0}:=(\underline{\mathbb{Q}}_{U^{\mathrm{an}}},(\mathcal{O}_{U},d), \rho_{m}=\mathrm{id})$. Hence there exist $F_{\mathrm{B}}: \underline{\mathbb{Q}}_{U^{\mathrm{an}}}\cong \mathbb{L}_{\mathrm{B}}$ an isomorphism of $\mathbb{Q}$-local systems and $F_{\mathrm{dR}}: \mathcal{O}_{U}\cong\mathcal{L}_{\mathrm{dR}}$ an isomorphism of $\mathcal{O}_{U}$-modules such that the following diagram is commutative \begin{equation}\label{nocontent}\begin{tikzcd}
\underline{\mathbb{Q}}_{U^{\mathrm{an}}}\otimes_{\mathbb{Q}}\mathcal{O}_{U^{\mathrm{an}}} \arrow[rr, "\rho_{m}=\mathrm{id}"] \arrow[d, "F_{\mathrm{B}}\otimes1"'] &  & \mathcal{O}_{U}\otimes_{\mathcal{O}_{U}}\mathcal{O}_{U^{\mathrm{an}}} \arrow[d, "F_{\mathrm{dR}}\otimes1"] \\
\mathbb{L}_{\mathrm{B}}\otimes_{\mathbb{Q}}\mathcal{O}_{U^{\mathrm{an}}} \arrow[rr, "\rho_{m}'"]                                                 &  & \mathcal{L}_{\mathrm{dR}}\otimes_{\mathcal{O}_{U}}\mathcal{O}_{U^{\mathrm{an}}}                                     
\end{tikzcd}\end{equation} Note that the isomorphism $F_{\mathrm{dR}}$ gives rise to an invertible global section $f \in \mathrm{H}^0(U,\mathcal{L}_{\mathrm{dR}})$. Therefore the commutativity of diagram (\ref{nocontent}) implies that 
$$\rho_{m}'(F_{\mathrm{B}}(1_{U^{\mathrm{an}}}))=F_{\mathrm{dR}}(1_{\mathrm{U}})$$ Thus up to a scalar in $\mathbb{Q}^{\times}$, we obtain that $$\rho_{m}'(\alpha_{\mathrm{B}})=f\in\mathrm{H}^{0}(U, \mathcal{L}_{\mathrm{dR}})$$ Moreover $F_{\mathrm{dR}}$ is compatible with the algebraic connection $\nabla'$ on $\mathcal{L}_{\mathrm{dR}}$ and the connection $d$ on $\mathcal{O}_{U}$. Hence the following diagram commutes as sheaves of abelian groups on $U$
$$\begin{tikzcd}
\mathcal{O}_{U} \arrow[d, "\cdot f"'] \arrow[r, "d"] & \mathcal{O}_{U} \otimes \Omega_{U} \arrow[d, "\cdot f\otimes1"] \\
\mathcal{L}_{\mathrm{dR}} \arrow[r, "\nabla^{'}"]                  & \mathcal{L}_{\mathrm{dR}} \otimes \Omega_{U}                             
\end{tikzcd}$$ Therefore for any $v$ a local section of $\mathcal{O}_{U}$ we have $\nabla'(vf)=f\otimes dv+v\nabla'(f)=f\otimes dv$ which implies that $\nabla'(f)=0$. 
\end{proof}
\begin{definition}\label{drbsystemrestriction}
    Given a de Rham-Betti system $\mathcal{V}=(\localsystem,(\mathcal{V}_{\mathrm{dR}},\nabla), \rho_{m})$ and a $\qbar$-point $u$ of $U$, one can define the \textit{restriction of $\mathcal{V}$ at $u$} $$\mathcal{V}|_{u}:=(\mathbb{V}_{\mathrm{B},u},\mathcal{V}_{\mathrm{dR},u},\rho_{m,u})$$where $\mathbb{V}_{\mathrm{B},u}$ is the stalk of the local system $\mathbb{V}_{\mathrm{B}}$ at $u$, $\mathcal{V}_{\mathrm{dR},u}$ is the fiber of the vector bundle $\mathcal{V}_{\mathrm{dR}}$ at $u$ and $\rho_{m,u}$ is the fiber of the holomorphic vector bundle isomorphism $\rho_{m}$ at $u$. Then $\mathcal{V}|_{u}$ is a de Rham-Betti structure (cf. Definition \ref{drb}).
\end{definition}
\begin{remark}
 By Lemma \ref{idenitiysystem}, given an identity de Rham-Betti system $\mathcal{L}_{0}$ and a global section of its Betti part $\mathbb{L}_{0,\mathrm{B}}$, for any $u\in U(\qbar)$, the stalk of its Betti part at $u$ is a de Rham-Betti class with respect to the dRB structure $\mathcal{L}_{0}|_{u}$.
\end{remark}

The category of dRB systems actually satisfies Tannakian type properties. The following proposition summarizes relevant results from \cite{saito1997determinant}.
\begin{proposition}[p.869, \cite{saito1997determinant}]\label{drbsystemtannakian}
Denote the tensor category of triples $$(\localsystem, (\algebraicderham,\nabla), \rho_{m})$$ on $U$ by $\mathcal{C}_{\mathrm{dRB}}(U)$. Then $\mathcal{C}_{\mathrm{dRB}}(U)$ satisfies the following properties.\begin{itemize}
    \item When endowed with the tensor structure specified in (\ref{tensorconstuction}) from above, $\mathcal{C}_{\mathrm{dRB}}(U)$ is a rigid tensor category whose identity objects satisfy the property described in (\ref{identityofdrbsystems}) from above.
    \item Fix a $\qbar$-point $u_{0}$ on $U$. Then for any dRB system $(\localsystem, (\algebraicderham,\nabla), \rho_{m})$, by taking the stalk of $\localsystem$ at $u_{0}$, we obtain an exact and faithful $\mathbb{Q}$-linear functor $\omega$ from $\mathcal{C}_{\mathrm{dRB}}(U)$ to the category of $\mathbb{Q}$-vector spaces $\mathrm{Vec}_{\mathbb{Q}}$.
    \item The category of dRB systems $\mathcal{C}_{\mathrm{dRB}}(U)$ together with the fiber functor $\omega$ forms a neutral Tannakian category.
\end{itemize}   
\end{proposition}

Using the construction and notation from the above Proposition \ref{drbsystemtannakian}, we can define the de Rham-Betti group associated with a de Rham-Betti system as follows.

\begin{definition}\label{drbgroupoffamily}
We use the notation of Proposition \ref{drbsystemtannakian}. Let $\mathcal{V}$ be a dRB system on $U$. Then we define $$\mathcal{G}_{\mathrm{dRB}}(\mathcal{V}):=\underline{\mathrm{Aut}}(\omega|_{\langle\mathcal{V}\rangle^{\otimes}})$$ Denote the stalk $\mathbb{V}_{\mathrm{B},u_0}$ by $V$. Then $\mathcal{G}_{\mathrm{dRB}}(\mathcal{V})$ is a $\mathbb{Q}$-algebraic subgroup of $\mathrm{GL}(V)$.
\end{definition}

Given $u_0$ a $\qbar$-point of $U$, taking the restriction of a de Rham-Betti system $\mathcal{V}=(\localsystem, (\algebraicderham,\nabla), \rho_{m})$ at $u_0$ (see Definition \ref{drbsystemrestriction}) defines a functor $\iota|_{u_0}$ from the category of dRB systems to $\mathcal{C}_{\mathrm{dRB}}$, the category of dRB structures. Moreover, remembering only $\localsystem$ and forgetting the rest of data in $\mathcal{V}$ also defines a functor $\iota_{\mathrm{mon}}$ from the category of dRB systems to the tensor category of $\mathbb{Q}$-coefficient local systems on $U$, denoted by $\mathrm{Loc}_{\mathbb{Q}}(U)$. Note that $\mathrm{Loc}_{\mathbb{Q}}(U)$ also admits a forgetful functor to $\mathrm{Vec}_{\mathbb{Q}}$ by taking the stalk of a local system at $u_0$. Both $\iota|_{u_0}$ and $\iota_{\mathrm{mon}}$ are tensor functors. Moreover, both functors are compatible with the forgetful functors from $\mathcal{C}_{\mathrm{dRB}}$ and $\mathrm{Loc}_{\mathbb{Q}}(U)$ to $\mathrm{Vec}_{\mathbb{Q}}$. Hence we deduce the following.
\begin{corollary}\label{twoimmmersions}
Given a de Rham-Betti system $\mathcal{V}=(\localsystem, (\algebraicderham,\nabla), \rho_{m})$ on $U$ and $u_0\in U(\qbar)$, the two functors $\iota|_{u_0}$ and $\iota_{\mathrm{mon}}$ defined above induce closed immersions of algebraic groups $\gdrbmath(\mathcal{V}|_{u_0})\xhookrightarrow{}\mathcal{G}_{\mathrm{dRB}}(\mathcal{V})$ and $G_{\mathrm{Mon}}\xhookrightarrow{}\mathcal{G}_{\mathrm{dRB}}(\mathcal{V})$, where we define $G_{\mathrm{Mon}}$ as the algebraic monodromy group associated with the local system $\localsystem$. 
\end{corollary}

\begin{keyexample} \label{geometricsetup}
Suppose $Y$ and $U$ are smooth quasi-projective varieties defined over $\qbar$. Let $f:Y \rightarrow U$ be a smooth projective morphism. The triple $$(\localsystem,(\algebraicderham,\nabla),\rho_{m})=(\mathrm{R}^{j}f^{\mathrm{an}}_{*}\underline{\mathbb{Q}}_{Y},(\mathrm{R}^{j}f_{*}\Omega^{\bullet}_{Y/U},\nabla),\rho_{m})$$ is a dRB system on $U$. The integer $j$ satisfies that $0\leq j\leq 2n$ where $n$ is the relative dimension of $f$. In \cite{katz1970regularity}, $\nabla$ is canonically constructed and is shown to be an integrable connection which has regular singularities. By base change theorems, the restriction of $\mathcal{V}$ at a $\qbar$-point $u$ of $U$ is the dRB structure $(\mathrm{H}^{j}(Y_{u},\mathbb{Q}),\mathrm{H}^{j}_{\mathrm{dR}}(Y_{u}/\qbar),\rho_{m}|_{Y_{u}})$ associated with the $j$th cohomological groups of the fiber $Y_{u}$.
    
\end{keyexample}
In \cite{deligne1974theorie} Deligne proved that the local system $\localsystem$ from Key Example \ref{geometricsetup} above is in fact semisimple.
\begin{theorem}[Theorem 4.2.6, \cite{deligne1974theorie}]\label{semisimplemonodromy}
We keep the same notation from Key Example \ref{geometricsetup}. The Tannakian category $\langle \localsystem\rangle^{\otimes}$ is semisimple. In other words the algebraic monodromy group associated with $\localsystem$ is a reductive subgroup of $\mathrm{GL}(\mathrm{H}^{j}(Y_{u},\mathbb{Q}))$.   
\end{theorem}
We now give a more detailed interpretation of fixed tensors by $G:=\mathcal{G}_{\mathrm{dRB}}(\mathcal{V})$. Recall by Tannakian duality $\mathrm{Rep}_{\mathbb{Q}}(G)$ is tensor equivalent to $\langle\mathcal{V}\rangle^{\otimes}$. By definition a fixed tensor $\alpha$ in $\mathrm{Rep}_{\mathbb{Q}}(G)$ is a one-dimensional $G$-invariant subquotient of $\bigoplus_{n_{i},m_{i} \in \mathbb{Z}_{\geq0}} V^{\otimes n_{i}}\otimes V^{*\otimes m_{i}}$. In other words, we have the following diagram of $G$-representations
\begin{equation}\label{firstidenityobjectdiagram}
\begin{tikzcd}
M \arrow[d, two heads] \arrow[r, hook] & \bigoplus_{n_{i},m_{i} \in \mathbb{Z}_{\geq0}} V^{\otimes n_{i}}\otimes V^{*\otimes m_{i}} \\
L_{\alpha}                                           &                                                                    
\end{tikzcd}\end{equation} where every arrow is $G$-equivariant and $G$ acts on the one-dimensional $\mathbb{Q}$-vector space $L_{\alpha}$ as identity. By the Tannakian duality, this is equivalent to the existence of the following diagram of dRB systems 
\begin{equation}\label{idenityobjectdiagram}
\begin{tikzcd}
\mathcal{M} \arrow[d, two heads] \arrow[r, hook] & \bigoplus_{n_{i}, m_{i}\in \mathbb{Z}_{\geq0}} \mathcal{V}^{\otimes n_{i}}\otimes \mathcal{V}^{*\otimes m_{i}} \\
\mathcal{L}_{\alpha}                            &                                                              
\end{tikzcd}\end{equation} where $\mathcal{L}_{\alpha}$ is an identity object in $\langle\mathcal{V}\rangle^{\otimes}$. Moreover, $\omega(\mathcal{L}_{\alpha})=L_{\alpha}$ and $\omega(\mathcal{M})=M$. 
\begin{definition}\label{algextendable}
We keep the same notations from above. For a one-dimensional object $\alpha$ in $\mathrm{Rep}_{\mathbb{Q}}(G)$, if there exists a diagram in the form of (\ref{firstidenityobjectdiagram}) and (\ref{idenityobjectdiagram}), we say that $\alpha$ is \textit{$\qbar$-algebraic extendable}.
\end{definition}

In the sequel we will mostly be dealing with de Rham-Betti systems satisfying the following assumption.
\begin{assumption} \label{keyassumption} 
The de Rham-Betti system $\mathcal{V}=(\localsystem,(\algebraicderham,\nabla),\rho_{m})$ is equal to $$(\mathrm{R}^{j}f^{\mathrm{an}}_{*}\underline{\mathbb{Q}}_{Y},(\mathrm{R}^{j}f_{*}\Omega^{\bullet}_{Y/U},\nabla),\rho_{m})$$ where $f:Y \rightarrow U$ is a smooth projective morphism between smooth quasi-projective varieties defined over $\qbar$. The integer $j$ satisfies that $0\leq j\leq 2n$ where $n$ is the relative dimension of $f$. 
\end{assumption}

Given Assumption \ref{keyassumption} and an extra condition on the de Rham-Betti structure of the fiber, we have the following important observation.
\begin{corollary} \label{semisimple}
Suppose $\mathcal{V}$ satisfies Assumption \ref{keyassumption} and moreover suppose there exists a $\qbar$-point $u_0$ such that the dRB structure $\mathcal{V}|_{u_0}$ is simple. Then the category $\langle \mathcal{V}\rangle^{\otimes}$ is semisimple. In particular the group $\mathcal{G}_{\mathrm{dRB}}(\mathcal{V})$ is a reductive algebraic group. 
\end{corollary}
\begin{remark}
By Theorem \ref{reductive}, the condition of Corollary \ref{semisimple} is satisfied for example when $f: Y \rightarrow U$ is a family of abelian varieties defined over $\qbar$ and for some $u\in U(\qbar)$, $Y_{u}$ is a simple abelian variety.
\end{remark}

\begin{proof}
By Proposition \ref{drbsystemtannakian}, the category of dRB systems together with its fiber functor $\omega$ is a neutral Tannakian category. Then by Proposition 2.2 in \cite{kreutz2023rhambetti}, it suffices to show that $\mathcal{V}$ is a simple dRB system. Suppose there exists an inclusion $i: \mathcal{W} \xhookrightarrow{} \mathcal{V}$ of dRB systems. Then this induces $i|_{u_{0}}: \mathcal{W}|_{u_{0}} \xhookrightarrow{} \mathcal{V}|_{u_{0}}$ an inclusion of dRB structures. But this gives a contradiction to the assumption that $\mathcal{V}$ restricts to a simple dRB structure at the $\qbar$-point $u_{0}$.
\end{proof}
\begin{remark}
It is not clear how to show that $\mathcal{G}_{\mathrm{dRB}}(\mathcal{V})$ is a reductive algebraic group in a more general setting, even if $\mathcal{V}$ has a geometric origin. For example, suppose $\mathcal{V}|_{u_0}$ is a semisimple de Rham-Betti structure which is not simple. Then given an inclusion $i: \mathcal{W} \xhookrightarrow{} \mathcal{V}$ of dRB systems, its restriction $i|_{u_{0}}: \mathcal{W}|_{u_{0}} \xhookrightarrow{} \mathcal{V}|_{u_{0}}$ splits as morphism of dRB \textit{structures}. However it is not clear whether the fiberwise splitting can lead to a splitting of $i$ as de Rham-Betti \textit{systems}. Hence we cannot immediately deduce that $\mathcal{V}$ is a semisimple de Rham-Betti system. Thus in this case, it is not easy to check $\mathcal{G}_{\mathrm{dRB}}(\mathcal{V})$ is a reductive algebraic group or not.
\end{remark}

\subsection{Extendable De Rham-Betti Classes and Deligne's Principle B}
Given an identity de Rham-Betti system, its Betti part is a constant $\mathbb{Q}$-local system and its restriction at every $\qbar$-point of $U$ is a de Rham-Betti class. In this section, we are going to explore the partial inverse, i.e. given a constant one-dimensional local system, suppose its stalk at a $\qbar$-point $u_0$ is a de Rham-Betti class, in what setting does it necessarily spread out to an identity de Rham-Betti system?

Recall that in Lemma \ref{twoimmmersions}, for a de Rham-Betti system $\mathcal{V}$ and a $\qbar$-point $u_0$ of $U$, we have natural inclusions $\gdrbmath(\mathcal{V}|_{u_0})\xhookrightarrow{}\mathcal{G}_{\mathrm{dRB}}(\mathcal{V})$ and $G_{\mathrm{mon}}\xhookrightarrow{}\mathcal{G}_{\mathrm{dRB}}(\mathcal{V})$. The main goal of this section is to prove the following proposition. 

\begin{proposition} \label{principleB}
  Suppose the de Rham-Betti system $\mathcal{V}$ satisfies Assumption \ref{keyassumption}. Moreover, assume that $\mathcal{G}_{\mathrm{dRB}}(\mathcal{V})$ is a reductive algebraic group. Let $u_0$ be a $\qbar$-point of $U$. The fixed tensors of $\mathcal{G}_{\mathrm{dRB}}(\mathcal{V})\subset\mathrm{GL}(\mathbb{V}_{\mathrm{B}, u_0})$ in $\langle\mathbb{V}_{\mathrm{B}, u_0}\rangle^{\otimes}$ are monodromy invariant classes which are furthermore de Rham-Betti classes with respect to the dRB structure $\mathcal{V}|_{u_0}$.
\end{proposition}
\begin{remark}\label{subobjectsuffices}
By Corollary \ref{semisimple}, the condition that $\mathcal{G}_{\mathrm{dRB}}(\mathcal{V})$ is reductive is satisfied for example when the restriction of the de Rham-Betti system $\mathcal{V}$ at $u_0$ is simple. Moreover, by Theorem \ref{semisimplemonodromy}, the Tannakian category $\langle \mathbb{V}_{\mathrm{B}}\rangle^{\otimes}$ is semisimple. Therefore stalk-wise for any $\qbar$-point $u_0$ of $U$, a monodromy invariant class is an element $$\alpha_{\mathrm{B,{u_0}}}\in \bigoplus_{n_{i},m_{i} \in\mathbb{Z}_{\geq0}} \mathrm{H}^{j}(Y_{u_0},\mathbb{Q})^{\otimes n_{i}}\otimes \mathrm{H}^{j}(Y_{u_0},\mathbb{Q})^{*\otimes m_{i}}$$ invariant under the action of $G_{\mathrm{mon}}\subset \mathrm{GL}(\mathrm{H}^{j}(Y_{u_0},\mathbb{Q}))$ on the tensor products. Then $\alpha_{\mathrm{B,{u_0}}}$ is a de Rham-Betti class if it is invariant under the action of the de Rham-Betti group $\gdrbmath(\mathcal{V}|_{u_0})\subset \mathrm{GL}(\mathrm{H}^{j}(Y_{u_0},\mathbb{Q}))$ associated with the fiber at $u_0$. 
\end{remark}

The proof of Proposition \ref{principleB} will occupy the rest of this section. The following setup is the preliminary stage for the proof.

\begin{setup}\label{keysetup} Given a monodromy invariant class \begin{equation}\label{classinclusion}
\alpha \in\bigoplus_{n_{i},m_{i} \in \mathbb{Z}_{\geq0}} \mathrm{H}^{j}(Y_{u_{0}},\mathbb{Q})^{\otimes n_{i}}\otimes \mathrm{H}^{j}(Y_{u_{0}},\mathbb{Q})^{*\otimes m_{i}}\end{equation} invariant under the action of the de Rham-Betti group $\gdrbmath(\mathcal{V}|_{u_0})\subset\mathrm{GL}(\mathrm{H}^{j}(Y_{u_{0}},\mathbb{Q}))$, we are going to construct a morphism of dRB systems $$I_{\alpha}:\mathcal{L}_{0} \xhookrightarrow{} \bigoplus_{n_{i},m_{i} \in \mathbb{Z}_{\geq0}} \mathcal{V}^{\otimes n_{i}}\otimes \mathcal{V}^{*\otimes m_{i}}$$ such that $I_{\alpha}|_{u_0}$ is spanned by $\alpha$, where $\mathcal{L}_0$ is an identity object in the category of dRB systems (see Lemma \ref{idenitiysystem}).

 Formula (\ref{classinclusion}) induces a morphism of local systems with $\mathbb{Q}$-coefficients $$\tilde{i}_{\alpha}: \underline{\mathbb{Q}}_{U^{\mathrm{an}}} \xhookrightarrow{} \bigoplus_{n_{i},m_{i} \in \mathbb{Z}_{\geq0}} \localsystem^{\otimes n_{i}}\otimes \localsystem^{*\otimes m_{i}}$$ such that the stalk of $\tilde{i}_{\alpha}(1)$ is $\alpha$. Upon analytification, we obtain a morphism of flat holomorphic vector bundles
$$\tilde{i}_{\alpha}^{\mathrm{an}}:(\mathcal{O}_{\analyticu},d)\xhookrightarrow{} \bigoplus_{n_{i},m_{i} \in \mathbb{Z}_{\geq0}} (\localsystem^{\otimes n_{i}}\otimes \localsystem^{*\otimes m_{i}} \otimes_{\mathbb{Q}}\mathcal{O}_{\analyticu},\nabla^{\mathrm{GM}})$$ Post composing with the comparison isomorphism of $\mathcal{O}_{\analyticu}$-modules $\rho_{m}:\localsystem \otimes_{\mathbb{Q}} \mathcal{O}_{\analyticu} \cong \mathcal{V}_{\mathrm{dR}}^{\mathrm{an}}$ coming from the dRB system $\mathcal{V}$, we obtain a morphism of flat holomorphic vector bundles
\begin{equation}\label{aftercomparison}
\tilde{\tilde{i}}_{\alpha}^{\mathrm{an}}:(\mathcal{O}_{\analyticu},d)\xhookrightarrow{} \bigoplus_{n_{i},m_{i} \in \mathbb{Z}_{\geq0}} (\mathcal{V}_{\mathrm{dR}}^{\otimes n_{i}}\otimes \mathcal{V}_{\mathrm{dR}}^{*\otimes m_{i}} \otimes_{\mathcal{O}_{U}}\mathcal{O}_{\analyticu},\nabla^{\mathrm{an}})
\end{equation}
\end{setup}

\begin{lemma} \label{rhcorrespondence}
The morphism of flat holomorphic vector bundles $\tilde{\tilde{i}}_{\alpha}^{\mathrm{an}}$ constructed in Setup \ref{keysetup} above descends to a morphism of flat algebraic vector bundles defined over $\mathbb{C}$

$$\tilde{\tilde{i}}_{\alpha}: (\mathcal{O}_{U_{\mathbb{C}}},d)\xhookrightarrow{} \bigoplus_{n_{i},m_{i} \in \mathbb{Z}_{\geq0}} (\mathcal{V}_{\mathrm{dR},\mathbb{C}}^{\otimes n_{i}}\otimes \mathcal{V}_{\mathrm{dR},\mathbb{C}}^{*\otimes m_{i}},\nabla_{\mathbb{C}})$$
 
\end{lemma}
\begin{proof}
By \cite{katz1970regularity}, both $(\mathcal{V}_{\mathrm{dR},\mathbb{C}}=\mathrm{R}^{j}f_{*}\Omega^{\bullet}_{Y_{\mathbb{C}}/U_{\mathbb{C}}},\nabla)$ and $(\mathcal{O}_{U_{\mathbb{C}}},d)$ are $\mathbb{C}$-algebraic vector bundles with integrable connections which has regular singularities. Hence both sides of formula (\ref{aftercomparison}) are analytification of $\mathbb{C}$-algebraic vector bundles together with integrable connections which has regular singularities.
Now the Riemann-Hilbert correspondence (for example see Theorem 11.7 of \cite{peters2008mixed}) states that the analytification functor induces an equivalence between the category of holomorphic vector bundles together with an integrable connection on $\analyticu$ and the category of $\mathbb{C}$-algebraic vector bundles with an integrable connection which has regular singularities. Therefore, there exists a morphism 
$$\tilde{\tilde{i}}_{\alpha}: (\mathcal{O}_{U_{\mathbb{C}}},d)\xhookrightarrow{} \bigoplus_{n_{i},m_{i} \in \mathbb{Z}_{\geq0}} (\mathcal{V}_{\mathrm{dR},\mathbb{C}}^{\otimes n_{i}}\otimes \mathcal{V}_{\mathrm{dR},\mathbb{C}}^{*\otimes m_{i}} \otimes_{\mathcal{O}_{U_{\mathbb{C}}}}\Omega_{U_{\mathbb{C}}}^1,\nabla_{\mathbb{C}})$$ of flat algebraic vector bundles defined over $\mathbb{C}$, whose analytification is precisely formula (\ref{aftercomparison}). 
\end{proof}

To finish the proof of Proposition \ref{principleB} i.e. the $\qbar$-algebraic extendability, we are going to use the next proposition, which is the analog of ``Deligne's Principle B" in the realm of dRB systems.

\begin{proposition}\label{mainprincipleB}
Let $f:Y \rightarrow U$ be a smooth projective morphism between smooth quasi-projective varieties defined over $\qbar$ and let $\mathcal{V}=(\localsystem, (\algebraicderham,\nabla), \rho_{m})=(\mathrm{R}^{j}f^{\mathrm{an}}_{*}\underline{\mathbb{Q}}_{Y},(\mathrm{R}^{j}f_{*}\Omega^{\bullet}_{Y/U},\nabla), \rho_{m})$. The integer $j$ satisfies that $0\leq j\leq 2n$ where $n$ is the relative dimension of $f$. Given any identity object $L_{\alpha} \in \mathrm{ob}\langle\localsystem\rangle^{\otimes}$, let $\alpha_{\mathrm{B}}$ be a global section of $L_{\alpha}$. Suppose at a $\qbar$-point $u_{0}$ of $U$, $\alpha_{\mathrm{B},u_{0}}$ is a dRB class in the category $\langle\mathcal{V}|_{u_0}\rangle^{\otimes}$. Then for any $\qbar$-point $u$ of $U$, $\alpha_{\mathrm{B,{u}}}$ is a dRB class in the category $\langle\mathcal{V}|_{u}\rangle^{\otimes}$.
\end{proposition}

\begin{remark}
The proof of Proposition \ref{mainprincipleB} will closely follow from the page 15 to 17 of \cite{charles2011notes} which proves the analogous statement for absolute Hodge classes. In particular, we will be using Deligne's invariant cycle theorem (for example see Theorem 4.18 of \cite{voisin2003hodge}) in a crucial way, hence it is essential that the de Rham-Betti system $\mathcal{V}$ has a geometric origin.
\end{remark}
We begin with the following easy lemma.
\begin{lemma} \label{inverseisdrb}
Suppose $f:V_{\mathrm{dRB}}=(V_{\mathrm{B}},V_{\mathrm{dR}},\rho_{m}) \rightarrow V'_{\mathrm{dRB}}=(V'_{\mathrm{B}},V'_{\mathrm{dR}},\rho'_{m})$ is a morphism of dRB structures. If furthermore its Betti part is an isomorphism of $\mathbb{Q}$-vector spaces $f_{\mathrm{B}}:V_{\mathrm{B}} \rightarrow V'_{\mathrm{B}}$, then $f$ can be upgraded to an isomorphism of dRB structures.
\end{lemma}
One can use the Tannakian formalism to prove it. But we attach a direct proof as well.
\begin{proof}
By definition, we have the following commutative diagram of $\mathbb{C}$-vector spaces
$$\begin{tikzcd}
V_{\mathrm{B}} \otimes_{\mathbb{Q}} \mathbb{C} \arrow[d, "\rho_{m}"'] \arrow[r, "f_{\mathrm{B}}\otimes1"] & V'_{\mathrm{B}} \otimes_{\mathbb{Q}} \mathbb{C} \arrow[d, "\rho'_{m}"] \\
V_{\mathrm{dR}} \otimes_{\qbar} \mathbb{C} \arrow[r, "f_{\mathrm{dR}}\otimes1"]                           & V'_{\mathrm{dR}} \otimes_{\qbar} \mathbb{C}                           
\end{tikzcd}$$ where we denote by $f_{\mathrm{B}}\otimes1$ and $f_{\mathrm{dR}}\otimes1$ the $\mathbb{C}$-linear extension of $f_{\mathrm{B}}$ and $f_{\mathrm{dR}}$. Because $f_{\mathrm{B}} \otimes1$ is an isomorphism of $\mathbb{C}$-vector spaces by the assumption in the lemma, we have $f_{\mathrm{dR}} \otimes1$ is an isomorphism of $\mathbb{C}$-vector spaces. Hence $f_{\mathrm{dR}}$ induces an isomorphism of $\qbar$-vector spaces. Moreover note that $f_{\mathrm{B}}^{-1}\otimes_{\mathbb{Q}}1=(f_{\mathrm{B}}\otimes_{\mathbb{Q}}1)^{-1}$ and $f_{\mathrm{dR}}^{-1}\otimes_{\qbar}1=(f_{\mathrm{dR}}\otimes_{\qbar}1)^{-1}$. Therefore, $\rho'_{m}\circ (f_{\mathrm{B}} \otimes1)=(f_{\mathrm{dR}} \otimes1)\circ \rho_{m}$ implies that $\rho_{m}\circ (f^{-1}_{\mathrm{B}} \otimes1)=(f^{-1}_{\mathrm{dR}} \otimes1)\circ \rho'_{m}$. Hence $(f^{-1}_{\mathrm{B}},f^{-1}_{\mathrm{dR}})$ is a morphism of dRB structures which is the inverse of $f$.

\end{proof}

\begin{lemma}\label{polarizationdrb}
Let $X$ be a smooth projective variety over $\qbar$ of dimension $n$. Then for $0\leq j\leq 2n$,
there exists a bilinear form $\mathcal{P}_{X}: \mathrm{H}^{j}(X,\mathbb{Q}) \times \mathrm{H}^{j}(X,\mathbb{Q}) \rightarrow \mathbb{Q}(-j)$ preserving dRB structures while also polarizing the Hodge structure on $\mathrm{H}^{j}(X,\mathbb{Q})$.
\end{lemma}

\begin{proofsketch}
The same proof strategy of Proposition 22 of \cite{charles2011notes} can be applied. Namely, fixing a very ample line bundle $L$ on $X$, one can write down a polarization form on each component of the Lefschetz decomposition on $\mathrm{H}^{*}(X,\mathbb{Q})$ and combine them together with twists to obtain a polarization form on $\mathrm{H}^{j}(X,\mathbb{Q})$.

Now because the Lefschetz operators on $\mathrm{H}^{*}(X,\mathbb{Q})$ are given by cup product by algebraic classes, they preserve dRB structures. Then by Lemma \ref{inverseisdrb}, their inverses also preserve dRB structures. Then the proof of Lemma 21 of \cite{charles2011notes} can be applied to show that the projection of $\mathrm{H}^{*}(X,\mathbb{Q})$ to each component of the Lefschetz decomposition is a morphism of de Rham-Betti structures. Therefore, the same polarization form constructed previously also preserves dRB structures.
\end{proofsketch}

\begin{lemma}\label{orthogonalprojs}
Given two smooth projective varieties $X$ and $Y$ defined over $\qbar$ of dimension $n$ and $n'$ respectively and a morphism $$f: \mathrm{H}^{j}(X,\mathbb{Q}) \rightarrow \mathrm{H}^{j}(Y,\mathbb{Q})$$ which preserves both de Rham-Betti and Hodge structures. Then we have orthogonal decompositions with respect to the polarization forms constructed in Lemma \ref{polarizationdrb} $$\mathrm{H}^{j}(X,\mathbb{Q})=\mathrm{Ker}(f) \oplus \mathrm{Ker}(f)^{\perp}$$ and $$\mathrm{H}^{j}(Y,\mathbb{Q})=\mathrm{Im}(f) \oplus \mathrm{Im}(f)^{\perp}$$ Moreover, every summand in above decompositions is a de Rham-Betti substructure.
\end{lemma}
\begin{proof} Because $f$ preserves both the dRB structure and the Hodge structures, both  $\mathrm{Ker}(f)$ and $\mathrm{Im}(f)$ are sub-dRB structures and sub-Hodge structures simultaneously. Now the polarization form $$\mathcal{P}_{X}: \mathrm{H}^{j}(X,\mathbb{Q}) \times \mathrm{H}^{j}(X,\mathbb{Q}) \rightarrow \mathbb{Q}(-j)$$ constructed from Lemma \ref{polarizationdrb} preserves both dRB structure and Hodge structures. Since $\mathrm{Ker}(f)$ is a sub-Hodge structure, $\mathcal{P}_{X}|_{\mathrm{Ker}(f)}$ is also a polarization form (see Lemma 7.26 of \cite{voisin2003hodge}). In particular it is non-degenerate. Therefore, we have the orthogonal decomposition $$\mathrm{H}^{j}(X,\mathbb{Q})=\mathrm{Ker}(f) \oplus \mathrm{Ker}(f)^{\perp}$$ as dRB and Hodge structures. The case of $\mathrm{Im}(f)$ follows similarly. 
\end{proof}
\begin{remark}
It is crucial that $\mathrm{Ker}(f)$ is a sub-Hodge structure and sub-dRB structure simultaneously. If $\mathrm{Ker}(f)$ is only a sub de Rham-Betti structure, then there is no reason for the bilinear form $\mathcal{P}_{X}|_{\mathrm{Ker}(f)}$ on $\mathrm{Ker}(f)$ to be non-degenerate.
\end{remark}

An important corollary of the above linear algebraic computation is the following.
\begin{corollary} \label{drbclassliftable} 
 Suppose we are given two smooth projective varieties $X$ and $Y$ defined over $\qbar$ and a morphism $$f: \mathrm{H}^{j}(X,\mathbb{Q})\rightarrow \mathrm{H}^{j}(Y,\mathbb{Q})$$ which preserves Hodge structures and de Rham-Betti structures simultaneously. Then for any $n\in\mathbb{Z}$ such that the set of de Rham-Betti class in $\mathrm{H}_{\mathrm{dRB}}^{j}(Y,\mathbb{Q})\otimes\mathbb{Q}_{\mathrm{dRB}}(n)$ is not equal to $\{0\}$ and any $\beta$ a dRB class in $\mathrm{H}_{\mathrm{dRB}}^{j}(Y,\mathbb{Q})\otimes\mathbb{Q}_{\mathrm{dRB}}(n)$ which lies in the image of $f\otimes 1$, there exists a dRB class $\alpha$ in $\mathrm{H}_{\mathrm{dRB}}^{j}(X,\mathbb{Q})\otimes\mathbb{Q}_{\mathrm{dRB}}(n)$ which satisfies $(f\otimes 1)(\alpha)=\beta$.
\end{corollary}

\begin{proof}
Again the proof will be very similar in nature to the proof of Corollary 24 in \cite{charles2011notes}. Consider the composition of morphisms $$g: \mathrm{Ker}(f)^{\perp} \rightarrow \mathrm{H}^{j}(X,\mathbb{Q}) \xrightarrow{f} \mathrm{Im}(f)$$ which preserves the dRB structure by Lemma \ref{orthogonalprojs} and is an injective morphism of vector spaces. Then by dimension consideration, $g$ is also an isomorphism of vector spaces, and therefore $$g^{-1}\otimes 1: \mathrm{Im}(f)\otimes\mathbb{Q}_{\mathrm{dRB}}(n)\rightarrow  \mathrm{Ker}(f)^{\perp}\otimes\mathbb{Q}_{\mathrm{dRB}}(n)$$ preserves dRB structures by Lemma \ref{inverseisdrb}. Hence $\alpha=(g^{-1}\otimes 1)(\beta)$ is a dRB class in $\mathrm{Ker}(f)^{\perp}\otimes\mathbb{Q}_{\mathrm{dRB}}(n) \subset \mathrm{H}^{j}(X,\mathbb{Q})\otimes\mathbb{Q}_{\mathrm{dRB}}(n)$. But by construction, we have $$(f\otimes 1)(\alpha)=\beta$$ and therefore we win.
     
\end{proof}

\begin{corollary}\label{liftablewhenquasiproj}
    Let $\iota:Y_{0} \rightarrow Y$ be a closed immersion of algebraic varieties defined over $\qbar$ where $Y_{0}$ is a smooth projective variety and $Y$ is a smooth quasi-projective variety. Denote by $h:Y \rightarrow \overline{Y}$ the smooth compactification of $Y$, where $\overline{Y}$ is a smooth projective variety over $\qbar$ as well. Then we denote the following composition by $$f: \mathrm{H}^{j}(\overline{Y},\mathbb{Q}) \xrightarrow{h^{*}} \mathrm{H}^{j}(Y,\mathbb{Q}) \xrightarrow{\iota^{*}} \mathrm{H}^{j}(Y_{0},\mathbb{Q})$$ If $\beta$ is a dRB class with respect to the dRB structure $\mathrm{H}_{\mathrm{dRB}}^{j}(Y_{0},\mathbb{Q})\otimes\mathbb{Q}_{\mathrm{dRB}}(n)$ where $n\in\mathbb{Z}$ and $\beta$ lies in the image of $f\otimes 1$, then there exists a dRB class $\alpha$ in  $\mathrm{H}_{\mathrm{dRB}}^{j}(\overline{Y},\mathbb{Q})\otimes\mathbb{Q}_{\mathrm{dRB}}(n)$ such that $(f\otimes 1)(\alpha)=\beta$.
\end{corollary}

\begin{proof}
  Note that by Page 22 and (1.4) and (1.5) of \cite{bost2016some} there is still a Grothendieck comparison isomorphism for a smooth quasi-projective variety $Y$ defined over $\qbar$ $$\mathrm{H}^{j}(Y,\mathbb{Q}) \otimes_{\mathbb{Q}} \mathbb{C} \cong \mathbb{H}^{j}(Y^{\mathrm{an}},\Omega^{\bullet}_{Y^{\mathrm{an}}}) \cong_{\mathrm{GAGA}} \mathbb{H}^{j}(Y_{\mathbb{C}},\Omega^{\bullet}_{Y_{\mathbb{C}}}) \cong \mathbb{H}^{j}(Y,\Omega^{\bullet}_{Y}) \otimes \mathbb{C}$$ Hence we have a natural de Rham-Betti structure on the $j$th singular and de Rham cohomology of $Y$
  $$\mathrm{H}^{j}_{\mathrm{dRB}}(Y,\mathbb{Q})=(\mathrm{H}^{j}(Y^{\mathrm{an}},\mathbb{Q}), \mathbb{H}^{j}(Y,\Omega^{\bullet}_{Y}),\rho_{m,Y})$$ Moreover the closed immersion $\iota: Y_{0} \rightarrow Y$ induces a morphism of dRB structures $$\iota^{*}: (\mathrm{H}^{j}(Y_{0},\mathbb{Q}),\mathbb{H}^{j}(Y_{0},\Omega^{\bullet}_{Y_{0}}), \rho_{m,Y_{0}}) \rightarrow (\mathrm{H}^{j}(Y,\mathbb{Q}),\mathbb{H}^{j}(Y,\Omega^{\bullet}_{Y}), \rho_{m,Y})$$ It is well known that $\mathrm{H}^{j}(Y,\mathbb{Q})$ admits a mixed Hodge structures and pulling back along the closed immersion $\iota$ also induces a morphism of mixed Hodge structures (for example see \cite{peters2008mixed}).

  Similarly the open immersion $h: Y \rightarrow \overline{Y}$ induces a morphism $h^{*}$ of dRB structures and mixed Hodge structures on the $j$th singular and de Rham cohomology groups.

Therefore the composition $\iota^{*} \circ h^{*}$ is simultaneously a morphism of dRB structures and Hodge structures. Moreover, $Y_0$ and $\overline{Y}$ are smooth projective varieties. Hence we can use Corollary \ref{drbclassliftable} to conclude.
\end{proof}

\begin{proof}[Proof of Proposition \ref{mainprincipleB}]
Let $\alpha_{\mathrm{B}}$ be a non-zero global section of a constant local system $L_{\alpha}$ in $\mathrm{ob}\langle\mathbb{V}_{B}\rangle^{\otimes}$. By Remark \ref{subobjectsuffices}, such $L_{\alpha}$ is a subobject in the tensor category, i.e. there exists a finite indexing set $I$ such that \begin{equation}\label{ithcomponent}
L_{\alpha} \xhookrightarrow{}  \bigoplus_{n_{i},m_{i} \in \mathbb{Z}_{\geq0},i\in I}\mathbb{V}_{\mathrm{B}}^{\otimes n_{i}}\otimes\mathbb{V}_{\mathrm{B}}^{*\otimes m_{i}}\end{equation} Under this decomposition, denote by $L_{\alpha}^{i}$ and $\alpha_{\mathrm{B}}^{i}$ the image of $L_{\alpha}$ and $\alpha_{\mathrm{B}}$ under the projection to the $i$th component of the right hand side of above. Then the idea of the proof is to explicitly write down the geometric interpretations of $\mathbb{V}_{\mathrm{B},u_{0}}^{\otimes n_{i}}\otimes\mathbb{V}_{\mathrm{B},u_{0}}^{*\otimes m_{i}}$ so that Corollary \ref{liftablewhenquasiproj} can be applied.

Consider products of the original fibrations 
$$f^{n_{i}+m_{i}}: Y^{n_{i}+m_{i}} \rightarrow U$$
  Recall that $\mathbb{V}_{\mathrm{B}}=\mathrm{R}^{j}f_{*}\underline{\mathbb{Q}}_{Y}$ and its dual $\mathbb{V}_{\mathrm{B}}^{*}$ is isomorphic to $\mathrm{R}^{2n-j}f_{*}\underline{\mathbb{Q}}_{Y}$ as local systems on $\analyticu$ by the family version of Poincaré  duality.

  Then by the family version of Künneth formula we have $$L_{\alpha}^{i} \xhookrightarrow{}  \mathbb{V}_{\mathrm{B}}^{\otimes n_{i}}\otimes\mathbb{V}_{\mathrm{B}}^{*\otimes m_{i}} \xhookrightarrow{} \mathrm{R}^{jn_{i}+(2n-j)m_{i}}f^{n_{i}+m_{i}}_{*}\underline{\mathbb{Q}}_{Y^{n_{i}+m_{i}}}$$ Denote by $\alpha_{\mathrm{B},u_{0}}^{i}$ the stalk of the global section $\alpha_{\mathrm{B}}^{i}$ at $u_{0}$. Then we have 
  $$\alpha_{\mathrm{B},u_{0}}^{i} \in \mathrm{H}^{jn_{i}+(2n-j)m_{i}}(Y_{u_{0}}^{n_{i}+m_{i}},\mathbb{Q})$$ By assumption, $\alpha_{\mathrm{B},u_{0}}^{i}$ is monodromy invariant with respect to the family $f^{n_{i}+m_{i}}: Y^{n_{i}+m_{i}} \rightarrow U$. Denote the product of closed immersions of algebraic varieties $Y_{u_0}^{n_{i}+m_{i}} \rightarrow Y^{n_{i}+m_{i}}$ by $\iota_{u_0}$. Then by Deligne's invariant cycle theorem (for example see Theorem 4.18 \cite{voisin2003hodge}), we have the surjection of vector spaces $$\iota_{u_{0}}^{*}: \mathrm{H}^{jn_{i}+(2n-j)m_{i}}(Y^{n_{i}+m_{i}},\mathbb{Q}) \twoheadrightarrow \mathrm{H}^{jn_{i}+(2n-j)m_{i}}(Y_{u_{0}}^{n_{i}+m_{i}},\mathbb{Q})^{\mathrm{Mon}}$$ Moreover, if we denote by $h$ the product of open immersions $Y^{n_{i}+m_{i}} \rightarrow \overline{Y}^{n_{i}+m_{i}}$, then by Corollary 4.22 of \cite{peters2008mixed}, we have \begin{align*}
  \iota_{u_{0}}^{*}(\mathrm{H}^{jn_{i}+(2n-j)m_{i}}(Y^{n_{i}+m_{i}},\mathbb{Q}))&=(h \circ \iota_{u_{0}})^{*}(\mathrm{H}^{jn_{i}+(2n-j)m_{i}}(\overline{Y}^{n_{i}+m_{i}},\mathbb{Q}))\\&=\mathrm{H}^{jn_{i}+(2n-j)m_{i}}(Y_{u_{0}}^{n_{i}+m_{i}},\mathbb{Q})^{\mathrm{Mon}} \end{align*} Now we turn to the de Rham-Betti side of the picture. By the assumption in the proposition, $\alpha_{\mathrm{B},u_{0}}^{i}$ is a de Rham-Betti class in $\mathcal{V}|_{u_0}^{\otimes n_{i}}\otimes\mathcal{V}|_{u_0}^{*\otimes m_{i}}$. By Poincaré duality, we have the isomorphism of dRB structures: $$\mathrm{H}^{2n-j}_{\mathrm{dRB}}(Y_{u_{0}},\mathbb{Q})\otimes\mathbb{Q}_{\mathrm{dRB}}(n)\cong\mathrm{H}^{j}_{\mathrm{dRB}}(Y_{u_{0}},\mathbb{Q})^{*}$$ Hence $\alpha_{\mathrm{B},u_{0}}^{i}$ is a dRB class in $$\mathcal{V}|_{u_0}^{\otimes n_{i}}\otimes\mathcal{V}|_{u_0}^{*\otimes m_{i}}\subset\mathrm{H}^{jn_{i}+(2n-j)m_{i}}_{\mathrm{dRB}}(Y_{u_{0}}^{n_{i}+m_{i}},\mathbb{Q})\otimes\mathbb{Q}_{\mathrm{dRB}}^{\otimes m_{i}}(n)$$ Moreover, both $\iota_{u_0}^{*}$ and $h^{*}$ preserve Hodge and de Rham-Betti structures, hence so does their composition $\iota_{u_{0}}^{*} \circ h^{*}$. Since both $Y_{u_0}$ and $\overline{Y}$ are smooth projective varieties defined over $\qbar$, we can apply Corollary \ref{liftablewhenquasiproj} to conclude that there exists a dRB class $$\overline{\alpha^{i}} \in \mathrm{H}^{jn_{i}+(2n-j)m_{i}}_{\mathrm{dRB}}(\overline{Y}^{n_{i}+m_{i}},\mathbb{Q})\otimes\mathbb{Q}_{\mathrm{dRB}}^{\otimes m_{i}}(n)$$ which restricts to $\alpha_{\mathrm{B},u_{0}}^{i}$. But for any $\qbar$-point $u_{t}$ of $U$, the restriction of $\overline{\alpha^{i}}$ at $u_{t}$ i.e. $\iota_{u_{t}}^{*}(\overline{\alpha^{i}})$ is a dRB class with respect to the dRB structure $$\mathrm{H}^{jn_{i}+(2n-j)m_{i}}_{\mathrm{dRB}}(Y_{u_{t}}^{n_{i}+m_{i}},\mathbb{Q})\otimes\mathbb{Q}_{\mathrm{dRB}}^{\otimes m_{i}}(n)$$ Note that the restriction of $\overline{\alpha^{i}}$ at $u_{t}$ can be identified with the stalk of $\alpha_{\mathrm{B}}^{i}$ at $u_{t}$ i.e. $$\iota_{u_{t}}^{*}(\overline{\alpha^{i}})=\alpha_{\mathrm{B},u_{t}}^{i}$$ Therefore 
  the stalk of $\alpha_{\mathrm{B}}^{i}$ at each $\qbar$-point $u_{t}$ is a dRB class with respect to the de Rham-Betti structure $\mathcal{V}|_{u_t}^{\otimes n_{i}}\otimes\mathcal{V}|_{u_t}^{*\otimes m_{i}}\subset\mathrm{H}^{jn_{i}+(2n-j)m_{i}}_{\mathrm{dRB}}(Y_{u_{t}}^{n_{i}+m_{i}},\mathbb{Q})\otimes\mathbb{Q}_{\mathrm{dRB}}^{\otimes m_{i}}(n)$. The same argument can be applied to each $i$ from the indexing set $I$. Hence $\alpha_{\mathrm{B}}$ is indeed a de Rham-Betti class when evaluated at each $\qbar$-point of $U$.
\end{proof}

\begin{proof}[Proof of Proposition \ref{principleB}]
By the assumption that $\mathcal{G}_{\mathrm{dRB}}(\mathcal{V})$ is a reductive group, $\langle\mathcal{V}\rangle^{\otimes}$ is a semisimple category. Hence every subquotient object in $\langle\mathcal{V}\rangle^{\otimes}$ is a subobject. Therefore a fixed tensor of $\mathcal{G}_{\mathrm{dRB}}(\mathcal{V})\subset \mathrm{GL}(\mathbb{V}_{\mathrm{B},u_0})$ is a an element of $\bigoplus_{n_{i},m_{i} \in \mathbb{Z}_{\geq0},i\in I} \mathrm{H}^{j}(Y_{u_{0}},\mathbb{Q})^{\otimes n_{i}}\otimes \mathrm{H}^{j}(Y_{u_{0}},\mathbb{Q})^{*\otimes m_{i}}$ invariant under the induced action of $\mathcal{G}_{\mathrm{dRB}}(\mathcal{V})$ where $I$ is a finite indexing set. In particular it is a monodromy invariant de Rham-Betti class by Corollary \ref{twoimmmersions}.  

Now we use Proposition \ref{mainprincipleB} (more precisely, its proof) to show the inverse. Suppose we are given a monodromy invariant class \begin{equation}\label{initialdrbclass}\alpha_{\mathrm{B}}\in\bigoplus_{n_{i},m_{i} \in \mathbb{Z}_{\geq0},i\in I} \mathrm{H}^{j}(Y_{u_{0}},\mathbb{Q})^{\otimes n_{i}}\otimes \mathrm{H}^{j}(Y_{u_{0}},\mathbb{Q})^{*\otimes m_{i}}\end{equation} which is furthermore a de Rham-Betti class. In Setup \ref{keysetup} and subsequently in Lemma \ref{rhcorrespondence}, we have constructed a flat section $$s_{\alpha} \in \bigoplus_{n_{i},m_{i} \in \mathbb{Z}_{\geq0},i\in I}\mathrm{H}^{0}(U_{\mathbb{C}},\mathcal{V}_{\mathrm{dR},\mathbb{C}}^{\otimes n_{i}}\otimes \mathcal{V}_{\mathrm{dR},\mathbb{C}}^{*\otimes m_{i}})^{\nabla_{\mathbb{C}}}$$ whose fiber at $u_0$ after applying the inverse comparison isomorphism is precisely $\alpha_{\mathrm{B}}$. 
Then we need to show that $s_{\alpha}$ descends to a flat $\qbar$-algebraic section 
\begin{equation}\label{targetsection}t_{\alpha} \in \bigoplus_{n_{i},m_{i} \in \mathbb{Z}_{\geq0},i\in I}\mathrm{H}^{0}(U,\mathcal{V}_{\mathrm{dR}}^{\otimes n_{i}}\otimes \mathcal{V}_{\mathrm{dR}}^{*\otimes m_{i}})^{\nabla}\end{equation}
It suffices to show that $i$th component of $s_{\alpha}$ descends to $\qbar$ for each $i\in I$. Denote the $i$th component of formula (\ref{initialdrbclass}) by $\alpha_{\mathrm{B}}^{i}$.

Consider the following diagram $$\begin{tikzcd}
\overline{Y}^{n_{i}+m_{i}} & Y^{n_{i}+m_{i}} \arrow[d] \arrow[l, "h"'] & Y_{u_0}^{n_{i}+m_{i}} \arrow[d] \arrow[l, "\iota_{u_0}"'] \\
                           & U                                         & u_{0} \arrow[l]                                      
\end{tikzcd}$$ Then by the argument (i.e. Künneth formula and Poincaré  duality) in Proposition \ref{principleB}, $\alpha_{\mathrm{B}}^{i}$ is a dRB class in \begin{equation*}\begin{split}\mathrm{H}^{j}_{\mathrm{dRB}}(Y_{u_0},\mathbb{Q})^{\otimes n_{i}}\otimes\mathrm{H}^{j}_{\mathrm{dRB}}(Y_{u_0},\mathbb{Q})^{*\otimes m_{i}}&\subset\mathrm{H}^{jn_{i}}_{\mathrm{dRB}}(Y_{u_0}^{n_{i}},\mathbb{Q})\otimes\mathrm{H}^{jm_{i}}_{\mathrm{dRB}}(Y_{u_0}^{m_{i}},\mathbb{Q})^{*}\\
&\subset\mathrm{H}^{jn_{i}+(2n-j)m_{i}}_{\mathrm{dRB}}(Y_{u_0}^{n_{i}+m_{i}},\mathbb{Q})\otimes\mathbb{Q}_{\mathrm{dRB}}^{\otimes m_i}(n)\end{split}\end{equation*} which is furthermore monodromy invariant. Then by Corollary \ref{liftablewhenquasiproj}, there exists a dRB class $$\overline{\alpha^{i}} \in \mathrm{H}^{j}_{\mathrm{dRB}}(\overline{Y},\mathbb{Q})^{\otimes n_{i}}\otimes\mathrm{H}^{j}_{\mathrm{dRB}}(\overline{Y},\mathbb{Q})^{*\otimes m_{i}}\subset \mathrm{H}^{jn_{i}+(2n-j)m_{i}}_{\mathrm{dRB}}(\overline{Y}^{n_{i}+m_{i}},\mathbb{Q})\otimes\mathbb{Q}_{\mathrm{dRB}}^{\otimes m_i}(n)$$ such that $$\iota_{u_{0}}^{*}\circ h^{*}(\overline{\alpha^{i}})=\alpha_{\mathrm{B}}^{i}$$ Moreover since the open immersion $h$ is also a morphism of dRB structures, we have that $$h^{*}(\overline{\alpha^{i}}) \in \mathrm{H}^{j}_{\mathrm{dRB}}(Y,\mathbb{Q})^{\otimes n_{i}}\otimes\mathrm{H}^{j}_{\mathrm{dRB}}(Y,\mathbb{Q})^{*\otimes m_{i}}$$ is a dRB class. Consequently under the comparison isomorphism of the de Rham-Betti structure we obtain that $$h^{*}(\overline{\alpha^{i}})_{\mathrm{dR}} \in \mathrm{H}_{\mathrm{dR}}^{j}(Y/\qbar)^{\otimes n_{i}}\otimes \mathrm{H}_{\mathrm{dR}}^{j}(Y/\qbar)^{*\otimes m_{i}}$$ Then under the natural map $\Gamma$ induced by the restriction map $\mathrm{H}^{j}_{\mathrm{dR}}(Y/\qbar)\rightarrow \mathrm{H}^{0}(U,\mathcal{V}_{\mathrm{dR}})$ given by the Leray spectral sequence (see Section 4.3.1 of \cite{voisin2003hodge} for example) we have that $$\Gamma(h^{*}(\overline{\alpha^{i}})_{\mathrm{dR}}) \in  \mathrm{H}^{0}(U,\mathcal{V}_{\mathrm{dR}}^{\otimes n_{i}}\otimes \mathcal{V}_{\mathrm{dR}}^{*\otimes m_{i}})$$ and this is the desired $t_{\alpha}$ (see formula (\ref{targetsection})) we are looking for. 
\end{proof}

\begin{remark}
In Proposition \ref{principleB}, it is crucial that $\mathcal{G}_{\mathrm{dRB}}(\mathcal{V})$ is a reductive group. Otherwise, recall from formula (\ref{idenityobjectdiagram}), fixed tensors in $\langle\mathcal{V}\rangle^{\otimes}$ are one-dimensional subquotient de Rham-Betti systems of the form $$\begin{tikzcd}
\mathcal{M} \arrow[d, two heads] \arrow[r, hook] & \bigoplus_{n_{i}, m_{i}\in \mathbb{Z}_{\geq0}} \mathcal{V}^{\otimes n_{i}}\otimes \mathcal{V}^{*\otimes m_{i}} \\
\mathcal{L}_{\alpha}                            &                                                              
\end{tikzcd}$$ However, suppose we are given a one-dimensional subquotient whose arrows are $\gdrbmath(\mathcal{V}|_{u_0})$ and $G_{\mathrm{mon}}$-equivariant i.e. $$\begin{tikzcd}
M \arrow[d, two heads] \arrow[r, hook] & \bigoplus_{n_{i},m_{i} \in \mathbb{Z}_{\geq0}} \mathrm{H}^{j}(Y_{u_{0}},\mathbb{Q})^{\otimes n_{i}}\otimes \mathrm{H}^{j}(Y_{u_{0}},\mathbb{Q})^{*\otimes m_{i}} \\
L_{\alpha}                                           &                                                                    
\end{tikzcd}$$ where $M$ is a dRB \textit{substructure} stable under the monodromy action. Then from Proposition \ref{mainprincipleB}, it is not immediately clear that $M$ can be extended to a de Rham-Betti \textit{subsystem}.
\end{remark}

We end this section with the following question.
\begin{question}
    Given a de Rham-Betti system $\mathcal{V}$ satisfying the geometric constraint from Assumption \ref{keyassumption}, does there exist a $\qbar$-point $u_{0}$ of $U$ such that $\mathcal{G}_{\mathrm{dRB}}(\mathcal{V})$ is equal to $\gdrbmath(\mathcal{V}|_{u_{0}})$? 
\end{question}
\printbibliography[
heading=bibintoc,
title={Bibliography}
]
\end{document}